\documentclass{amsart}

\usepackage{amsmath,amssymb,amsfonts,amssymb,amsthm}

\usepackage{verbatim}
\usepackage[usenames]{color}
\usepackage{hyperref}
\usepackage{url}
\usepackage{tikz,tikz-qtree,ifthen,cancel}
\usepackage{array,tikz-qtree,ifthen,cancel}
\usepackage{graphicx}
\usepackage{adjustbox}
\usepackage{amsthm,graphicx,tikz,appendix,tikz-qtree,ifthen,cancel}
\usetikzlibrary{calc,shapes,patterns,positioning}

\newtheorem{thm}{Theorem}[section]

\newtheorem{mainthm}[thm]{Main Theorem}
\newtheorem{lem}[thm]{Lemma}

\newtheorem{observation}[thm]{Observation}
\newtheorem*{thm7.3}{Theorem \ref{thm.MillikenSWP}}
\newtheorem*{thmR2}{Theorem \ref{thm.mainRamsey}}
\newtheorem*{thmfinalthm}{Theorem \ref{finalthm}}
\newtheorem{claim}{Claim}
\newtheorem{fact}[thm]{Fact}

\theoremstyle{remark}
\newtheorem{rem}[thm]{Remark}

\theoremstyle{definition}
\newtheorem{defn}[thm]{Definition}
\newtheorem{conv}[thm]{Convention}

\newtheorem{notation}[thm]{Notation}
\newtheorem{term}[thm]{Terminology}
\newtheorem{assumption}[thm]{Assumption}
\newtheorem{example}[thm]{Example}
\newtheorem{question}[thm]{Question}

\theoremstyle{remark}

\newcommand{\al}{\alpha}
\newcommand{\om}{\omega}

\newcommand{\sse}{\subseteq}
\newcommand{\contains}{\supseteq}
\newcommand{\forces}{\Vdash}

\DeclareMathOperator{\SP}{SP}

\DeclareMathOperator{\ran}{ran}
\DeclareMathOperator{\dom}{dom}

\DeclareMathOperator{\Spl}{Spl}

\DeclareMathOperator{\Sim}{Sim}
\DeclareMathOperator{\Ext}{Ext}

\DeclareMathOperator{\splitpred}{splitpred}
\DeclareMathOperator{\MPE}{MPE}

\newcommand{\re}{\restriction}
\newcommand{\bP}{\mathbb{P}}

\newcommand{\bD}{\mathbb{D}}

\newcommand{\bN}{\mathbb{N}}
\newcommand{\bQ}{\mathbb{Q}}

\newcommand{\bT}{\mathbb{T}}
\newcommand{\bS}{\mathbb{S}}

\newcommand{\G}{\mathrm{G}}
\newcommand{\HH}{\mathrm{H}}

\newcommand{\A}{\mathrm{A}}
\newcommand{\B}{\mathrm{B}}
\newcommand{\C}{\mathrm{C}}
\newcommand{\D}{\mathrm{D}}
\newcommand{\E}{\mathrm{E}}

\newcommand{\Seq}{\mathcal{S}}

\newcommand{\ssim}{\stackrel{s}{\sim}}

\newcommand{\sssim}{\stackrel{ss}{\sim}}

\newcommand{\ra}{\rightarrow}

\newcommand{\llra}{\longleftrightarrow}
\newcommand{\Llra}{\Longleftrightarrow}
\newcommand{\lgl}{\langle}
\newcommand{\rgl}{\rangle}

\newcommand{\Erdos}{Erd{\H{o}}s}
\newcommand{\Fraisse}{Fra{\"{i}}ss{\'{e}}}
\newcommand{\Lauchli}{L{\"{a}}uchli}
\newcommand{\Masulovic}{Ma{\v{s}}ulovi{\'{c}}}
\newcommand{\Nesetril}{Ne{\v{s}}et{\v{r}}il}

\newcommand{\Rodl}{R{\"{o}}dl}

\newcommand{\noprint}[1]{\relax}

\title[The Ramsey theory of  Henson graphs]{The Ramsey theory of
Henson graphs}

\author{Natasha  Dobrinen}
\address{Department of Mathematics\\
 University of Denver \\
C.M.\ Knudson Hall, Room 300\\
2390 S.\ York St.\\ Denver, CO \ 80208 U.S.A.}
\email{natasha.dobrinen@du.edu}
  \urladdr{\url{http://cs.du.edu/~ndobrine}}
\thanks{This research was supported  by National Science Foundation Grants DMS-1600781 and DMS-1901753}


\begin{document}

\begin{abstract}
Analogues of Ramsey's Theorem for infinite structures
such as the rationals or the Rado graph have been known for some time.
In this context, one looks for optimal bounds, called degrees,  for the number of colors in an isomorphic substructure rather than one color, as that is often impossible. 
Such  theorems for 
Henson  graphs however remained elusive, due to lack of techniques for handling forbidden cliques. 
 Building on the author's recent result for 
 the triangle-free Henson graph, 
we  prove
 that for each $k\ge 4$,
  the $k$-clique-free Henson graph
 has finite big Ramsey degrees, 
the appropriate    analogue of Ramsey's Theorem.

We  develop  a   method for coding copies of  Henson graphs into a new class of trees, called strong coding trees,  and
prove   Ramsey
 theorems  for  these  trees which are applied to deduce finite big Ramsey degrees.
  The approach here      provides  a general methodology   opening
  further study of big Ramsey degrees for
   ultrahomogeneous structures.
 The results have bearing on topological dynamics via  work of Kechris, Pestov, and Todorcevic  
 and of Zucker.
\end{abstract}

\maketitle

\section{Overview}

A central program of the theory of infinite structures
is to find which structures have  partition properties  resembling  Ramsey's Theorem.
In this context, one colors the copies of a  finite structure $\A$ inside the infinite structure $\mathcal{S}$ into finitely many colors and looks for an
infinite substructure  $\mathcal{S}'$, isomorphic to  $\mathcal{S}$,
in which   the  copies of $\A$ have the same color.
A wide collection of infinite structures have the Ramsey property for colorings of singletons.
However,
even the rationals as a linearly ordered structure do not have the Ramsey property for colorings of pairsets, as seen by Sierpi\'{n}ski's
two-coloring of pairs of rationals  where   each subcopy of the rationals retains both colors on its pairsets.
This leads to the following
 question:
 \begin{question}
Given an infinite structure $\mathcal{S}$ and
a finite substructure $\mathrm{A}$,
is there a  positive integer $T(\mathrm{A}, \mathcal{S})$ such that for any coloring of all copies of $\mathrm{A}$ in $\mathcal{S}$ into finitely many colors,
there is a substructure $\mathcal{S}'$ of $\mathcal{S}$, isomorphic to  $\mathcal{S}$,
in which all copies of $\mathrm{A}$  take no more than $T(\mathrm{A}, \mathcal{S})$ colors?
\end{question}
The smallest such  number $T(\mathrm{A}, \mathcal{S})$, when it exists, is called the
{\em big Ramsey degree} of $\A$ in $\mathcal{S}$.
Research in this area
has gained recent momentum, as it was
 highlighted by
 Kechris, Pestov, and Todorcevic
 in \cite{Kechris/Pestov/Todorcevic05}.
Big Ramsey degrees  have  implications for topological dynamics, as shown in \cite{Kechris/Pestov/Todorcevic05} and  further developed in Zucker's work  \cite{Zucker19}.

In contrast to finite structural Ramsey theory,
the development of Ramsey theory for infinite structures has  progressed quite slowly.
After  Sierpi\'{n}ski's coloring  for pairs of rationals, work of  Laver
and Devlin (see \cite{DevlinThesis}) established the exact big Ramsey degrees for  finite sets of rationals by 1979.
In the mid 1970's,  \Erdos, Hajnal, and Pos\'{a}
began work on the big Ramsey degrees of the Rado graph, establishing
 an analogue of Sierpi\'{n}ski's coloring for edges in  \cite{Erdos/Hajnal/Posa75}.
Building on work of Pouzet and Sauer in \cite{Pouzet/Sauer96}, the   full  Ramsey theory of the Rado graph for colorings of copies of any finite graph
was finally established
 in 2006  in the two papers
  \cite{Sauer06} by Sauer  and \cite{Laflamme/Sauer/Vuksanovic06}  by Laflamme, Sauer, and Vuksanovic.
Around that time, driven  by  the interest  generated by
  \cite{Kechris/Pestov/Todorcevic05}, the Ramsey theory of   other
 ultrahomogeneous structures was established in
  \cite{NVT08} and \cite{Laflamme/NVT/Sauer10}.
 A principal  component in the work in  \cite{DevlinThesis} and \cite{Sauer06} is a Ramsey  theorem for strong trees due to Milliken  \cite{Milliken79}, while
\cite{Laflamme/NVT/Sauer10} depended on the authors' development of a colored
 version  of this theorem.
The lack of similar  means for coding infinite  structures and the lack of   Ramsey theorems  for such  coded structures
 have been the  largest obstacles in the further development of this area, especially for ultrahomogeneous structures with forbidden configurations.
As stated in   Nguyen Van Th\'{e}'s   habilitation   \cite{NVTHabil},
``so far, the lack of tools to represent ultrahomogeneous structures is the major obstacle towards a better understanding of their infinite partition properties."

In this paper, we prove that for each $k\ge 4$, the  $k$-clique-free  Henson graph $\mathcal{H}_k$ has finite big Ramsey degrees, extending work of the author in \cite{DobrinenJML20} for the triangle-free Henson graph.
Given   $k\ge 3$,
the Henson graph $\mathcal{H}_k$ is the universal ultrahomogeneous $K_k$-free graph; that is, the $k$-clique-free analogue of the Rado graph.
The only prior work on the big Ramsey degrees of $\mathcal{H}_k$ for $k\ge 4$  was work of
 El-Zahar and Sauer
   in  \cite{El-Zahar/Sauer89} for vertex colorings in 1989.
In \cite{DobrinenJML20}, we proved that the triangle-free Henson graph has finite big Ramsey degrees.
The work in this paper follows  the general outline in \cite{DobrinenJML20}, but
the extension of Ramsey theory to all Henson graphs
required  expanded ideas, a better understanding of the nature of coding structures with forbidden configurations, and many new lemmas.
 This article
  presents a unified framework for the Ramsey theory of Henson graphs.
We develop new
  techniques for coding copies of $\mathcal{H}_k$ via {\em strong $\mathcal{H}_k$-coding trees} and  prove    Ramsey
 theorems  for these trees, forming a family of   Milliken-style  theorems.
  The approach here   streamlines  the one  in  \cite{DobrinenJML20} for $\mathcal{H}_3$ and     provides  a general methodology   opening
  further study of big Ramsey degrees for ultrahomogeneous structures with  and without forbidden configurations.


\section{Introduction}

The field  of
Ramsey theory
was established by
 the following celebrated result.

\begin{thm}[Infinite Ramsey Theorem, \cite{Ramsey30}]\label{thm.RamseyInfinite}
Given positive integers $m$ and $j$,
suppose the  collection of  all $m$-element subsets of $\mathbb{N}$
is colored by  $j$ colors.
Then there is an infinite set  $N$ of natural numbers
such that all $m$-element subsets of $N$  have the same color.
\end{thm}

From this,  Ramsey  deduced  the following finite version,  which also can be proved directly.  
Throughout the paper, each natural number is identified with the set of its predecessors.

\begin{thm}[Finite Ramsey Theorem, \cite{Ramsey30}]\label{thm.RamseyFinite}
Given positive integers $m,n,j$ with
$m\le n$, there is an integer $r>n$ such that for any coloring of the $m$-element subsets of $r$ into $j$ colors,
there is a subset  $N\sse r$ of cardinality $n$  such that  all $m$-element subsets of $N$ have the same color.
\end{thm}

In both cases, we say that the coloring is {\em monochromatic} on
 $N$.
Interestingly, Theorem \ref{thm.RamseyFinite}  was motivated by Hilbert's  Entscheidungsproblem: to find a decision procedure deciding which formulas in first order logic are valid.
 Ramsey applied Theorem \ref{thm.RamseyFinite} to prove that  the validity, or lack of it, for
 certain types of formulas in first order logic (those with no existential quantifiers) can be ascertained algorithmically.
Later, Church and Turing  each showed that a general solution to Hilbert's problem is impossible, so Ramsey's success for the class of existential formulas is quite remarkable.
Ever since the inception of Ramsey theory, its connections with
logic have continually spurred  progress in both fields.
This phenomenon occurs once again
  in Sections \ref{sec.5} and \ref{sec.1SPOC}, where methods of logic are used to deduce Ramsey theorems.

Structural Ramsey theory   investigates which structures satisfy versions of  Ramsey's Theorem.
All structures in this paper are 
first order, and  embeddings are as in model theory. 
In this setting,  one tries to find a substructure  isomorphic to some fixed  structure on which the coloring is monochromatic.
Given  structures $\A$ and $\B$,
we write $\A\le \B$ if and only if there is an embedding of $\A$ into $\B$.
A substructure $\A'$ of  $\B$ is  called a  {\em copy} of $\A$  if and only if $\A'$ is the image of some embedding of $\A$ into $\B$.
The collection of all copies of $\A$ in $\B$ is denoted by  ${\B \choose \A}$.
Given structures  $\A,\B,\C$ with $\A\le \B\le \C$
and an integer $j\ge 1$, we write
\begin{equation}
C\ra (B)^A_j
\end{equation}
to mean that for each $c:{\C\choose \A}\ra j$, there is a $\B'\in {\C\choose \B}$ for which $c$ takes only one color on ${\B'\choose \A}$.
A  class $\mathcal{K}$ of finite structures  is said to have  the {\em Ramsey property}
if given  $\A,\B\in\mathcal{K}$ with $\A\le \B$,
  for any integer $j\ge 1$,
there is some  $\C\in\mathcal{K}$ for which $B\le C$ and
$\C\ra (\B)^{\A}_j$.

Some classic examples of classes of  structures
 with the Ramsey property include
finite Boolean algebras (Graham and Rothschild \cite{Graham/Rothschild71}),
finite vector spaces over a finite field (Graham, Leeb, and Rothschild \cite{Graham/Leeb/Rothschild72} and
\cite{Graham/Leeb/Rothschild73}),
 finite  linearly ordered relational structures  (independently, Abramson and Harrington, \cite{Abramson/Harringon78} and \Nesetril\ and \Rodl, \cite{Nesetril/Rodl77}, \cite{Nesetril/Rodl83}),
in particular, the class of finite linearly ordered graphs.
The papers \cite{Nesetril/Rodl77} and \cite{Nesetril/Rodl83} further  proved
 that  all set-systems of  finite linearly ordered relational structures omitting some irreducible substructure have the Ramsey property.
This includes
 the  classes of finite linearly ordered graphs omitting $k$-cliques, denoted $\mathcal{G}_k^{<}$, for each $k\ge 3$.
 \Fraisse\ classes are  natural objects for  structural Ramsey theory investigations, for as shown by \Nesetril, any class with the Ramsey property must satisfy the amalgamation property.
Since \Fraisse\ theory is not central to the proofs  in this article, we refer the interested reader to \cite{Fraisse54} and Section 2 of the more recent \cite{Kechris/Pestov/Todorcevic05} for background; the properties of the specific examples contained in this article will be clear.

Most   classes of finite unordered structures do not have
the Ramsey property.
However, if  equipping
the class with an additional linear order  produces  the Ramsey property, then some
 remnant of it  remains in the unordered reduct.
 This is the idea behind {\em small Ramsey degrees}.
Following notation  in \cite{Kechris/Pestov/Todorcevic05},
given any \Fraisse\ class $\mathcal{K}$ of finite structures,
for  $A\in\mathcal{K}$,
$t(A,\mathcal{K})$ denotes  the smallest number $t$, if it exists, such that
for each $B\in \mathcal{K}$ with $A\le B$ and for each $j\ge 2$,
there is some $C\in\mathcal{K}$ into which $B$ embeds  such that for
 any coloring $c:{C \choose A}\ra j$,
 there is a $B'\in {C \choose B}$ such that the restriction of $c$ to ${B'\choose A}$ takes no more than $t$ colors.
In the arrow notation, this is written as
\begin{equation}
C\ra (B)^A_{j,t(A,\mathcal{K})}.
\end{equation}
A class  $\mathcal{K}$ has  {\em finite (small) Ramsey degrees} if
for each $\A\in\mathcal{K}$ the number
 $t(\A,\mathcal{K})$  exists.
The number $t(\A,\mathcal{K})$ is called the {\em Ramsey degree of $A$} in $\mathcal{K}$ \cite{Fouche98}.
Notice that $\mathcal{K}$ has the Ramsey property if and only if $t(A,\mathcal{K})=1$ for each $A\in\mathcal{K}$.

The  connection between \Fraisse\ classes with finite small Ramsey degrees and ordered expansions is made explicit in Section 10 of \cite{Kechris/Pestov/Todorcevic05},
where it is shown that if an ordered expansion $\mathcal{K}^{<}$ of a \Fraisse\ class $\mathcal{K}$ has the Ramsey property,
then $\mathcal{K}$ has finite small Ramsey degrees. Furthermore, the degree of $\A\in\mathcal{K}$  can be computed from the number of non-isomorphic order expansions it has in $\mathcal{K}^{<}$.
Nguyen Van Th\'{e} has extended this to the more  general  notion of
pre-compact expansions (see \cite{NVTHabil}).
In particular, the
 classes of  finite  (unordered) graphs and finite (unordered) graphs omitting $k$-cliques have   finite small Ramsey degrees.

Continuing this expansion of Ramsey theory leads to investigations of which infinite structures have  properties similar to Theorem \ref{thm.RamseyInfinite}.
Notice that the  infinite homogeneous subset $N\sse\mathbb{N}$ in
 Theorem  \ref{thm.RamseyInfinite}
 is actually isomorphic to
$\mathbb{N}$ as a linearly ordered structure.
Ramsey theory on infinite structures  is concerned with finding  substructures isomorphic to the original infinite  structure in which a given coloring is as simple as possible.
Many infinite structures have been proved to be
 {\em indivisible}:  given a coloring of its single-element  substructures
into finitely many colors, there is an infinite substructure isomorphic to the original structure in which all single-element substructures have the same color.
The natural numbers and the rational numbers as linearly ordered structures are indivisible, the proofs being straightforward.
Similarly, it is folklore that the Rado graph is indivisible, the proof following naturally from the defining properties of this graph.

In contrast,  it took much more effort to prove  the indivisibility of the Henson graphs, and this was achieved first for the triangle-free Henson graphs in
 \cite{Komjath/Rodl86},  and for   all other Henson graphs in \cite{El-Zahar/Sauer89}.
When  one
considers colorings of structures of two or more elements, more complexity begins to emerge.
Even for  the simple structure of  the rationals, there is a coloring of pairsets  into two colors  such that each subset isomorphic to the rationals has pairsets in both colors.
This is  the infamous example of Sierpi\'{n}ski,
and it immediately  leads to the notion of big Ramsey degree.
We take the definition from \cite{Kechris/Pestov/Todorcevic05}, slightly changing some notation.

\begin{defn}[\cite{Kechris/Pestov/Todorcevic05}]\label{defn.bRd}
Given an infinite structure $\mathcal{S}$ and a finite substructure $\A\le \mathcal{S}$,
let $T(\A,\mathcal{S})$ denote the least integer $T\ge 1$, if it exists, such that
given any coloring of ${\mathcal{S}\choose \A}$ into finitely many colors, there is a
 substructure $\mathcal{S}'$ of $\mathcal{S}$, isomorphic to $\mathcal{S}$,  such that ${\mathcal{S}'\choose \A}$ takes no more than $T$ colors.
This may be written succinctly as
\begin{equation}\label{eq.bRd}
\forall j\ge 1,\ \ {\mathcal{S}}\ra ({\mathcal{S}})^{\A}_{j,T(\A,\mathcal{S})}.
\end{equation}
We say that
 $\mathcal{S}$ has {\em finite big Ramsey degrees} if for each finite substructure
$\A\le\mathcal{S}$,
there is an integer $T(A,\mathcal{S})\ge 1$
such that (\ref{eq.bRd}) holds.
\end{defn}

Infinite structures which have been investigated in this light include the rationals
(\cite{DevlinThesis}),
the Rado graph (\cite{Erdos/Hajnal/Posa75},
\cite{Pouzet/Sauer96}, \cite{Sauer06}, \cite{Laflamme/Sauer/Vuksanovic06}),
ultrametric spaces (\cite{NVT08}), the
rationals with  a fixed finite number of  equivalence relations,  and
 the tournaments $\bf{S}(2)$ and $\bf{S}(3)$
(\cite{Laflamme/NVT/Sauer10}),
and recently, the triangle-free   Henson graph (\cite{DobrinenJML20}).
These results will be discussed below.
See \cite{NVTHabil} for an overview of results on big Ramsey degrees obtained prior to 2013.
Each of these structures is {\em ultrahomogeneous}: any isomorphism between two finitely generated substructures can be extended to an automorphism of the infinite structure.
Recently, \Masulovic\ has
 widened the  investigation   of  big Ramsey degrees to    universal structures,  regardless of ultrahomogeneity, and proved transfer principles  in \cite{Masulovic18} from which big Ramsey degrees for one structure may be transferred to other categorically related structures.
More background on the development of Ramsey theory on infinite structures will be given below, but first, we present some recent  motivation
from topological dynamics for further exploration of big Ramsey degrees.

For the specialist, we briefly remark on 
connections between
topological dynamics  and Ramsey theory.
These connections  have been known for some time (see for instance
\cite{Glasner/Weiss02} and 
\cite{Pestov98_2})
and  many of the  previously known phenomena 
were subsumed in the work of
Kechris, Pestov, and Todorcevic in \cite{Kechris/Pestov/Todorcevic05}, where they proved
 several general correspondences between Ramsey theory and topological dynamics.
A \Fraisse\  class which
has at least one relation which is a linear order  is called a {\em \Fraisse\ order class}.
One of the main theorems  in \cite{Kechris/Pestov/Todorcevic05} (Theorem 4.7)
shows that  the extremely amenable
 closed subgroups of the infinite symmetric group $S_{\infty}$ are exactly those of the form Aut$(\mathbf{F})$, where $\mathbf{F}$
is the \Fraisse\ limit (and hence an ultrahomogeneous structure) of some \Fraisse\ order class satisfying the Ramsey property.
Another significant theorem (Theorem 10.8)
provides a way to compute the universal minimal flow of
topological groups which arise as the automorphism groups of \Fraisse\ limits of \Fraisse\ classes with the Ramsey property and the ordering property.
That the ordering property can be relaxed to the expansion property was proved by Nguyen Van Th\'{e} in \cite{NVT13}.

In \cite{Kechris/Pestov/Todorcevic05}, Kechris, Pestov, and Todorcevic also demonstrated how  big Ramsey degrees for \Fraisse\  structures $\mathbf{F}$ are related to
big oscillation degrees for their automorphism groups, Aut$(\mathbf{F})$.
Recently,
Zucker proved  in \cite{Zucker19} that
if  a  \Fraisse\ structure  $\mathbf{F}$ has finite big Ramsey degrees and moreover, $\mathbf{F}$  admits a big Ramsey structure,
then
any big Ramsey flow of Aut$(\mathbf{F})$ is a universal completion flow, and further,  any two universal completion flows are isomorphic.


\subsection{A brief history of big Ramsey degrees and main obstacles to its development}

In contrast to the robust  development  for  finite structures,
results on  the Ramsey theory of  infinite structures  have been meager and the development  quite slow.
Motivated by  Sierpi\'{n}ski's  coloring for pairs of rationals which admits no isomorphic copy in one color,
Laver investigated the more general problem of finding whether or not there are bounds for colorings of $m$-sized  subsets of rationals, for any positive integer $m$.
In the 1970's,
Laver showed that the rationals have finite big Ramsey degrees, finding good upper bounds.
Guided by Laver's results and methods,
Devlin found
the exact bounds  in
\cite{DevlinThesis}.
Interestingly, these numbers  turn  out to be coefficients of the Taylor series for the tangent function.
Around the same time,
 \Erdos, Hajnal, and Pos\'{a}  initiated investigations of the  Rado graph,  the universal ultrahomogeneous graph on countably many vertices.
 In 1975, they proved in \cite{Erdos/Hajnal/Posa75}
 that there is a coloring of edges into two colors in which
each subcopy of the Rado graph has edges in both colors.
That the upper bound for the big Ramsey degree of edges in the Rado graph is exactly two was proved much  later (1996)
 by
 Pouzet and Sauer in \cite{Pouzet/Sauer96}.
The  full  Ramsey theory of the Rado graph
was finally established a decade later  by
 Sauer in  \cite{Sauer06} and by Laflamme, Sauer, and Vuksanovic in \cite{Laflamme/Sauer/Vuksanovic06}.
 Together, these two papers  gave a full description of the big Ramsey degrees of the Rado graph in terms of types of certain trees.
A recursive procedure for computing  these numbers was given by Larson in \cite{Larson08} soon after.
 It is notable that while the big Ramsey degrees for the rationals are computed  by a closed formula,
 there is no closed formula producing the big Ramsey degrees for the Rado graph.

The successful completion of the Ramsey theory of the Rado graph  so soon after the work of
of Kechris, Pestov, and Todorcevic stimulated more interest in Ramsey theory of infinite structures, especially ultrahomogeneous structures,  which  are obtained  as limits of \Fraisse\ classes.
In 2008,
Nguyen Van Th\'{e}
investigated big Ramsey degrees for ultrahomogeneous ultrametric spaces.
Given $S$ a set of positive real numbers,
$\mathcal{U}_S$ denotes the class of all finite ultrametric spaces  with strictly positive distances in $S$.
Its \Fraisse\ limit, denoted
$\mathbf{Q}_S$, is called the {\em Urysohn space associated with} $\mathcal{U}_S$.
In  \cite{NVT08},
Nguyen Van Th\'{e} proved  that
$\mathbf{Q}_S$ has finite big Ramsey degrees whenever $S$ is finite.
Moreover, if $S$ is infinite, then any member of $\mathcal{U}_S$ of size greater than or equal to $2$ does not have a big Ramsey degree.
Soon after this,
 Laflamme, Nguyen Van Th\'{e}, and Sauer proved
in
\cite{Laflamme/NVT/Sauer10}
that enriched structures of the rationals, and two related directed graphs, have  finite big Ramsey degrees.
For each $n\ge 1$,
 $\mathbb{Q}_n$ denotes the structure $(\mathbb{Q},  Q_1,\dots, Q_n,<)$, where  $Q_1,\dots, Q_n$ are disjoint dense subsets of $\mathbb{Q}$  whose union is $\mathbb{Q}$.
This is the \Fraisse\ limit of the class $\mathcal{P}_n$ of all finite linear orders equipped with an equivalence relation with  $n$ many equivalence classes.
Laflamme, Nguyen Van Th\'{e}, and Sauer proved  that
each member of $\mathcal{P}_n$ has a finite big Ramsey degree in $\mathbb{Q}_n$.
Further,
using the bi-definability between $\mathbb{Q}_n$ and the circular directed graphs $\mathbf{S}(n)$, for $n=2,3$,
they proved that
 $\mathbf{S}(2)$ and  $\mathbf{S}(3)$
have finite big Ramsey degrees.
Central to these results is a colored verision of  Milliken's theorem which they  proved in order to deduce the big Ramsey degrees.
For more details,
we recommend the paper \cite{NVTHabil} containing a good overview
of  these results.





A common theme emerges when one looks at the proofs in \cite{DevlinThesis}, \cite{Sauer06}, and
\cite{Laflamme/NVT/Sauer10}.
The first two rely  in an essential way  on Milliken's Theorem,
(see Theorem \ref{thm.Milliken}  in Section \ref{sec.2}).
The third proves a new colored version of Milliken's Theorem and uses it to deduce the results.
The results in \cite{NVT08} use Ramsey's theorem.
This would lead one to conclude or at least conjecture that, aside from Ramsey's Theorem itself,  Milliken's Theorem contains the core combinatorial content of  big Ramsey degree results, at least for binary relational structures.
The lack of useful representations and lack of Milliken-style
 theorems  for infinite structures in general
pose  the two main obstacles to broader  investigations of  big Ramsey degrees.
Upon the author's initial interest in  the Ramsey theory  of the triangle-free Henson graph, these  two challenges
were pointed out to the author   by Todorcevic in 2012 and by Sauer in 2013;
this idea is also expressed in \cite{NVTHabil}, quoted in the Overview.


\subsection{Big Ramsey degrees for Henson graphs:  Main theorem  and prior results}

For  $k\ge 3$, the Henson graph $\mathcal{H}_k$ is the universal ultrahomogeneous $k$-clique free graph.
These graphs  were first constructed by
Henson  in 1971
in \cite{Henson71}.
It was later noticed that $\mathcal{H}_k$ is  isomorphic to the \Fraisse\ limit of the \Fraisse\ class of finite $k$-clique free graphs, $\mathcal{G}_k$.
Henson proved in \cite{Henson71} that
these graphs are {\em weakly indivisible}; given a coloring of the vertices into two colors, either there is a subgraph isomorphic to $\mathcal{H}_k$ in which all vertices have the first  color, or else every finite $k$-clique free graph has a copy whose vertices all have the second color.
 However, the indivisibility of  $\mathcal{H}_k$ took  longer  to prove.
 In 1986,
Komj\'{a}th and \Rodl\ proved
in \cite{Komjath/Rodl86}
 that
given a coloring of the vertices of $\mathcal{H}_3$ into finitely many colors, there is an induced subgraph isomorphic to $\mathcal{H}_3$ in which all vertices have the same color.
A few years later,
El-Zahar and Sauer proved
 more generally  that
$\mathcal{H}_k$  is indivisible for each $k\ge 4$ in  \cite{El-Zahar/Sauer89}.

Prior to the author's work in  \cite{DobrinenJML20},
the only further progress on big Ramsey degrees for Henson graphs was for edge relations on the triangle-free Henson graph.
In 1998,
Sauer proved in \cite{Sauer98} that the big Ramsey degree for edges in $\mathcal{H}_3$ is two.
There, progress stalled for lack of techniques.
This intrigued the author for several reasons.
Sauer and Todorcevic each stated to the author that the heart of the problem was to
find the correct analogue of Milliken's Theorem applicable to Henson graphs.
This would of course
 help solve the problem of whether Henson graphs have finite big Ramsey degrees, but it would moreover  have broader repercussions, as Ramsey theorems for trees are combinatorially strong objects and Milliken's Theorem has already found numerous applications.
 The problem was that it was unclear to experts what such an analogue of Milliken's Theorem should be.

In  \cite{DobrinenJML20}, the author  developed 
an analogue of Milliken's Theorem applicable to the 
 triangle-free Henson graph, and used it to
prove  that $\mathcal{H}_3$ has finite big Ramsey degrees.
In this paper, we provide a unified  approach to the Ramsey theory of the  $k$-clique-free Henson graphs  $\mathcal{H}_k$, for each $k\ge 3$.
This includes the
 development of new types of  trees which  code Henson graphs and  new Milliken-style theorems for these classes of trees which are applied to determine upper bounds for the big Ramsey degrees.
Our  presentation encompasses and streamlines work in \cite{DobrinenJML20} for $\mathcal{H}_3$.
New obstacles  arose   for $k\ge 4$;
these  challenges and their solutions are discussed  as the sections of the paper are delineated below.


\subsection{Outline of paper}

In Section \ref{sec.2}, we introduce  basic definitions
and notation, and review strong trees and
 the
Halpern-\Lauchli\ and Milliken Theorems
(Theorems \ref{thm.HL} and \ref{thm.Milliken})
and their use in obtaining upper bounds for the finite big Ramsey degrees  of the Rado graph.
In Subsection \ref{subsec.HLHarrington}, we include a version of Harrington's forcing proof of the Halpern-\Lauchli\ Theorem.
Many new issues arise due to  $k$-cliques being forbidden in Henson graphs,
but the proof of Theorem \ref{thm.HL}  will at least provide the reader with a toehold into 
the  
 proof strategies   for  Theorem \ref{thm.matrixHL} and Lemma  \ref{lem.Case(c)}, which lead to Theorem \ref{thm.MillikenSWP}, an analogue of Milliken's Theorem.

The article consists of three main phases.
Phase I  occurs in Sections
\ref{sec.3}--\ref{sec.ExtLem}, where we define the tree structures and prove extension lemmas.
In  Section \ref{sec.3},
we introduce
the  notion of {\em strong  $K_k$-free trees},
 analogues of Milliken's strong trees
capable of  coding $k$-clique free graphs.
These trees contain certain distinguished nodes, called  {\em coding nodes},
which  code the vertices of a given graph.
These trees branch maximally,  subject to the constraint of  the coding nodes not coding any $k$-cliques, and thus are the analogues of strong trees for the $K_k$-free setting.
Model-theoretically, such trees are simply coding all  (incomplete)  $1$-types over initial segments of a Henson graph, where the vertices are indexed by the natural numbers.
Although it is not possible to fully develop Ramsey theory on
strong $K_k$-free trees (see Example \ref{ex.badcoloring}),
 they have the  main structural aspects of the trees    for which we will prove analogues of Halpern-\Lauchli\ and Milliken Theorems,
defined in Section \ref{sec.4}.
Section \ref{sec.3} is given for the sole purpose of building the reader's understanding of strong  $\mathcal{H}_k$-coding trees.

Section \ref{sec.4} presents the  definition of {\em strong $\mathcal{H}_k$-coding trees} as  subtrees of a given tree  $\mathbb{T}_k$ which are strongly isomorphic
(Definition \ref{defn.stable})
 to $\mathbb{T}_k$.
The class of these trees is denoted $\mathcal{T}_k$, and these trees are best thought of as skew versions of the trees presented in
Section \ref{sec.3}.
Secondarily, an internal description of the trees in $\mathcal{T}_k$ is given.
This is a new  and simpler approach than the one we took in  \cite{DobrinenJML20} for the triangle-free Henson graph.
An important property  of 
strong $\mathcal{H}_k$-coding
trees is 
the Witnessing Property
(Definition \ref{defn.PrekCrit}).
This means that certain configurations which can give rise to codings of pre-cliques (Definition \ref{defn.newll1s})
 are witnessed by coding nodes.
 The effect is a type of book-keeping to guarantee when finite trees can be extended within a given tree $T\in\mathcal{T}_k$ to another tree in $\mathcal{T}_k$.

Section \ref{sec.ExtLem} presents  Extension Lemmas,
guaranteeing when a given finite tree can be extended to a desired configuration.
For $k\ge 4$, some new  difficulties arise   which did not exist for $k=3$.
The lemmas in this section extend  work in
\cite{DobrinenJML20}, while addressing  new complexities.
Further,  this section includes some new extension lemmas not in \cite{DobrinenJML20}.
These have the
added benefit of streamlining proofs in Section \ref{sec.5}
in which analogues of the Halpern-\Lauchli\ Theorem are proved.

Phase II of the paper takes place in Sections \ref{sec.5} and  \ref{sec.1SPOC},
the goal being to
prove the Milliken-style  Theorem \ref{thm.MillikenSWP}  for colorings of certain finite trees, namely those with the Strict Witnessing Property (see Definition \ref{defn.SWP}).
First, we prove analogues of
 the Halpern-\Lauchli\ Theorem for strong $\mathcal{H}_k$-coding trees
in Theorem
 \ref{thm.matrixHL}.
 The proof builds on ideas  from Harrington's forcing proof of the Halpern-\Lauchli\ Theorem,
 but now we must use
distinct  forcings for two
 cases:  the level sets being colored   have either  a coding node  or else a splitting node.
 The Extension Lemmas from  Section \ref{sec.ExtLem} and the Witnessing Property are essential to  these proofs.
 A new ingredient for $k\ge 4$ is that all pre-$a$-cliques  for $3\le a\le k$ need to be considered and witnessed, not just pre-$k$-cliques.
 It is important to note that
 the technique of forcing is used to conduct  unbounded searches for  finite objects;
  the proof takes place entirely within the standard axioms of set theory and does not involve passing to a generic model.

In Section \ref{sec.1SPOC},
we apply induction and fusion lemmas and
a third  Harrington-style  forcing argument   to  obtain  our  first Ramsey Theorem for colorings of finite trees.

\begin{thm7.3}
Let $k\ge 3$ be given and let  $T\in\mathcal{T}_k$ be a strong  $\mathcal{H}_k$-coding tree and let $A$ be a finite subtree of $T$ satisfying the Strict Witnessing Property.
Then for any coloring of the copies of $A$ in $T$ into finitely many colors,
there is a strong  $\mathcal{H}_k$-coding tree $S$ contained in $T$
such that all  copies of $A$ in $S$  have the same color.
\end{thm7.3}

Phase III of the article takes place in  Sections \ref{sec.MainThm} and \ref{sec.7}.
There,
we  prove a Ramsey theorem  for  finite antichains of coding nodes (Theorem \ref{thm.mainRamsey}), which is then applied to deduce that each Henson graph has finite big Ramsey degrees.
To do this, we  must first  develop a  way to transform
antichains of coding nodes  into  finite trees with the Strict Witnessing Property.
This is accomplished
in Subsections
\ref{sec.squiggle}
   and
 \ref{sec.1color}, where
we develop  the   notions of  incremental  new pre-cliques
and envelopes.
Given any   finite  $K_k$-free graph $\G$, there are only finitely many strict similarity types (Definition
\ref{defn.ssimtype})
of antichains coding $\G$.
Given a
coloring  $c$ of all copies of $\G$ in $\mathcal{H}_k$ into finitely many colors,
we
transfer the coloring to the envelopes of copies of $\G$ in a given strong coding tree $T$.
Then we
 apply the results in previous sections to obtain a strong $\mathcal{H}_k$-coding tree $T'\le T$ in which all  envelopes  encompassing the same strict similarity type  have the same color.
Upon  thinning
to an incremental strong subtree $S\le T'$ while simultaneously choosing a set $W\sse T'$ of {\em witnessing coding nodes},
each  finite
antichain $X$ of nodes in  $S$  is incremental and
has an  envelope comprised of nodes from  $W$  satisfying the Strict Witnessing Property.
Applying Theorem \ref{thm.MillikenSWP} finitely many times, once for each strict similarity type of envelope,
 we obtain our second Ramsey theorem for strong $\mathcal{H}_k$-coding trees, extending the first one.

\begin{thmR2}
[Ramsey Theorem for Strict Similarity Types]
Fix $k\ge 3$.
Let $Z$ be a finite antichain of coding nodes in a strong $\mathcal{H}_k$-coding tree $T$,
and let $h$ be a coloring of all subsets of $T$ which are strictly similar to $Z$ into finitely many colors.
Then  there is an incremental strong $\mathcal{H}_k$-coding tree $S$ contained in $T$ such that all  subsets of $S$
strictly similar to  $Z$  have the same $h$ color.
\end{thmR2}

Upon taking an antichain of coding nodes  $D\sse S$  coding $\mathcal{H}_k$,
the only sets of coding nodes in $D$ coding a given finite $K_k$-free graph $\G$ are automatically antichains which are incremental.
Applying  Theorem \ref{thm.mainRamsey} to the finitely many strict similarity types of antichains  coding $\G$, we arrive at the main theorem.
\begin{thmfinalthm}
The universal  homogeneous $k$-clique free graph has finite big Ramsey degrees.
\end{thmfinalthm}

For each $\G\in\mathcal{G}_k$,
the number $T(\G,\mathcal{G}_k)$
is bounded by  the number   of strict similarity types of antichains of coding nodes coding $\G$.
It is  presently open to see whether this number is in fact the lower bound.
If so, then recent work of Zucker in \cite{Zucker19} would provide an interesting connection with topological dynamics, as the  colorings obtainable from our structures  cohere in the manner necessary  to apply  Zucker's work.

\vskip.1in

\it Acknowledgements. \rm  The author would like to thank 
Andy Zucker for careful reading of a previous version of this paper, pointing  out an oversight which led the author to consider singleton pre-cliques when $k\ge 4$, an issue  that does not arise in the  triangle-free Henson graph. 
Much gratitude  also goes to
 Dana Barto\v{s}ov\'a and Jean Larson for listening to early and later  stages of these results and for useful feedback; 
Norbert Sauer for discussing key aspects of the homogeneous triangle-free graph with me during a research visit in Calgary in 2014;    Stevo Todorcevic for pointing out to her in 2012 that any proof of finite big Ramsey degrees for $\mathcal{H}_3$ would likely involve a new type of Halpern-\Lauchli\ Theorem;
 and to the organizers and participants of the
Ramsey Theory Doccourse at Charles University, Prague, 2016, for their encouragement.
The author would like to thank the referee for  helpful suggestions which improved  this paper. 
Most of all, the author is  grateful for  and much indebted to
  Richard Laver  for
providing  for  her  in 2011 the main  points of
 Harrington's forcing proof of the Halpern-\Lauchli\ Theorem, 
opening  the path  of applying forcing methods  to develop Ramsey theory on  infinite structures.




\section{Background and the Milliken and   Halpern-\Lauchli\ Theorems}\label{sec.2}

This section provides background and sets some notation and terminology for the paper.
We review
Milliken's Ramsey  Theorem for trees and its application to proving that the Rado graph has finite big Ramsey degrees.
Then we discuss why this theorem is not sufficient for handling Henson graphs.
In Subsection \ref{subsec.HLHarrington}, we present
the Halpern-\Lauchli\ Theorem, which is a Ramsey theorem on
products of trees.
This theorem forms the basis for
Milliken's Theorem.
We present a version of
Harrington's forcing proof of the Halpern-\Lauchli\ Theorem in order to offer the reader  an introduction to 
the tack we take in later sections toward proving Theorem \ref{thm.MillikenSWP}.


\subsection{Coding  graphs via finite binary sequences}\label{subsection.treescodinggraphs}

The following notation shall be used throughout.
The  set of all natural numbers  $\{0,1,2,\dots\}$ is denoted by $\mathbb{N}$.
Each natural number  $n$
is equated with  the set  of all natural numbers strictly less than $n$; thus,  $n=\{0,\dots,n-1\}$ and, in particular,
 $0$ is the emptyset.
 It follows that for $m,n\in \mathbb{N}$, $m<n$ if and only if $m\in n$.
For  $n\in \mathbb{N}$,  we shall let  $\Seq_n$  denote the collection of
 all  binary sequences of length $n$.
 Thus, each $s\in \Seq_n$ is a
 function from $n$ into $\{0,1\}$, and
we often  write $s$ as $\lgl s(0), \dots, s(n-1)\rgl$.
For  $i<n$, $s(i)$ denotes the $i$-th {\em value} or {\em entry}  of the sequence $s$.
The {\em length} of $s$, denoted  $|s|$, is the domain of $s$.
Note that $\Seq_0$ contains just the empty sequence, denoted $\lgl \rgl$.

\begin{notation}\label{notn.Seq}
We shall  let   $\Seq$   denote  $\bigcup_{n\in \mathbb{N}}\Seq_n$,
the collection of all binary sequences of finite length.
\end{notation}

For nodes $s,t\in \Seq$, we write
$s\sse t$ if and only if $|s|\le |t|$ and for each $i<|s|$, $s(i)=t(i)$.
In this case, we say that  $s$ is an {\em initial segment} of $t$, or that $t$ {\em extends} $s$.
If $s$ is an initial segment of  $t$ and $|s|<|t|$, then we write $s\subset t$ and  say that $s$ is a {\em proper initial segment} of $t$.
For $i\in \mathbb{N}$, we let $s\re i$ denote
the  sequence  $s$ restricted to domain $i$.
Thus, if $i< |s|$, then $s\re i$ is
 the proper initial segment of $s$ of length $i$, $s\re i=\lgl s(0),\dots, s(i-1)\rgl$;
if $i\ge |s|$, then $s\re i$  equals $s$.
Let $\Seq_{\le n}$ denote $\bigcup_{i\le n}\Seq_i$, the set of all binary sequences of length at most $n$.

In \cite{Erdos/Hajnal/Posa75},
\Erdos, Hajnal and P\'{o}sa
used the edge relation on  a given ordered graph to induce a
 lexicographic order
 on the vertices, which
 they employed
  to solve problems regarding strong embeddings of graphs.
With this lexicographic order,
vertices in a given graph  can
   can be  represented by 
    finite sequences of $0$'s and $1$'s,  a
view made explicit  in \cite{Sauer98} which we review now, using terminology from \cite{Sauer06}.

\begin{defn}[\cite{Sauer06}]\label{defn.pn}
Given nodes $s,t\in \Seq$ with $|s|<|t|$,
 we call the integer $t(|s|)$
  the {\em passing number} of $t$ at $s$.
  Passing numbers represent the edge relation in a graph as follows:
Two vertices  $v,w$  in a  graph  can be
{\em represented}  by nodes  $s,t\in \Seq$ with $|s|<|t|$,  respectively, if
 \begin{equation}
 v\, E\, w \Longleftrightarrow t(|s|)=1.
 \end{equation}
 \end{defn}

Using this correspondence between passing numbers and the edge relation $E$, any graph can be coded by nodes in a binary tree as follows.
Let $\G$ be a  graph with $N$ vertices,
where either $N\in \mathbb{N}$ or $N=\mathbb{N}$,
and
let  $\lgl v_n:n\in N\rgl$ be any enumeration of
the vertices of  $\G$.
Choose any node  $t_0\in \Seq$ to  represent the vertex $v_0$.
For  $n>0$,
given nodes $t_0,\dots,t_{n-1}$  in $\Seq$ coding the vertices $v_0,\dots,v_{n-1}$,
take  $t_n$  to be any node in $\Seq$  such that
 $|t_n|>|t_{n-1}|$ and
 for all $i< n$,
$v_n$ and $ v_i$ have an edge between them if and only if $t_n(|t_i|)=1$.
Then the set of nodes $\{t_n:n\in N\}$  codes the graph $\G$.
For  the purposes of developing the Ramsey theory of Henson graphs, we make the convention that
the nodes in a tree used to code the vertices in a
 graph  have different lengths.
Figure 1.\  shows a  set of nodes $\{t_0,t_1,t_2,t_3\}$  from $\Seq$  coding the four-cycle $\{v_0,v_1,v_2,v_3\}$.

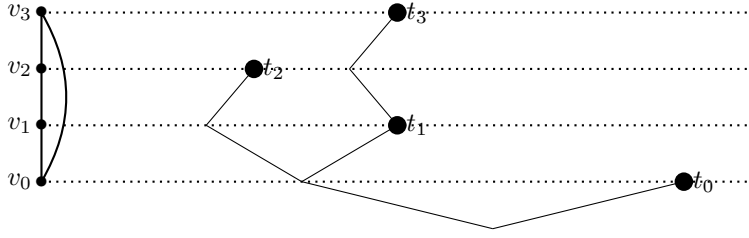
\begin{figure}
\begin{tikzpicture}[grow'=up,scale=.5]
\tikzstyle{level 1}=[sibling distance=4in]
\tikzstyle{level 2}=[sibling distance=2in]
\tikzstyle{level 3}=[sibling distance=1in]
\tikzstyle{level 4}=[sibling distance=1in]
\node {}
 child{
	child{
		child{edge from parent[draw=none]
			child{edge from parent[draw=none]}
			child{edge from parent[draw=none]}}
		child{ coordinate (t2)
			child{edge from parent[draw=none]}
			child {edge from parent[draw=none]}}
		}
	child{ coordinate (t1)
		child{
			child{edge from parent[draw=none]}
			child{coordinate (t3)}}
		child {edge from parent[draw=none]} }}
 child{coordinate (t0)
	child{edge from parent[draw=none]
		child{edge from parent[draw=none]
			child{edge from parent[draw=none]}
			child{edge from parent[draw=none]}
		}
		child{edge from parent[draw=none]
			child{edge from parent[draw=none]}
			child {edge from parent[draw=none]}}
	}
	child {edge from parent[draw=none]}};
		
\node[right] at (t0) {$t_{0}$};
\node[right] at (t1) {$t_{1}$};
\node[right] at (t2) {$t_{2}$};
\node[right] at (t3) {$t_{3}$};

\node[circle, fill=black,inner sep=0pt, minimum size=7pt] at (t0) {};
\node[circle, fill=black,inner sep=0pt, minimum size=7pt] at (t1) {};
\node[circle, fill=black,inner sep=0pt, minimum size=7pt] at (t2) {};
\node[circle, fill=black,inner sep=0pt, minimum size=7pt] at (t3) {};

\draw[thick, dotted] let \p1=(t1) in (-12,\y1) node (v1) {$\bullet$} -- (7,\y1);
\draw[thick, dotted] let \p1=(t2) in (-12,\y1) node (v2) {$\bullet$} -- (7,\y1);
\draw[thick, dotted] let \p1=(t3) in (-12,\y1) node (v3) {$\bullet$} -- (7,\y1);
\draw[thick, dotted] let \p1=(t0) in (-12,\y1) node (v0) {$\bullet$} -- (7,\y1);

\node[left] at (v1) {$v_1$};
\node[left] at (v2) {$v_2$};
\node[left] at (v3) {$v_3$};
\node[left] at (v0) {$v_0$};

\draw[thick] (v0.center) to (v1.center) to (v2.center) to (v3.center) to [bend left] (v0.center);

\end{tikzpicture}
\caption{A tree with nodes $\{t_0,t_1,t_2,t_3\}$ coding the 4-cycle $\{v_0,v_1,v_2,v_3\}$}
\end{figure}


\subsection{Trees and related notation}

In Ramsey theory on trees, it is standard to use the  slightly relaxed  definition of tree below.
(See for instance Chapter 6 of \cite{TodorcevicBK10}.)
Recall that
the {\em meet}  of two nodes $s$ and $t$ in $\Seq$, denoted $s\wedge t$,
is the  maximal  initial segment of both $s$ and $t$.
In particular, if $s\sse t$ then $s\wedge t=s$.
A set  of nodes $A\sse \Seq$ is {\em closed under meets}
if $s\wedge t$ is in $A$, for every pair $s,t\in A$.

\begin{defn}[Tree]\label{defn.tree}
A  subset $T\sse \Seq$  is a {\em tree}
if $T$
 is closed under meets and  for each pair $s,t\in T$
with $|s|\le |t|$,
$t\re |s|$ is also in $T$.
\end{defn}

Thus,  in our context,  a {\em tree} $T$  is not necessarily closed under initial segments in $\Seq$, but rather  is closed under initial segments having  length  in $L(T)=\{|s|:s\in T\}$.
Given a tree $T\sse \Seq$, let
\begin{equation}\label{eq.That}
\widehat{T}=\{
s\in \Seq: \exists t\in T\  \exists n\le |t|\, (s=t\re n)\}.
\end{equation}
Thus,   $\widehat{T}$ is closed under initial segments in
the usual sense.

Given  a tree $T\sse \Seq$, for each node $t\in T$, define
\begin{equation}\label{eq.ht_T}
\mbox{ht}_T(t)
=|\{s\in T:s\subset t\}|,
\end{equation}
where we recall that $s\subset t$ denotes that $s$  is a proper initial segment of $t$.
For
$n\in\bN$, define
 \begin{equation}
T(n)=\{t\in T:\mbox{ht}_T(t)=n\}.
\end{equation}
A set $X\sse T$ is  called a {\em level set}
 if $X\sse T(n)$ for some  $n\in\bN$.
 Given $l\in \bN$, let
 \begin{equation}\label{eq.Trel}
 T\re l=\{t\in\widehat{T}:|t|=l\}.
 \end{equation}
 If $l$ is in $L(T)$, then $T\re l$ is a level subset of  $T$.  
 In this case, there is some $n$ such that $T\re l=T(n)$.


\subsection{Milliken's  Theorem, its use in  proving upper bounds for big Ramsey degrees of the Rado graph,  and its  insufficiency for Henson graphs}\label{subsection.HLM}

A Ramsey theorem for colorings of finite subtrees of a given tree was proved by Milliken in 1979 (Theorem \ref{thm.Milliken} below).
This theorem  has turned out to be central to proving upper bounds for
big Ramsey degrees of several types  ultrahomogeneous structures,
 including the rationals  as a linearly ordered structure in \cite{DevlinThesis},
the Rado graph  and other simple  binary relational structures   in \cite{Sauer06} and \cite{Laflamme/Sauer/Vuksanovic06},
and
circular directed graphs and   rationals with finitely many equivalence relations $\mathbb{Q}_n$  in \cite{Laflamme/NVT/Sauer10}.
A good review of these results appears in \cite{NVTHabil}.
Chapter 6 of \cite{TodorcevicBK10} provides a solid  foundation  for understanding how  Milliken's theorem is used to  deduce upper bounds on  big Ramsey degrees for both  the rationals and the Rado graph.

Given a tree $T\sse\Seq$,
 a node  $t\in T$  {\em splits} in $T$ if and only if  there are $u,v\in T$  properly extending  $t$ with
 $u(|t|)=0$ and $v(|t|)=1$.

\begin{defn}[Strong tree]\label{defn.strongtree}
Given a tree
 $T\sse \Seq$,  recall that  $L(T)=\{|s|:s\in T\}$, the set of lengths of nodes in $T$.
 We say that
$T$ is a
{\em strong tree} if and only if
for each $t\in T$ with $|t|\ne\sup L(T)$,
 $t$  splits in $T$.
 We say that $T$ is
a {\em strong tree of height $k$}, or simply a {\em $k$-strong tree}, if and only if $L(T)$ has exactly $k$ members.
\end{defn}

Note that a strong tree  $T$ has no maximal nodes if and only if $L(T)$ is infinite.
Further, for
each $k$-strong subtree of $\Seq$
is isomorphic as a tree to some binary tree of height $k$, where the isomorphism preserves relative lengths  of nodes.
In particular, a $1$-strong tree   is simply a node in $\Seq$.
See Figure 2.\ for an example of a $3$-strong tree $T$  with $L(T)=\{0,2,5\}$.


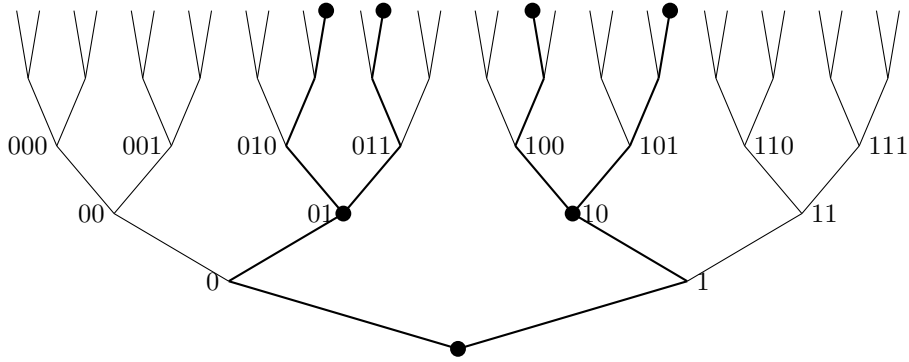
\begin{figure}
\begin{tikzpicture}[grow'=up,scale=.6]
\tikzstyle{level 1}=[sibling distance=4in]
\tikzstyle{level 2}=[sibling distance=2in]
\tikzstyle{level 3}=[sibling distance=1in]
\tikzstyle{level 4}=[sibling distance=0.5in]
\tikzstyle{level 5}=[sibling distance=0.2in]
\node {} coordinate (t9)
child{coordinate (t0) edge from parent[thick]
			child{coordinate (t00) edge from parent[thin]
child{coordinate (t000)
child {coordinate(t0000)
child{coordinate(t00000) edge from parent[color=black] }
child{coordinate(t00001)}}
child {coordinate(t0001) edge from parent[color=black]
child {coordinate(t00010)}
child{coordinate(t00011)}}}
child{ coordinate(t001)
child{ coordinate(t0010)
child{ coordinate(t00100)}
child{ coordinate(t00101) edge from parent[color=black] }}
child{ coordinate(t0011) edge from parent[color=black]
child{ coordinate(t00110)}
child{ coordinate(t00111)}}}}
			child{ coordinate(t01)  edge from parent[color=black]
child{ coordinate(t010)
child{ coordinate(t0100) edge from parent[thin]
child{ coordinate(t01000)}
child{ coordinate(t01001)}}
child{ coordinate(t0101)
child{ coordinate(t01010) edge from parent[thin]}
child{ coordinate(t01011)}}}
child{ coordinate(t011)
child{ coordinate(t0110)
child{ coordinate(t01100) edge from parent[thin]}
child{ coordinate(t01101)}}
child{ coordinate(t0111) edge from parent[thin]
child { coordinate(t01110)}
child{ coordinate(t01111)}}}}}
		child{ coordinate(t1)  edge from parent[thick]
			child{ coordinate(t10)
child{ coordinate(t100)
child{ coordinate(t1000) edge from parent[color=black, thin]
child{ coordinate(t10000)}
child{ coordinate(t10001)}}
child{ coordinate(t1001)
child{ coordinate(t10010)}
child{ coordinate(t10011) edge from parent[color=black, thin] }}}
child{ coordinate(t101)
child{ coordinate(t1010) edge from parent[color=black, thin]
child{ coordinate(t10100) }
child{ coordinate(t10101)}}
child{ coordinate(t1011)
child{ coordinate(t10110) edge from parent[color=black, thin] }
child{ coordinate(t10111)}}}}
			child{  coordinate(t11)  edge from parent[color=black, thin]
child{ coordinate(t110)
child{ coordinate(t1100)
child{ coordinate(t11000)}
child{ coordinate(t11001)}}
child{ coordinate(t1101)
child{ coordinate(t11010)}
child{ coordinate(t11011)}}}
child{  coordinate(t111)
child{  coordinate(t1110)
child{  coordinate(t11100)}
child{  coordinate(t11101)}}
child{  coordinate(t1111)
child{  coordinate(t11110)}
child{  coordinate(t11111)}}}} };

\node[left] at (t0) {$0$};
\node[left] at (t00) {$00$};
\node[left] at (t000) {$000$};
\node[left] at (t001) {$001$};
\node[left] at (t01) {$01$};
\node[left] at (t010) {$010$};
\node[left] at (t011) {$011$};
\node[right] at (t1) {$1$};
\node[right] at (t10) {$10$};
\node[right] at (t100) {$100$};
\node[right] at (t101) {$101$};
\node[right] at (t11) {$11$};
\node[right] at (t110) {$110$};
\node[right] at (t111) {$111$};

\node[circle, fill=black,inner sep=0pt, minimum size=6pt] at (t9) {};

\node[circle, fill=black,inner sep=0pt, minimum size=6pt] at (t01) {};
\node[circle, fill=black,inner sep=0pt, minimum size=6pt] at (t01011) {};
\node[circle, fill=black,inner sep=0pt, minimum size=6pt] at (t01101) {};
\node[circle, fill=black,inner sep=0pt, minimum size=6pt] at (t10010) {};
\node[circle, fill=black,inner sep=0pt, minimum size=6pt] at (t10) {};
\node[circle, fill=black,inner sep=0pt, minimum size=6pt] at (t10111) {};

\end{tikzpicture}
\caption{A strong subtree of $\Seq$ of height $3$}
\end{figure}


In \cite{Milliken79}, Milliken proved a Ramsey theorem  for finite strong subtrees of   finitely branching  trees with no maximal nodes.
Here, we present a restricted version of that theorem relevant to this paper.

\begin{thm}[Milliken, \cite{Milliken79}]\label{thm.Milliken}
Let $k\ge 1$ be given and let  all  $k$-strong subtrees of $\Seq$   be colored by finitely many colors.
Then there is an infinite strong subtree $T$ of $\Seq$ such that all $k$-strong subtrees of $T$  have the same color.
\end{thm}

\begin{rem}
A theorem  stronger than Theorem \ref{thm.Milliken},  also due to Milliken in \cite{Milliken81},
shows that the collection of all infinite strong subtrees of
an infinite finitely branching tree
forms a topological Ramsey space, meaning that it satisfies an infinite-dimensional Ramsey theorem for Baire sets when equipped with its version of the  Ellentuck  topology (see \cite{TodorcevicBK10}).
This  fact informed some of our intuition when  approaching  the present work.
\end{rem}

Now, we present the basic ideas behind  the  upper bounds for the big Ramsey degrees of the Rado graph.
Upper bounds for the rationals are similar.
The work in \cite{Laflamme/NVT/Sauer10}  relied on   a stronger variation of Milliken's theorem proved in that paper;   given that theorem, the basic idea behind  the upper bounds is similar to what we now present.

The Rado graph, denoted by $\mathcal{R}$, is universal for all countable graphs.
However, it is not the only universal countable graph.
The following graph $\mathcal{G}$  is also universal.
Let $\mathcal{G}$   denote the graph with  $\Seq$  as its countable vertex set and  the edge relation $E_{\mathcal{G}}$ defined as follows:
For $s,t\in \Seq$,
\begin{equation}
s\, E_{\mathcal{G}}\, t \ \Llra (|s|>|t|\mathrm{\  and\ } s(|t|)=1)\mathrm{\  or \ } (|t|>|s|\mathrm{\ and\  }t(|s|)=1).
\end{equation}
It turns out that this graph $\mathcal{G}$ is also universal, so $\mathcal{G}$ embeds into $\mathcal{R}$ and vice versa.

Suppose that
$\mathrm{A}$ is  a finite graph and
 all copies of $\mathrm{A}$ in $\mathcal{R}$ are colored into finitely many colors.
Take a copy of  $\mathcal{G}$   in $\mathcal{R}$ and restrict our attention to those copies of $\mathrm{A}$ in $\mathcal{G}$.
Each copy of $\mathrm{A}$ in $\mathcal{G}$
has vertices which are nodes in $\Seq$.
There are only finitely many {\em  strong similarity types} of  embeddings    of $\mathrm{A}$ into $\Seq$,
(Definition 3.1 in  \cite{Sauer06}), which we now review.

\begin{defn}[Strong Similarity Map, \cite{Sauer06}]\label{def.3.1.Sauer}
Let $k\ge 3$ be given and let  $S,T\sse \Seq$ be meet-closed subsets.
A function $f:S\ra T$ is a {\em strong similarity map} from $S$ to $T$ if for all nodes $s,t,u,v\in S$, the following hold:
\begin{enumerate}
\item
$f$ is a bijection.
\item
$f$ preserves lexicographic order: $s<_{\mathrm{lex}}t$ if and only if $f(s)<_{\mathrm{lex}}f(t)$.

\item
$f$ preserves meets, and hence splitting nodes:
$f(s\wedge t)=f(s)\wedge f(t)$.

\item
$f$ preserves relative lengths:
$|s\wedge t|<|u\wedge v|$ if and only if
$|f(s)\wedge f(t)|<|f(u)\wedge f(v)|$.

\item
$f$ preserves initial segments:
$s\wedge t\sse u\wedge v$ if and only if $f(s)\wedge f(t)\sse f(u)\wedge f(v)$.

\item
$f$ preserves passing numbers:
If  $|s|<|t|$,
then $f(t)(|f(s)|)=t(|s|)$.
\end{enumerate}
We  say that $S$ and $T$ are {\em strongly similar}, and write $S\ssim T$, exactly when there is a strong similarity map between $S$ and $T$.
\end{defn}

The relation $\ssim$ is an equivalence relation, and given a fixed finite graph $\mathrm{A}$, there are only finitely many different  equivalence classes of strongly similar  copies of  $\mathrm{A}$ in  $\mathcal{G}$.
Each equivalence class is called a {\em strong similarity type}.
Thus, each  copy  of $\mathrm{A}$ in $\mathcal{G}$
is in exactly one of finitely many strong similarity types.

 Fix one strong similarity type for $\mathrm{A}$, call it $\tau$.
For each copy $\mathrm{B}$ of $\mathrm{A}$ in
$\mathcal{G}$ of type $\tau$,  as the vertices of
 $\mathrm{B}$ are nodes in the tree $\Seq$, we
 let $T_{\mathrm{B}}$ denote the   tree    induced by the meet-closure in $\Seq$ of the vertices in  $\mathrm{B}$;
let $k$ be the number of levels of
$T_{\mathrm{B}}$.
 There are finitely many $k$-strong subtrees of $\Seq$ which contain $T_{\mathrm{B}}$.
 Moreover, each $k$-strong subtree of $\Seq$ contains  $T_{\mathrm{B}}$ for exactly one $\mathrm{B}$ in $\tau$.
 (A proof of this fact can be found in Section 6 of \cite{TodorcevicBK10}.)
 Define a coloring $h$ on the $k$-strong subtrees of $\Seq$ as follows:
 Given a $k$-strong subtree $T\sse\Seq$, let
  $\mathrm{B}$ be the unique copy of $\mathrm{A}$ in $\tau$
such that $T_{\mathrm{B}}$ is contained in $T$.
Let $h(T)$ be the color of $\mathrm{B}$.
Applying Milliken's Theorem \ref{thm.Milliken}, we obtain an infinite  strong subtree $S_{\tau}$  of $\Seq$  with all of its $k$-strong subtrees having the same color.
  Thus, all copies of $\mathrm{A}$ in $\tau$ with vertices in $S_{\tau}$ have the same color.

  Repeating this argument for each strict similarity type,
  after finitely many applications of Milliken's Theorem, we obtain an infinite strong subtree $S\sse\Seq$ with the following property:
  For each strong similarity type $\tau$ for $\mathrm{A}$,
  all copies of $\mathrm{A}$ of type $\tau$ with vertices in $S$ have the same color.
 Since $S$ is an infinite strong subtree of  $\Seq$, the subgraph $\mathcal{G}'$ of
 $\mathcal{G}$ coded by the nodes in $S$ is isomorphic to $\mathcal{G}$.
 Since $\mathcal{R}$ embeds into $\mathcal{G}$,
 we may take a copy of $\mathcal{R}$ whose nodes come from $S$; call this copy $\mathcal{R}'$.
 Then every copy of $\mathrm{A}$ in  $\mathcal{R}'$ has the color of its strong similarity type in $S$.
The number of strong similarity types for copies of  $\mathrm{A}$ is the upper bound for the number of colors of copies of $\mathrm{A}$ in $\mathcal{R}'$.

To find the exact big Ramsey degrees, Laflamme, Sauer and Vuksanovic use an additional argument in \cite{Laflamme/Sauer/Vuksanovic06}.
We do not reproduce their argument here, as the lower bounds for  big Ramsey degrees   of the Henson graphs are not the subject of this article.
However,
we do point out that  at the end of the article
 \cite{Sauer06},
Sauer takes
an  antichain
 $\mathcal{A}$ of nodes in $S$  coding a copy of $\mathcal{R}$
with the further properties:
(a)  The  tree induced by the meet-closure of the nodes in
 $\mathcal{A}$ has at each level  at most one splitting node or one
maximal node, but never both.
(b)
Passing numbers at splitting nodes are always zero, except of course for the right extension of the splitting node itself.
These  crucial properties (a) and (b)  were used
to reduce the upper bounds found in
\cite{Sauer06} to the number of strong similarity types  of copies of a finite graph $\mathrm{A}$
occuring in the copy $\mathcal{A}$ of $\mathcal{R}$.
This number was
then proved to be the exact lower bound for the big Ramsey degrees
in  \cite{Laflamme/Sauer/Vuksanovic06}.

Milliken's Theorem is not able to handle big Ramsey degrees of Henson graphs for the following reasons:
First, there is no natural way to  code a Henson graph using all nodes in a strong subtree of $\Seq$, nor is there a nicely defined graph which is bi-embeddable with a Henson graph.
Second, even if there were,  there is no way to guarantee that the strong subtree  obtained by Milliken's Theorem  would contain a Henson graph.
Thus, the need for a new Milliken-style theorem able to handle $k$-clique-free graphs.

We   begin with the properties (a) and (b) in mind when we construct trees with special distinguished nodes which code Henson graphs.
 The reader will notice that our  strong $\mathcal{H}_k$-coding trees in Section
 \ref{subsec.T_k}
 have the property that each level of the tree has at most one splitting node or one {\em coding node}
 (Definition \ref{defn.treewcodingnodes}).
 While our  $\mathcal{H}_k$-coding trees are certainly not antichains (there are no maximal nodes),
 they set the stage for taking an antichain of coding nodes which code $\mathcal{H}_k$
 and have the properties (a) and (b)
 (Lemma \ref{lem.bD}).
 This serves to reduce the  upper  bound on the number of colors.
We conjecture that the upper bounds found in this article - the number of strict similarity types of incremental antichains of coding nodes - are in fact the big Ramsey degrees.


\subsection{Halpern-\Lauchli\ Theorem and  Harrington's forcing proof}\label{subsec.HLHarrington}

An important  Ramsey theorem for trees is Theorem
\ref{thm.HL},
 due to
Halpern and \Lauchli.
This theorem was found as a key step in the celebrated result of Halpern and L\'{e}vy in \cite{Halpern/Levy71}, proving
that
 the Boolean Prime Ideal Theorem (the statement that any filter can be extended to an ultrafilter)  is strictly weaker than the Axiom of Choice, assuming  the Zermelo-Fraenkel axioms of set theory.

The Halpern-\Lauchli\ Theorem
 is a Ramsey theorem for colorings of products of level sets  of finitely many trees, forming
the basis for
Milliken's
Theorem \ref{thm.Milliken},  discussed in the previous subsection.
In-depth presentations and proofs of  various versions of the Halpern-\Lauchli\ Theorem as well as Milliken's Theorem can be found in  \cite{Farah/TodorcevicBK}, \cite{TodorcevicBK10}, and  \cite{DodosKBK}.
The book \cite{Farah/TodorcevicBK} contains
the first published
 version of a proof due to Harrington using the method of forcing to produce a result inside the standard axioms of set theory, Zermelo-Fraenkel + Axiom of Choice.
Harrington's novel approach is central to the methods we developed in \cite{DobrinenJML20} for the triangle-free Henson graph and the more general approach developed in this paper for all Henson graphs.
To provide the reader  with a warm-up  for our proof of Theorem
\ref{thm.MillikenSWP},
we reproduce  here
a forcing proof  from \cite{DobrinenRIMS17}.
This proof  was outlined for us in 2011 by Richard Laver.
It is simpler than the one given in  \cite{Farah/TodorcevicBK} (at the expense of using $\kappa=\beth_{2d-1}(\aleph_0)^+$ instead of the $\kappa=\beth_{d-1}(\aleph_0)^+$
used in \cite{Farah/TodorcevicBK})
 and provides the starting point  towards obtaining Theorem \ref{thm.MillikenSWP}.

The Halpern-\Lauchli\ Theorem holds for finitely many finitely branching trees with no maximal nodes;
here, we  restrict attention to  binary trees since they are  sufficient for applications to  graphs.
The following is the simplest version of the Halpern-\Lauchli\ Theorem for strong trees, which   provides the reader with a basic understanding of  the starting point for our Ramsey theorems in Sections
\ref{sec.5} and  \ref{sec.1SPOC}.
Recall that for a tree $T\sse\Seq$, $L(T)$ denotes the set of lengths of nodes in $T$ and $T(n)$ is a level set.

\begin{thm}[Halpern-\Lauchli, \cite{Halpern/Lauchli66}]\label{thm.HL}
Let $d\ge 1$ be fixed, and for each $i<d$, let $T_i$ denote
$\Seq$, the tree of all binary sequences of finite length.
Suppose
\begin{equation}
h:\bigcup_{n\in \mathbb{N}}\prod_{i<d} T_i(n)\ra r
\end{equation}
is  a given coloring, where $r$ is any positive integer.
Then there are infinite  strong subtrees $S_i\sse T_i$, where  $L(S_i)=L(S_j)$ for all $i<j<d$,
such that $h$ is monochromatic on
\begin{equation}
\bigcup_{n\in \mathbb{N} }\prod_{i<d} S_i(n).
\end{equation}
\end{thm}

Harrington's proof  uses a cardinal $\kappa$ large enough to satisfy a partition relation guaranteed by the following theorem.
Recall that
given cardinals $d,\sigma,\kappa,\lambda$,
$[\lambda]^d$ denotes the collection of all subsets of $\lambda$ of cardinality $d$, and 
\begin{equation}
\lambda\ra(\kappa)^d_{\sigma}
\end{equation}
means that for each coloring of $[\lambda]^d$ into $\sigma$ many colors,
there is a subset $X$ of $\lambda$ such that $|X|=\kappa$ and all members of $[X]^d$ have the same color.
The following is  a ZFC result guaranteeing cardinals large enough to have the Ramsey property for colorings into infinitely many colors.

\begin{thm}[\Erdos-Rado]\label{thm.ER}
For any non-negative integer $d$ and  infinite cardinal $\mu$,
$$
\beth_d(\mu)^+\ra(\mu^+)_{\mu}^{d+1}.
$$
\end{thm}
\vskip.1in

\noindent \bf Proof of Theorem \ref{thm.HL}. \rm
It  is  sufficient  to consider the case $r=2$.
Let
$h:\bigcup_{n\in\mathbb{N}}\prod_{i<d} T_i(n) \ra 2$
be given.
Notice that since each $T_i=\Seq$, it follows that each  $T_i(n)=T_i\re n=\Seq_n$.
Let $\kappa=\beth_{2d-1}(\aleph_0)^+$.
(Recall that
$\beth_1(\aleph_0)=2^{\aleph_0}$ and  in general, $\beth_{n+1}(\aleph_0)=2^{\beth_n(\aleph_0)}$.)
Define  $\bP$ to be the set of   functions
$p$
of the form
\begin{equation}
p: d\times\vec{\delta}_p\ra \bigcup_{i<d} T_i\re l_p,
\end{equation}
where 
\begin{enumerate}
\item[(i)]
$\vec{\delta}_p$ is a finite subset of $\kappa$ and
$l_p\in\mathbb{N}$;
\item [(ii)]
for each $i<d$,
 $\{p(i,\delta) : \delta\in  \vec{\delta}_p\}\sse  T_i\re l_p$.
\end{enumerate}
The partial  ordering on $\bP$ is  inclusion:
$q\le p$ if and only if
$l_q\ge l_p$, $\vec\delta_q\contains \vec\delta_p$,
and for each $(i,\delta)\in d\times \vec\delta_p$,
$q(i,\delta)\contains p(i,\delta)$.

Forcing with
$\bP$ adds $\kappa$ branches through the  tree $T_i$, for each  $i<d$.
For  $\al<\kappa$,
let  $\dot{b}_{i,\al}$ denote the $\al$-th generic branch through $T_i$.
Thus,
\begin{equation}
\dot{b}_{i,\al}=\{\lgl p(i,\al),p\rgl :p\in \bP,\ \mathrm{and}\ (i,\al)\in\dom(p)\}.
\end{equation}
Note that for each $p\in \bP$ with $(i,\al)\in\dom(p)$, $p$ forces that $\dot{b}_{i,\al}\re l_p= p(i,\al)$.
Let $\dot{\mathcal{U}}$ be a $\bP$-name for a non-principal ultrafilter on $\mathbb{N}$.
To simplify notation, we   write sets $\{\al_i:i< d\}$ in
$[\kappa]^d$ as vectors $\vec{\al}=\lgl \al_0,\dots,\al_{d-1}\rgl$ in strictly increasing order.
For $\vec{\al}=\lgl\al_0,\dots,\al_{d-1}\rgl\in[\kappa]^d$,
we let
\begin{equation}
\dot{b}_{\vec{\al}}\mathrm{\  \  denote\ \ }
\lgl \dot{b}_{0,\al_0},\dots, \dot{b}_{d-1,\al_{d-1}}\rgl,
\end{equation}
 and for any $l\in \mathbb{N}$, let 
\begin{equation}
\dot{b}_{\vec\al}\re l
\mathrm{\ \ denote \  \ }
\{\dot{b}_{i,\al_i}\re l:i<d\}.
\end{equation}

The goal now is to find disjoint
 infinite  sets $K_i\sse \kappa$, for $i<d$,
and a set of conditions $\{p_{\vec\al}:\vec\al\in \prod_{i<d}K_i\}$ which are compatible,
have the same images in $T$,
and such that for some fixed $\varepsilon^*$,
for each $\vec\al\in\prod_{i<d}K_i$,
$p_{\vec\al}$ forces
$h(\dot{b}_{\vec\al}\re l)=\varepsilon^*$  for   $\dot{\mathcal{U}}$-many  $l$.
Moreover, we will find nodes $t^*_i$, $i\le d$, such that for each $\vec\al\in\prod_{i<d}K_i$,
$p_{\vec\al}(i,\al_i)=t^*_i$.
These will  serve as the basis  for  the  process of building the strong subtrees  $S_i\sse T_i$ on which $h$ is monochromatic.

For each $\vec\al\in[\kappa]^d$,
choose a condition $p_{\vec{\al}}\in\bP$ such that
\begin{enumerate}
\item
 $\vec{\al}\sse\vec{\delta}_{p_{\vec\al}}$;
\item
$p_{\vec{\al}}\forces$ ``There is an $\varepsilon\in 2$  such that
$h(\dot{b}_{\vec{\al}}\re l)=\varepsilon$
for $\dot{\mathcal{U}}$ many $l$";
\item
$p_{\vec{\al}}$ decides a value for $\varepsilon$, label it  $\varepsilon_{\vec{\al}}$; and
\item
$h(\{p_{\vec\al}(i,\al_i):i< d\})=\varepsilon_{\vec{\al}}$.
\end{enumerate}

Such conditions $p_{\vec\al}$ may be obtained as follows.
Given $\vec\al\in[\kappa]^d$,
take $p_{\vec\al}^1$ to be any condition such that $\vec\al\sse\vec{\delta}_{p_{\vec\al}^1}$.
Since $\bP$ forces $\dot{\mathcal{U}}$ to be an ultrafilter on $\mathbb{N}$, there is a condition
 $p_{\vec\al}^2\le p_{\vec\al}^1$ such that
$p_{\vec\al}^2$ forces that $h(\dot{b}_{\vec\al}\re l)$ is the same color for $\dot{\mathcal{U}}$ many $l$.
Furthermore,  there must be a stronger
condition deciding which of the colors
$h(\dot{b}_{\vec\al}\re l)$ takes on $\dot{\mathcal{U}}$ many levels $l$.
Let $p_{\vec\al}^3\le p_{\vec\al}^2$  be a condition which decides this  color, and let $\varepsilon_{\vec\al}$ denote  that color.
Finally, since $p_{\vec\al}^3$ forces that for  $\dot{\mathcal{U}}$ many $l$ the color
 $h(\dot{b}_{\vec\al}\re l)$
will equal $\varepsilon_{\vec{\al}}$,
there is some $p_{\vec\al}^4\le p_{\vec\al}^3$ which decides some level $l$ so that
$h(\dot{b}_{\vec\al}\re l)=\varepsilon_{\vec{\al}}$.
If $l_{p_{\vec\al}^4}<l$,
let $p_{\vec\al}$ be any member of $\bP$ such that
$p_{\vec\al}\le p_{\vec\al}^4$ and $l_{p_{\vec\al}}=l$.
If $l_{p_{\vec\al}^4}\ge l$,
let $p_{\vec\al}=\{((i,\delta), p_{\vec\al}^4(i,\delta)\re l):(i,\delta)\in d\times\vec\delta_{p_{\vec\al}^4}\}$,
the truncation of $p_{\vec\al}^4$ to images that have length $l$.
Then $p_{\vec\al}$ forces that $\dot{b}_{\vec\al}\re l=\{p_{\vec\al}(i,\al_i):i< d\}$, and hence
 $p_{\vec\al}$  forces that
$h(\{p_{\vec\al}(i,\al_i):i< d\})=\varepsilon_{\vec{\al}}$.

Recall that we chose  $\kappa$ large enough so that   $\kappa\ra(\aleph_1)^{2d}_{\aleph_0}$ holds.
Now we prepare  for an application of the \Erdos-Rado Theorem.
Given two sets of ordinals $J,K$ we shall write $J<K$ if and only if every member of $J$ is less than every member of $K$.
Let $D_e=\{0,2,\dots,2d-2\}$ and  $D_o=\{1,3,\dots,2d-1\}$, the sets of  even and odd integers less than $2d$, respectively.
Let $\mathcal{I}$ denote the collection of all functions $\iota: 2d\ra 2d$ such that
\begin{equation}
\{\iota(0),\iota(1)\}<\{\iota(2),\iota(3)\}<\dots<\{\iota(2d-2),\iota(2d-1)\}.
\end{equation}
Thus, each $\iota$ codes two strictly increasing sequences $\iota\re D_e$ and $\iota\re D_o$, each of length $d$.
For $\vec{\theta}\in[\kappa]^{2d}$,
$\iota(\vec{\theta}\,)$ determines the pair of sequences of ordinals
\begin{equation}
 (\theta_{\iota(0)},\theta_{\iota(2)},\dots,\theta_{\iota(2d-2))}), (\theta_{\iota(1)},\theta_{\iota(3)},\dots,\theta_{\iota(2d-1)}),
\end{equation}
both of which are members of $[\kappa]^d$.
Denote these as $\iota_e(\vec\theta\,)$ and $\iota_o(\vec\theta\,)$, respectively.
To ease notation, let $\vec{\delta}_{\vec\al}$ denote
$\vec\delta_{p_{\vec\al}}$,
 $k_{\vec{\al}}$ denote $|\vec{\delta}_{\vec\al}|$,
and let $l_{\vec{\al}}$ denote  $l_{p_{\vec\al}}$.
Let $\lgl \delta_{\vec{\al}}(j):j<k_{\vec{\al}}\rgl$
denote the enumeration of $\vec{\delta}_{\vec\al}$
in increasing order.

Define a coloring  $f$ on $[\kappa]^{2d}$ into countably many colors as follows:
Given  $\vec\theta\in[\kappa]^{2d}$ and
 $\iota\in\mathcal{I}$, to reduce the number of subscripts,  letting
$\vec\al$ denote $\iota_e(\vec\theta\,)$ and $\vec\beta$ denote $\iota_o(\vec\theta\,)$,
define
\begin{align}\label{eq.fseq}
f(\iota,\vec\theta\,)= \,
&\lgl \iota, \varepsilon_{\vec{\al}}, k_{\vec{\al}},
\lgl \lgl p_{\vec{\al}}(i,\delta_{\vec{\al}}(j)):j<k_{\vec{\al}}\rgl:i< d\rgl,\cr
& \lgl  \lgl i,j \rgl: i< d,\ j<k_{\vec{\al}},  \vec{\delta}_{\vec{\al}}(j)=\al_i \rgl,\cr
&\lgl \lgl j,k\rgl:j<k_{\vec{\al}},\ k<k_{\vec{\beta}},\ \delta_{\vec{\al}}(j)=\delta_{\vec{\beta}}(k)\rgl\rgl.
\end{align}
Let $f(\vec{\theta}\,)$ be the sequence $\lgl f(\iota,\vec\theta\,):\iota\in\mathcal{I}\rgl$, where $\mathcal{I}$ is given some fixed ordering.
Since the range of $f$ is countable,
applying the \Erdos-Rado Theorem \ref{thm.ER},
there is a  subset $K\sse\kappa$ of cardinality $\aleph_1$
which is homogeneous for $f$.
Take $K'\sse K$ such that between each two members of $K'$ there is a member of $K$.
Take subsets $K_i\sse K'$ such that  $K_0<\dots<K_{d-1}$
and   each $|K_i|=\aleph_0$.

\begin{lem}\label{lem.HLonetypes}
There are $\varepsilon^*\in 2$, $k^*\in\mathbb{N}$,
and $ \lgl t_{i,j}: j<k^*\rgl$, $i< d$,
 such that
 $\varepsilon_{\vec{\al}}=\varepsilon^*$,
$k_{\vec\al}=k^*$,   and
$\lgl p_{\vec\al}(i,\delta_{\vec\al}(j)):j<k_{\vec\al}\rgl
=
 \lgl t_{i,j}: j<k^*\rgl$,
for each $i< d$,
for all $\vec{\al}\in \prod_{i<d} K_i$.
\end{lem}

\begin{proof}
Let  $\iota$ be the member in $\mathcal{I}$
which is the identity function on $2d$.
For any pair $\vec{\al},\vec{\beta}\in \prod_{i<d}K_i$, there are $\vec\theta,\vec\theta'\in [K]^{2d}$
such that
$\vec\al=\iota_e(\vec\theta\,)$ and $\vec\beta=\iota_e(\vec\theta'\,)$.
Since $f(\iota,\vec\theta\,)=f(\iota,\vec\theta'\,)$,
it follows that $\varepsilon_{\vec\al}=\varepsilon_{\vec\beta}$, $k_{\vec{\al}}=k_{\vec{\beta}}$,
and $\lgl \lgl p_{\vec{\al}}(i,\delta_{\vec{\al}}(j)):j<k_{\vec{\al}}\rgl:i< d\rgl
=
\lgl \lgl p_{\vec{\beta}}(i,\delta_{\vec{\beta}}(j)):j<k_{\vec{\beta}}\rgl:i< d\rgl$.
\end{proof}

Let $l^*$ denote the length of the nodes $t_{i,j}$.

\begin{lem}\label{lem.HLj=j'}
Given any $\vec\al,\vec\beta\in \prod_{i<d}K_i$,
if $j,j'<k^*$ and $\delta_{\vec\al}(j)=\delta_{\vec\beta}(j')$,
 then $j=j'$.
\end{lem}

\begin{proof}
Let $\vec\al,\vec\beta$ be members of $\prod_{i<d}K_i$   and suppose that
 $\delta_{\vec\al}(j)=\delta_{\vec\beta}(j')$ for some $j,j'<k^*$.
For each $i<d$, let  $\rho_i$ be the relation from among $\{<,=,>\}$ such that
 $\al_i\,\rho_i\,\beta_i$.
Let   $\iota$ be a member of  $\mathcal{I}$  such that for each $\vec\zeta\in[K]^{2d}$ and each $i<d$,
$\zeta_{\iota(2i)}\ \rho_i \ \zeta_{\iota(2i+1)}$.
Take
$\vec\theta\in[K']^{2d}$ satisfying
$\iota_e(\vec\theta)=\vec\al$ and $\iota_o(\vec\theta)= \vec\beta$.
Since between any two members of $K'$ there is a member of $K$, there is a
 $\vec\gamma\in[K]^{d}$ such that  for each $i< d$,
 $\al_i\,\rho_i\,\gamma_i$ and $\gamma_i\,\rho_i\, \beta_i$.
Given that  $\al_i\,\rho_i\,\gamma_i$ and $\gamma_i\,\rho_i\, \beta_i$ for each $i<d$,
there are  $\vec\mu,\vec\nu\in[K]^{2d}$ such that $\iota_e(\vec\mu)=\vec\al$,
$\iota_o(\vec\mu)=\vec\gamma$,
$\iota_e(\vec\nu)=\vec\gamma$, and $\iota_o(\vec\nu)=\vec\beta$.
Since $\delta_{\vec\al}(j)=\delta_{\vec\beta}(j')$,
the pair $\lgl j,j'\rgl$ is in the last sequence in  $f(\iota,\vec\theta)$.
Since $f(\iota,\vec\mu)=f(\iota,\vec\nu)=f(\iota,\vec\theta)$,
also $\lgl j,j'\rgl$ is in the last  sequence in  $f(\iota,\vec\mu)$ and $f(\iota,\vec\nu)$.
It follows that $\delta_{\vec\al}(j)=\delta_{\vec\gamma}(j')$ and $\delta_{\vec\gamma}(j)=\delta_{\vec\beta}(j')$.
Hence, $\delta_{\vec\gamma}(j)=\delta_{\vec\gamma}(j')$,
and therefore $j$ must equal $j'$.
\end{proof}

For any $\vec\al\in \prod_{i<d}K_i$ and any $\iota\in\mathcal{I}$, there is a $\vec\theta\in[K]^{2d}$ such that $\vec\al=\iota_o(\vec\theta)$.
By homogeneity of $f$ and  by the first sequence in the second line of equation  (\ref{eq.fseq}), there is a strictly increasing sequence
$\lgl j_i:i< d\rgl$  of members of $k^*$ such that for each $\vec\al\in \prod_{i<d}K_i$,
$\delta_{\vec\al}(j_i)=\al_i$.
For each $i< d$, let $t^*_i$ denote $t_{i,j_i}$.
Then  for each $i<d$ and each $\vec\al\in \prod_{i<d}K_i$,
\begin{equation}
p_{\vec\al}(i,\al_i)=p_{\vec{\al}}(i, \delta_{\vec\al}(j_i))=t_{i,j_i}=t^*_i.
\end{equation}

\begin{lem}\label{lem.HLcompat}
The set of conditions  $\{p_{\vec{\al}}:\vec{\al}\in \prod_{i<d}K_i\}$ is  compatible.
\end{lem}

\begin{proof}
Suppose toward a contradiction that there are $\vec\al,\vec\beta\in\prod_{i<d}K_i$ such that $p_{\vec\al}$ and
 $p_{\vec\beta}$ are incompatible.
By Lemma \ref{lem.HLonetypes},
for each $i<d$ and $j<k^*$,
\begin{equation}
 p_{\vec{\al}}(i,\delta_{\vec{\al}}(j))
=t_{i,j}
=p_{\vec{\beta}}(i,\delta_{\vec{\beta}}(j)).
\end{equation}
Thus,
 the only way $p_{\vec\al}$ and $p_{\vec\beta}$ can be incompatible is if
there are  $i< d$ and $j,j'<k^*$ such that
$\delta_{\vec\al}(j)=\delta_{\vec\beta}(j')$
but
$p_{\vec\al}(i,\delta_{\vec\al}(j))\ne p_{\vec\beta}(i,\delta_{\vec\beta}(j'))$.
Since
$p_{\vec\al}(i,\delta_{\vec\al}(j))=t_{i,j}$ and
$p_{\vec\beta}(i,\delta_{\vec\beta}(j'))= t_{i,j'}$,
this would imply
 $j\ne j'$.
But by Lemma \ref{lem.HLj=j'},
$j\ne j'$ implies that $\delta_{\vec\al}(j)\ne\delta_{\vec\beta}(j')$, a contradiction.
Therefore,
 $p_{\vec\al}$ and $p_{\vec\beta}$ must be  compatible.
\end{proof}

We now construct  the strong subtrees $S_i\sse T_i$,
for each $i<d$, by induction on the number of levels in the trees.
For each $i<d$, let $S_i(0)=\{t_i^*\}$ and
let $l_0=|t_i^*|$, the length of the $t_i^*$,
which is well defined since all nodes in the range of any $p\in\mathbb{P}$ have the same length.

Assume now that  $n\ge 1$,
there  are  lengths  $l_0<\dots< l_{n-1}$,
and
 we have constructed   finite strong subtrees
 $\bigcup_{m<n}S_i(m)$ of $T_i$, $i<d$,
such that
for each $m<n$,
$h$ takes color $\varepsilon^*$ on each member of
$\prod_{i<d} S_i(m)$.

For each $i<d$, let $X_i$ denote the set of immediate extensions  in $T_i$ of  the nodes in $S_i(n-1)$.
For each $i<d$,
let $J_i$ be a subset of $K_i$ with the same size as $X_i$.
For each $i< d$, label the nodes in $X_i$ as
 $\{q(i,\delta):\delta\in J_i\}$.
Let $\vec{J}$ denote  $\prod_{i< d}J_i$.
Notice that for each
$\vec\al\in \vec{J}$ and $i<d$, $q(i,\al_i)\contains t^*_i=p_{\vec{\al}}(i,\al_i)$.

We now construct a condition $q\in\bP$ such that
for each $\vec\al\in\vec{J}$,
$q\le p_{\vec\al}$.
Let
 $\vec{\delta}_q=\bigcup\{\vec{\delta}_{\vec\al}:\vec\al\in \vec{J}\}$.
For each pair $(i,\gamma)$ with $i<d$ and $\gamma\in\vec{\delta}_q\setminus
J_i$,
there is at least one $\vec{\al}\in\vec{J}$ and some $j'<k^*$ such that $\delta_{\vec\al}(j')=\gamma$.
For any other $\vec\beta\in\vec{J}$ for which $\gamma\in\vec{\delta}_{\vec\beta}$,
since the set $\{p_{\vec{\al}}:\vec{\al}\in\vec{J}\}$ is pairwise compatible by Lemma \ref{lem.HLcompat},
it follows
 that $p_{\vec\beta}(i,\gamma)$ must  equal $p_{\vec{\al}}(i,\gamma)$, which is exactly $t^*_{i,j'}$.
Let $q(i,\gamma)$ be the leftmost extension
 of $t_{i,j'}^*$ in $T$.
Thus, $q(i,\gamma)$ is defined for each pair $(i,\gamma)\in d\times \vec{\delta}_q$.
Define
\begin{equation}
q= \{\lgl (i,\delta),q(i,\delta)\rgl: i<d,\  \delta\in \vec{\delta}_q\}.
\end{equation}

\begin{lem}\label{lem.HLqbelowpal}
For each $\vec\al\in \vec{J}$,
$q\le p_{\vec\al}$.
\end{lem}

\begin{proof}
Given $\vec\al\in\vec{J}$,
by  our construction
for each pair $(i,\gamma)\in d\times\vec{\delta}_{\vec\al}$, we have
$q(i,\gamma)\contains p_{\vec{\al}}(i,\gamma)$.
\end{proof}

To construct  the $n$-th level  of the strong trees $S_i$,
take an $r\le q$ in  $\bP$ which  decides some $l_n\ge l_q$  for which   $h(\dot{b}_{\vec\al}\re l_n)=\varepsilon^*$, for all $\vec\al\in\vec{J}$.
By extending or truncating $r$, we may assume without
 loss of generality  that $l_n$ is equal to the length of the nodes in the image of $r$.
Notice that since
$r$ forces $\dot{b}_{\vec{\al}}\re l_n=\{r(i,\al_i):i<d\}$ for each $\vec\al\in \vec{J}$,
and since the coloring $h$ is defined in the ground model,
it is simply true in the ground model that
$h(\{r(i,\al_i):i<d\})=\varepsilon^*$ for each $\vec\al\in \vec{J}$.
For each $i<d$ and $\al_i\in J_i$,
extend the nodes in $X_i$ to length $l_n$ by extending  $q(i,\delta)$ to $r(i,\delta)$.
Thus, for each $i<d$,
we define $S_i(n)=\{r(i,\delta):\delta\in J_i\}$.
It follows that $h$ takes value $\varepsilon^*$ on each member of $\prod_{i<d} S_i(n)$.

For each $i<d$,  let
 $S_i=\bigcup_{n\in\mathbb{N}} S_i(n)$.
Then  each $S_i$ is a strong subtree of $T_i$
with the same set of lengths
$L(S_i)=\{l_n:n\in\bN\}$, and
$h$ takes value $\varepsilon^*$ on $\bigcup_{n\in \bN}\prod_{i<d}S_i(n)$.
\hfill$\square$

\begin{rem}
This theorem of Halpern and \Lauchli\ was applied by Laver in
 \cite{Laver84}
to prove that
given $k\ge 2$ and given
any coloring of the product of $k$ many copies of the rationals  $\mathbb{Q}^k$
into finitely many colors,
there are subsets $X_i$ of the rationals which again are  dense linear orders without endpoints such that
$X_0\times\dots\times X_{k-1}$ has at most $k!$ colors.
Laver further proved that  $k!$ is  the lower bound.
Thus, the big Ramsey degree for the simplest object ($k$-length sequences) in the \Fraisse\ class of products of $k$-many  finite linear orders
has been found.

Shelah extended the arguments above, applying forcing methods to prove consistent
versions of the Halpern-\Lauchli\ Theorem
at a measurable cardinal  in \cite{Shelah91}.
Modifications were used to prove big Ramsey degrees for the $\kappa$-rationals and $\kappa$-Rado graph by D\v{z}amonja, Larson, and Mitchell in
\cite{Dzamonja/Larson/MitchellQ09} and \cite{Dzamonja/Larson/MitchellRado09}.
Further work on the Halpern-\Lauchli\ Theorem at uncountable cardinals has been continued in \cite{Dobrinen/Hathaway16}, \cite{Dobrinen/Hathaway18} and by
Zhang who proved the analogue of Laver's result
 \cite{Laver84}  for measurable cardinals in \cite{Zhang17}.
\end{rem}


\section{Trees coding Henson graphs}\label{sec.3}

This section introduces a  unified approach for coding the Henson graphs via trees with special  distinguished nodes.
We call these trees {\em strong $K_k$-free trees} (Definition \ref{defn.stft}), since they branch as fully as
 possible
 (like strong trees) subject to never  coding $k$-cliques.
The constructions extend and streamline the construction of
 strong triangle-free trees  in \cite{DobrinenJML20}.
 The distinguished nodes will code the vertices of a Henson graph.
 The nodes in a given level of a strong $K_k$-free tree will
  code all possible  types over the
  finite graph  coded by the  lower levels of the tree.
Example 3.18 of \cite{DobrinenJML20} showed that there is a bad coloring for strong $K_3$-free  trees which thwarts their development of Ramsey theory.
This will be  overcome by using  skewed versions of strong $K_k$-free   trees,  called strong $\mathcal{H}_k$-coding trees, developed in Section \ref{sec.4}.
The work in the current  section
  provides the reader with the   essential structure of and intuition behind   strong coding trees
 utilized in the remainder of the paper.


\subsection{Henson's Criterion}\label{subsection.HC}

Recall that $K_k$ denotes  a {\em $k$-clique}, a
complete graph on $k$ vertices.
In \cite{Henson71},  for each $k\ge 3$, Henson constructed  an ultrahomogeneous $K_k$-free graph which is universal for all $K_k$-free graphs on countably many vertices.
We denote these graphs by $\mathcal{H}_k$.
It was later seen that  $\mathcal{H}_k$ is isomorphic to the \Fraisse\ limit of the \Fraisse\ class of finite $K_k$-free graphs.
Given a graph $\HH$ and a subset $V_0$ of the vertices of $\HH$,
let $\HH|V_0$ denote the induced subgraph  of $\HH$ on the vertices in $V_0$.
In  \cite{Henson71},
Henson proved that a countable graph $\HH$ is ultrahomogeneous and  universal for  countable $K_k$-free graphs if and only if $\HH$ satisfies the following property ($A_k$).
\begin{enumerate}
\item[($A_k$)]
\begin{enumerate}
\item[(i)]
$\HH$ does not admit any $k$-cliques.
\item[(ii)]
If $V_0,V_1$ are disjoint finite sets of vertices of $\HH$, and $\HH|V_0$ has no copy of $K_{k-1}$,
then there is another vertex which is connected in $\HH$ to every member of $V_0$ and to no member of $V_1$.
\end{enumerate}
\end{enumerate}

The following equivalent  modification  will be useful for our constructions.
\begin{enumerate}
\item[$(A_k)'$]
\begin{enumerate}
\item[(i)]
$\HH$ does not admit any $k$-cliques.
\item[(ii)]
Let $\lgl v_n:n\in\mathbb{N}\rgl$ enumerate the vertices of $\HH$, and
let $\lgl F_i:i\in\mathbb{N}\rgl$ be any enumeration of the finite subsets of $\mathbb{N}$ such that for each $i \in\mathbb{N}$, $\max(F_i)<i$,
 and each finite set  appears  infinitely many times in the enumeration.
Then there is a strictly increasing sequence $\lgl n_i: i \in\mathbb{N}\rgl$
such that for each $i \in\mathbb{N}$, if  $\HH|\{v_m:m\in F_i\}$
has no copy of $K_{k-1}$,
then  for all $m<i$, $v_{n_i} \, E\, v_m\longleftrightarrow m\in F_i$.
\end{enumerate}
\end{enumerate}
It is routine to check that any countably infinite graph
$\HH$
is ultrahomogeneous and universal for countable $K_k$-free graphs if and only if
$(A_k)'$ holds.


\subsection{Trees with coding nodes and strong $K_k$-free trees}\label{subsection.K_kfree}

As seen  for the case of triangle-free graphs in \cite{DobrinenJML20},
enriching trees with a collection of distinguished nodes allows for coding graphs with forbidden configurations into trees  which have properties similar to strong trees.
Recall that $\Seq$ denotes the set  of all finite sequences of $0$'s and $1$'s.

\begin{defn}[\cite{DobrinenJML20}]\label{defn.treewcodingnodes}
A {\em tree with coding nodes}
is a structure $(T,N;\sse,<,c^T)$ in the language
$\mathcal{L}=\{\sse,<,c\}$,
 where $\sse$ and $<$ are binary relation symbols  and $c$ is a unary function symbol,
 satisfying the following:
 $T$ is a subset of  $\Seq$ satisfying that   $(T,\sse)$ is a tree (recall Definition \ref{defn.tree}),
either   $N\in \mathbb{N}$ or $N=\mathbb{N}$, $<$ is the usual linear order on $N$,  and $c^T:N\ra T$ is an injective  function such that whenever   $m<n$ in $N$, then $|c^T(m)|<|c^T(n)|$.
\end{defn}

\begin{notation}
 The  {\em $n$-th coding node} in $T$, $c^T(n)$, will usually be denoted as
$c^T_n$.
The length $|c^T_n|$ of the $n$-th coding node  in $T$ shall be denoted by  $l^T_n$.
Whenever no ambiguity arises, we shall drop the superscript $T$.

Throughout this paper,
 we use
$N$ to denote either a member of $\bN=\{0,1,2,\dots\}$, or $\bN$ itself.
We  shall treat the natural numbers as von Neumann ordinals.
Thus,
 for $N\in\mathbb{N}$,  $N$ denotes the set  of natural numbers less than $N$; that is, $N=\{0,\dots,N-1\}$.
 Hence, in either case that $N\in\bN$ or $N=\bN$,
 it makes sense to  write $n\in N$.
\end{notation}

\begin{defn}[\cite{DobrinenJML20}]\label{def.rep}
A graph $\G$ with vertices enumerated as $\lgl v_n:n\in N\rgl$ is {\em  represented}  by a tree $T$ with  coding nodes $\lgl c_n:n\in N\rgl$
if and only if
for each pair $i<n$ in $N$,
 $v_n\, \E\, v_i\Longleftrightarrow  c_n(l_i)=1$.
We will often simply say that $T$ {\em codes} $\G$.
\end{defn}

For each copy of $\mathcal{H}_k$ with vertices indexed by $\mathbb{N}$, there is a  tree with coding nodes representing the graph.
In fact, this is true more generally for any graph, finite or infinite.

\begin{defn}[The tree $T_{\mathrm{G}}$ coding $\mathrm{G}$]\label{defn.T_H}
  Let $\mathrm{G}$ be any  graph with vertices ordered as $\{v_i:i\in N\}$.
  Define $T_{\mathrm{G}}$
 as follows:
 Let $c_0=\lgl\rgl$, the empty sequence.
For $n\ge 1$ with  $n\in N$,  given
 coding nodes $\{c_i:i<n\}$ coding $\mathrm{G}|\{v_i:i<n\}$, with each  $|c_i|=i$,
 define $c_n$ to be the node in $\Seq$ of length $n$ such that for each $i<n$, $c_n(i)=1\ \llra \ v_n\,  E \, v_i$.
 Let
 \begin{equation}
 T_{\mathrm{G}}=\{t\in\Seq:\exists n\in N,\ \exists l\le n\, (t=c_n\re l)\}.
\end{equation}
\end{defn}

\begin{observation}\label{obs_{T_H0}}
Let  $\mathrm{G}$ and $T_{\mathrm{G}}$ be  as in  Definition \ref{defn.T_H}.
Notice that
for each $n\ge 1$ with $n\in N$,  each  node in  $T_{\mathrm{G}}\re n$  (the nodes in  $T_{\mathrm{G}}$ of length $n$)
 represents a model-theoretic (incomplete)  $1$-type over the graph
$\mathrm{G}|\{v_i:i<n\}$.
Moreover, each such  $1$-type is represented by a unique node in  $T_{\mathrm{G}}\re n$.
In particular,  if $\mathrm{G}$ is a Rado graph or a Henson graph, then
  $T_{\mathrm{G}}$ has no maximal nodes and
 the coding nodes in
 $T_{\mathrm{G}}$  are dense.
\end{observation}

Our goal is to  develop   a means for  working with subtrees of trees like $T_{\mathcal{H}_k}$, where $\mathcal{H}_k$ is a $k$-clique-free Henson graph,  for  which we can prove Ramsey  theorems like the Halpern-\Lauchli\ Theorem \ref{thm.HL} and the Milliken Theorem \ref{thm.Milliken}.
There are several reasons why  the most na\"{i}ve approach does not work; these   will be pointed out as they arise.
In this and the next  two sections, we develop tools for
recognizing which    trees  coding $\mathcal{H}_k$ and which of their  subtrees
are able to carry a robust  Ramsey theory.
These can be interpreted model-theoretically  in terms of types over finite subgraphs, but the language of trees will be simpler and  easier to visualize.

In Definition \ref{defn.T_H}, we showed how to make a tree with coding nodes coding a particular
copy of $\mathcal{H}_k$; this is a ``top-down'' approach.
To develop Ramsey theory for colorings of finite trees,
we will need to consider all subtrees of a given  tree $T$ coding $\mathcal{H}_k$  which are  ``similar'' enough to $T$  to make a Ramsey theorem possible.
In order to prove the Ramsey theorems,
 we will further  need  criteria for  how and when we can
 extend a finite subtree   $A$ of a given tree  $S$, which is a subtree of some $T$, where $T$ codes a
copy of $\mathcal{H}_k$,  to a subtree of $S$ coding another
copy of $\mathcal{H}_k$.
This will provide a ``bottom-up'' approach for constructing trees coding $\mathcal{H}_k$.
The potential obstacles are   cliques   coded by coding nodes  in $T$, but which are not coded by coding nodes in  $S$.
To begin, we observe  exactly how  cliques are coded.

\begin{observation}\label{obs.k-cliquecoding}
For $a\ge 2$,
given  an index set $I$ of size $a$,
a  collection of coding  nodes
$\{c_{i}:i\in I\}$
 {\em codes an $a$-clique} if and only if for each pair $i<j$ in $I$,
$c_{j}(l_{i})=1$.
\end{observation}

\begin{defn}[$K_k$-Free   Criterion]\label{defn.trianglefreeextcrit}
Let $T\sse \Seq$ be a  tree with coding nodes $\lgl c_n:n<N\rgl$.
We say that
$T$
{\em satisfies the $K_k$-Free Criterion}
if the following holds:
For each $n\ge k-2$,
for all increasing sequences
$i_0<i_1<\dots<i_{k-2}=n$ such that $\{c_{i_j}:j<k-1\}$ codes  a $(k-1)$-clique,
for each $t\in T$ such that $|t|>l_n$,
\begin{equation}
(\forall j<k-2)\ t(l_{i_j})=1\ \ \Longrightarrow \ \ t(l_n)=0.
\end{equation}
\end{defn}

In words,  a tree  $T$ with coding  nodes $\lgl c_n:n\in N\rgl$ satisfies the  $K_k$-Free   Criterion  if  for each $n\in N$,  whenever a node $t$  in $T$ has the same length as the coding node $c_n$, the following holds:
If   $t$ and $c_n$ both
code edges with
some collection of  $k-2$ many coding nodes  in $T$ which themselves code a $(k-2)$-clique,
then $t$ does not split in $T$;  its only allowable extension in $\widehat{T}$ is
$t^{\frown}0$.

The next lemma  characterizes tree representations of $K_k$-free  graphs.
We say that the coding nodes in $T$ are {\em dense} in $T$, if
for each $t\in T$, there is some coding node $c_n\in T$ such that $t\sse c_n$.
Note that a  finite tree $T$ in which the coding nodes are dense will necessarily  have  coding nodes (of differing lengths) as its maximal nodes.

\begin{lem}\label{lem.trianglefreerep}
Let $T\sse\Seq$ be a tree
with coding nodes $\lgl c_n:n\in N\rgl$
 coding  a countable   graph $\G$  with  vertices $\lgl v_n:n\in N\rgl$.
Assume  that  the coding nodes in $T$ are dense in $T$.
Then
 $\G$ is a $K_k$-free graph  if and only if the tree
$T$ satisfies the $K_k$-Free Criterion.
\end{lem}

\begin{proof}
If $T$ does not satisfy the $K_k$-Free Criterion,
then there are $i_0<\dots<i_{k-2}$ in $N$
 and  $t\in T$ with $|t|>l_{i_{k-2}}$ such that
 $\{c_{i_j}:j<k-1\}$ codes a $(k-1)$-clique
 and $t(l_{i_j})=1$ for all $j<k-1$.
Since the coding nodes are dense in $T$,
there is an $n>i_{k-2}$ such that $c_n\contains t$.
Then    $\{c_{i_j}:j<k-1\}\cup\{c_n\}$ codes a $k$-clique.
On the other hand, if
$\G$ contains a $k$-clique,
then there are $i_0<\dots<i_{k-1}$ such that the coding nodes $\{c_{i_j}:j<k\}$ in $T$ code a $k$-clique, and these coding nodes  witness the failure of
the  $K_k$-Free Criterion  in $T$.
\end{proof}

\begin{defn}[$K_k$-Free Branching Criterion]\label{defn.kFSC}
A tree $T$ with coding nodes $\lgl c_n:n\in N\rgl$
satisfies the {\em $K_k$-Free Branching  Criterion ($k$-FBC)}
if   $T$  is maximally branching, subject to satisfying the $K_k$-Free Criterion.
\end{defn}

Thus,   $T$ satisfies the $K_k$-Free Branching Criterion if and only if
$T$  satisfies  the $K_k$-Free Criterion,  and
 given any $n\in N$ and  non-maximal node $t\in T$ of length $l_n$,
 (a)
there is a node $t_0\in T$  such that $t_0\contains t^{\frown}0$, and (b)
there is a $t_1\in T$ such that
$t_1\contains t^{\frown}1$ if and only
if for all sequences  $i_0<\dots<i_{k-2}=n$ such that $\{c_{i_j}:j<k-1\}$ codes a copy of $K_{k-1}$,
$t(l_{i_j})=0$ for at least one $j<k-1$.

As we move toward defining strong $K_k$-free  trees in Definition \ref{defn.stft},
we recall that
the
modified Henson criterion $(A_k)'$   is satisfied by an infinite  $K_k$-free  graph if and only if it is ultrahomogeneous and  universal for all countable $K_k$-free graphs.
The following reformulation translates $(A_k)'$ in terms of trees with coding nodes.
We say that a tree
$T\sse \Seq$ with coding  nodes $\lgl c_n:n\in \mathbb{N}\rgl$ {\em satisfies property $(A_k)^{\tt tree}$} if the following hold:\vskip.1in

\begin{enumerate}
\item[$(A_k)^{\tt{tree}}$]
\begin{enumerate}
\item[(i)]
$T$ satisfies the $K_k$-Free Criterion.
\item[(ii)]
Let  $\lgl F_n:n \in \mathbb{N}\rgl$   be any enumeration  of finite subsets of $\mathbb{N}$ such that
for each $n \in \mathbb{N}$, $\max(F_n)<n-1$, and
 each finite subset of $\mathbb{N}$ appears as $F_n$ for infinitely many indices $n$.
Given $n \in \mathbb{N}$,
if  for each subset $J\sse F_k$ of size $k-1$,
$\{c_j:j\in J\}$ does not code a $(k-1)$-clique,
then
there is some $m\ge n$ such that
for all $j<i$,
$c_m(l_j)=1$ if and only if $j\in F_n$.
\end{enumerate}
\end{enumerate}

\begin{observation}\label{obs.ttimplieshomog}
If $T$ satisfies $(A_k)^{\tt tree}$, then the coding nodes in $T$ code $\mathcal{H}_k$.
\end{observation}

To see this,
suppose  that $T$ satisfies $(A_k)^{\tt tree}$, and
let $\mathcal{H}$ be the graph with vertices $\lgl v_n:n \in \mathbb{N}\rgl$ where for $m<n$, $v_n\ E\ v_m$ if and only if $c_n(l_m)=1$.
Then
 $\mathcal{H}$ satisfies Henson's property $(A_k)$, and hence is ultrahomogeneous and universal for countable $k$-clique-free graphs.

\begin{observation}\label{obs.anyHcoded}
Let $\mathcal{H}$ be  a  copy of $\mathcal{H}_k$
 with vertices ordered as $\lgl v_n:n\in\mathbb{N}\rgl$.
 Then   $T_{\mathcal{H}}$  (recall Definition \ref{defn.T_H})
 satisfies the $K_k$-Free Branching Criterion.
\end{observation}

The next lemma
shows that any finite tree with coding nodes satisfying the $k$-FBC, where  all  maximal nodes have height $l_{N-1}$,
has the property  that  every $1$-type over the graph represented by $\{c_i:i<N-1\}$ is represented by a  maximal node in the tree.
This
is the  vital step toward  proving
Theorem \ref{thm.A_3treeimpliestrianglefreegraph}:
Any tree  satisfying the $k$-FBC
with no maximal nodes and with
 coding nodes forming a dense subset
codes  the $k$-clique-free Henson graph.

\begin{lem}\label{lem.stftextension}
Let $T$ be a finite  tree with coding nodes $\lgl c_n:n\in N\rgl$,
where $N\ge 1$,
satisfying the $K_k$-Free Branching Criterion
with all maximal nodes of length $l_{N-1}$.
Given any $F\sse N-1$ for which the set $\{c_n:n\in F\}$ codes no $(k-1)$-cliques,
there is a maximal node $t\in T$ such that
for all $n\in N-1$,
\begin{equation}
t(l_n)=1\ \ \Longleftrightarrow \ \ n\in F.
\end{equation}
\end{lem}

\begin{proof}
The proof is by induction on $N\ge 1$ over all  such trees with $N$ coding nodes.
For $N=1$,  $N-1=0$, so
  the lemma trivially holds.

Now suppose $N\ge 2$ and suppose the lemma holds for all trees with less than $N$ coding nodes.
Let $T$ be a tree  with  coding nodes $\lgl c_n:n\in N \rgl$ satisfying the $k$-FBC.
Let $F$ be a subset of $N-1$ such that $\{c_n:n\in F\}$ codes no $(k-1)$-cliques.
By the induction hypothesis,
$\{t\in T:|t|\le l_{N-2}\}$
 satisfies the lemma.
So
 there is a node $t$ in $T$ of length $l_{N-2}$ such that
for all $n\in N-2$, $t(l_n)=1$ if and only if $n\in F\setminus \{N-2\}$.
If  $N-2\not\in F$,
then
the maximal node  in $T$ extending $t^{\frown}0$
 satisfies the lemma; this node is guaranteed to exist
by the $k$-FBC.

Now  suppose $N-2\in F$.
We claim that there is a maximal node $t'$ in $T$ which extends $t^{\frown}1$.
If not, then $t^{\frown}1$ is not in $\widehat{T}$.
By the $k$-FBC, this implies that
there is some  sequence
$i_0<\dots<i_{k-2}= N-2$  such that
$\{c_{i_j}:j<k-1\}$ codes a $(k-1)$-clique
and $t(l_{i_j})=1$ for each $j<k-2$.
Since for all $i<N-2$,  $t(l_i)=1$ if and only if $i\in F\setminus\{N-2\}$,
it follows that $\{i_j:j<k-2\}\sse F$.
But then $F\contains \{i_j:j<k-1\}$,  which  contradicts that $F$ codes no $(k-1)$-cliques.
Therefore, $t^{\frown}1$ is in $\widehat{T}$.
Taking $t'$ to be maximal in $T$  and extending $t^{\frown}1$ satisfies the lemma.
\end{proof}

\begin{rem}
Lemma \ref{lem.stftextension} says the following:
Suppose
$T$ is a tree with coding nodes $\lgl c_n:n\in N\rgl$ satisfying the $K_k$-FBC, $m+1\in N$,
 and $\mathrm{G}$ is the graph represented by $\lgl c_n:n\in m\rgl$.
 Let $G'$ be any graph on $m+1$ vertices such that $G'\re m=G$.
 Then
there is a node $t\in T\re l_m$
such that letting $c'_{m}=t$,
the graph  represented by $\{c_n:n\in m\}\cup\{c'_{m}\}$
is isomorphic to $\mathrm{G}'$.
\end{rem}

\begin{thm}\label{thm.A_3treeimpliestrianglefreegraph}
Let $T$ be a tree  with infinitely many coding nodes
satisfying the $K_k$-Free Branching Criterion.
If $T$ has no maximal nodes and the coding nodes are dense in $T$,
then
 $T$ satisfies $(A_k)^{\tt tree}$, and hence codes $\mathcal{H}_k$.
\end{thm}

\begin{proof}
Since $T$ satisfies the $k$-FBC, it automatically satisfies (i) of $(A_k)^{\tt tree}$.
Let   $\lgl F_n:n\in\mathbb{N}\rgl$  be an enumeration of finite subsets of $\mathbb{N}$
where each set is repeated infinitely many times, and each $\max(F_n)<n-1$.
For $n=0$, $F_n$ is the emptyset, so every coding node in $T$ fulfills  (ii) of $(A_k)^{\tt tree}$.
Let $n\ge 1$ be given and suppose that for each subset
$J\sse F_n$ of size $k-1$, $\{c_j:j\in J\}$ does not code a
$(k-1)$-clique.
By Lemma \ref{lem.stftextension},
there is some node $t\in T$ of length $l_{n-1}$
such that for all $i<n-1$, $t(l_i)=1$ if and only if $i\in F_n$.
Since the coding nodes are dense in $T$, there is some
$m\ge n$ such that $c_m$ extends $t$.
This coding node $c_m$ fulfills (ii) of $(A_k)^{\tt tree}$.
\end{proof}

At this point, we  have developed enough ideas and terminology  to define strong $K_k$-free trees.
These will be special types of trees   coding copies of $\mathcal{H}_k$ with  additional properties which set the stage for their skew versions in Section \ref{sec.4} on which we will be able to develop a viable Ramsey theory.
We shall use {\em ghost coding nodes} for the first $k-3$ levels.
Coding nodes will start at length $k-2$, and all coding nodes of length at least $k-2$ will end in a sequence of $(k-2)$ many $1$'s.
The effect is that coding nodes will only be extendible by $0$; coding nodes will never split.  
This will serve to reduce the upper bound on the big Ramsey degrees for $\mathcal{H}_k$.

Recall that $\mathcal{S}_{\le n}$ is the set of all sequences of $0$'s and $1$'s of length $\le n$.

 \begin{notation}
Throughout this paper, we use the notation  $0^{(i)}$ and $1^{(i)}$ to denote sequences of length $i$ where all entries are $0$, or all entries are $1$, respectively.
\end{notation}

\begin{defn}[Strong $K_k$-Free Tree]\label{defn.stft}
A  {\em strong $K_k$-free tree} is  a  tree with coding nodes, $(T,\mathbb{N};\sse,<,c)$,  satisfying the following:
\begin{enumerate}
\item
$T$ has no maximal nodes, the coding nodes are dense in $T$, and no coding node splits in $T$.
\item
The first $k-2$ levels of $T$ are exactly $\Seq_{\le k-2}$,
and the least coding node $c_0$ is exactly $1^{(k-2)}$.
\item
For each $n\in \mathbb{N}$,
the $n$-th coding node $c_n$ has length $n+k-2$, and has as final segment a sequence of $k-2$ many $1$'s.
\item
$T$
satisfies the $K_k$-free Branching Criterion.
\end{enumerate}
Moreover, $T$ has ghost coding nodes $c_{-k+2},\dots,c_{-1}$ defined by $c_n=1^{(k+n-2)}$ for $n\in[-k+2,-1]$, where $1^{(0)}$ denotes the empty sequence. 
A
{\em finite strong $K_k$-free tree}  is the restriction
of a strong $K_k$-free tree to some finite level.
\end{defn}

By Theorem
\ref{thm.A_3treeimpliestrianglefreegraph},
each strong $K_k$-free tree codes $\mathcal{H}_k$.

\begin{rem}
Items (1)  and (4) in Definition
\ref{defn.stft} ensure that the tree represents a $K_k$-free Henson graph.
Items (2) and (3) serve to reduce our bounds on the big Ramsey degrees by ensuring that coding nodes never split.
For any  node $t$ in $T$   with all entries being $0$, the   subtree $S$ of all nodes in  $T$ extending $t$   codes a copy of $\mathcal{H}_k$, by Theorem
\ref{thm.A_3treeimpliestrianglefreegraph}.
Moreover, 
 the structure of the  first $k-2$  levels of   such a subtree  $S$ are tree isomorphic to $\Seq_{\le k-2}$.
 Thus, it makes sense to require (2) in Definition
\ref{defn.stft}.
The ghost coding nodes provide the  structure which subtrees  coding $\mathcal{H}_k$ in the same order as $T$ automatically inherit.
This will enable us to build the collection of all subtrees of a given tree $T$ which are isomorphic to $T$ in a strong way, to be made precise in the next section. 
\end{rem}

We now present a  method for constructing   strong $K_k$-free trees.
For $k=3$,
this construction method   simplifies the construction of a strong triangle-free tree coding $\mathcal{H}_3$ in
Theorem 3.14 of
 \cite{DobrinenJML20}  and accomplishes the same goals.
 The aim of Example \ref{thm.stftree}
 is  to
 build  the reader's understanding of the principal structural properties
 of the trees on which we will develop Ramsey theory, before defining their skew versions in the next section.

\begin{example}[Construction Method for  a Strong $K_k$-Free Tree, $\bS_k$]\label{thm.stftree}
Recall that by Theorem
\ref{thm.A_3treeimpliestrianglefreegraph},
each strong $K_k$-free tree codes $\mathcal{H}_k$.
Let
 $\lgl u_i: i\in \mathbb{N}\rgl$ be any enumeration of  $\Seq$
  such  that $|u_i|\le |u_j|$ whenever $i<j$.
 Notice that in particular, $|u_i|\le i$.
We will build a strong $K_k$-free tree $\bS_k\sse \Seq$
with the $n$-th coding node $c_n$ of length $l_n=n+k-2$ and 
 satisfying the following additional conditions for  $n\ge  k-2$:
\begin{enumerate}
\item[(i)]
If $n\equiv 0 \ (\mathrm{mod\ } k-1)$, $i=n/(k-1)$,
and
$u_i $ is in  $\bS_k\cap \Seq_{\le  i+1}$,
then $c_n\contains u_i$.
\item[(ii)]
Otherwise, 
$c_n={0^{(n)}}^{\frown}1^{(k-2)}$.
\end{enumerate}

The first $k-2$ levels of $\bS_k$ are exactly $\Seq_{\le k-2}$.
The ghost coding nodes are defined as in Definition \ref{defn.stft}, with $c_{-k+2}$ being the empty sequence and $c_{-1}$ being $1^{(k-3)}$.
The shortest coding node is $c_0=1^{(k-2)}$.
 Notice that since $u_0=\lgl\rgl$, $c_0$ extends $u_0$.

$\bS_k$ will have nodes of every finite length, so
 $\bS_k(n)=\bS_k \re n$ for each $n\in \bN$.
We shall let $r_n(\bS_k)$ denote $\bigcup_{m<n}\bS_k(m)$;
this notation comes from topological Ramsey space theory  in \cite{TodorcevicBK10} and will be used again in the next section.
Then   $r_{k-1}(\bS_k)=\Seq_{\le k-2}$.
  Extend the nodes of $\bS_k(k-2)$ according to the $k$-FBC with respect to
 $\{c_{-k+2},\dots, c_0\}$
 to form the next level
  $\bS_k(k-1)$.
 Let $c_{1}= \lgl 0\rgl^{\frown}1^{(k-2)}$.
 This node is in
 $\bS_k(k-1)$, since it codes no $k$-cliques with 
  $\{c_{-k+2},\dots, c_0\}$.
 Extend the nodes of $\bS_k(k-1)$  according to the $k$-FBC with respect to
  $\{c_{-k+2},\dots, c_1\}$
  to form  the next level $\bS_k(k)$.
So far,
  (1)--(4) of a Strong $K_k$-Free Tree and (i) and (ii)  above  are satisfied.

 Given $n\ge  k-1$,  suppose we
have defined
$r_{n+2}(\bS_k)$ and
$\{c_{-k+2},\dots,c_{n-k+2}\}$ so that  (1)--(4) and (i) and (ii) hold so far.
If $n\not\equiv 0 \ (\mathrm{mod\ } k-1)$,
or $n\equiv 0 \ (\mathrm{mod\ } k-1)$ but $u_i\not\in r_{i+1}(\bS_k)$, where $i=n/(k-1)$,
then  define
$c_n={0^{(n)}}^{\frown}1^{(k-2)}$.
This node is in $\bS_k(n+1)$ by the $k$-FBC, since
the only nodes
$c_n$  codes   edges with are
exactly  the coding nodes $c_{n-k+2},\dots, c_{n-1}$.

Now suppose that  $n\equiv 0 \ (\mathrm{mod\ } k-1)$ and 
$u_i$ is in
$r_{i+1}(\bS_k)$, where $i=n/(k-1)$.
Let $q$ denote $n- |u_i|$
and
define $v={u_i}^{\frown} {0^{(q)}}^{\frown}1^{(k-2)}$.
We claim that $v$ is in
$\bS_k(n+1)$.
Let $-k+2\le m_0<\dots<m_{k-2}\le n-1$ and suppose that
$\{c_{m_j}: j\in k-1\}$ is a set of coding nodes in $r_n(\mathbb{S}_k)$ coding a $(k-1)$-clique.
If $m_{k-2}< |u_i|$,  then  $v(m_j)=0$ for at least one $j\in  k-1$, since $u_i$ is in $r_{  |u_i|+1 }(\bS_k)$ which satisfies the $k$-FBC.
Let $w=
{u_i}^{\frown} {0^{(q)}}$.
If for some $j\in k-1$, $|u_i|\le m_j<|w|$,
then  $v(m_j)=0$.
For the following, it is important to notice that
$q\ge k-2$.
If for some $j<j'$, $ |c_{m_j}|<|u_i|$ and $|w|\le |c_{m_{j'}}|$,
then $c_{m_{j'}}(m_j)=0$, contrary to our assumption that $\{c_{m_j}: j\in k-1\}$ codes a  $(k-1)$-clique.
Lastly, suppose $m_0\ge |w|$.
Then  the nodes $c_{m_j}$, $j\in k-1$, are exactly the coding nodes
$c_{n- k+1},\dots, c_{n-1}$.
Thus, $v(c_{m_0})=0$.
Therefore, by the $k$-FBC, $v$ is in $\bS_k(n+1)$.
Let $c_n=v$ and split according to the $k$-FBC to construct $\bS_k(n+2)$.
This satisfies (1)--(4) and (i) and (ii).

This inductive process  constructs a tree  $\bS_k=\bigcup_{n<\om}\bS_k(n)$  which is a strong $K_k$-free tree satisfying (i) and (ii).
\end{example}

\begin{rem}
For $k=3$, the previous construction of a strong triangle-free tree produces a strong triangle-free tree in the sense of \cite{DobrinenJML20}, albeit in a more streamlined fashion.
\end{rem}

\begin{example}[A Strong $K_3$-Free Tree, $\bS_3$]\label{ex.stft}
In keeping with the construction method above,
we present the first several  levels of the construction of  
$\bS_3$.
Let $u_0$ denote the empty sequence, and suppose  $u_1=\lgl 1\rgl$ and  $u_2=\lgl 0\rgl$.
The ghost coding node is $c_{-1}=\lgl\rgl$, the empty sequence.
The  coding nodes $c_n$ where $n$ is odd will be $c_n={0^{(n)}}^{\frown}1$.
In particular,  $c_1=\lgl 0,1\rgl$, $c_3=\lgl 0,0,0,1\rgl$, $c_5=\lgl 0,0,0,0,0,1\rgl$, etc.
These nodes are in every tree satisfying the $K_3$-Free Branching Criterion.

Let  $c_0=\lgl 1\rgl$;
this node extends
$u_0=\lgl\rgl$.
Let $c_2=\lgl 1,0,1\rgl$; this node extends $u_1=\lgl 1\rgl$.
Let $c_4=\lgl 0,1,0,0,1\rgl$; this node extends $u_2=\lgl 0\rgl$.
If $u_3=\lgl 1,1\rgl$, since this node is not in  $\bS_3(7)$, we let $c_6={0^{(6)}}^{\frown}1$.
If $u_3=\lgl 1,0\rgl$, we
 can let $c_6$ be any node in $\bS_3(7)$ extending $u_3$.
For instance, if we care about making $\bS_3$ recursively definable with respect to the sequence $\lgl u_i:i\in \mathbb{N}\rgl$,  we  can let $c_6$ be the rightmost extension of $u_3$ in $\bS_3(7)$ which has last entry $1$, namely
$c_6= \lgl 1,0,1,0,1,0,1\rgl$.
In this manner, one constructs a  tree such as  in Figure 3.


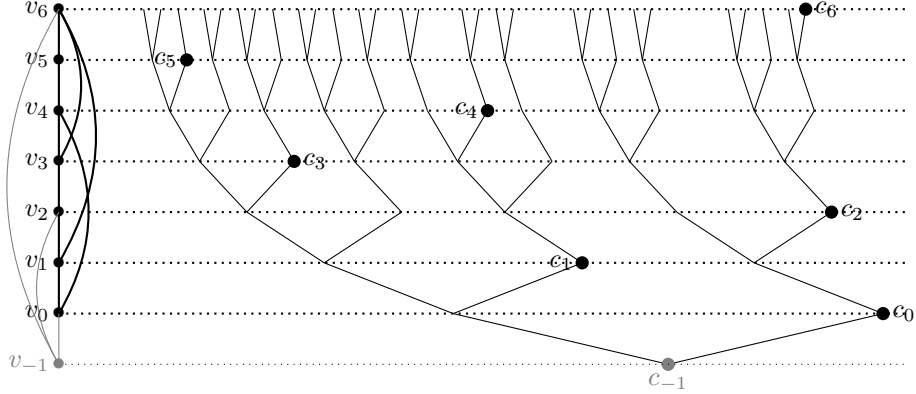
\begin{figure}\label{fig.bS3}
\begin{tikzpicture}[grow'=up,scale=.45]

\tikzstyle{level 1}=[sibling distance=5in]
\tikzstyle{level 2}=[sibling distance=3in]
\tikzstyle{level 3}=[sibling distance=1.8in]
\tikzstyle{level 4}=[sibling distance=1.1in]
\tikzstyle{level 5}=[sibling distance=.7in]
\tikzstyle{level 6}=[sibling distance=0.4in]
\tikzstyle{level 7}=[sibling distance=0.2in]

\node {} coordinate(t)
child{coordinate (t0)
			child{coordinate (t00)
child{coordinate (t000)
child {coordinate(t0000)
child{coordinate(t00000)
child{coordinate(t000000)
child{coordinate(t0000000)}
child{coordinate(t0000001)}
}
child{coordinate(t000001)
child{coordinate(t0000010)}
child{ edge from parent[draw=none]  coordinate(t0000011)}
}
}
child{coordinate(t00001)
child{coordinate(t000010)
child{coordinate(t0000100)}
child{coordinate(t0000101)}
}
child{ edge from parent[draw=none]  coordinate(t000011)
}
}}
child {coordinate(t0001)
child {coordinate(t00010)
child {coordinate(t000100)
child {coordinate(t0001000)}
child {coordinate(t0001001)}
}
child {coordinate(t000101)
child {coordinate(t0001010)}
child { edge from parent[draw=none]  coordinate(t0001011)}
}
}
child{coordinate(t00011) edge from parent[draw=none] }}}
child{ coordinate(t001)
child{ coordinate(t0010)
child{ coordinate(t00100)
child{ coordinate(t001000)
child{ coordinate(t0010000)}
child{ coordinate(t0010001)}
}
child{ coordinate(t001001)
child{ coordinate(t0010010)}
child{ edge from parent[draw=none] coordinate(t0010011)}
}
}
child{ coordinate(t00101)
child{ coordinate(t001010)
child{ coordinate(t0010100)}
child{ coordinate(t0010101)}
}
child{   edge from parent[draw=none]coordinate(t001011)
}
}}
child{  edge from parent[draw=none] coordinate(t0011)}}}
			child{ coordinate(t01)
child{ coordinate(t010)
child{ coordinate(t0100)
child{ coordinate(t01000)
child{ coordinate(t010000)
child{ coordinate(t0100000)}
child{   coordinate(t0100001)}
}
child{ edge from parent[draw=none] coordinate(t010001)
}
}
child{ coordinate(t01001)
child{ coordinate(t010010)
child{ coordinate(t0100100)}
child{ coordinate(t0100101)}
}
child{edge from parent[draw=none]  coordinate(t010011)}
}}
child{ coordinate(t0101)
child{ coordinate(t01010)
child{ coordinate(t010100)
child{ coordinate(t0101000)}
child{coordinate(t0101001)}
}
child{ edge from parent[draw=none]  coordinate(t010101)
}
}
child{  edge from parent[draw=none]  coordinate(t01011)
}}}
child{ edge from parent[draw=none]  coordinate(t011)}}}
		child{ coordinate(t1)
			child{ coordinate(t10)
child{ coordinate(t100)
child{ coordinate(t1000)
child{ coordinate(t10000)
child{ coordinate(t100000)
child{ coordinate(t1000000)}
child{ coordinate(t1000001)}
}
child{ coordinate(t100001)
child{ coordinate(t1000010)}
child{ edge from parent[draw=none] coordinate(t1000011)}
}
}
child{  coordinate(t10001)
child{coordinate(t100010)
child{coordinate(t1000100)}
child{coordinate(t1000101)}}
child{edge from parent[draw=none] coordinate(t100011)}
}}
child{edge from parent[draw=none] coordinate(t1001)}}
child{ coordinate(t101)
child{ coordinate(t1010)
child{ coordinate(t10100)
child{ coordinate(t101000)
child{ coordinate(t1010000)}
child{ coordinate(t1010001)}
}
child{ coordinate(t101001)
child{ coordinate(t1010010)}
child{   edge from parent[draw=none]   coordinate(t1010011)}
}
}
child{  coordinate(t10101)
child{  coordinate(t101010)
child{  coordinate(t1010100)}
child{  coordinate(t1010101)}
}
child{ edge from parent[draw=none]   coordinate(t101011)}
}}
child{ edge from parent[draw=none] coordinate(t1011)}}}
child{  edge from parent[draw=none] coordinate(t11)} };

\node[below] at (t) {${\color{gray}c_{-1}}$};
\node[right] at (t1) {$c_0$};
\node[left] at (t01) {$c_1$};
\node[right] at (t101) {$c_2$};
\node[right] at (t0001) {$c_3$};
\node[left] at (t01001) {$c_4$};
\node[left] at (t000001) {$c_5$};
\node[right] at (t1010101) {$c_6$};

\node[circle, fill=gray,inner sep=0pt, minimum size=5pt] at (t) {};
\node[circle, fill=black,inner sep=0pt, minimum size=5pt] at (t1) {};
\node[circle, fill=black,inner sep=0pt, minimum size=5pt] at (t01) {};
\node[circle, fill=black,inner sep=0pt, minimum size=5pt] at (t101) {};
\node[circle, fill=black,inner sep=0pt, minimum size=5pt] at (t0001) {};
\node[circle, fill=black,inner sep=0pt, minimum size=5pt] at (t01001) {};
\node[circle, fill=black,inner sep=0pt, minimum size=5pt] at (t000001) {};
\node[circle, fill=black,inner sep=0pt, minimum size=5pt] at (t1010101) {};

\draw[dotted] let \p1=(t) in (-18,\y1) node (v00) {${\color{gray}\bullet}$} -- (7,\y1);
\draw[thick, dotted] let \p1=(t1) in (-18,\y1) node (v0) {$\bullet$} -- (7,\y1);
\draw[thick, dotted] let \p1=(t01) in (-18,\y1) node (v1) {$\bullet$} -- (7,\y1);
\draw[thick, dotted] let \p1=(t001) in (-18,\y1) node (v2) {$\bullet$} -- (7,\y1);
\draw[thick, dotted] let \p1=(t0001) in (-18,\y1) node (v3) {$\bullet$} -- (7,\y1);
\draw[thick, dotted] let \p1=(t01001) in (-18,\y1) node (v4) {$\bullet$} -- (7,\y1);
\draw[thick, dotted] let \p1=(t000001) in (-18,\y1) node (v5) {$\bullet$} -- (7,\y1);
\draw[thick, dotted] let \p1=(t1010101) in (-18,\y1) node (v6) {$\bullet$} -- (7,\y1);

\node[left] at (v00) {${\color{gray}v_{-1}}$};
\node[left] at (v0) {$v_0$};
\node[left] at (v1) {$v_1$};
\node[left] at (v2) {$v_2$};
\node[left] at (v3) {$v_3$};
\node[left] at (v4) {$v_4$};
\node[left] at (v5) {$v_5$};
\node[left] at (v6) {$v_6$};

\draw[thick] (v0.center) to (v1.center) to (v2.center) to (v3.center);
\draw[thick] (v3.center) to (v4.center) to (v5.center) to  (v6.center) ;
\draw[thick] (v6.center) to [bend left] (v3.center);
\draw[thick] (v4.center) to [bend left] (v0.center);
\draw[thick] (v6.center) to [bend left] (v1.center);
\draw[gray] (v00.center) to [bend left] (v2.center);
\draw[gray] (v00.center) to [bend left] (v6.center);
\draw[gray] (v00.center) to (v0.center);

\end{tikzpicture}
\caption{A strong triangle-free tree $\bS_3$ densely coding $\mathcal{H}_3$}
\end{figure}

\end{example}


\begin{example}[A Strong $K_4$-Free Tree, $\bS_4$]\label{ex.S4}

The following tree $\bS_4$  in Figure 4.\ is an example of a strong $K_4$-free  tree.
Let $u_0$ denote the empty sequence, and suppose $u_1=\lgl 1\rgl$ and $u_2=\lgl 0\rgl$.
The ghost coding nodes are  $c_{-2}=\lgl \rgl$ and  $c_{-1}=\lgl 1\rgl$.
According to the construction in Example \ref{thm.stftree},
the first three coding nodes of $\bS_4$ are
$c_0=\lgl 1,1\rgl$,  which extends $u_0$,
$c_1=\lgl 0,1,1\rgl$, and $c_2=\lgl 0,0,1,1\rgl$,
each time splitting according to the $K_4$-Free Splitting Criterion ($4$-FSC) to construct a tree $r_5(\bS_4)$.
Split again according to the $4$-FSC to construct the next level of the tree, $\bS_4(5)$.
Since $u_1$ is in $r_3(\bS_4)$, 
letting $c_3=\lgl 1,1,0,1,1\rgl$ satisfies requirement (i) in Example \ref{thm.stftree}.
Let $c_4=\lgl 0,0,0,0,1,1\rgl$ and $c_5=\lgl 0,0,0,0,0,1,1\rgl$.
Since $u_2=\lgl 0\rgl$ is in $r_4(\bS_4)$,  taking $c_6$ to be  
$\lgl 0,1,1,0,0,0,1,1\rgl$ satisfies our requiremnts. 
One can check that this node is in $\bS_4(8)$.
(One could also simply let $c_6$ be ${0^{(6)}}^{\frown}\lgl 1,1\rgl$. 
Continue the construction according to Example \ref{thm.stftree}.

\begin{figure}\label{fig.bS4}
\begin{tikzpicture}[grow'=up,scale=.4]

\tikzstyle{level 1}=[sibling distance=6in]
\tikzstyle{level 2}=[sibling distance=3.2in]
\tikzstyle{level 3}=[sibling distance=1.6in]
\tikzstyle{level 4}=[sibling distance=.8in]
\tikzstyle{level 5}=[sibling distance=.4in]
\tikzstyle{level 6}=[sibling distance=0.2in]
\tikzstyle{level 7}=[sibling distance=0.1in]

\node {} coordinate(t)
child{coordinate (t0)
			child{coordinate (t00)
child{coordinate (t000)
child {coordinate(t0000)
child{coordinate(t00000)
child{coordinate(t000000)
child{coordinate(t0000000)}
child{coordinate(t0000001)}
}
child{coordinate(t000001)
child{coordinate(t0000010)}
child{ coordinate(t0000011)}
}
}
child{coordinate(t00001)
child{coordinate(t000010)
child{coordinate(t0000100)}
child{coordinate(t0000101)}
}
child{ coordinate(t000011)
child{ coordinate(t0000110)}
child{ coordinate(t0000111) edge from parent[draw=none] }
}
}}
child {coordinate(t0001)
child {coordinate(t00010)
child {coordinate(t000100)
child {coordinate(t0001000)}
child {  coordinate(t0001001)}
}
child {coordinate(t000101)
child {coordinate(t0001010)}
child { coordinate(t0001011)}
}
}
child{coordinate(t00011)
child{coordinate(t000110)
child{coordinate(t0001100)}
child{coordinate(t0001101)}
 }
child{coordinate(t000111)edge from parent[draw=none] }
 }}}
child{ coordinate(t001)
child{ coordinate(t0010)
child{ coordinate(t00100)
child{ coordinate(t001000)
child{ coordinate(t0010000)}
child{ coordinate(t0010001)}
}
child{ coordinate(t001001)
child{ coordinate(t0010010)}
child{ coordinate(t0010011)}
}
}
child{ coordinate(t00101)
child{ coordinate(t001010)
child{ coordinate(t0010100)}
child{ coordinate(t0010101)}
}
child{ coordinate(t001011)
child{ coordinate(t0010110)}
child{ coordinate(t0010111) edge from parent[draw=none] }
}
}}
child{  coordinate(t0011)
child{  coordinate(t00110)
child{  coordinate(t001100)
child{  coordinate(t0011000)}
child{  coordinate(t0011001)}
}
child{  coordinate(t001101)
child{  coordinate(t0011010)}
child{  coordinate(t0011011)}
}
}
child{  coordinate(t00111)edge from parent[draw=none]}
}}}
			child{ coordinate(t01)
child{ coordinate(t010)
child{ coordinate(t0100)
child{ coordinate(t01000)
child{ coordinate(t010000)
child{ coordinate(t0100000)}
child{  coordinate(t0100001)}
}
child{ coordinate(t010001)
child{ coordinate(t0100010)}
child{ coordinate(t0100011)}
}
}
child{ coordinate(t01001)
child{ coordinate(t010010)
child{ coordinate(t0100100)}
child{ coordinate(t0100101)}
}
child{ coordinate(t010011)
child{ coordinate(t0100110)}
child{ coordinate(t0100111)edge from parent[draw=none]  }
}
}}
child{ coordinate(t0101)
child{ coordinate(t01010)
child{ coordinate(t010100)
child{ coordinate(t0101000)}
child{ coordinate(t0101001)}
}
child{ coordinate(t010101)edge from parent[draw=none]}
}
child{  coordinate(t01011)
child{  coordinate(t010110)
child{  coordinate(t0101100)}
child{  coordinate(t0101101)}
}
child{  coordinate(t010111)edge from parent[draw=none]  }
}}}
child{ coordinate(t011)
child{ coordinate(t0110)
child{ coordinate(t01100)
child{ coordinate(t011000)
child{ coordinate(t0110000)}
child{ coordinate(t0110001)}
}
child{ coordinate(t011001)
child{ coordinate(t0110010)}
child{ coordinate(t0110011)}
}
}
child{ coordinate(t01101)
child{ coordinate(t011010)
child{ coordinate(t0110100)}
child{ coordinate(t0110101)}
}
child{ coordinate(t011011)
child{ coordinate(t0110110)}
child{ coordinate(t0110111)edge from parent[draw=none] }}
}
}
child{ coordinate(t0111)edge from parent[draw=none]
}
}}}
		child{ coordinate(t1)
			child{ coordinate(t10)
child{ coordinate(t100)
child{ coordinate(t1000)
child{ coordinate(t10000)
child{ coordinate(t100000)
child{ coordinate(t1000000)}
child{ coordinate(t1000001)}
}
child{ coordinate(t100001)
child{ coordinate(t1000010)}
child{ coordinate(t1000011)}
}
}
child{ coordinate(t10001)
child{ coordinate(t100010)
child{ coordinate(t1000100)}
child{ coordinate(t1000101)}
}
child{ coordinate(t100011)
child{ coordinate(t1000110)}
child{ coordinate(t1000111)edge from parent[draw=none]  }
}
}}
child{ coordinate(t1001)
child{ coordinate(t10010)
child{ coordinate(t100100)
child{ coordinate(t1001000)}
child{ coordinate(t1001001)}
}
child{ coordinate(t100101)
child{ coordinate(t1001010)}
child{ coordinate(t1001011)}
}
}
child{  coordinate(t10011)
child{  coordinate(t100110)
child{  coordinate(t1001100)}
child{  coordinate(t1001101)}
}
child{  coordinate(t100111)edge from parent[draw=none] }
}}}
child{ coordinate(t101)
child{ coordinate(t1010)
child{ coordinate(t10100)
child{ coordinate(t101000)
child{ coordinate(t1010000)}
child{ coordinate(t1010001)}
}
child{ coordinate(t101001)
child{ coordinate(t1010010)}
child{   coordinate(t1010011)}
}
}
child{  coordinate(t10101)
child{  coordinate(t101010)
child{  coordinate(t1010100)}
child{  coordinate(t1010101)}
}
child{  coordinate(t101011)
child{  coordinate(t1010110)}
child{  coordinate(t1010111)edge from parent[draw=none]}
}
}}
child{ coordinate(t1011)
child{ coordinate(t10110)
child{ coordinate(t101100)
child{ coordinate(t1011000)}
child{ coordinate(t1011001)}
}
child{ coordinate(t101101)
child{ coordinate(t1011010)}
child{ coordinate(t1011011)}
}
}
child{ coordinate(t10111)
edge from parent[draw=none]  }
}}}
child{ coordinate(t11)
child{coordinate (t110)
child{coordinate (t1100)
child{coordinate (t11000)
child{coordinate(t110000)
child{coordinate(t1100000)}
child{coordinate(t1100001)}
}
child{coordinate(t110001) edge from parent[draw=none]  }}
child{coordinate(t11001)
child{coordinate(t110010)
child{coordinate(t1100100)}
child{coordinate(t1100101)}}
child{coordinate(t110011) edge from parent[draw=none]  }
}}
child{coordinate (t1101)
child{coordinate(t11010)
child{coordinate(t110100)
child{coordinate(t1101000)}
child{coordinate(t1101001)}}
child{coordinate(t110101) edge from parent[draw=none]  }
}
child{coordinate(t11011)
child{coordinate(t110110)
child{coordinate(t1101100)}
child{coordinate(t1101101)}
}
child{coordinate(t110111) edge from parent[draw=none]  }
}}}
child{coordinate(t111) edge from parent[draw=none] }} };

\node[below] at (t) {${\color{gray} c_{-2}}$};
\node[right] at (t1) {$\ \ {\color{gray}c_{-1}}$};
\node[right] at (t11) {$c_0$};
\node[left] at (t011) {$c_1$};
\node[right] at (t0011) {$c_2$};
\node[right] at (t11011) {$c_3$};
\node[left] at (t000011) {$c_4$};

\node[circle, fill=gray,inner sep=0pt, minimum size=5pt] at (t) {};
\node[circle, fill=gray,inner sep=0pt, minimum size=5pt] at (t1) {};
\node[circle, fill=black,inner sep=0pt, minimum size=5pt] at (t11) {};
\node[circle, fill=black,inner sep=0pt, minimum size=5pt] at (t011) {};
\node[circle, fill=black,inner sep=0pt, minimum size=5pt] at (t0011) {};
\node[circle, fill=black,inner sep=0pt, minimum size=5pt] at (t11011) {};
\node[circle, fill=black,inner sep=0pt, minimum size=5pt] at (t000011) {};

\draw[dotted] let \p1=(t) in (-17,\y1) node (v02) {${\color{gray}\bullet}$} -- (13,\y1);
\draw[dotted] let \p1=(t1) in (-17,\y1) node (v01) {${\color{gray}\bullet}$} -- (13,\y1);
\draw[thick, dotted] let \p1=(t11) in (-17,\y1) node (v0) {$\bullet$} -- (13,\y1);
\draw[thick, dotted] let \p1=(t011) in (-17,\y1) node (v1) {$\bullet$} -- (13,\y1);
\draw[thick, dotted] let \p1=(t0011) in (-17,\y1) node (v2) {$\bullet$} -- (13,\y1);
\draw[thick, dotted] let \p1=(t11011) in (-17,\y1) node (v3) {$\bullet$} -- (13,\y1);
\draw[thick, dotted] let \p1=(t000011) in (-17,\y1) node (v4) {$\bullet$} -- (13,\y1);

\node[left, gray] at (v02) {$v_{-2}$};
\node[left, gray] at (v01) {$v_{-1}$};
\node[left] at (v0) {$v_0$};
\node[left] at (v1) {$v_1$};
\node[left] at (v2) {$v_2$};
\node[left] at (v3) {$v_3$};
\node[left] at (v4) {$v_4$};

\draw[gray] (v02.center) to (v01.center) to (v0.center) to [bend right] (v02.center);
\draw[thick] (v0.center) to (v1.center) ;
\draw[gray] (v1.center) to [bend right] (v01.center);
\draw[thick] (v1.center) to (v2.center)to [bend left] (v0.center) ;
\draw[thick] (v2.center) to (v3.center)to [bend left] (v1.center) ;
\draw[gray] (v3.center) to [bend right] (v01.center);
\draw[gray] (v3.center) to [bend right] (v02.center);
\draw[thick] (v3.center) to (v4.center)to [bend left] (v2.center) ;

\end{tikzpicture}
\caption{A strong $K_4$-free tree $\bS_4$ densely coding $\mathcal{H}_4$}
\end{figure}
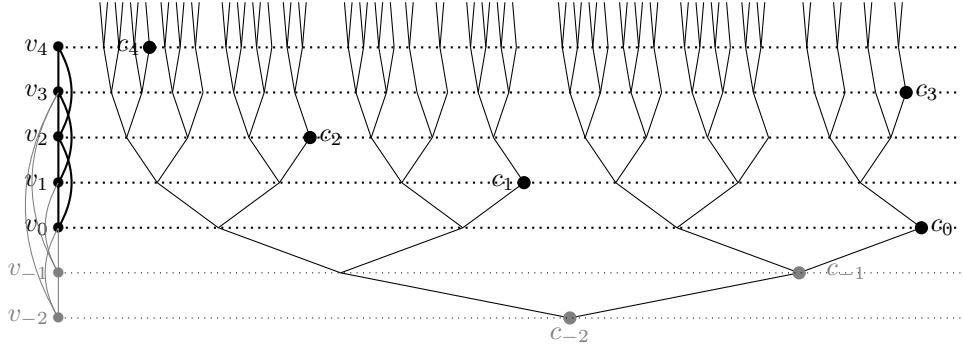



\end{example}

As in the case of $\mathcal{H}_3$ in \cite{DobrinenJML20},
the purpose  of not allowing coding nodes to split is to
reduce the   number of different  types of trees coding  a given finite $K_k$-free graph.
Having the coding nodes be dense in the tree  is necessary  development of Ramsey theory.
However,
the same example of a bad coloring as given in Example 3.18 of \cite{DobrinenJML20}
provides
 a bad coloring for any strong $K_k$-free tree, for any $k\ge 3$.

 \begin{example}[A bad coloring of vertices in $\mathcal{H}_k$]\label{ex.badcoloring}
 Let $k\ge 3$ be fixed.
 Color all coding nodes in $\bS_k$  extending $\lgl 1\rgl$ red.
 In particular, $c_0=\lgl 1\rgl$ is colored red.
 Given $n\ge 0$, suppose that for each $i\le n$, all coding nodes in $\bS_k$ extending ${0^{(i)}}^{\frown}1$ have been colored either red or blue.
  Look at the coding node $c_{n}$.
  This node has length $n+1$ and has already been assigned a color.
 If $c_{n}$ is red, then color every coding node extending
 ${0^{(n+1)}}^{\frown}1$ blue;
if  $c_{n}$ is blue, then color every coding node extending
 ${0^{(n+1)}}^{\frown}1$ red.
 Notice that any subtree of $\bS_k$ which is strongly similar to $\bS_k$ in the sense of Definition \ref{def.3.1.Sauer}
 where additionally coding nodes are sent to coding nodes, has nodes of both colors.
 (See Definition \ref{def.3.1.likeSauer} for the precise definition of {\em strongly similar} for trees with coding nodes.)
 Equivalently, any subtree of $\bS_k$ which is again a strong $K_k$-tree which represents a copy of $\mathcal{H}_k$ has coding nodes with both colors.
 \end{example}

 This coloring in  Example \ref{ex.badcoloring}  is equivalent to a coloring of the vertices of $\mathcal{H}_k$.
Recall that the work of Komj\'{a}th and \Rodl\ in \cite{Komjath/Rodl86}  for $k=3$ and work of El-Zahar and Sauer in \cite{El-Zahar/Sauer89}  for $k\ge 4$, shows that  for any  coloring of the vertices in
 $\mathcal{H}_k$  into  two colors, there is a subgraph $\mathcal{H}'$ which is again a $K_k$-free Henson graph in which all vertices have the same color.
 However, the previous example shows that we cannot expect the subgraph $\mathcal{H}'$ to have induced tree with coding nodes
 $T_{\mathcal{H}'}$ strongly similar to $\bS_k$.
 Since we are aiming to prove Ramsey theorems on collections of trees with coding nodes which are all strongly similar to each other,
  we immediately turn to the next section where we present  the skewed version of these trees  on which  the relevant Ramsey theory can be developed.


\section{Strong $\mathcal{H}_k$-coding trees}\label{sec.4}

The classes of trees coding Henson graphs on which we develop Ramsey theory are presented in this section.
For each $k\ge 3$,
fixing a   tree  $\bS_k$  constructed as  in Example \ref{thm.stftree}, we construct its skew version, denoted $\bT_k$, the skewing being necessary to avoid the  bad colorings seen  in Example
 \ref{ex.badcoloring}.
The coding nodes in $\bT_k$ code a $k$-clique-free Henson graph in the same way as the coding nodes of $\bS_k$.
 In Definition \ref{defn.T_pspace}, we present the space of strong $\mathcal{H}_k$-coding subtrees of $\bT_k$.
These are subtrees of $\bT_k$ which are isomorphic to $\bT_k$ in a strong way, and consequently code a copy of $\mathcal{H}_k$ in the same way that $\bT_k$ does.

By the end of Section \ref{sec.1SPOC},
these spaces of strong $\mathcal{H}_k$-coding trees will be shown to
have  Ramsey theorems similar to  the Milliken space of strong trees \cite{Milliken79}.
The added
 difficulty for $k>3$ 
 will be seen and addressed  from here throughout the rest of the paper. 
This section extends results of Section 4 in \cite{DobrinenJML20} to $\mathcal{H}_k$ for all $k\ge 4$, while providing  a new, more streamlined approach for the $k=3$ case.


\subsection{Definitions, notation, and  maximal strong $\mathcal{H}_k$-coding trees, $\bT_k$}\label{subsection.T_p}

The following terminology and notation will be used throughout, some of which is recalled from previous sections for  ease of reading. 
A subset $X\sse \Seq$ is a {\em level set} if all nodes in $X$ have the same length.
We continue to use the notions of {\em tree}
and {\em tree with coding nodes}
given  in Definitions  \ref{defn.tree}
 and
\ref{defn.treewcodingnodes}, respectively,  augmented to include {\em ghost coding nodes}, as was the case in the definition of $\bS_k$ in Example \ref{thm.stftree}. 
Recall  from equation (\ref{eq.That}) (just after Definition \ref{defn.tree})
that  for a tree $T$, we define
 $\widehat{T}$ to be the tree of all  nodes in  $ s\in \Seq$   
for which there is some $t\in T$ such that $s\sse t$.

Let  $T\sse \Seq$  be a finite or infinite  tree with coding nodes
$\lgl c^T_n:n\in N\rgl$, where either $N\in\mathbb{N}$ or $N=\mathbb{N}$.
If $T$ is to be a strong $\mathcal{H}_k$-coding tree, then  $T$ will also have 
ghost coding nodes $\lgl c^T_{-k+2},\dots,c^T_{-1}\rgl$. 
We  let  $l^T_n$ denote  $|c^T_n|$, the length of $c^T_n$.
Recall that $|c^T_n|$ is  the domain of $c^T_n$, as the sequence $c^T_n$ is a  function from some natural number into $\{0,1\}$.
We sometimes drop the superscript $T$ when it is clear from the context.
A node $s\in T$  is called  a {\em splitting node}  if
both $s^{\frown}0$ and $s^{\frown}1$  are in $\widehat{T}$;
equivalently,  $s$ is a splitting node in $T$ if there are nodes $s_0,s_1\in T$ such that
$s_0\contains s^{\frown}0$ and $s_1\contains s^{\frown}1$.
The {\em critical nodes} of $T$ is the set of all splitting and coding nodes, as well as any ghost coding nodes of $T$.
Given $t$ in  $T$,
let $T\re |t|$ denote
the set of all $s\in T$ such that $|s|=|t|$, and
 call $T\re |t|$ a {\em level} of $T$.

 We will say that a  tree
 $T$   is {\em skew} if
each level of $T$ has  exactly one of either a  coding node or a splitting node.
The set of {\em levels} of a skew tree $T\sse \Seq$, denoted $L(T)$, is the set  of those $l\in\mathbb{N}$ such that $T$ has either a splitting or a coding node
of length $l$.
A skew tree
$T$ is {\em strongly skew} if  additionally
for each splitting node $s\in T$,
every
 $t\in T$ such that $|t|>|s|$ and  $t\not\supset s$ also
satisfies
$t(|s|)=0$;
that is, the passing number  of any node  passing by, but not extending, a splitting node
is $0$.

Given a skew  tree $T$,
we let $\lgl d^T_m:m\in M\rgl$ denote 
the enumeration of  all   critical  nodes of $T$ in increasing order; 
$d^T_0$ will be the stem of $T$,  that is, the first splitting node of $T$.
Appropriating the standard notation  for Milliken's strong trees,
for each $m\in M$,
the {\em $m$-th level of $T$} is
\begin{equation}\label{eq.T(m)}
T(m)=\{s\in T:|s|=|d^T_m|\}.
\end{equation}
Then for any skew tree $T$,
\begin{equation}
T=\bigcup_{m\in M} T(m).
 \end{equation}
 For $m\in M$, the {\em $m$-th approximation of $T$} is defined to be
 \begin{equation}\label{eq.r_m(T)}
r_m(T)=\bigcup_{j<m}T(j).
 \end{equation}
Let $m_n$ denote the integer
such that $c^T_n\in T(m_n)$; thus,
 $d^T_{m_n}=c^T_n$.

For $n\ge 1$,
the {\em $n$-th interval} of $T$
is $\bigcup\{T(j): m_{n-1}<j\le m_n\}$; 
the {\em interval of $T$ between $c^T_{n-1}$ and $c^T_n$} is 
 $\bigcup\{T(j): m_{n-1}<j < m_n\}$.
The {\em $0$-th interval} of $T$ is 
 $\bigcup\{T(j): j\le m_0\}$.
 We call a skew tree $T$ {\em regular} if for each $n\in N$,
the lengths of the splitting nodes in the $n$-th interval of $T$ increase as their lexicographic order decreases.


In contrast to our approach in  \cite{DobrinenJML20} where we defined strong $\mathcal{H}_3$-coding trees via several structural properties,
in this paper we shall construct a  particular strong
$\mathcal{H}_k$-coding tree $\bT_k$ and then define a subtree to be a strong $\mathcal{H}_k$-coding tree if it is isomorphic to $\bT_k$ in a strong sense,  to be  made precise in  Definition
\ref{defn.T_pspace}.
The coding structure of $\bT_k$ is the same as that of the strong $K_k$-free tree $\bS_k$ given in Example \ref{thm.stftree}.  The best way to think about $\bT_k$ is that it is simply the strongly skew, regular version of $\bS_k$.


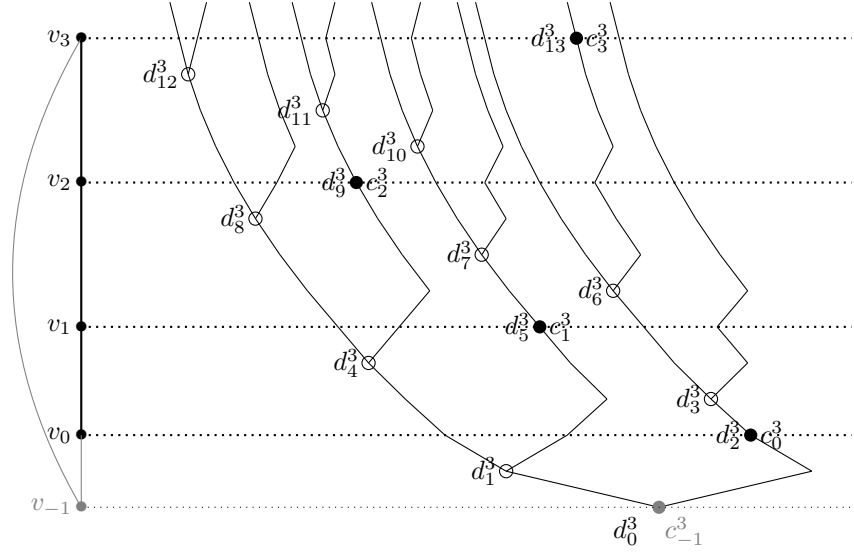
\begin{figure}\label{fig.bT}
\begin{tikzpicture}[grow'=up,scale=.32]
\tikzstyle{level 1}=[sibling distance=5in]
\tikzstyle{level 2}=[sibling distance=2in]
\tikzstyle{level 3}=[sibling distance=1.3in]
\tikzstyle{level 4}=[sibling distance=1.2in]
\tikzstyle{level 5}=[sibling distance=1in]
\tikzstyle{level 6}=[sibling distance=1in]
\tikzstyle{level 7}=[sibling distance=.9in]
\tikzstyle{level 8}=[sibling distance=.8in]
\tikzstyle{level 9}=[sibling distance=.7in]
\tikzstyle{level 10}=[sibling distance=.6in]
\tikzstyle{level 11}=[sibling distance=.5in]
\tikzstyle{level 12}=[sibling distance=.4in]
\tikzstyle{level 13}=[sibling distance=.3in]
\tikzstyle{level 14}=[sibling distance=.3in]
\node {} coordinate(t)
child{coordinate (t0)
			child{coordinate (t00)
child{coordinate (t000)
child{coordinate (t0000)
child{coordinate (t00000)
child{coordinate (t000000)
child{coordinate (t0000000)
child{coordinate (t00000000)
child{coordinate (t000000000)
child{coordinate (t0000000000)
child{coordinate (t00000000000)
child{coordinate (t000000000000)
child{coordinate (t0000000000000)
child{coordinate (t00000000000000)}
child{edge from parent[draw=none] coordinate (t00000000000001)}
}
child{coordinate (t0000000000001)
child{edge from parent[draw=none]  coordinate (t00000000000010)}
child{coordinate (t00000000000011)}
}
}
child{edge from parent[draw=none] coordinate (t000000000001)}
}
child{edge from parent[draw=none]  coordinate (t00000000001)}
}
child{edge from parent[draw=none]  coordinate (t0000000001)}
}
child{coordinate (t000000001)
child{edge from parent[draw=none]  coordinate (t0000000010)}
child{coordinate (t0000000011)
child{coordinate (t00000000110)
child{coordinate (t000000001100)
child{coordinate (t0000000011000)
child{coordinate (t00000000110000)}
child{edge from parent[draw=none] coordinate (t00000000110001)}
}
child{edge from parent[draw=none] coordinate (t0000000011001)}
}
child{edge from parent[draw=none]  coordinate (t000000001101)}
}
child{edge from parent[draw=none]   coordinate (t00000000111)}
}
}
}
child{edge from parent[draw=none] coordinate (t00000001)}
}
child{edge from parent[draw=none] coordinate (t0000001)}
}
child{edge from parent[draw=none]  coordinate (t000001)}
}
child{coordinate (t00001)
child{edge from parent[draw=none]  coordinate (t000010)}
child{coordinate (t000011)
child{coordinate (t0000110)
child{coordinate (t00001100)
child{coordinate (t000011000)
child{coordinate (t0000110000)
child{coordinate (t00001100000)
child{coordinate (t000011000000)
child{coordinate (t0000110000000)
child{coordinate (t00001100000000)}
child{edge from parent[draw=none] coordinate (t00001100000001)}
}
child{edge from parent[draw=none] coordinate (t0000110000001)}
}
child{coordinate (t000011000001)
child{coordinate (t0000110000010)
child{edge from parent[draw=none]  coordinate (t00001100000100)}
child{coordinate (t00001100000101)}
}
child{edge from parent[draw=none] coordinate (t0000110000011)}
}
}
child{edge from parent[draw=none] coordinate (t00001100001)}
}
child{edge from parent[draw=none]  coordinate (t0000110001)}
}
child{edge from parent[draw=none] coordinate (t000011001)}
}
child{edge from parent[draw=none]  coordinate (t00001101)}
}
child{edge from parent[draw=none]  coordinate (t0000111)}
}
}
}
child{edge from parent[draw=none] coordinate (t0001)}
}
child{ edge from parent[draw=none]coordinate(t001)}}
			child{ coordinate(t01)
child{ edge from parent[draw=none] coordinate(t010)}
child{coordinate(t011)
child{coordinate(t0110)
child{coordinate(t01100)
child{coordinate(t011000)
child{coordinate(t0110000)
child{coordinate(t01100000)
child{coordinate(t011000000)
child{coordinate(t0110000000)
child{coordinate(t01100000000)
child{coordinate(t011000000000)
child{coordinate(t0110000000000)
child{coordinate(t01100000000000)}
child{edge from parent[draw=none] coordinate(t01100000000001)}
}
child{ edge from parent[draw=none]coordinate(t0110000000001)}
}
child{ edge from parent[draw=none]  coordinate(t011000000001)}
}
child{coordinate(t01100000001)
child{coordinate(t011000000010)
child{coordinate(t0110000000100)
child{edge from parent[draw=none] coordinate(t01100000001000)}
child{coordinate(t01100000001001)}
}
child{edge from parent[draw=none] coordinate(t0110000000101)}
}
child{edge from parent[draw=none]   coordinate(t011000000011)}
}
}
child{ edge from parent[draw=none]  coordinate(t0110000001)}
}
child{ edge from parent[draw=none] coordinate(t011000001)}
}
child{coordinate(t01100001)
child{coordinate(t011000010)
child{edge from parent[draw=none] coordinate(t0110000100)}
child{coordinate(t0110000101)
child{coordinate(t01100001010)
child{coordinate(t011000010100)
child{coordinate(t0110000101000)
child{coordinate(t01100001010000)}
child{edge from parent[draw=none] coordinate(t01100001010001)}
}
child{edge from parent[draw=none]  coordinate(t0110000101001)}
}
child{edge from parent[draw=none]  coordinate(t011000010101)}
}
child{edge from parent[draw=none] coordinate(t01100001011)}
}
}
child{edge from parent[draw=none]  coordinate(t011000011)}
}
}
child{ edge from parent[draw=none] coordinate(t0110001)}
}
child{ edge from parent[draw=none]  coordinate(t011001)}
}
child{ edge from parent[draw=none]  coordinate(t01101)}
}
child{ edge from parent[draw=none]  coordinate(t0111)}
}}}
		child{ coordinate(t1)
			child{ coordinate(t10)
child{ coordinate(t100)
child{ coordinate(t1000)
child{ coordinate(t10000)
child{ coordinate(t100000)
child{ coordinate(t1000000)
child{ coordinate(t10000000)
child{ coordinate(t100000000)
child{ coordinate(t1000000000)
child{ coordinate(t10000000000)
child{ coordinate(t100000000000)
child{ coordinate(t1000000000000)
child{ coordinate(t10000000000000)}
child{ edge from parent[draw=none]   coordinate(t10000000000001)}
}
child{edge from parent[draw=none]   coordinate(t1000000000001)}
}
child{ edge from parent[draw=none]   coordinate(t100000000001)}
}
child{edge from parent[draw=none]  coordinate(t10000000001)}
}
child{ edge from parent[draw=none]  coordinate(t1000000001)}
}
child{edge from parent[draw=none]  coordinate(t100000001)}
}
child{   edge from parent[draw=none] coordinate(t10000001)}
}
child{ coordinate(t1000001)
child{ coordinate(t10000010)
child{ coordinate(t100000100)
child{ edge from parent[draw=none]  coordinate(t1000001000)}
child{ coordinate(t1000001001)
child{ coordinate(t10000010010)
child{ coordinate(t100000100100)
child{ coordinate(t1000001001000)
child{ coordinate(t10000010010000)}
child{ edge from parent[draw=none] coordinate(t10000010010001)}
}
child{edge from parent[draw=none]   coordinate(t1000001001001)}
}
child{edge from parent[draw=none]  coordinate(t100000100101)}
}
child{ edge from parent[draw=none] coordinate(t10000010011)}
}
}
child{  edge from parent[draw=none] coordinate(t100000101)}
}
child{   edge from parent[draw=none] coordinate(t10000011)}
}
}
child{  edge from parent[draw=none] coordinate(t100001)}
}
child{ edge from parent[draw=none]  coordinate(t10001)}
}
child{ coordinate(t1001)
child{ coordinate(t10010)
child{ edge from parent[draw=none]   coordinate(t100100)}
child{ coordinate(t100101)
child{ coordinate(t1001010)
child{ coordinate(t10010100)
child{ coordinate(t100101000)
child{ coordinate(t1001010000)
child{ coordinate(t10010100000)
child{ coordinate(t100101000000)
child{ coordinate(t1001010000000)
child{ coordinate(t10010100000000)}
child{edge from parent[draw=none] coordinate(t10010100000001)}
}
child{ edge from parent[draw=none]coordinate(t1001010000001)}
}
child{edge from parent[draw=none]  coordinate(t100101000001)}
}
child{edge from parent[draw=none]  coordinate(t10010100001)}
}
child{ edge from parent[draw=none]  coordinate(t1001010001)}
}
child{ edge from parent[draw=none] coordinate(t100101001)}
}
child{edge from parent[draw=none]   coordinate(t10010101)}
}
child{ edge from parent[draw=none]   coordinate(t1001011)}
}
}
child{  edge from parent[draw=none] coordinate(t10011)}
}
}
child{edge from parent[draw=none]  coordinate(t101)}}
child{edge from parent[draw=none]   coordinate(t11)
} };

\node[left] at (t0) {$d^3_1$};
\node[right] at (t10) {$c^3_0$};
\node[left] at (t10) {$d^3_2$};
\node[left] at (t100) {$d^3_3$};
\node[left] at (t0000) {$d^3_4$};
\node[left] at (t01100) {$d^3_5$};
\node[right] at (t01100) {$c^3_1$};
\node[left] at (t100000) {$d^3_6$};
\node[left] at (t0110000) {$d^3_7$};
\node[left] at (t00000000) {$d^3_8$};
\node[right] at (t000011000) {$c^3_2$};
\node[left] at (t000011000) {$d^3_9$};
\node[left] at (t0110000000) {$d^3_{10}$};
\node[left] at (t00001100000) {$d^3_{11}$};
\node[left] at (t000000000000) {$d^3_{12}$};
\node[left] at (t1000001001000) {$d^3_{13}$};
\node[right] at (t1000001001000)  {$c^3_3$};

\node[circle, draw,inner sep=0pt, minimum size=5pt] at (t0)  {};
\node[circle, draw,inner sep=0pt, minimum size=5pt] at (t100)  {};
\node[circle, draw,inner sep=0pt, minimum size=5pt] at (t0000)  {};
\node[circle, draw,inner sep=0pt, minimum size=5pt] at (t100000)  {};
\node[circle, draw,inner sep=0pt, minimum size=5pt] at (t0110000)  {};
\node[circle, draw,inner sep=0pt, minimum size=5pt] at (t00000000)  {};
\node[circle, draw,inner sep=0pt, minimum size=5pt] at (t0110000000)  {};
\node[circle, draw,inner sep=0pt, minimum size=5pt] at (t00001100000)  {};
\node[circle, draw,inner sep=0pt, minimum size=5pt] at (t000000000000)  {};

\node[below] at (t) {$d^3_0\ \ \ {\color{gray}c^3_{-1}}$};
\node[circle, fill=gray,inner sep=0pt, minimum size=5pt] at (t) {};
\node[circle, fill=black,inner sep=0pt, minimum size=5pt] at (t10) {};
\node[circle, fill=black,inner sep=0pt, minimum size=5pt] at (t01100) {};
\node[circle, fill=black,inner sep=0pt, minimum size=5pt] at (t000011000) {};
\node[circle, fill=black,inner sep=0pt, minimum size=5pt] at (t1000001001000) {};

\draw[dotted] let \p1=(t) in (-24,\y1) node (v01) {$\color{gray}{\bullet}$} -- (8,\y1);
\draw[thick, dotted] let \p1=(t10) in (-24,\y1) node (v0) {$\bullet$} -- (8,\y1);
\draw[thick, dotted] let \p1=(t01100) in (-24,\y1) node (v1) {$\bullet$} -- (8,\y1);
\draw[thick, dotted] let \p1= (t000011000) in (-24,\y1) node (v2) {$\bullet$} -- (8,\y1);
\draw[thick, dotted] let \p1=  (t1000001001000) in (-24,\y1) node (v3) {$\bullet$} -- (8,\y1);

\node[left, gray] at (v01) {$v_{-1}$};
\node[left] at (v0) {$v_0$};
\node[left] at (v1) {$v_1$};
\node[left] at (v2) {$v_2$};
\node[left] at (v3) {$v_3$};

\draw[gray] (v0.center) to (v01.center)  to [bend left] (v3.center);
\draw[thick] (v0.center) to (v1.center) to (v2.center) to (v3.center);

\end{tikzpicture}
\caption{Strong $\mathcal{H}_3$-coding tree $\bT_3$}
\end{figure}


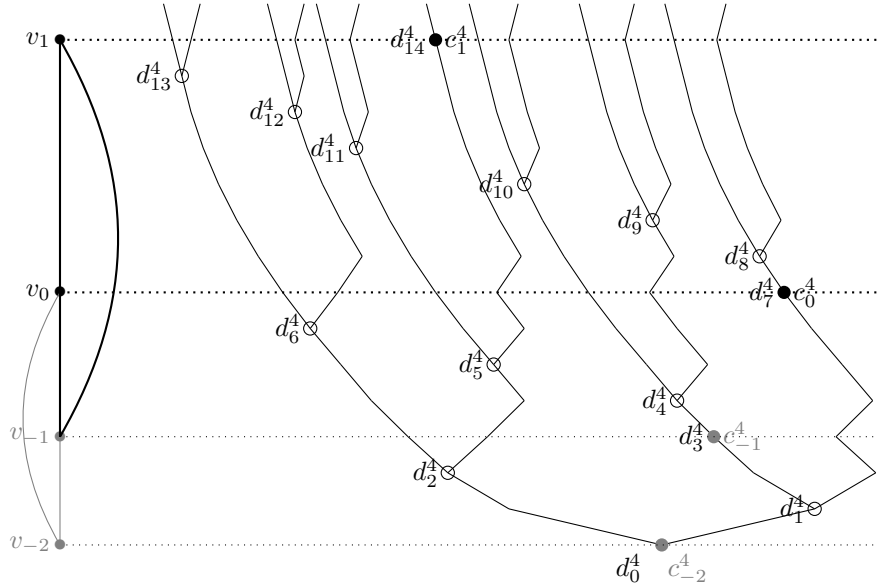
\begin{figure}\label{fig.bT4}
\begin{tikzpicture}[grow'=up,scale=.32]
\tikzstyle{level 1}=[sibling distance=5in]
\tikzstyle{level 2}=[sibling distance=2in]
\tikzstyle{level 3}=[sibling distance=1.3in]
\tikzstyle{level 4}=[sibling distance=1.2in]
\tikzstyle{level 5}=[sibling distance=1in]
\tikzstyle{level 6}=[sibling distance=1in]
\tikzstyle{level 7}=[sibling distance=.9in]
\tikzstyle{level 8}=[sibling distance=.8in]
\tikzstyle{level 9}=[sibling distance=.7in]
\tikzstyle{level 10}=[sibling distance=.6in]
\tikzstyle{level 11}=[sibling distance=.5in]
\tikzstyle{level 12}=[sibling distance=.4in]
\tikzstyle{level 13}=[sibling distance=.3in]
\tikzstyle{level 14}=[sibling distance=.3in]
\tikzstyle{level 15}=[sibling distance=.3in]
\node {} coordinate(t)
child{coordinate(t0)
child{coordinate(t00)
child{coordinate(t000)
child{coordinate(t0000)
child{coordinate(t00000)
child{coordinate(t000000)
child{coordinate(t0000000)
child{coordinate(t00000000)
child{coordinate(t000000000)
child{coordinate(t0000000000)
child{coordinate(t00000000000)
child{coordinate(t000000000000)
child{coordinate(t0000000000000)
child{coordinate(t00000000000000)
child{coordinate(t000000000000000)}child{coordinate(t000000000000001)edge from parent[draw=none]}
}
child{coordinate(t00000000000001)
child{coordinate(t000000000000010)edge from parent[draw=none]}
child{coordinate(t000000000000011)}
}}
child{coordinate(t0000000000001)edge from parent[draw=none]}
}
child{coordinate(t000000000001)edge from parent[draw=none]}
}
child{coordinate(t00000000001)
edge from parent[draw=none]}
}
child{coordinate(t0000000001)edge from parent[draw=none]}
}
child{coordinate(t000000001)edge from parent[draw=none]}
}
child{coordinate(t00000001)edge from parent[draw=none]}}
child{coordinate(t0000001)
child{coordinate(t00000010)edge from parent[draw=none]}
child{coordinate(t00000011)
child{coordinate(t000000110)
child{coordinate(t0000001100)
child{coordinate(t00000011000)
child{coordinate(t000000110000)
child{coordinate(t0000001100000)
child{coordinate(t00000011000000)
child{coordinate(t000000110000000)}
child{coordinate(t000000110000001)edge from parent[draw=none]}
}
child{coordinate(t00000011000001)edge from parent[draw=none]}
}
child{coordinate(t0000001100001)
child{coordinate(t00000011000010)
child{coordinate(t000000110000100)edge from parent[draw=none]}
child{coordinate(t000000110000101)}
}
child{coordinate(t00000011000011)edge from parent[draw=none]}
}}
child{coordinate(t000000110001)edge from parent[draw=none]}
}
child{coordinate(t00000011001)edge from parent[draw=none]}
}
child{coordinate(t0000001101)
edge from parent[draw=none]}
}
child{coordinate(t000000111)edge from parent[draw=none]}
}}}
child{coordinate(t000001)edge from parent[draw=none]}}
child{coordinate(t00001)edge from parent[draw=none]}
}
child{coordinate(t0001)edge from parent[draw=none]}
}
child{coordinate(t001)
child{coordinate(t0010) edge from parent[draw=none]}
child{coordinate(t0011)
child{coordinate(t00110)
child{coordinate(t001100)
child{coordinate(t0011000)
child{coordinate(t00110000)
child{coordinate(t001100000)
child{coordinate(t0011000000)
child{coordinate(t00110000000)
child{coordinate(t001100000000)
child{coordinate(t0011000000000)
child{coordinate(t00110000000000)
child{coordinate(t001100000000000)}
child{coordinate(t001100000000001)edge from parent[draw=none]}
}
child{coordinate(t00110000000001)edge from parent[draw=none]}
}
child{coordinate(t0011000000001)edge from parent[draw=none]}
}
child{coordinate(t001100000001)
child{coordinate(t0011000000010)
child{coordinate(t00110000000100)
child{coordinate(t001100000001000)edge from parent[draw=none]}
child{coordinate(t001100000001001)}
}
child{coordinate(t00110000000101)edge from parent[draw=none]}
}
child{coordinate(t0011000000011)edge from parent[draw=none]}
}}
child{coordinate(t00110000001)edge from parent[draw=none]}
}
child{coordinate(t0011000001)edge from parent[draw=none]}
}
child{coordinate(t001100001)edge from parent[draw=none]}
}
child{coordinate(t00110001)edge from parent[draw=none]}
}
child{coordinate(t0011001)edge from parent[draw=none]}
}
child{coordinate(t001101)
child{coordinate(t0011010)
child{coordinate(t00110100)edge from parent[draw=none]}
child{coordinate(t00110101)
child{coordinate(t001101010)
child{coordinate(t0011010100)
child{coordinate(t00110101000)
child{coordinate(t001101010000)
child{coordinate(t0011010100000)
child{coordinate(t00110101000000)
child{coordinate(t001101010000000)}
child{coordinate(t001101010000001)edge from parent[draw=none]}
}
child{coordinate(t00110101000001)edge from parent[draw=none]}
}
child{coordinate(t0011010100001)edge from parent[draw=none]}
}
child{coordinate(t001101010001)edge from parent[draw=none]}
}
child{coordinate(t00110101001)edge from parent[draw=none]}
}
child{coordinate(t0011010101)edge from parent[draw=none]}
}
child{coordinate(t001101011)edge from parent[draw=none]}
}}
child{coordinate(t0011011)edge from parent[draw=none]}
}}
child{coordinate(t00111)edge from parent[draw=none]}
}}}
child{coordinate(t01) edge from parent[draw=none]}
}
child{coordinate(t1)
child{coordinate(t10)
child{coordinate(t100)
child{coordinate(t1000)
child{coordinate(t10000)
child{coordinate(t100000)
child{coordinate(t1000000)
child{coordinate(t10000000)
child{coordinate(t100000000)
child{coordinate(t1000000000)
child{coordinate(t10000000000)
child{coordinate(t100000000000)
child{coordinate(t1000000000000)
child{coordinate(t10000000000000)
child{coordinate(t100000000000000)}
child{coordinate(t100000000000001)edge from parent[draw=none]}
}
child{coordinate(t10000000000001)edge from parent[draw=none]}
}
child{coordinate(t1000000000001)edge from parent[draw=none]}
}
child{coordinate(t100000000001)edge from parent[draw=none]}
}
child{coordinate(t10000000001)
child{coordinate(t100000000010)
child{coordinate(t1000000000100)
child{coordinate(t10000000001000)
child{coordinate(t100000000010000)edge from parent[draw=none]}
child{coordinate(t100000000010001)}
}
child{coordinate(t10000000001001)edge from parent[draw=none]}
}
child{coordinate(t1000000000101)edge from parent[draw=none]}
}
child{coordinate(t100000000011)edge from parent[draw=none]}
}}
child{coordinate(t1000000001)edge from parent[draw=none]}
}
child{coordinate(t100000001)edge from parent[draw=none]}
}
child{coordinate(t10000001)edge from parent[draw=none]}
}
child{coordinate(t1000001)edge from parent[draw=none]}
}
child{coordinate(t100001)edge from parent[draw=none]}
}
child{coordinate(t10001)
child{coordinate(t100010)
child{coordinate(t1000100)
child{coordinate(t10001000)edge from parent[draw=none]}
child{coordinate(t10001001)
child{coordinate(t100010010)
child{coordinate(t1000100100)
child{coordinate(t10001001000)
child{coordinate(t100010010000)
child{coordinate(t1000100100000)
child{coordinate(t10001001000000)
child{coordinate(t100010010000000)}
child{coordinate(t100010010000001)edge from parent[draw=none]}
}
child{coordinate(t10001001000001)edge from parent[draw=none]}
}
child{coordinate(t1000100100001)edge from parent[draw=none]}
}
child{coordinate(t100010010001)edge from parent[draw=none]}
}
child{coordinate(t10001001001)edge from parent[draw=none]}
}
child{coordinate(t1000100101)
child{coordinate(t10001001010)
child{coordinate(t100010010100)
child{coordinate(t1000100101000)
child{coordinate(t10001001010000)
child{coordinate(t100010010100000)edge from parent[draw=none]}
child{coordinate(t100010010100001)}
}
child{coordinate(t10001001010001)edge from parent[draw=none]}
}
child{coordinate(t1000100101001)edge from parent[draw=none]}
}
child{coordinate(t100010010101)edge from parent[draw=none]}
}
child{coordinate(t10001001011)edge from parent[draw=none]}
}}
child{coordinate(t100010011)edge from parent[draw=none]}
}}
child{coordinate(t1000101)edge from parent[draw=none]}
}
child{coordinate(t100011)edge from parent[draw=none]}
}}
child{coordinate(t1001)edge from parent[draw=none]}
}
child{coordinate(t101)edge from parent[draw=none]}
}
child{coordinate(t11)
child{coordinate(t110)
child{coordinate(t1100)edge from parent[draw=none]}
child{coordinate(t1101)
child{coordinate(t11010)
child{coordinate(t110100)
child{coordinate(t1101000)
child{coordinate(t11010000)
child{coordinate(t110100000)
child{coordinate(t1101000000)
child{coordinate(t11010000000)
child{coordinate(t110100000000)
child{coordinate(t1101000000000)
child{coordinate(t11010000000000)
child{coordinate(t110100000000000)}
child{coordinate(t110100000000001)edge from parent[draw=none]}
}
child{coordinate(t11010000000001)edge from parent[draw=none]}
}
child{coordinate(t1101000000001)edge from parent[draw=none]}
}
child{coordinate(t110100000001)edge from parent[draw=none]}
}
child{coordinate(t11010000001)edge from parent[draw=none]}
}
child{coordinate(t1101000001)edge from parent[draw=none]}
}
child{coordinate(t110100001)
child{coordinate(t1101000010)
child{coordinate(t11010000100)
child{coordinate(t110100001000)
child{coordinate(t1101000010000)
child{coordinate(t11010000100000)
child{coordinate(t110100001000000)edge from parent[draw=none]}
child{coordinate(t110100001000001)}
}
child{coordinate(t11010000100001)edge from parent[draw=none]}
}
child{coordinate(t1101000010001)edge from parent[draw=none]}
}
child{coordinate(t110100001001)edge from parent[draw=none]}
}
child{coordinate(t11010000101)edge from parent[draw=none]}
}
child{coordinate(t1101000011)edge from parent[draw=none]}
}}
child{coordinate(t11010001)edge from parent[draw=none]}
}
child{coordinate(t1101001)edge from parent[draw=none]}
}
child{coordinate(t110101)edge from parent[draw=none]}
}
child{coordinate(t11011)edge from parent[draw=none]}
}}
child{coordinate(t111)edge from parent[draw=none]}
}}
;

\node[below] at (t) {$d^4_0\ \ \ {\color{gray}c^4_{-2}}$};
\node[left] at (t1) {$d^4_1$};
\node[left] at (t00) {$d^4_2$};
\node[left] at (t100) {$d^4_3$};
\node[right] at (t100) {${\color{gray}c^4_{-1}}$};
\node[left] at (t1000) {$d^4_4$};
\node[left] at (t00110) {$d^4_5$};
\node[left] at (t000000) {$d^4_6$};
\node[left] at (t1101000) {$d^4_7$};
\node[right] at (t1101000) {$c^4_0$};
\node[left] at (t11010000) {$d^4_8$};
\node[left] at (t100010010) {$d^4_9$};
\node[left] at (t1000000000) {$d^4_{10}$};
\node[left] at (t00110000000) {$d^4_{11}$};
\node[left] at (t000000110000) {$d^4_{12}$};
\node[left] at (t0000000000000) {$d^4_{13}$};
\node[right] at (t00110101000000) {$c^4_1$};
\node[left] at (t00110101000000) {$d^4_{14}$};

\node[circle, fill=gray,inner sep=0pt, minimum size=5pt] at (t) {};
\node[circle, fill=gray,inner sep=0pt, minimum size=5pt] at (t100) {};
\node[circle, fill=black,inner sep=0pt, minimum size=5pt] at (t1101000) {};
\node[circle, fill=black,inner sep=0pt, minimum size=5pt] at (t00110101000000) {};

\node[circle, draw,inner sep=0pt, minimum size=5pt] at (t1)  {};
\node[circle, draw,inner sep=0pt, minimum size=5pt] at (t00)  {};
\node[circle, draw,inner sep=0pt, minimum size=5pt] at (t1000)  {};
\node[circle, draw,inner sep=0pt, minimum size=5pt] at (t00110)  {};
\node[circle, draw,inner sep=0pt, minimum size=5pt] at (t000000)  {};
\node[circle, draw,inner sep=0pt, minimum size=5pt] at (t11010000)  {};
\node[circle, draw,inner sep=0pt, minimum size=5pt] at (t100010010)  {};
\node[circle, draw,inner sep=0pt, minimum size=5pt] at (t1000000000)  {};
\node[circle, draw,inner sep=0pt, minimum size=5pt] at (t00110000000)  {};
\node[circle, draw,inner sep=0pt, minimum size=5pt] at (t000000110000)  {};
\node[circle, draw,inner sep=0pt, minimum size=5pt] at (t0000000000000)  {};

\draw[dotted] let \p1=(t) in (-25,\y1) node (v01) {$\color{gray}{\bullet}$} -- (9,\y1);
\draw[dotted] let \p1=(t100) in (-25,\y1) node (v0) {$\color{gray}{\bullet}$} -- (9,\y1);
\draw[thick, dotted] let \p1=(t1101000) in (-25,\y1) node (v1) {$\bullet$} -- (9,\y1);
\draw[thick, dotted] let \p1=(t00110101000000) in (-25,\y1) node (v2) {$\bullet$} -- (9,\y1);

\node[left, gray] at (v01) {$v_{-2}$};
\node[left, gray] at (v0) {$v_{-1}$};
\node[left] at (v1) {$v_0$};
\node[left] at (v2) {$v_1$};

\draw[gray] (v0.center) to (v01.center)  to [bend left] (v1.center);
\draw[thick] (v0.center) to (v1.center) to (v2.center) to [bend left] (v0.center) ;

\end{tikzpicture}
\caption{Strong $\mathcal{H}_4$-coding tree $\bT_4$}
\end{figure}



\begin{example}[Construction Method for  Maximal Strong $\mathcal{H}_k$-Coding Trees, $\bT_k$]\label{ex.bTp}

Fix $k\ge 3$, and let
 $\lgl u_i:i\in \mathbb{N}\rgl$ be any enumeration  of the nodes in $\Seq$ such that $|u_i|\le|u_j|$ whenever $i<j$.
 Let $\bS_k$ be a strong $K_k$-free tree  constructed in Example  \ref{thm.stftree}.
 Recall that the graph represented by the coding nodes $\lgl c^{\bS_k}_n:n\in\bN\rgl$  in
  $\bS_k$  is the $k$-clique-free Henson graph.
  (The graph represented by the coding nodes in $\bS_k$ along with the ghost coding nodes in $\bS_k$ is also a $k$-clique-free Henson graph, since  $(A)^{\tt tree}_k$ holds.)
Define  $\bT_k$ to  be  the tree 
obtained by stretching $\bS_k$ to make it strongly skew and regular while preserving the passing numbers so that $\bT_k$ represents the same  copy of the Henson graph $\mathcal{H}_k$ as $\bS_k$ does.
$\bT_k$ will have 
coding nodes  $\lgl c^k_n:n\in\mathbb{N}\rgl$ 
and ghost coding nodes $\lgl c^k_{-k+2},\dots, c^k_{-1}\rgl$.
For all pairs $j<n$, we will have
$c^k_n(|c^k_j|)=c^{\bS_k}_n(|c^{\bS_k}_j|)$.
The $m$-th critical node $d^k_m$ in $\bT_k$ will be a node of length $m$, so that
 $\bT_k$ will have nodes in every length in $\mathbb{N}$.
 The critical nodes consist of the splitting nodes, coding nodes, and ghost coding nodes. 
 In particular,  we will have
 $\widehat{\bT}_k=\bT_k$, and $\bT_k(m)=\bT_k\re m$ for each $m\in\bN$.

 We now show precisely how to construct $\bT_k$, given $\bS_k$.
Set $m_{-k+2}=0$ and $\bT_k(0)=\bS_k(0)$, which is the singleton $\{\lgl\rgl\}$.
 Let  the ghost coding node of $\bT_k$ be $c^k_{-k+2}=\lgl \rgl$.
 This node splits so that $\bT_k(1)=\bS_k(1)$, which has exactly two nodes,
 so
 there is a bijection between these level  sets of nodes.
 As in the previous section,
 let $r_m(\bT_k)$ denote $\bigcup_{j<m}\bT_k(j)$.
 Then 
 $r_2(\bT_k)=\bT_k(0)\cup\bT_k(1)$. 
 The $0$-th critical node $d^k_0$ equals the ghost coding node $c_{-k+2}^k$.

 Suppose that $n\ge -k+2$ and that  $r_{m_n+2}(\bT_k)$ has been constructed,
 $c^k_n\in \bT_k(m_n)$ has been fixed, and there is a bijection between
 $\bT_k(m_n+1)$ and $\bS_k(n+k-1)$.
 Let $J$ be the number of nodes in
 $\bS_k(n+k-1)$ which split into two extensions in $\bS_k(n+k)$.
  Let $m_{n+1}=m_n+J+1$.
 (Notice that
 for $k=3$,   $m_0=2$,  and
  for   $k=4$,  $m_{-1}=3$;
  more generally, for $k\ge 4$, $m_{-k+3}=3$.)
Let $\{\tilde{s}_j : m_n<j<m_{n+1}\}$ enumerate in reverse lexicographic order those members of
  $\bS_k(n+k-1)$ which split into two extensions in $\bS_k(n+k)$.
  Let $\varphi$ be the lexicographic preserving bijection from
    $\bS_k(n+k-1)$ onto
   $\bT_k(m_n+1)$.
For each     $m_n<j<m_{n+1}$, define $s_j=\varphi(\tilde{s}_j)$, and set
 \begin{equation}
 d^k_j={s_j}^{\frown}0^{(i_j)},
 \end{equation}
  where $i_j=j-m_n-1$.
 Thus, $ d^k_{m_n+1}=s_{m_n+1}$, $d^k_{m_n+2}={s_{m_n+2}}^{\frown}\lgl 0\rgl$, $d^k_{m_n+3}={s_{m_n+2}}^{\frown}\lgl 0,0\rgl$,
 and finally, $d^k_{m_{n+1}-1}={s_{m_{n+1}-1}}^{\frown} 0^{(J-1)}$.
 These are the splitting nodes in the interval of $\bT_k$ between $c^k_{n}$ and $c^k_{n+1}$.
For each   $i\in\{0,1\}$,
 define
 $t_j^i$
 to be the binary sequence of length
  $m_{n+1}$
  which  extends
 ${d^k_j}^{\frown}i$ by all $0$'s.
 Define
 \begin{equation}
 \bT_k(m_{n+1})=
 \{t_j^i: m_n<j<m_{n+1},\ i\in\{0,1\}\}\cup
 \{t_{\tilde{s}}:\tilde{s}\in S_{n+1}'\}
 \end{equation}
 where
  $S_{n+1}'=\bS_k(n+k)\setminus \{\tilde{s}_j:m_n<j<m_{n+1}\}$
  and for each $\tilde{s}\in S_{n+1}'$,
  $t_{\tilde{s}}$ is the
   extension
   of $\varphi(\tilde{s})$ by $0$'s to length $m_{n+1}$.
  Let $c^k_{n+1}$ be  the leftmost  extension    in   $\bT_k(m_{n+1})$ of $\varphi(\tilde{s}_c)$,
  where $\tilde{s}_c$
  is the leftmost immediate successor of
  $c_{n+1}^{\mathbb{S}_k}$
  in $\mathbb{S}_k(n+k-1)$.
 Define
  \begin{equation}
  \bT_k(m_{n+1}+1)=
 \{{t_j^i}^{\frown}i: m_n<j<m_{n+1},\ i\in\{0,1\}\}\cup
 \{{t_{\tilde{s}}}^{\frown}0:\tilde{s}\in S_{n+1}'\}.
 \end{equation}
 Notice that
 the level sets
 $\bS_k(n+k)$, $  \bT_k(m_{n+1})$,  and
 $  \bT_k(m_{n+1}+1)$  have the same cardinality.
 Moreover, 
  the lexicographic preserving bijection from
 $\bS_k(n+k)$ onto
 $  \bT_k(m_{n+1})$
 preserves passing numbers.
By our construction,  for each $m\in\mathbb{N}$, $|d^k_m|=m$.

 Let $\bT_k=\bigcup_{m\in\bN}\bT_k(m)$
 with coding nodes $\lgl c^k_n:n\in\bN\rgl$ and ghost coding nodes $\lgl c^k_{-k+2},\dots, c^k_{-1}\rgl$. 
$\bT_k$ is a    strongly skew, regular tree with coding nodes representing a copy of $\mathcal{H}_k$ and is {\em maximal} in the sense that
 $\widehat{\bT}_k=\bT_k$
\end{example}

\begin{rem}\label{rem.awesome!}
Notice that the
coding nodes (along with the  ghost coding nodes) in $\bT_k$  represent  the same ordered copy of $\mathcal{H}_k$ as the one represented by the coding nodes (along with the  ghost coding nodes)  in $\bS_k$.
That is, given $j<n$, $c^k_n(|c^k_j|)=c^{\mathbb{S}_k}_n(|c^{\mathbb{S}_k}_j|)$.
A simple  way to  think about $\bT_k$ is that it is the skew tree such that if  one ``zips up'' the splitting nodes in the $n$-th interval of $\bT_k$   to the length of $c^k_n$, then one recovers a tree which is strongly similar (in fact, strongly isomorphic--see Definition \ref{defn.stable})  to $\bS_k$, in the sense of the upcoming Definition \ref{def.3.1.likeSauer}.
\end{rem}


\subsection{The space of strong $\mathcal{H}_k$-coding trees}\label{subsec.T_k}

Let $k\ge 3$ be given, and fix  $\bT_k$ as constructed as in Example \ref{ex.bTp}.
In  preparation for defining  the space of strong $\mathcal{H}_k$-coding subtrees  of $\bT_k$,  we provide
the following  definitions.

A subset $A$ of $\bT_k$ is an {\em antichain} if
 for all $s,t\in A$,
$s\sse t$ implies $s=t$.
Recall that  a subset $X$ of $\bT_k$ is a {\em level set} if all members of $X$ have the same length.
Thus, each level set is an antichain.
Given a subset $S\sse \bT_k$,
recall that
 the {\em meet closure of $S$}, denoted $S^{\wedge}$,
is the set of all meets of (not necessarily distinct) pairs of nodes in $S$; in particular,  $S^{\wedge}$ contains $S$.
We say that  $S$ is {\em meet-closed} if $S=S^{\wedge}$.
Recall that for  $s\in\Seq$ and $l\in\mathbb{N}$ with $|s|\ge l$,
$s\re l$ denotes the initial segment of $s$ of length  $l$.
Similarly to  equation (\ref{eq.Trel}),
given a subset $S\sse \bT_k$ and any $l\in\mathbb{N}$,
we define
 \begin{equation}
 S\re l=\{s\re l:s\in S\mathrm{\ and\ }|s|\ge l\}.
 \end{equation}
That is, $S\re l$ is 
 the set of nodes in $\widehat{S}$ of length exactly $l$ 
which are initial segments of some node in $S$.
 Thus,
$S\re l= \widehat{S}\re l$,
whether or not $S$ has nodes of length $l$.

 If
 $s\in\bT_k$ has the same  length as some  (ghost) coding node $c^k_n$ in $\bT_k$, then $s$ has a unique immediate successor in $\bT_k$; denote this as $s^+$.
We  say that
$s$
{\em  has passing number
 $i$  at $c^k_n$} exactly when
$s^+(|c^k_n|)=i$.

The following is   Definition 4.9 in \cite{DobrinenJML20} augmented to include ghost coding nodes; it extends Definition \ref{def.3.1.Sauer} of Sauer.
It is important to note that a ghost coding node in some subset  $T\sse\bT_k$ is never a coding node in $\bT_k$.

\begin{defn}[Strong Similarity Map]\label{def.3.1.likeSauer}
Let $k\ge 3$ be fixed,  and let  $S,T\sse \bT_k$ be meet-closed subsets.
A function $f:S\ra T$ is a {\em strong similarity map} of $S$ to $T$ if for all nodes $s,t,u,v\in S$, the following hold:
\begin{enumerate}
\item
$f$ is a bijection.
\item
$f$ preserves lexicographic order: $s<_{\mathrm{lex}}t$ if and only if $f(s)<_{\mathrm{lex}}f(t)$.

\item
$f$ preserves meets, and hence splitting nodes:
$f(s\wedge t)=f(s)\wedge f(t)$.

\item
$f$ preserves relative lengths:
$|s\wedge t|<|u\wedge v|$ if and only if
$|f(s)\wedge f(t)|<|f(u)\wedge f(v)|$.

\item
$f$ preserves initial segments:
$s\wedge t\sse u\wedge v$ if and only if $f(s)\wedge f(t)\sse f(u)\wedge f(v)$.

\item $f$ preserves (ghost)  coding  nodes:
$s$ is a   coding node in $S$ if and only if $f(s)$ is a  coding node in $T$.
If $S$ has also ghost coding nodes, then 
$s$ is a ghost   coding node in $S$ if and only if $f(s)$ is a   ghost  coding node in $T$.

\item
$f$ preserves passing numbers at  (ghost) coding nodes:
If $c$ is a (ghost) coding  node in $S$ and $s$ is a node in $S$ with $|s|=|c|$,
then $(f(s))^+(|f(c)|)=s^+(|c|)$.
In words, the passing number of  $f(s)$ at $f(c)$  equals the passing number of  $s$ at $c$.
\end{enumerate}
We  say that $S$ and $T$ are {\em strongly similar}, and write $S\ssim T$, exactly when there is a strong similarity map between $S$ and $T$.
\end{defn}

It follows from  (3) that
 $s\in S$ is a splitting node in $S$ if and only if  $f(s)$ is a splitting node in $T$.
In all cases above,   it may be that $s=t$ and $u=v$,
 so in particular,
   (5) implies that  $s\sse u$ if and only if $f(s)\sse f(u)$.
Notice that strong similarity is an equivalence relation, since the inverse of a strong similarity map is a strong similarity map, and the composition of two strong similarity maps is a strong similarity map.
If $T'\sse T$ and $f$ is a strong similarity of $S$ to $T'$, then we say that  $f$ is a {\em strong similarity embedding} of $S$ into $T$.

Our goal in this subsection is to define a space of subtrees of
 $\bT_k$ for which the development of a Ramsey theory  akin to the Halpern-\Lauchli\ and Milliken Theorems is possible.
 The first potential obstacle arises due to the fact that $k$-cliques are forbidden in $\mathcal{H}_k$.
 This manifests in terms of trees in the following way:
 There are finite subtrees of $\bT_k$ which are strongly similar to an initial segment of $\bT_k$, and yet cannot be extended within $\bT_k$ to a subtree which is strongly similar to  $\bT_k$.
In this subsection we
make precise what the possible obstructions are.
We then define the set of strong $\mathcal{H}_k$-coding trees to be those trees which avoid  the possible obstructions.

The next definitions  are new to $k$-clique-free graphs  for  $k>3$,
and  are necessary for the work in this paper.
When $k=3$, the rest of this section  simply reproduces the concepts of sets of parallel $1$'s and the Parallel $1$'s Criterion used throughout \cite{DobrinenJML20}, though in a new and more streamlined manner.
Fix $k\ge 3$ throughout the rest of Section \ref{sec.4}.

If $X$ is a level subset of some  meet-closed  $S\sse \bT_k$, let $l_X$ denote the length of the members of $X$.
If the nodes in $X$ are not maximal in $S$,
  let
the set of immediate successors of $X$ in $\widehat{S}$ be denoted by $X^+$.
Thus, when $S$ is understood,
 \begin{equation}
 X^+=\{s^{\frown}i:s\in X,\ i\in 2,\mathrm{\ and \ } s^{\frown}i\in \widehat{S}\}.
 \end{equation}
 Note that if $l_X$ is the length of a (ghost) coding  node in $\bT_k$, then  each node in $X$  has a unique extension in $X^+$ which is determined by $\bT_k$, irregardless of $S$.

\begin{defn}[Pre-$a$-Clique and Witnesses]\label{defn.prepclique}
Let  $a\in [3,k]$, and
let  $X\sse \bT_k$
 be a  level subset of  $\bT_k$.  
 ($X$ is allowed to consist of  a single node.)
  Let $l_X$ denote the length of the nodes in $X$.
We say that $X$
{\em  has pre-$a$-clique}
 if there is an index set $I\sse[-k+2,-1]\cup\mathbb{N}$ of size $a-2$ such that,
letting  $i_*=\max(I)$ and $l=
|c^k_{i_*}|$,
the following hold:
\begin{enumerate}
\item
The set
$\{c^k_{i}:i\in I\}$ codes an $(a-2)$-clique; that is, 
for each pair $i<j$ in $I$,
 $c^k_{j}(|c^k_{i}|)=1$;

\item
For each node  $x\in X$ and  each $i \in I$,
$x^+(|c^k_{i}|)=1$.
\end{enumerate}
We say  that
$X$  has a pre-$a$-clique {\em at $l$}, and that $X\re l$ {\em is} a   pre-$a$-clique.
The set of nodes $\{c^k_{i}:i\in I\}$  is said to {\em witness} that  $X$ has a pre-$a$-clique at $l$.
 \end{defn}

\begin{notation}\label{notn.WP_a}
We write $\mbox{P}_a(X)$ exactly when $X$ has a  pre-$a$-clique.
\end{notation}

\begin{rem}\label{rem.central}
Whenever a level set $X$ has a pre-$k$-clique,  then
for  any set $C$  of coding (and/or ghost coding) nodes of  $\bT_k$ witnessing that
$\mbox{P}_k(X)$,
 the  set $C$  codes a $(k-2)$-clique,
and  each  $x\in X$ has passing number $1$ at each  $c\in C$.
It follows that, if $X$ has more than one node, then  for  any coding node  $c$ in $\bT_k$ extending some $x\in X$,
any extension of any $y\in X\setminus\{x\}$ to some node $y'$ of  length greater than $|c|$ must satisfy $y'(|c|)=0$.

Thus, the nodes in a  pre-$k$-clique are `entangled':
The splitting possibilities in the cone above one of these nodes depends  on  the cones above the other nodes  in the pre-$k$-clique.
If $X$ is contained in some finite subtree $A$ of $\bT_k$ and
$\mbox{P}_k(X)$ is not witnessed by coding nodes in $A$,
then the graph coded by $A$ has no knowledge
 that the cones above  $X$   in $\bT_k$ are  entangled.
Then no
extension of $A$ into $\bT_k$ can be strongly similar to $\bT_k$.
This is one of the main reasons  that developing Ramsey theory for Henson graphs is more complex than for the Rado graph.

Even if $a<k$, pre-$a$-cliques are also entangled.
In the set-up to the space of strong coding trees, we must consider pre-$a$-cliques for each $a\in [3,k]$;
it is  necessary to witness them in order to guarantee the existence of  extensions 
within   a given strong $\mathcal{H}_k$-coding tree  $T$ 
which are strongly similar to some 
tree which should exist according to the $K_k$-Free  Criterion. 
The guarantee of such extensions is the heart of  the extension lemmas in Section \ref{sec.ExtLem}.
\end{rem}

\begin{defn}[New Pre-$a$-Clique]\label{defn.newll1s}
Let $a\in [3,k]$, and let  $X\sse \bT_k$
be a level set.  ($X$ can consist of  a single coding node.)
We say that $X$
 {\em has a  new   pre-$a$-clique at $l$} if
$X\re l$ is
 a pre-$a$-clique and
for each $l'< l$ for which $X\re l$ and $X\re l'$ have the same number of nodes,
$X\re l'$
 is not a pre-$a$-clique.
 \end{defn}

The reasoning behind  the requirement that  $X\re l$ and $X\re l'$ have the same number of nodes will become more apparent in latter sections, when we want to color finite antichains of coding nodes coding a copy of some finite $K_k$-free graph.

\begin{defn}\label{defn.nonewpkc}
Let
$a\in[3, k]$, and   let $X, Y\sse \bT_k$  be   level sets with $l_Y>l_X$.
We say that    $Y$
 {\em   has no new pre-$a$-cliques over $X$}
if and only if the following holds:
 For each $j\in(l_X,l_Y]$
and each $Z\sse Y$,
if $Z\re j$ is a pre-$a$-clique,
 then
 $Z$ end-extends $Z\re l_X$,
  and
$Z\re l_X$ already has a pre-$a$-clique.
We say that $Y$ {\em  has  no new pre-cliques over $X$} if $Y$ has no new pre-$a$-cliques over $X$ for any $a\in[3,k]$.
\end{defn}

For example, suppose $X$ is a  singleton which is a new pre-$a$-clique for some $a\in[3,k]$ and $Z$ is a level set with at least two members such that $Z\re l_X=X$.
If for some 
$l_X<j\le l_Z$,  $Z\re j$ has at least two distinct nodes and for each $z\in Z$, $z(j)=1$, then $Z$ 
has at least  a new pre-$3$-clique over $X$.

The next definition
gives precise conditions under which a new  pre-$a$-clique at $l$ in a subtree $T$ of $\bT_k$ is maximal in the interval
of $T$ containing  $l$.

\begin{defn}[Maximal New Pre-$a$-Clique]\label{defn.newmpkc}
Let $T$ be a subtree of $\bT_k$,
let   $X\sse T$
be a level set,
and let $a\in[3,k]$.
We say that $X$
has   a {\em  maximal new   pre-$a$-clique  in $T$ at $l$}
if
$X\re l$
is   a new  pre-$a$-clique  in $T$ which  is  also maximal in $T$ in the following sense:
Let
$d$ denote  the critical  node  in $T$ of maximum length satisfying $|d|<l$.
If $m$ is  the index so that $d=d^T_m$,
let $e$ denote $d^T_{m+1}$ and
note that
  $l\le |e|$.
Then for any $l'\in (l,|e|]$ and any new  pre-$a$-clique  $Y\sse T\re l'$,
if $Y\re l$ contains $X\re l$
then  these sets are equal;  hence $l'=l$, since $T$ has no splitting nodes in the interval $(|d|,|e|)$.

We write  $\mbox{MP}_a(X;T)$ if
 $X$ has a maximal new pre-$a$-clique in $T$ in the interval of $T$ containing the length of the nodes in $X$.
Thus, if $l_X=|d^T_m|$,
then  $\mbox{MP}_a(X;T)$ means that  for some $l\in(l^T_{m-1},l^T_m]$,
$X$ has a maximal
 new pre-$a$-clique at $l$.
\end{defn}

In Definition \ref{defn.newmpkc},  for any level set $Z$ end-extending $X$, we say that $Z$ has a maximal new   pre-$a$-clique  in $T$ at $l$.
We will say that a  set $Y\sse T$ {\em contains} a maximal new   pre-$a$-clique at $l$  if
$\mbox{MP}_a(X;T)$ for some
 subset $X\sse Y\re l$.

\begin{defn}[Strong Isomorphism]\label{defn.stable}
Let  $S$ and $T$  be strongly similar subtrees of $\bT_k$ with $M$  many critical nodes.
The  strong similarity map  $f:T\ra S$ is a
{\em strong isomorphism}   if  for each positive
$m\in M$, the following hold:
For each $a\in [3,k]$,
a level subset $X\sse T\re |d^T_{m}|$ has a  maximal  new
pre-$a$-clique
 in $T$ in the interval $(|d^T_{m-1}|,|d^T_{m}|]$
if and only if $f[X]$ has a
 maximal  new  pre-$a$-clique
 in $S$ in the interval $(|d^S_{m-1}|,|d^S_{m}|]$.

When there is a  strong isomorphism between $S$ and $T$,
we say that $S$ and $T$ are {\em strongly isomorphic} and
 write $S\cong T$.
\end{defn}

\begin{observation}\label{obs.urcool}
The relation $\cong$ is an equivalence relation, since the
inverse of a strong isomorphism  is a strong isomorphism, and
 composition of two strong isomorphisms  is   again a strong isomorphism.
 Moreover, since   strong similarity maps preserve (ghost) coding nodes and passing numbers,
any strong isomorphism 
$f :T\ra S$  will map a set of  coding (and/or ghost coding) nodes in $T$ witnessing a  pre-$a$-clique $X$  in $T$  to a   set of coding (and/or ghost coding) nodes in $S$ witnessing that $f[X]$ is a   pre-$a$-clique $X$  in $S$.
\end{observation}


Strong isomorphisms preserve all relevant structure regarding the shape of the tree, (ghost) coding nodes and passing numbers,  and maximal new   pre-$a$-cliques and their witnesses.
They provide  the essential structure   of   our  strong 
$\mathcal{H}_k$-coding trees.

\begin{defn}[The Space of Strong $\mathcal{H}_k$-Coding Trees $(\mathcal{T}_k,\le, r)$]\label{defn.T_pspace}
A tree $T\sse\bT_k$ 
(with designated ghost coding nodes) is a member of $\mathcal{T}_k$  if  and only if there is  a  strong isomorphism, denoted  $f_T$, from   $\bT_k$ onto $T$.
Thus, $\mathcal{T}_k$ consists of all subtrees of $\bT_k$ which are strongly  isomorphic to $\bT_k$.
Call
the members of $\mathcal{T}_k$ {\em strong $\mathcal{H}_k$-coding trees}, or
 just   {\em strong coding trees} when $k$ is clear.
The partial ordering $\le$ on $\mathcal{T}_k$ is simply inclusion:
For $S,T\in\mathcal{T}_k$, we write 
$S\le T$ if and only if $S$ is a subtree of $T$.
If $T\in\mathcal{T}_k$, the notation $S\le T$ implies that $S$ is also a member of $\mathcal{T}_k$.

Given $T\in\mathcal{T}_k$, for $n\in[-k+2,-1]\cup\bN$,
we let $c^T_n$ denote $ f_T(c^k_n)$;  hence,
$\lgl c^T_{-k+2},\dots,c^T_{-1}\rgl$ enumerates the ghost coding nodes and 
 $\lgl c^T_n:n\in\mathbb{N}\rgl$   enumerates the coding nodes in $T$ in increasing order of length. 
  Letting 
$\lgl d^T_m:m\in\mathbb{N}\rgl$ enumerate  the  critical nodes of $T$ in order of increasing length, 
note that the $m$-th critical node in $T$, $d^T_m$, equals $f_T(d^k_m)$. 
In particular,   the {\em ghost coding node} $c^T_{-k+2}$ equals  $d^T_0$, the minimum splitting node in $T$.

It follows from the definition of strong isomorphism and the structure of $\bT_k$ that 
each coding node in $T$  is a singleton pre-k-clique, while  each ghost coding node is not.
Precisely,  $c_{-k+2}^T$ is not a pre-clique at all.
If $k\ge 4$, then $c_{-k+3}^T$ is a pre-$3$-clique, and in general, for $3\le i<k$, 
 the ghost coding node $c^T_{-k+i}$ is a  
pre-$i$-clique.
Thus, the ghost coding nodes in $T$ cannot be  coding nodes in $\bT_k$; hence the name.
It is the case, though,  that each ghost coding node  $c^T_n$ with $n>-k+2$ in $T$ has the same length as  some ghost  or  coding node in $\bT_k$.

Let $T(m)$ denote the set of nodes in $T$ of length $|d^T_m|$.
Recalling the notation in  (\ref{eq.T(m)}) and (\ref{eq.r_m(T)}) in Subsection \ref{subsection.T_p}, the $m$-th
 {\em finite approximation} of $T$  is
\begin{equation}\label{eq.r_m}
r_m(T)= \bigcup_{j<m} T(j),
\end{equation}
for $m\in\mathbb{N}$.
This is equal to $\{t\in T: |t|< |d^T_m|\}$, since $T$ is a tree.
Thus for $m<n$, $r_n(T)$ end-extends $r_m(T)$,
and $T=\bigcup_{m\in\mathbb{N}}r_m(T)$.
Notice that  for any tree $T$,   $r_0(T)$ is the emptyset  and $r_1(T)=\{d^T_0\}$,  the splitting node of smallest length in $T$.

For each $m\in\mathbb{N}$, define
\begin{equation}
\mathcal{AT}^k_m=\{r_m(T):T\in\mathcal{T}_k\},
\end{equation}
 and  let
\begin{equation}
\mathcal{AT}^k=\bigcup_{m\in\mathbb{N}}\mathcal{AT}^k_m.
\end{equation}
Given
$A\in\mathcal{AT}^k$ and
$T\in\mathcal{T}_k$,
define
\begin{equation}
[A,T]=\{S\in \mathcal{T}_k:  \exists m\, (r_m(S)=A)\mathrm{\ and\ }S\le T\}.
\end{equation}
Given $j<m$, $A\in\mathcal{AT}^k_j$ and $T\in\mathcal{T}_k$, define
\begin{equation}\label{eq.r_m}
r_m[A,T]=\{r_m(S):S\in [A,T]\}.
\end{equation}
For $A\in \mathcal{AT}^k$ and $B\in \mathcal{AT}^k\cup\mathcal{T}_k$,
if
for some $m$, $r_m(B)=A$, then we write $A\sqsubseteq B$ and say that $A$ is an {\em initial segment} of $B$.
If $A\sqsubseteq B$ and $A\ne B$, then we write $A\sqsubset B$ and say that $A$ is a {\em proper initial segment} of $B$.
\end{defn}

\begin{rem}
If a  subset  $A\sse \bT_k$ does not  contain sequences of $0$'s of unbounded length, then
there is an $n\ge 0$ such that each node in $A$  has passing number $1$ at  $c^k_i$, for some  $i\in [-k+2,n]$.
Such an  $A$ cannot satisfy  property $(A_k)^{\tt tree}$ so it  does not code $\mathcal{H}_k$;
hence it is not strongly similar to $\bT_k$.
Thus, the leftmost path through  any member of $\mathcal{T}_k$  is the infinite sequence of $0$'s.
It follows that
for $T\in\mathcal{T}_k$,
 the strong isomorphism  $f_T:\bT_k\ra T$
must take each splitting node in $\bT_k$  which is a sequence  of  $0$'s to a splitting node in $T$ which is a sequence  of  $0$'s.
In particular, $f_T$ takes
 $c_{-k+2}^k$ to the  ghost coding node $c_{-k+2}^T$ of $T$, which is a splitting node in $\bT_k$ consisting of a sequence of $0$'s.
\end{rem}

Next,  we define what it means for a pre-clique to be witnessed inside a given subtree  of $\bT_k$. 
This sets the stage for the various {\em Witnessing  Properties} to follow.
The Strong Witnessing Property  will be  a structural characteristic of strong $\mathcal{H}_k$-coding trees. 
The Witnessing Property 
will be   useful for extending a given finite subtree  of some $T\in\mathcal{T}_k$ as should be possible according to the $K_k$-free Branching Criterion.
Both will be utilized in following sections.

\begin{defn}[Witnessing new pre-cliques in $T$]\label{def.WinT}
Let   $T$ be a finite or infinite subtree of $\bT_k$ with 
some 
coding and/or ghost coding nodes.
Suppose that $X$ is a level subset of $T$ which has
 a new
 pre-$a$-clique at $l$, for some $a\in [3,k]$.
 We say that a  set of coding and/or ghost coding  nodes $\{c^T_{i}: i\in I\}$ in $T$,
$|I|=a-2$, {\em witnesses in $T$}
that $X$ has
 a new
 pre-$a$-clique at $l$
if, letting $i_*=\max(I)$,
\begin{enumerate}
\item
$\{c^T_{i}: i\in I\}$ codes an $(a-2)$-clique;
\item
$|c^T_{i_*}|\ge l$ and
$T$ has no critical  nodes in
the interval  $[l,|c^T_{i_*}|)$;
\item
 For each $x\in X$,
the node $y_x$ in
$T\re |c^T_{i_*}|$  extending $x\re l$
has passing number $1$ at
$c^T_{i}$, for  each $i\in I$.
\end{enumerate}
In this case, we write $WP_a(X;T)$.  
If  $X\re l$ is a maximal new pre-$a$-clique and   is witnessed  by a set of coding nodes in $T$, we write $\mbox{WMP}_a(X;T)$.
\end{defn}

\begin{rem}
Note that the set   $\{y_x:x\in X\}$  in (3) above is  well-defined, since by (2), $T$ has no critical  nodes  
 (in particular no splitting nodes) in
the interval  $[l,|c^T_{i_*}|)$.
The passing number  of $y_x$ at $c_{i_*}^T$
is uniquely determined by $\bT_k$ to be  the passing number of $(y_x)^+$ at  $c_{i_*}^T$.
We point out that  there may be other sets of coding and/or ghost coding  nodes in $\bT_k$ witnessing $WP_a(X;T)$. 
However,
   any such  set of witnesses in $T$ must contain $c^T_{i_*}$.
\end{rem}

In the following, given a finite or infinite subtree $T$ of $\bT_k$,
recall that $\lgl d^T_m:m\in M\rgl$ is the enumeration of the critical nodes in $T$ in order of increasing length. 
Here,  either $M\in\bN$ or $M=\bN$.
Enumerate the coding nodes in $T$ as $\lgl c^T_n:n\in N\rgl$, where $N\in\bN$ or $N=\bN$. 
If $T$ has ghost coding nodes, enumerate these as $c^T_n$, where $n\in [-k+2,j]$ for some $j\le  -1$.
The number $m_n$ is the index such that $d^T_{m_n}=c^T_n$.

\begin{defn}[Strong Witnessing Property]\label{defn.PrekCrit}
A subtree  $T$ of  $\bT_k$
 has the {\em Strong Witnessing Property}  if the following hold:
\begin{enumerate}
\item
Each     new pre-clique in $T$
takes  place in some interval  in $T$  of the form
$(|d^T_{m_n-1}|,|c^T_n|]$.
\item
Each new pre-clique in $T$  is witnessed in  $T$.
\end{enumerate}
\end{defn}

Notice that the coding node $c^T_n$  in Definition
\ref{defn.PrekCrit}
is obligated (by Definition \ref{def.WinT}) to be among the set of coding nodes witnessing $P_a(X;T)$.
Further,  in order to  satisfy Definition \ref{defn.PrekCrit},   it suffices that the maximal new pre-$a$-cliques are witnessed  in $T$,
as this automatically guarantees that every new pre-$a$-clique is  witnessed in $T$.

\begin{defn}[Witnessing Property]\label{defn.WP}
A subtree  $T$ of  $\bT_k$ has the {\em  Witnessing Property (WP)}  if
each     new pre-clique  in $T$ of size at least two
takes  place in some interval  in $T$  of the form
$(|d^T_{m_n-1}|,|c^T_n|]$ and 
 is witnessed by a set of coding nodes in $T$.
\end{defn}

The idea behind the Witnessing Property is that we want a property strong enough to guarantee that the finite subtree $T$ of $\bT_k$ can  extend  within $\bT_k$ according to the $k$-FBC.
However, 
each coding node is a pre-$k$-clique, and  if $T$ is an antichain of coding nodes in $\bT_k$ coding some finite $K_k$-free graph, we  cannot require that all singleton pre-$k$-cliques be witnessed in $T$.
The Witnessing Property will 
allow just the right amount of flexibility to achieve 
our Ramsey Theorem for antichains of coding nodes.

\begin{lem}\label{lem.concpresWP}
If $A$ is a  subtree of $\bT_k$ which  has the (Strong) Witnessing Property and $A\cong B$,
then $B$ has the (Strong) Witnessing Property.
\end{lem}

\begin{proof}
Given the hypotheses,
let $f:B\ra A$ be a  strong isomorphism from $B$ to $A$.
Suppose $X\sse B$ is a level set   which has a new pre-$a$-clique, for some $a\in [3,k]$.
Let $m$ be the index such that the new pre-$a$-clique in $X$ takes place in the interval 
$(|d^B_{m-1}|,|d^B_m|]$.
Without loss of generality, assume that $X$ has a maximal new pre-$a$-clique in $B$ in  this interval.
In the case for WP, assume that $X$ has at least two members. 
Since $f$ is a strong isomorphism,
$f[X]$ has a maximal new pre-$a$-clique in $A$ in  the interval $(|d^A_{m-1}|,|d^A_m|]$.
Since $A$ has the (Strong) Witnessing Property, $d^A_m$ must be a coding node in $A$;
moreover, 
this coding node must be among the set of coding nodes in $A$ witnessing that $f[X]$ has a new pre-$a$-clique.
Therefore, $d^B_m$ is  a coding node,
and $d^B_m$ 
is among the set of coding nodes witnessing that $X$ has
a new pre-$a$-clique,
 since $f$ being a strong similarity map implies $f$  preserves coding nodes and passing numbers.
Furthermore, 
 each new  pre-$a$-clique in $B$ (of size at least two in the case of WP) 
 is witnessed in $B$.
Hence, $B$ has the (Strong) Witnessing Property.
\end{proof}

\begin{lem}\label{lem.need}
Suppose $A,B$ are subtrees of $\bT_k$ and that $A$ has the Strong Witnessing Property.
Then $A\cong B$ if and only if
$A\ssim B$ and $B$ also has the  Strong Witnessing Property.
\end{lem}

\begin{proof}
For the  forward direction,
note that $A\cong B$ implies $A\ssim B$, by the definition of strongly isomorphic.
If moreover, $A$ has the  Strong Witnessing Property then Lemma
\ref{lem.concpresWP} implies $B$ also has the  Strong Witnessing Property.

Now suppose that $A\ssim B$ and both $A$ and $B$ have  the Strong Witnessing Property.
Let $f:A\ra B$ be the strong similarity map.
Suppose
 $X$ is a level set in $A$ which has a maximal new
pre-$a$-clique, for some $a\in [3,k]$.
Since $A$ has the Strong Witnessing Property,
there is a set of coding nodes $C\sse A$ witnessing that $X$  has a new
pre-$a$-clique.
Furthermore,
$l_X$ must be the length of some coding node in the set $C$.
Since
$f$ preserves  coding nodes and passing numbers,
it follows that $f[C]$ is a set of coding nodes in $B$ witnessing  that $f[X]$ has a pre-$a$-clique.
It remains to show that $f[X]$ has a new pre-$a$-clique and is maximal in $B$.

If $f[X]$ does not have  a new pre-$a$-clique in $B$,
then there is  some critical node $d$  in $B$ below  $f(c)$ such that $f[X]\re |d|$ has a new pre-$a$-clique in $B$, where $c$ is the longest coding node in $C$.
Since $B$ satisfies the Strong Witnessing Property, this new pre-$a$-clique in $f[X]$ appears at some coding node in $B$ below $d$.
Further, $f[X]$ must be witnessed by some set of coding nodes $D$ in $B$.
But then $f^{-1}[D]$ is a set of coding nodes in $A$ witnessing a pre-$a$-clique in $X$.
Since the longest length of a coding node in $f^{-1}[D]$ is shorter than $|c|$,
the pre-$a$-clique in $X$ occurs first at some coding node below $c$, a contradiction to $X$ having a new pre-$a$-clique.
Therefore, $f[X]$ is a new pre-$a$-clique in $B$.

If $f[X]$  is not maximal in $B$, then
there is some level set $Z$ of nodes of length $l_X$ properly containing $f[X]$ which has a new pre-$a$-clique in $B$.
Since $B$ has the Strong Witnessing Property,
 there is some set of coding nodes $D\sse B$ witnessing $Z$.
Then $f^{-1}[D]$ witnesses that $f^{-1}[Z]$ is a pre-$a$-clique in $A$ properly containing $X$, contradicting the maximality of $X$ in $A$.
Therefore, $f$ preserves maximal new pre-$a$-cliques, and hence is a strong isomorphism.
Hence, $A\cong B$.
\end{proof}

The $K_k$-Free Branching Criterion from Definition \ref{defn.kFSC}
 naturally gives rise to a  version for skew trees.
A tree  $T\sse\bT_k$  with (ghost) coding  nodes  $c_n$ for $n\in [-k+2,N)$, where $N\in\bN$ or $N=\bN$ and $c_{-k+2}$ is the stem of $T$,
satisfies the {\em $K_k$-Free Branching Criterion  ($k$-FBC)}
if and only if  the following holds:
For  $n\ge -k+2$,
letting $u$ denote the node $c_{n+1}\re |c_n|$, 
each node $t\in T$ of length $|c_n|$ branches  in $T$ before reaching length $|c_{n+1}|$ if and only if for each $I\sse [-k+2,n-1]$ of size $k-2$  for which the set $C=\{c_i:i\in I\}$ codes a $(k-2)$-clique
and such that $u(|c_i|)=1$ for each $i\in I$,
$t(|c_i|)=0$ for at least one $i\in I$.

Notice that if a skew tree $T$
satisfies the $k$-FBC, then
Theorem \ref{thm.A_3treeimpliestrianglefreegraph}
implies the following fact:

\begin{observation}\label{obs.kfbcH_k}
Any skew tree satisfying the $K_k$-Free Branching Criterion in which the coding nodes are dense  codes a copy of $\mathcal{H}_k$.
\end{observation}

\begin{lem}\label{lem.internalcharacterization}
(1) If $T\sse \bT_k$ is strongly similar to $\bT_k$,
then $T$ satisfies the $K_k$-Free Branching Criterion.

(2) If $T\sse \bT_k$ is strongly similar to $\bT_k$ and
 has the Strong Witnessing  Property,
then  the strong similarity map from  $\bT_k$ to  $T$ is a strong isomorphism,  and hence $T$ is a member of $\mathcal{T}_k$.
\end{lem}

\begin{proof}
(1)  follows in a straightforward manner from the definitions of $k$-FBC  and strong similarity map, along with the structure of $\bT_k$, as we now show.
Suppose $T\sse \bT_k$  is strongly similar to $T$, and let  $f:\bT_k\ra T$ be the strong similarity map.
Note that  for each integer $n\ge -k+2$, $c^T_n=f(c^k_n)$.
Fix $n\ge -k+2$ and a  node
$t\in T\re l^T_n$ which does not extend to $c^T_{n+1}$.
Then $s:=f^{-1}(t)$ is in $\bT_k\re l^k_n$.
Since $f$ is a strong similarity map, $s$ does not extend to the coding node $c^k_{n+1}$ in $\bT_k$.
Since $\bT_k$ satisfies the $k$-FBC,
$s$ splits in $\bT_k$ before reaching the level of $c^k_{n+1}$ if and only if,
letting  $u=c^k_{n+1}\re (l_n^k+1)$,
for each subset $I\sse [-k+2,n]$  of size $k-2$ such that
$C=\{c^k_i:i\in I\}$ codes a $(k-2)$-clique
and $u$ has passing number $1$ at each
$c\in C$,
there is some $c\in C$ at  which $s^+$ has passing number $0$.
Since $t=f(s)$ and $f$ is a strong similarity map,
$t$ splits in $T$ before reaching the level of $c^T_{n+1}$ if and only if,
letting  $v=c^T_{n+1}\re (l^T_n+1)$,
for each subset $I\sse [-k+2,n]$  of size $k-2$  for which
$D=\{c^T_i:i\in I\}$ codes a $(k-2)$-clique
and $v$ has passing number $1$ at each
$c\in D$,
there is some $c\in D$ at  which $t^+$ has passing number $0$.

For
(2), if $T\sse\bT_k$ is strongly similar to $\bT_k$ and has the Strong Witnessing Property, then
it  follows from Lemma \ref{lem.need}
 that $T\cong \bT_k$
  since $\bT_k$ has  the Strong Witnessing Property.
\end{proof}

\begin{lem}\label{lem.psim.properties}
Every $T\in\mathcal{T}_k$ has
 the following properties:
\begin{enumerate}
\item
$T\ssim \bT_k$.

\item
$T$ satisfies the $K_k$-Free Branching Criterion.

\item
$T$ has  the  Strong Witnessing  Property.
\end{enumerate}
\end{lem}

\begin{proof}
(1) is immediate from the definition of $\mathcal{T}_k$.
(2) follows from  Lemma \ref{lem.internalcharacterization} part (1).
(3) follows from (1) and Lemma \ref{lem.need}.
\end{proof}


\section{Extension Lemmas}\label{sec.ExtLem}

Unlike Milliken's strong trees,
not every finite subtree of a strong $\mathcal{H}_k$-coding tree can be extended  within that ambient tree
to another member of $\mathcal{T}_k$, nor necessarily even  to another finite tree of a desired configuration.
This section provides   structural properties of finite subtrees
which are necessary and sufficient  to extend to a larger tree of a particular strong similarity type.
The first subsection lays the groundwork for these properties and the second subsection proves extension lemmas which are fundamental  to developing Ramsey theory on strong $\mathcal{H}_k$-coding trees.
The extension lemmas extend and streamline similar lemmas in \cite{DobrinenJML20}, taking care of new issues that arise when $k\ge 4$.
Furthermore, these lemmas lay new groundwork for general extension principles, with the benefit of a simpler
 proof of Theorem \ref{thm.matrixHL} than the proof of its  instance for  $\mathcal{H}_3$ in \cite{DobrinenJML20}.

\subsection{Free level sets}\label{subsec.valid}

In this subsection, we provide criteria which will aid in the extension lemmas in Subsection \ref{subsec.extlemmas}.
These requirements will guarantee that a finite subtree of a strong coding tree $T$  can be extended {\em within $T$} to another strong coding tree.

Recall that given a tree $T\in\mathcal{T}_k$ and $m\in\bN$, $T(m)$ denotes the set of nodes in $T$ of length equal to $|d^T_m|$, the length of the $m$-th critical node in $T$.
The length of a (ghost)  coding node $c^T_n$ is denoted $l^T_n$, which equals $|d^T_{m_n}|$.
Thus, $T(m_n)=T\re l^T_n$.
Throughout this section,   $n \ge -k+2$, and when we write ``coding node'' we are including ghost coding nodes.

\begin{defn}[Free]\label{defn.nonewppc}
Let $T\in\mathcal{T}_k$ be fixed.
We say that a level set $X\sse \widehat{T}$ with length $l$
is {\em free} in $T$
if
given $n$ least such that $l^T_n\ge  l$,
letting  $Y$ consist of the leftmost extensions
 of members of $X$
in $T\re l^T_n$,
then
$Y$ has no new pre-$a$-cliques over $X$, for any $a\in [3,k]$.
\end{defn}

In particular, any level set $X\sse T$ with length  that of some   coding node in $T$ is free in $T$.

\begin{rem}
For $k=3$, this is equivalent to the concept of ``$X$ has no pre-determined new parallel $1$'s in $T$" in \cite{DobrinenJML20}.
\end{rem}

\begin{term}\label{term.pc}
For a level set $Y$ end-extending a level set $X$,
we  say that
$Y$ has {\em no new pre-cliques over $X$}
if $Y$ has no new pre-$a$-cliques over $X$, for any $a\in [3,k]$.
\end{term}

\begin{lem}\label{lem.leftmostfree}
Let $T\in\mathcal{T}_k$ be fixed,
$X\sse \widehat{T}$  be a level set
which is free in $T$.
Then for  any $l>l_X$, 
the set of leftmost extensions in $T$ of the nodes in $X$ to $T\re l$ contains  no new pre-cliques over $X$.
Furthermore,  for any $n$ such that $l^T_n> l_X$, 
 the leftmost extensions of $X$ in $T\re l^T_n$
 have passing numbers   $0$ at  $c^T_n$.
 It follows that any set of leftmost extensions of $X$ is free in $T$. 
\end{lem}

\begin{proof}
This lemma  follows from the fact that $T\cong \bT_k$.
To see this,
let $f:\bT_k\ra T$ be the strong isomorphism
 witnessing that $T\in\mathcal{T}_k$,
and let $j$ be least such that $l^T_j\ge l_X$.
Let $n\ge j$ and $a\in[3,k]$ be given, and let $Y$ be the
end-extension of $X$ in
$T\re l^T_n$
consisting of the nodes
which are leftmost extensions in $T$ of the nodes in $X$.
Since $X$ is free in $T$, $Y$ has no new pre-$a$-cliques over $X$.
Since
$f^{-1}$  is a strong similarity map, $f^{-1}[Y]$
is the collection of leftmost extensions in $\bT_k\re l^k_n$ of the level set $f^{-1}[X]$.
In particular, $f^{-1}[Y]$ has no new  pre-$a$-cliques  in the interval $[l^k_j,l^k_n]$.
In particular, the passing numbers of members of $f^{-1}[Y]$  in this interval $[l^k_j,l^k_n]$ are all $0$.
Since $f$ is a strong isomorphism,
$Y=f\circ f^{-1}[Y]$ has no new pre-$a$-cliques over $Y\re l_X$, and all passing numbers of the leftmost extensions of $X$ in $T$ are $0$.
In particular, any set of leftmost extensions of $X$ in $T$ is free in $T$. 
\end{proof}

An important property of $\mathcal{T}_k$ is that
 all of its members contain
 unbounded sequences of $0$'s.

\begin{lem}
Suppose $T\in\mathcal{T}_k$ and $s$ is a node in the leftmost branch of $T$.
Then  $s$ is a sequence of $0$'s.
\end{lem}

\begin{proof}
If not, then 
for some $n$, no node of $T$ extends $0^{(n)}$.
Let $n$ be the least such integer. 
Then for each $t\in T\re n$, there is some coding node 
$c\in r_n(\bT_k)$ such that $t(|c|)=1$.

If there is a $t\in T\re n$ such that for all coding nodes $c^T_i$ in $T$ of length less than $n$, $t(|c^T_i|)=0$,
then by the $K_k$-Free Branching Criterion of $T$,
$t$ extends in $T$ to $k-1$ many coding nodes forming a $(k-1)$-clique.
But this  $(k-1)$-clique
forms a $k$-clique with the vertex represented by any 
 coding node in $r_n(\bT_k)$ for which  $t(|c|)=1$, contradicting that $T\in\mathcal{T}_k$.

Otherwise, for each $t\in T\re n$, there is a coding node $c\in T$ with $|c|<n$ such that $t(|c|)=1$.
But then $T$ cannot code a copy of $\mathcal{H}_k$, contradicting that $T\in\mathcal{T}_k$.
\end{proof}

\begin{notation}\label{notn.maxl_A}
Let $A$ be a finite subtree of $\bT_k$.
We let  $l_A$  denote the   maximum of the lengths of the nodes in $A$, and  let
\begin{equation}
\max(A)=\{t\in A:|t|=l_A\}.
\end{equation}
Thus, $\max(A)$ is a level set.  
If the maximal nodes in $A$ do not all have the same length, then $\max(A)$ will not contain all the maximal nodes in $A$. 
\end{notation}

\begin{defn}[Finite strong coding  tree]\label{defn.fsct}
A finite subtree $A\sse\bT_k$ is a
 {\em finite strong coding  tree} if and only if
$A\in\mathcal{AT}^k_{m+1}$
 for  $m$ such that   either $d^k_m$ is a  coding node or else $m=0$.
\end{defn}

\begin{lem}\label{lem.fsctvalid}
Given $T\in\mathcal{T}_k$, each finite strong coding tree $A$ contained in $T$
has the Strong Witnessing Property, and $\max(A)$ is free
 in $T$.
\end{lem}

\begin{proof}
Fix $T\in \mathcal{T}_k$ and
let $A$ be  a finite strong coding tree contained in $T$.
If $A\in \mathcal{AT}_0^k$, then
$A$ is the empty set, so the lemma vacuously holds.
Otherwise,
by Definition \ref{defn.fsct},
$\max(A)$ contains a coding node, so
$\max(A)$ is  free in $T$.
Further,  $A\cong r_{m+1}(\bT_k)$ for some $m$ such that $d^k_m$ is a coding node.
As $r_{m+1}(\bT_k)$ has the Strong Witnessing Property,  Lemma \ref{lem.need} implies 
that $A$ also has the Strong Witnessing Property.
\end{proof}


\subsection{Extension Lemmas}\label{subsec.extlemmas}

The  next series of lemmas
 will be used extensively throughout the rest of the paper.
 As  consequences, 
  these lemmas ensure that  every tree in $\mathcal{T}_k$  contains infinitely many subtrees  which are also members of $\mathcal{T}_k$,
and that  for any  $A\in \mathcal{AT}_m^k$ 
with 
$A\sse T\in\mathcal{T}_k$ and $\max(A)$ free in $T$,
the set $r_j[A,T]$, defined in  (\ref{eq.r_m}) of Definition   \ref{defn.T_pspace}, is infinite for each $j> m$.

\begin{lem}\label{lem.poc}
Suppose $T\in \mathcal{T}_k$ and
 $X\sse T $   is a level set  with  $X$  free in $T$. 
Fix a    subset  $X'\sse X$. 
Let
$Y'\sse T$   
be any level  set 
 end-extending  $X'$ such that $Y'$ is free in $T$.
Let $Y''$ denote the set of leftmost extensions of $X\setminus X'$ in $T \re l_{Y'}$.
Then  $Y=Y'\cup Y''$ is free in $T$, and any new pre-cliques in $Y$ occur in $Y'$. 
In particular,  if 
$Y'$ has no new pre-cliques over $X'$,
 then $Y$ has no new pre-cliques over $X$.
\end{lem}

\begin{proof}
Since   $Y'$ is free in $T$ and
 Lemma \ref{lem.leftmostfree} implies  that $Y''$ is free in $T$,  it follows that $Y$ is free in $T$.
Suppose that for  some $a\in[3,k]$  there is a 
   $Z\sse Y$  
such that $Z$ has a  new  pre-$a$-clique in the interval $(l_X,l_Y]$.
Let $X''=X\setminus X'$,  $Z'=Z\cap Y'$, and $Z''=Z\cap Y''$.
 By Lemma \ref{lem.leftmostfree},
$Y''$ has no new pre-cliques over $X''$, 
so $Z'$ must be non-empty.

Suppose toward a contradiction that also  $Z''$ is  non-empty. 
Let $l$ be the minimal  length at which
 this new  pre-$a$-clique occurs, and
let $m$ be the least integer such that
$|d^T_m|<l \le |d^T_{m+1}|$.
Since $T$ has the Strong Witnessing Property (recall Lemma \ref{lem.psim.properties}),
 $d^T_{m+1}$ must be  a  coding node in $T$, say $c^T_j$.
 In the case that $l_Y<l^T_j$,
 by Lemma \ref{lem.leftmostfree} we may extend the nodes in $Z$ leftmost  to nodes in $T\re l^T_j$ without adding any new pre-cliques.
 Thus,
 without loss of generality, assume that $l_Y\ge l^T_j$. 
Since the new  pre-$a$-clique $Z\re l$  must be witnessed  in $T$ at the level of $c^T_j$,
it follows that all nodes in $Z$ have passing number $1$ at $c^T_j$.
Let $f$ be the strong isomorphism from $T$ to $\bT_k$ and
let $V$ denote $f[Z\re l^T_j]$.
Then  $V$ is a level set in $\bT_k$ of length $l^k_j$.
Since $f$ is a strong similarity map,
it  preserves passing numbers.
Hence,
all nodes in $V$ have passing number $1$ at $c^k_j$ in $\bT_k$.

However,
since
$Y''$ consists of the leftmost extensions in $T$
of the nodes in  $X''$,
it follows that
$f[Y'']$ is the set of
 leftmost extensions in $\bT_k$ of $f[X'']$.
Thus, each member of $f[Y'']$ has passing number $0$ at $c_j^k$.
Since $Z''\ne\emptyset$,
also
$f[Z'']\ne\emptyset$, so $V$ has at least one node with passing number $0$ at $c^k_j$, contradicting the previous paragraph.
Therefore,  $Z''$ must be empty, so $Z$ resides entirely within $Y'$.
In particular, if  $Y'$ has no new pre-cliques over $X'$, then  $Y$ has no new  pre-cliques over $X$. 
\end{proof}

\begin{lem}\label{lem.perfect}
Suppose $s$ is a node in a strong coding tree $T\in\mathcal{T}_k$.
If $n$ is least such that $|s|\le l^T_n$,
then there is splitting node $t$ in $T$  extending $s$ such that
 $|t|\le  l^T_{n+k}$.
In particular,
every strong coding tree is perfect.
\end{lem}

\begin{proof}
It suffices to  work with $\bT_k$, since each member of $\mathcal{T}_k$ is strongly isomorphic to
$\bT_k$.
We make use here of  the
particular construction of $\bT_k$ from Example
\ref{ex.bTp}.

Let $s$ be a node in $\bT_k$, and let $n$ be least such that $l^k_n\ge |s|$.
Let $p> n$ be least such that $p=i(k-1)$ for some $i\ge 1$,
and let $s'$ be the leftmost extension of $s$ in
$\bT_k\re (l^k_p +1)$.
Note that $p< n+k$ and
 that $s'$ has passing number $0$ at
$c^k_p$.
By the construction of $\bT_k$,
 $c^k_{p+1}$ in $\bT_k$
will have passing number $1$ at precisely  $c^k_{p-k+3},\dots, c^k_p$, and at no others.
Let $v$ denote the truncation of $c^k_{p+1}$ to length  $l^k_p+1$.
The number of coding nodes in $\bT_k$ at which both
 $s'$ and $v$ have passing number $1$ is at most $k-3$.
Therefore, $s'$ and $v$ do not code a pre-$k$-clique.
So by the $k$-FBC,
$s'$ extends to a splitting node  $t$ in $\bT_k$ before reaching the level of $c^k_{p+1}$.
\end{proof}

Given a set of nodes $Z\sse \bT_k$,
let $Z^{\wedge}$ denote the {\em meet-closure} of $Z$,
which is the set of nodes $\{s\wedge t:s,t\in Z\}$.
By
the {\em tree induced by $Z$} we mean
 the set of nodes
$\{t\re |v| :t\in Z,\  v\in Z^{\wedge}\}$.

\begin{lem}\label{lem.factssplit}
Suppose $A$ is a finite   subtree of some  strong coding tree  $T\in \mathcal{T}_k$ with $\max(A)$ free in $T$. 
Let $X$ be any nonempty   subset of $ \max(A)$, and
 let $Z$ be any subset of
 $\max(A)\setminus X$.
Let
$\{s_i:i<\tilde{i}\}$
be  any enumeration
 of $X$ and suppose
$l\ge l_A$ is given.
Then
 there exist $l_*>l$
 and extensions
$t_i^0,t_i^1\supset s_i$ ($i<\tilde{i}$)
and $t_z\supset z$ ($z\in Z$), each  in $T\re l_*$,
such that letting
\begin{equation}
Y=\{t_i^j:i<\tilde{i},\ j<2\}\cup\{t_z:z\in Z\},
\end{equation}
and  letting $B$ denote the tree induced by $A\cup Y$,
 the following hold:
\begin{enumerate}
\item
The splitting in $B$ above level $l_A$ occurs in the order of the enumeration of $ X$.
Thus, for $i<i'<\tilde{i}$,
$|t_i^0\wedge t_i^1|<|t_{i'}^0\wedge t_{i'}^1|$.
\item
$Y$ has no new  pre-cliques   over $\max(A)$ and is free in $T$.
\end{enumerate}
\end{lem}

\begin{proof}
If $l_A$ is not the level of some coding node in
$T$, begin by extending each member of $X$ leftmost in $T$ to the level of the
very next coding node in $T$.
In this case, abuse notation and
let  $X$  denote this set of extensions.
Since $\max(A)$ is free  in $T$,
this adds no new pre-cliques over $\max(A)$.

By Lemma \ref{lem.perfect},
every node in $X$ extends to a splitting node in $T$.
Let $s_0^*$ be the splitting node of least length in $T$ extending $s_0$,
and let $c^T_{n_0}$  be the coding node in $T$ of least length above $|s_0^*|$.
Extend all nodes in $\{s_i:1\le i<\tilde{i}\}$ leftmost  in $T$ to length $l^T_{n_0}$, and label their extensions $\{s^1_i:1\le i<\tilde{i}\}$.
Given $1\le p<\tilde{i}$ and the nodes $\{s^p_i:p\le i<\tilde{i}\}$,
let $s^*_p$ be the splitting node of least length in $T$ extending $s_p^p$,
and let $c^T_{n_p}$  be the coding node in $T$ of least length above $|s_p^*|$.
If $p<\tilde{i}-1$, then
extend all nodes in $\{s^p_i:p+1\le i<\tilde{i}\}$ leftmost  in $T$ to length $l^T_{n_p}$, and label these $\{s^{p+1}_i:p+1\le i<\tilde{i}\}$.

When  $p=\tilde{i}-1$, let $n=n_{\tilde{i}-1}$ and
for each $i<\tilde{i}$ and $j<2$, let $t_i^j$ be the leftmost extension in $T$
of ${s^*_i}^{\frown}j$ to length $l^T_n$.
For each $z\in Z$, let $t_z$ be the leftmost extension in $T$ of $z$ to length
$l^T_n$.
This collection of nodes composes the desired set $Y$.
By Lemma \ref{lem.poc},
$Y$ has no new pre-cliques over $\max(A)$.
$Y$   is free in $T$ since the nodes in $T$ have the length of a coding node in $T$.
\end{proof}

\begin{conv}\label{conv.POC}
Recall that when working within a strong coding tree $T\in\mathcal{T}_k$,
the passing numbers at coding nodes in  $T$ are completely determined by $T$; in fact, they are determined by $\bT_k$.
For a finite subset $A$ of  $T$ such that
 $l_A$ equals $l_n^T$ for some $n<\om$,
we shall say that
  $A$  {\em has  the (Strong) Witnessing Property} if and only if    the extension
$A\cup\{s^+:s\in\max(A)\}$ has the (Strong) Witnessing Property.
\end{conv}

The notion of  a valid subtree is central to the constructions in this paper.

\begin{defn}[Valid]\label{defn.valid}
Suppose  $T\in \mathcal{T}_k$  and  let $A$ be a finite subtree of $T$.
We say that $A$
is {\em valid} in $T$ if and only if
$A$ has the Witnessing Property and
$\max(A)$  is free in $T$.
\end{defn}

The next
Lemma \ref{lem.pnc} shows that given a valid subtree of a strong coding tree $T$,
any of its maximal nodes can be extended to some coding node $c_n^T$  in $T$ while the rest of the maximal nodes can be extended to length $l_n^T$ so that their passing numbers are anything desired,  subject only to the $K_k$-Free Criterion.

\begin{lem}[Passing Number Choice]\label{lem.pnc}
Fix  $T\in\mathcal{T}_k$  and
 a finite  subtree $A$  with $\max(A)$ free in $T$.
Let $\{s_i:i<\tilde{i}\}$ be any enumeration of
 $\max(A)$,
and
fix some  $d<\tilde{i}$.
To  each $i\in \tilde{i}\setminus\{d\}$ associate an $\varepsilon_i\in\{0,1\}$, with the stipulation that $\varepsilon_i$ must equal $0$ if $\{s_i,s_d\}$ has  a pre-$k$-clique.
In particular,  $\varepsilon_d=0$.

Then given any  $j$,
there is an $n\ge j$ such that the
coding node $c^T_n$
extends $s_d$,
 and
there are
 extensions $u_i\contains s_i$, $i\in\tilde{i}\setminus \{d\}$, in  $T\re l_n^T$
such that,
letting  $u_d$ denote $ c_n^T$ and letting $Y=\{u_i:i<\tilde{i}\}$,
the following hold:
\begin{enumerate}
\item
Each $u_i$ has
passing number   $\varepsilon_i$ at $u_d$.
\item

Any new  pre-cliques among  subsets of $Y$ (except possibly  for the singleton $\{s_d\}$)
have their first instances occurring in the interval
$(|d^T_{m_n-1}|,l^T_n]$.

\item
If $A$ has the Witnessing Property, then so does 
 $A\cup Y$.
 Thus, if $A$ is valid, then  $A\cup Y$ is also valid.
 \item
 If $A$ has the Strong Witnessing Property and $s_d$ has a pre-$(k-1)$-clique  witnessed by  coding nodes in $A$,
 then $A\cup Y$ has the Strong Witnessing Property. 
\end{enumerate}
\end{lem}

\begin{proof}
Assume the hypotheses  in the first paragraph of the lemma.
Let $m$ be  least such that
$l^T_m\ge l_A$, and
for each $i<\tilde{i}$, let $s'_i$ be the leftmost extension of $s_i$  in $T$ of length $l_m^T$.
Since  $\max(A)$ is free in $T$, the set   $\{s'_i:i< \tilde{i}\}$    has
no new pre-cliques  over
 $A$.
Given
$j$,
take $n$  minimal above $\max(j,m+1)$  such that
$c_n^T\contains s'_d$, and let $u_d=c_n^T$.
Such an $n$ exists, as the coding nodes in
any strong coding tree are dense in that tree, by its strong similarity to $\bT_k$.
For  $i\ne d$,
extend $s'_i$ via its leftmost extension in $T$ to  length
$l_{n-1}^T$ and label it $t_i$.
By Lemma \ref{lem.poc},
$\{t_i:i\in\tilde{i}\setminus\{d\}\}\cup\{u_d\re l^T_{n-1}\}$ has
no  new pre-cliques   over $\{s'_i:i< \tilde{i}\}$, with the possible exception of the singleton $u_d\re l^T_{n-1}$,  so (2) of the Lemma holds.

For  $i\in \tilde{i}\setminus\{d\}$ with  $\varepsilon_i=0$,
let $u_i$ be the leftmost extension of $t_i$ of length $l_n^T$.
For $i<\tilde{i}$ with
 $\varepsilon_i=1$, by our assumption,
 $\{t_i,u_d\re l^T_{n-1}\}$ has no  pre-$k$-cliques,  and their extensions to length
  $l^T_{n-1}$ have no new pre-cliques  by  Lemma \ref{lem.poc}.
  By the $k$-Free Branching Criterion of $T$,
  $t_i$ splits  in $T$ before reaching the level of $c^T_n$.
 Let $u_i$ be the rightmost extension of $t_i$ to length
 $l_n^T$.
Then  for each $i<\tilde{i}$,
the passing number
of $u_i$ at  $u_d$  is
 $\varepsilon_i$.
 Thus, (1) holds.

 Suppose now that $A$ has the Witnessing Property.
Let $Y=\{u_i:i<\tilde{i}\}$.
Since the nodes in $Y$ have the length of the coding node $u_d$,  $Y$ is free in $T$.
By construction, any new pre-cliques in $Y$ over $A$  occur in the interval $(l^T_{n-1},l^T_n]$.
Since $T$ has the Strong Witnessing Property, any new pre-cliques of $Y$ in the interval   
$(l^T_{n-1},l^T_n]$ must actually occur in the interval
$(|d^T_{m_{n-1}}|,l^T_n]$.
It remains to show that  any new pre-cliques of size at least two in $Y$ over $A$   in the interval $(|d^T_{m_{n-1}}|,l^T_n]$ are witnessed by coding nodes in $A\cup Y$.
We now show  slightly more.

Suppose
 $I\sse\tilde{i}$, where $d\not\in I$,  and
$\{u_i:i\in I\}$
has a new  pre-$a$-clique over $A$, for
 some $a\in[3,k]$.
Let $Z$ denote $\{u_i:i\in I\}$  and let $l$ be
 least
such that  $Z\re l$ is a  pre-$a$-clique, and note that
$l$ must be in  the interval $(|d^T_{m_n-1}|,l_n^T]$.
Since $T$ has the Strong Witnessing Property, 
 there is some set of coding nodes 
$C\sse T$ witnessing $Z\re l$.
As
$u_d$  is the least coding node in $T$ above $Z\re l$, $u_d$ must be in $C$, again by the Strong Witnessing Property of $T$.
It follows that each node in $Z$ must have passing number $1$ at $u_d$.

If $a=3$, then  the coding node  $u_d$, which  is contained in  $ Y$, witnesses the pre-$3$-clique in $Z$.
Now suppose that $a\ge 4$.
Then $C\setminus \{u_d\}$ witnesses that
$Z'=Z\cup\{u_d\}$ has a pre-$(a-1)$-clique.
The $l'$ at which $Z'\re l'$ is a new
pre-$(a-1)$-clique must be below $|d^T_{m_n-1}|$, since
$T$ cannot witness it at the level of $u_d$.
Since $Y\setminus\{u_d\}$ has no new pre-cliques over $A$ in  the interval $(l_A, |d^T_{m_n-1}|]$,
it must be that $l'\le l_A$.
Since $Z'\re l_A$  has size at least two and is contained in $A$, the Witnessing Property of $A$ implies that 
there  is a set of coding nodes $C'$ contained in $A$ witnessing the pre-$(a-1)$-clique $Z\re l'$.
Then $C'\cup\{u_d\}$ is contained in $A\cup Y$ and witnesses the pre-$a$-clique $Z$.

Now, suppose that  $d\in I\sse\tilde{i}$,  $I$ has size at least two, and $\{u_i:i\in I\}$ has a new pre-clique over $A$. 
We will show that this is impossible. 
We point out that in the interval 
$(l^T_{n-1} +1, l_n^T)$,
the coding node   $u_d$  has no   new pre-cliques with any other node in $T$ of length $l^T_n$,
(for $T$ has the Strong Witnessing Property and  such a new pre-clique could not be witnessed in $T$).
Since for each 
 $i\ne d$,  the node $u_i\re l^T_{n-1}$   is the leftmost extension of $s_i$ in $T$,  it follows that 
for each $l\in (l_A,  l^T_{n-1}]$, the set 
$\{u_i\re l:i\in I\}$ has 
no new pre-cliques of size two or more.
Thus, any new pre-clique occuring among   $\{u_i:i\in I\}$ above $A$ must exclude $u_d$. 
It follows that $A\cup Y$ has the Witnessing Property, so (3) holds.

We have already shown that, assuming that $A$ has the  Witnessing Property,   $u_d$  is not 
in any subset of $Y$ of size at least two which has a new pre-clique over $A$, and that for $i\ne d$, any new pre-clique in  $\{u_i\}$  over $A$ is witnessed in $A\cup Y$. 
Assuming the premise of (4),  $s_d$ has a pre-$(k-1)$-clique  witnessed by coding nodes in $A$.
Thus, the Strong Witnessing Property of $A$ carries over to $A\cup Y$.
\end{proof}

The next lemma  provides conditions under which a
subtree of a strong coding tree can be extended to another  subtree with a prescribed strong similarity type.
This will be central to the constructions involved in proving the Ramsey theorems for strong coding trees as well as in further sections.

\begin{lem}\label{lem.facts}
Suppose  $A$ is  a  finite  subtree of a strong coding tree  $T\in\mathcal{T}_k$ with $\max(A)$ free in $T$.
Fix any member $u\in\max(A)$.
Let $X$ be any subset of $\max(A)$ such that for each $s\in X$,
the pair $\{s,u\}$ has no pre-$k$-cliques,
and let $Z$ denote $\max(A)\setminus(X\cup\{u\})$.
Let $l\ge l_A$ be given.

Then there is an $l_*>l$
 and there are extensions $u_*\supset u$,
$s_*^0,s_*^1\supset s$ for all $s\in X$,
and $s_*\supset s$ for all $s\in Z$, each of length $l_*$,
such that, letting
\begin{equation}
Y=\{u_*\}\cup\{s_*^i:s\in X,\ i \le 1\}\cup\{s_*:s\in Z\},
\end{equation}
and letting $B$ be the tree induced by $A\cup Y$,
 the following hold:
\begin{enumerate}
\item
$u_*$ is a coding node.
\item
For each $s\in X$ and $i\le 1$, the passing number of $s_*^i$ at $u_*$ is $i$.
\item
For each $s\in  Z$,
the passing number of
 $s_*$ at $u_*$ is $0$.
\item
Splitting among the extensions of the $s\in X$ occurs in reverse lexicographic order:
For $s$ and $t$ in $X$,
 $|s_*^0\wedge s_*^1|<|t^0_*\wedge t^1_*|$
if and only if $s_*>_{\mathrm{lex}}t_*$.
\item
There are no new  pre-cliques
among  the  nodes  in $X$
below the length of  the longest splitting node in $B$  below $u_*$.
\end{enumerate}
If moreover $A$ has the (Strong) Witnessing Property, then $B$ also has the (Strong) Witnessing Property.
\end{lem}

\begin{proof}
Since $\max(A)$ is free in $T$,
apply Lemma \ref{lem.factssplit}
to extend $\max(A)$ to have splitting nodes in the desired order
without adding any new  pre-cliques  and so that this extension is free in $T$.
Then apply Lemma
\ref{lem.pnc} to extend to a level with a coding node and passing numbers as prescribed,  with the extension being  valid in $T$.
 Lemma
\ref{lem.pnc} also guarantees that  we can construct such a $B$ with  the (Strong) Witnessing Property, provided that  $A$ has the  (Strong) Witnessing Property. 
\end{proof}

This immediately yields  the main extension theorem of this section.

\begin{thm}\label{thm.GOODnonempty}
Suppose  $T\in\mathcal{T}_k$, $m<\om$, 
and $A$ is a member of $\mathcal{AT}_m$  with $\max(A)$ free in   $T$.
Then the set
$r_{m+1}[A,T]$ is infinite.
In particular,
for each $l<\om$, there is a member $B\in r_{m+1}[A,T]$ with  $\max(B)$ free in  $T$ and  $l_B\ge l$.
Furthermore,  $[A,T]$ is  infinite, and
 for each strictly increasing sequence of integers $(l_j)_{j>m}$, there  is a member $S\in [A,T]$ such that
$|d^S_j|>l_j$ and $\max(r_j(S))$ is free in $T$, for each $j>m$.
\end{thm}

\begin{proof}
Recall that every member of $\mathcal{AT}_m$ has the Strong Witnessing Property. 
The first part of the theorem follows from Lemmas \ref{lem.factssplit} and \ref{lem.pnc}.
The second part 
 follows  from  Lemma \ref{lem.facts}.
\end{proof}

The final lemmas of this section set up for constructions in    the main theorem of Section \ref{sec.5}.

\begin{lem}\label{lem.HLCasebtruncate}
Suppose $T\in \mathcal{T}_k$,
$X\sse T$ is a level set containing a coding node $c^T_j$, and $X'\cup X''$ is a partition of $X$  with
$c^T_j\in X'$.
Suppose further that $j<n$, $c^T_n$ extends $c^T_j$,  and $c_n^T\in Y'\sse T\re l^T_n$
end-extends $X'$ with the following properties:
$Y'$ has no new pre-cliques over $X'$,
and each node in $Y'$ has the same passing number at $c^T_n$ as it does at $c^T_j$.
Then there is a level set $Y''\sse T\re l^T_n$  end-extending $X''$ such that each node in $Y''$ has the same passing number at $c^T_n$
as it does at $c^T_j$,
 and $Y=Y'\cup Y''$ has no new pre-cliques over $X$.
\end{lem}

\begin{proof}
If $n>j+1$,
first extend the nodes in $X''$ leftmost in $T$ to length $l^T_{n-1}$, and label this set of nodes $Y''\re l^T_{n-1}$.
By
Lemma \ref{lem.poc}
$Y''\re l^T_{n-1} \cup Y'\re l^T_{n-1}$  has no new pre-cliques over $X$.
Apply Lemma \ref{lem.pnc} to extend the nodes in $Y''\re l^T_{n-1}$ to $Y''\in T\re l^T_n$ such that
for each node $t\in Y''$,
$t$ has the same passing number at $c^T_n$ as it does at $c^T_j$.
Let $Y=Y'\cup Y''$.

Suppose towards a contradiction that for some $a\in [3,k]$, there is a new pre-$a$-clique
 $Z\sse Y$
 above $X$.
 If $a=3$, then $c^T_n$ witnesses this pre-$3$-clique.
 Since each node in $Y$ has the same passing number at $c^T_n$ as it does at $c^T_j$,
 it follows that $Z\re l^T_j$ has a pre-$3$-clique which is witnessed by $c^T_j$.
 Thus, $Z$ was not new over $X$.

 Now suppose that $a\ge 4$.
 Then $Z\cup \{c^T_n\}$ has a pre-$b$-clique, where $b=a-1$.
 Since $Z$ is a new pre-$a$-clique and $T\cong \bT_k$,
 it must be that the
 the level where the pre-$b$-clique in $Z\cup \{c^T_n\}$ is new must be at some $l\le l^T_{n-1}$.
 Since $Y$ has no new pre-cliques in the interval
 $(l^T_j,l^T_{n-1}]$,
 this $l$ must be less than or equal to  $l^T_j$.
 Since the passing numbers of members in $Y$ are the same at $c^T_n$ as they are at $c^T_j$,
 it follows that $Z\cup\{c^T_j\}$ has a pre-$b$-clique.
 This pre-$b$-clique must occur at some level strictly below $l^T_j$, since the passing number of the coding node $c^T_j$ at itself is $0$.
 Hence, $Z\re l^T_j\cup\{c^T_j\}$ is a pre-$a$-clique.
 Therefore, $Z$ is not a  new  pre-$a$-clique over $X$, a contradiction.
\end{proof}


\begin{lem}\label{lem.HLconstruction}
Suppose $T\in \mathcal{T}_k$, $m\in\bN$,  and $B\in r_{m+1}[0,T]$ with $\max(B)$ free in $T$.
 Let $x_*$ be the critical node  of $\max(B)$,
let  $X\sse \max(B)$ with $x_*\in X$,
 and  let  $X'=\max(B) \setminus X$.
Suppose that $Y$ end-extends $X$ into $T$ so that $Y$ has no new pre-cliques over $X$, $Y$  is free in $T$,
and the critical node $x_*$ is extended to the same type of critical node $y_*$ in $Y$.
If $x_*$ is a coding node,  assume that for each
$y\in Y$, the passing number of $y$ at $y_*$ is the same as the passing number of $y$ at $x_*$.
Then there is a level set $Y'$ end-extending  $X'$  in $T$  to length  $l_Y$ such that
 $r_m(B)\cup (Y\cup Y')$
is a member of $r_{m+1}[r_m(B),T]$.
\end{lem}

\begin{proof}
Suppose first that $x_*$ is a splitting node.
By  Lemma \ref{lem.poc},
letting $Y'$ be the level set of leftmost nodes in $T\re |y_*|$,
we see  that $Y\cup Y'$ is free in $T$ and has no new pre-cliques over $\max(B)$.
In particular, $r_m(B)\cup(Y\cup Y')$ is a member of $r_{m+1}[r_m(B),T]$.

Now suppose that $x_*$ is a coding node.
Let $n$ be the integer such that $y_*=c^T_n$.
Then $l_X\le l^T_{n-1}$.
Let $W'$ denote the leftmost extensions  of the nodes in $X'$  in $T\re  l^T_{n-1}$.
Again by  Lemma \ref{lem.poc}, 
the set $W'\cup (Y\re l^T_{n-1})$ has no new pre-cliques over $B$.
For $i\in 2$, let  $W'_i$ be the set of those $w\in W'$ which have passing number  $i$ at $x_*$.
Note that for each $w\in W'_1$,
the set $\{w\re l_X,x_*\}$
has no pre-$k$-clique, and since no new pre-cliques are added between the levels of $l_X$ and $l^T_{n-1}$,
the set $\{w,y_*\re l^T_{n-1}\}$ has no pre-$k$-clique.
Since $T$ satisfies the $K_k$-Free Branching Criterion,
each $w\in W'_1$ can be extended to a node $y\in T\re l_Y$ with passing number $1$ at $y_*$.
Extend each node in $W'_0$ leftmost in $T$ to length $l_Y$.
Let $Y'=W'_0\cup W'_1$.
Then  $Y'$ end-extends $W'$ which end-extends $X'$,
and each $y\in Y'$ has the same passing number at $y_*$ as it does at $x_*$.

We claim that $Y\cup Y'$ has no new pre-cliques over $B$.
Suppose towards a contradiction that $Z\sse Y\cup Y'$ is a new pre-$a$-clique above $B$, for some $a\in [3,k]$.
Since $Z\re  l^T_{n-1}$ has no new pre-cliques over $B$,
this new pre-$a$-clique must take place at some level
$l\in (l^T_{n-1}, l_Y]$.
Since $T$ has the Strong Witnessing Property, $l$ must be in the interval $(|d|,l_Y]$, where $d$ is the longest splitting node in $T$ of length less than $l_Y$, and $y_*$ must be among the set of witnessing coding nodes in $T$ for this new pre-$a$-clique.
If $a=3$, then $y_*$ witnesses the pre-$3$-clique $Z$.
But then  $Z\re l_X$ must also be a pre-$3$-clique, since the passing numbers at $y_*$ are the same as at $x_*$, and $x_*$ witnesses the pre-$3$-clique $Z\re l_X$.
Hence, $Z$ is not new over $B$.
Now suppose that $a\ge 4$.
Then $Z\cup\{y_*\}$  has a new pre-$(a-1)$-clique at some level $l'<l$.
Since $T$ has the Strong Witnessing Property
this new pre-$(a-1)$-clique must be witnessed in $T$.
Since  $T$ has no new pre-cliques in  $(l^T_{n-1},|d|]$
and $Z\re l^T_{n-1}$ has no new pre-cliques over $B$,
 it must be that  $l'\le l_X$.
Therefore, the minimal level of a  pre-$(a-1)$-clique  in $Z\cup\{y_*\}$ is   at some level in $B$.
Since $B$ has the Strong Witnessing Property, this is witnessed in $B$.
Since $y_*\contains x_*$ and each $z\in Z$ has the same passing number at $y_*$ as at $x_*$,
$x_*$ cannot be a witness of the
 pre-$(a-1)$-clique  in $Z\cup\{y_*\}$.
 Therefore,  $Z\cup\{y_*\}$ must be witnessed in $r_m(B)$, say by coding nodes $\{c^B_{i_j}:j<a-3\}$, where $i_{a-4}<l^B_m$.
But then $\{x_*\}\cup \{c^B_{i_j}:j<a-3\}$ witnesses  the pre-$a$-clique $Z$.
Hence, $Z$ is not new over $B$.

Now we will show that $Y\cup Y'$ has no new pre-cliques over $r_m(B)$.
Suppose $Z\sse Y\cup Y'$ has  a pre-$a$-clique, for some $a\in[3,k]$.
Since this pre-$a$-clique is not new over $B$, there is some $l\le l_X$ where $Z\re l$ is a new pre-$a$-clique in $B$.
Since $B$  has the Strong Witnessing Property,  there are some coding nodes $c_{i_0}^B,\dots c_{i_{a-3}}^B$ in $B$ witnessing $Z\re l$.
If $i_{a-3}<m$, then these witnesses are in $r_m(B)$.
Now suppose that  $i_{a-3}=m$.
Note that $y_*\contains x_*=c^B_m$.
Thus, $\{y_*\}\cup\{c^B_{i_j}:j<a-3\}$ forms a
pre-$(a-1)$-clique which  witnesses
$Z$.
Therefore, $r_m(B)\cup Y'\cup Y$ has the Strong Witnessing Property.
Since it is strongly similar to $B$,
$r_m(B)\cup Y'\cup Y$ is a member of $r_{m+1}[r_m(B),T]$ by Lemma \ref{lem.need}.
\end{proof}

\begin{rem}
As was remarked  for $\mathcal{T}_3$ in \cite{DobrinenJML20},
each
space $(\mathcal{T}_k,\le, r)$, $k\ge 3$,
satisfies  Axioms \bf A.1\rm, \bf A.2\rm, and \bf A.3(1) \rm of Todorcevic's   axioms  in Chapter 5 of \cite{TodorcevicBK10} guaranteeing a  topological Ramsey space, and it is routine to check this.
However,
Axiom
\bf A.3(2) \rm does not hold.
The pigeonhole principle, Axiom \bf A.4\rm,  holds exactly when
 the finite  subtree  is  valid inside  the  given strong coding tree;  this  will follow from Theorems in Section \ref{sec.5} and \ref{sec.1SPOC}.
\end{rem}


\section{Halpern-\Lauchli-style Theorems for strong coding trees}\label{sec.5}

The
Ramsey theory content for strong coding trees
begins in this section.
The  ultimate goal
is to obtain Theorem \ref{thm.mainRamsey}.
This is a
 Ramsey theorem for colorings of strictly similar (Definition \ref{defn.ssimtype}) copies of any given finite antichain of coding nodes, as these are the structures which will code finite triangle-free graphs.

Phase II of the paper takes place in this and the next section.
 Theorem \ref{thm.matrixHL} is a Halpern-\Lauchli-style theorem for colorings of level sets extending a finite  valid subtree of some strong coding tree.
 Its proof begins  with Harrington's forcing proof of Theorem \ref{thm.HL} as a rough template, but  involves new forcings and new arguments  necessitated by  the $k$-clique-free
 nature of Henson graphs.
 This is a major step toward proving a Milliken-style theorem for strong coding trees, but it is not enough: in  the case when there is a coding node in the level sets being colored,
 Theorem \ref{thm.matrixHL} proves
  homogeneity on level sets extending some fixed set of nodes, but does not obtain homogeneity overall.
We will have a third Halpern-\Lauchli-style theorem
 in Section  \ref{sec.1SPOC}, which involves ideas unprecedented to our knowledge.
 This Lemma \ref{lem.Case(c)} will use a third type of forcing.
 Then using much induction on Theorem \ref{thm.matrixHL}
 and
 Lemma  \ref{lem.Case(c)},
 we will prove the Main Ramsey Theorem for Strictly Witnessed (see Definition \ref {defn.SWP}) finite subtrees of a given strong coding tree.
 This Theorem
\ref{thm.MillikenSWP}
 is the main theorem of Phase II of the paper.

Theorem  \ref{thm.matrixHL}  encompasses colorings of two different types  of level set  extensions of a fixed finite tree: The  level set  either contains a splitting node (Case (a)) or a coding node (Case (b)).
In Case (a), we obtain a direct analogue of the
Halpern-\Lauchli\ Theorem.
In Case (b), we obtain a weaker version of the
Halpern-\Lauchli\ Theorem, which is later strengthened to
the direct analogue in Lemma \ref{lem.Case(c)}.
The proof given here essentially  follows the  outline
of the proof 
of Theorem 5.2 
in \cite{DobrinenJML20}, but our argument  is now  more streamlined, due to having proved more general extension lemmas
in Section \ref{sec.4}.

Let  $k\ge 3$ be fixed, and
fix the following terminology and notation.
Given subtrees  $U,V$ of $\bT_k$  with $U$ finite,
we write $U\sqsubseteq V$ if and only if
$U=\{v\in V:|v|\le l_U\}$;
in this case we say that $V$ {\em extends} $U$, or that $U$ is an {\em initial subtree of} $V$.
We write $U\sqsubset V$ if $U$ is a proper initial subtree of $V$.
Recall the following notation from Definition \ref{defn.T_pspace} of the space $(\mathcal{T}_k,\le, r)$:
$S\le T$ means that  $S$ and $T$ are members of
$\mathcal{T}_k$ and  $S$ is a subtree of $T$.
Given  $A\in\mathcal{AR}_m$ for some $m$,
 $[A,T]$ denotes the set of all   $S\le T$ such that  $S$ extends $A$.
We now begin setting up for the two possible  cases before stating the theorem.
\vskip.1in

\noindent\underline{\bf{The Set-up for Theorem \ref{thm.matrixHL}.}}
Let
$T\in\mathcal{T}_k$ be given,
and
 let
$A$ be a finite valid subtree of $T$.
$A$  is allowed to  have terminal  nodes at different levels.
In order to simplify notation in  the proof, without loss of generality, we
assume that $0^{(l_A)}$ is in $A$.
Let  $A^+$ denote the set of   immediate extensions in            $\widehat{T}$ of the members of $\max(A)$; thus,
\begin{equation}
A^+=\{s^{\frown}i : s\in \max(A),\ i\in\{0,1\},\mathrm{\ and\ } s^{\frown}i\in \widehat{T}\}.
\end{equation}
Note that   $A^+$ is a level set of nodes of length $l_A+1$.
Let  $A_e$ be a   subset of $A^+$
containing $0^{(l_A+1)}$ and of size at least two.
(If $A^+$ has only one member, then
$\max(A)$ consists of one non-splitting node of the form $0^{(l)}$ for some $l$, and
the theorem in this section does not apply.)
Suppose that  $\tilde{X}$ is a level set of nodes in $T$ extending $A_e$ so that $A\cup\tilde{X}$
 is a  finite valid subtree of $T$.
Assume moreover that  $0^{(l_{\tilde{X}})}$ is a member of $\tilde{X}$, so that the node $0^{(l_A)}$ in $A_e$ is extended by $0^{(l_{\tilde{X}})}$ in $\tilde{X}$.
There are two possibilities:

\begin{enumerate}
\item[]
\begin{enumerate}
\item[\bf{Case (a).}]
  $\tilde{X}$ contains a splitting node.
\end{enumerate}
\end{enumerate}

\begin{enumerate}
\item[]
\begin{enumerate}
\item[\bf{Case (b).}]
 $\tilde{X}$ contains a coding  node.
\end{enumerate}
\end{enumerate}
In both cases,
define
\begin{align}\label{eq.ExtTAC}
\Ext_T(A,\tilde{X})= \{X\sse T: \ &
 X\sqsupseteq \tilde{X} \mathrm{\ is\ a\ level\ set}, \
A\cup X\cong A\cup\tilde{X}, \cr
&\mathrm{\ and\ } A\cup X \mathrm{\  is\ valid\ in\ } T\}.
\end{align}

The next lemma
 shows that seemingly weaker properties suffice to guarantee that a level set is in $\Ext_T(A,
\tilde{X})$.

\begin{lem}\label{lem.alternate}
Let $X$ be a level set in $T$  extending $\tilde{X}$. Then
  $X\in\Ext_T(A,\tilde{X})$ if and only if
$X$ is free in $T$,
 $A\cup X$ is strongly similar to $A\cup \tilde{X}$, and
$X$ has no new pre-cliques over $A\cup \tilde{X}$.
Moreover, $X\in\Ext_T(A,\tilde{X})$ implies that $A\cup X$ has the Witnessing Property. 
\end{lem}

\begin{proof}
The forward direction follows from the definition of
$\Ext_T(A,\tilde{X})$.
If $X\in \Ext_T(A,\tilde{X})$, then there 
is a strong isomorphism, say
$f:A\cup \tilde{X}\ra A\cup X$.
Then 
$f$ is a strong similarity map and moreover, 
 $f$ takes $\tilde{X}$ to $X$.
 Since $X$ extends $\tilde{X}$ and $f$ maps the new pre-cliques of $\tilde{X}$ over $A$ to the new pre-cliques of $X$ over $A$,
 $X$ must have no new pre-cliques over $A\cup \tilde{X}$. 
 Note that $X$ is free in $T$, since $A\cup X$ is valid in $T$.

Now suppose that
$X\sqsupseteq \tilde{X}$  is as in the second part of the statement.
Since $X$ is  free in $T$,
$A\cup X$ is valid in $T$.
 $A\cup X$ being  strongly similar to $A\cup \tilde{X}$ 
 implies that the strong similarity map $g:A\cup\tilde{X}\ra A\cup X$ takes $\tilde{X}$ to $X$.
Since $X$ has no new pre-cliques over $A\cup\tilde{X}$,
any new pre-cliques in $X$ over $A$ are already in $A\cup \tilde{X}$ and hence witnessed by coding nodes in $A$ along possibly with the coding node $c_*$ in $\tilde{X}$ (in Case (b)).
If this is the case, then the coding node $f(c_*)$ in $X$ along with those same coding nodes in $A$ witness the new pre-clique in $X$ over $A$.  Therefore, $f$ is a strong isomorphism. 
It follows that since  $A\cup\tilde{X}$ has the  Witnessing Property,  so does $A\cup X$.
Also note that if moreover $A\cup\tilde{X}$ has the  Strong Witnessing Property, then $A\cup X$ does as well.
\end{proof}

In the following, for a finite subtree $A$ of some  $T\in\mathcal{T}_k$, recall that $\max(A)$ denotes the set $\{t\in A:|t|=l_A\}$,
the set of all nodes in $A$ of the maximum length, and
 $A^+$ denotes the set of all immediate successors of $\max(A)$ in $\widehat{T}$.
We now prove the analogue of the Halpern-\Lauchli\ Theorem for strong coding trees.

\begin{thm}\label{thm.matrixHL}
Fix  $T\in\mathcal{T}_k$  and  $B$ a finite  valid subtree of $T$ such that $B\in \mathcal{AT}^k_m$, for some $m\ge 1$.
Let  $A$ be a  subtree of  $B$  with  $l_A=l_B$ and  $0^{(l_A)}\in A$
such that
$A$  is valid in $T$.
Let $A_e$ be a subset of $A^+$ of size at least two such that
$0^{(l_A+1)}$ is in $A_e$.
Let $\tilde{X}$ be a level set in $T$ end-extending $A_e$
with  at least two members, one of which is the node $0^{(l_{\tilde{X}})}$
such that
$A\cup\tilde{X}$ is a finite valid subtree of $T$.

Given any coloring
$h:\Ext_T(A,\tilde{X})\ra 2$,
 there is a strong coding tree $S\in [B,T]$ such that
$h$ is monochromatic on $\Ext_S(A,\tilde{X})$.
If $\tilde{X}$ has a coding node, then
the strong coding tree $S$ is, moreover, taken to be in $[r_{m_0-1}(B'),T]$,
where
$m_0$ is the integer
for which there is a $B'\in r_{m_0}[B,T]$ with
$\tilde{X}\sse\max(B')$.
\end{thm}

\begin{proof}
Let $T,A,A_e,B,\tilde{X}$ be given satisfying the hypotheses,
and
let $h$ be a coloring of the members of $\Ext_T(A,\tilde{X})$ into two colors, $\{0,1\}$.
Fix the following notation:
Let $d+1$ equal the number of  nodes in  $\tilde{X}$,
and enumerate the nodes in    $\tilde{X}$ as
 $s_0,\dots, s_d$
so that $s_d$  is the critical node in $\tilde{X}$.
Let $i_0$ denote the integer such that
 $s_{i_0}$ is  the node  which is a sequence of $0$'s.
Notice that  $i_0$ can equal $d$ only if we are  in  Case (a)  and
the  splitting node in $\tilde{X}$ is a sequence of $0$'s.
In Case (b), let $I_{0}$ denote  the set of all $i<d$ such that $s^+_i(l_{\tilde{X}})=0$
and let
$I_{1}$ denote  the set of all $i<d$ such that $s^+_i(l_{\tilde{X}})=1$.

Let $L$ denote the collection of all  $l\in\bN$ such that there is a member of
 $\Ext_T(A,\tilde{X})$ with  nodes of length $l$.
 In Case (a),
 since $B$ is valid in $T$,
it follows from 
 Lemmas \ref{lem.poc} and \ref{lem.perfect} that 
$L$  consists of those $l\in\bN$ for which
 there is a  splitting node  of length $l$ extending $s_d$ and  that $L$ 
is infinite.
 In Case (b),
 since $\tilde{X}$ contains a coding node, it follows from   Lemma \ref{lem.pnc} that
 $L$  is exactly the set of  all $l\in\bN$ for which
 there is a coding  node  of length $l$ extending $s_d$ and that $L$ is infinite.

For each
  $i\in (d+1)\setminus\{i_0\}$,   let  $T_i=\{t\in T:t\contains s_i\}$.
  Let $\Seq[0]$ denote the set of all sequences of $0$'s of finite length.
  Let $T_{i_0}=\{t\in T:t\contains s_{i_0}$ and $t\in \Seq[0]\}$, the collection of all leftmost nodes in $T$ extending $s_{i_0}$.
Let $\kappa=\beth_{2d}$.
The following forcing notion $\bP$    adds $\kappa$ many paths through  $T_i$, for each  $i\in d\setminus\{i_0\}$,
and one path through $T_d$.
If $i_0\ne d$, then $\bP$ will add one
 path through $T_{i_0}$, but with $\kappa$ many ordinals labeling this path.
 We allow this in order to simplify notation.
\vskip.1in

$\bP$ is the set of conditions $p$ such that
$p$ is a  function
of the form
$$
p:(d\times\vec{\delta}_p)\cup\{d\}\ra T\re l_p,
$$
where $\vec{\delta}_p$ is a finite subset of $\kappa$,
 $l_p\in L$,
 $\{p(i,\delta) : \delta\in  \vec{\delta}_p\}\sse  T_i\re l_p$ for each $i<d$, and the following hold:
\vskip.1in

\noindent  \underline{Case (a)}.
(i) $p(d)$ is {\em the} splitting  node extending $s_d$ of length  $l_p$;
\begin{enumerate}
\item [(ii)]
$\{p(i,\delta):(i,\delta)\in d\times\vec{\delta}_p\}\cup \{p(d)\}$  is free in $T$.
\end{enumerate}
\vskip.1in

\noindent  \underline{Case (b)}.  (i)
$p(d)$ is {\em the} coding  node extending $s_d$ of length $l_p$;
\begin{enumerate}
\item [(ii)]
For each $\delta\in\vec{\delta}_p$,
 $j\in \{0,1\}$,
and  $i\in I_j$, the passing number of $p(i,\delta)$ at $p(d)$ is $j$.
\end{enumerate}
\vskip.1in

Given $p\in\bP$,
 the {\em range of $p$} is defined as
$$
\ran(p)=\{p(i,\delta):(i,\delta)\in d\times \vec\delta_p\}\cup \{p(d)\}.
$$
If also $q\in \bP$ and $\vec{\delta}_p\sse \vec{\delta}_q$, then we let $\ran(q\re \vec{\delta}_p)$ denote
$\{q(i,\delta):(i,\delta)\in d\times \vec{\delta}_p\}\cup \{q(d)\}$.
 In both Cases (a) and (b), the partial
 ordering on $\bP$  is defined as follows:
$q\le p$ if and only if
$l_q\ge l_p$, $\vec{\delta}_q\contains \vec{\delta}_p$,  and the following hold:
\begin{enumerate}
\item[(i)]
$q(d)\contains p(d)$,
and
$q(i,\delta)\contains p(i,\delta)$ for each $(i,\delta)\in d\times \vec{\delta}_p$, and

\item[(ii)]
$\ran(q\re\vec{\delta}_p)$ has no new
pre-cliques over $\ran(p)$.
\end{enumerate}

Since all conditions in $\bP$ have ranges which are free in $T$,
we shall  say that  {\em $q$ is valid over $p$}
to mean that  (ii) holds.

The  theorem  will be proved in two main  parts.
In Part I,
we  check that    $\bP$ is an atomless partial order and then  prove  the main  Lemma \ref{lem.compat}.
In Part II, we apply Lemma \ref{lem.compat} to build
 the tree $S$
 such that
 $h$ is monochromatic on
 $\Ext_S(A,\tilde{X})$.
\vskip.1in

\noindent \underline{\bf {Part I.}}

\begin{lem}\label{lem.atomlesspo}
$(\bP,\le)$ is an atomless partial ordering.
\end{lem}

\begin{proof}
The  order  $\le$ on $\bP$ is clearly  reflexive and antisymmetric.
Transitivity follows from the fact that the requirement (ii) in the definition of the partial order on $\bP$ is a transitive property.
To see this, suppose that
 $p\ge q$ and $q\ge r$.
Then $\vec{\delta}_p\sse\vec{\delta}_q\sse \vec{\delta}_r$, $l_p\le l_q\le l_r$,
$r$ is valid over $q$, and $q$ is valid over $p$.
Since
$\ran(r\re \vec{\delta}_p)$ is   contained in $\ran(r\re \vec{\delta}_q)$ which
 has no new pre-cliques over $\ran(q)$,
  it follows that
$\ran(r\re \vec{\delta}_p)$ has  no new pre-cliques over $\ran(q\re \vec{\delta}_p)$.
Since
 $\ran(q\re \vec{\delta}_p)$ has  no new pre-cliques over $\ran(p)$,
 it follows that $\ran(r\re \vec{\delta}_p)$
  has  no new pre-cliques over $\ran(p)$.
Therefore, $r$ is valid over $p$, so
$p\ge r$.

\begin{claim}\label{claim.densehigh}
For each $p\in\bP$ and   $l>l_p$, there
are  $q,r\in\bP$ with $l_q,l_r> l$  such that  $q,r<p$ and $q$ and $r$ are incompatible.
\end{claim}

\begin{proof}
Let $p\in \bP$ and $l>l_p$ be given, and
let
$\vec{\delta}$ denote $\vec{\delta}_p$ and  let
$\vec{\delta}_r=\vec{\delta}_q=\vec{\delta}$.
In Case (a),
take $q(d)$ and $r(d)$ to be incomparable splitting nodes in $T$ extending $p(d)$ to some lengths greater than $l$.
Such splitting nodes exist  by Lemma \ref{lem.perfect},
showing that  strong coding trees are perfect.
Let $l_q=|q(d)|$ and $l_r=|r(d)|$.
For each $(i,\delta)\in d\times\vec{\delta}$,
let $q(i,\delta)$ be the leftmost extension in $T$ of $p(i,\delta)$ to length $l_q$,
and  let
$r(i,\delta)$ be the leftmost extension of $p(i,\delta)$ to length $l_r$.
Then $q$ and $r$ are members of $\bP$.
Since $\ran(p)$ is free in $T$,
both $\ran(q)$ and $\ran(r)$ are free in $T$ and
$\ran(q\re \vec{\delta}_p)$ and $\ran(r\re \vec{\delta}_p)$
have no new pre-cliques over $\ran(p)$, by Lemma \ref{lem.poc}.
 It follows
that $q$ and $r$ are both  valid over $p$.
Since neither of $q(d)$ and $r(d)$ extends the other,
$q$ and $r$ are incompatible.

In Case (b),
let $s$ be a splitting node in $T$ of length greater than $l$ extending $p(d)$.
Let $k$ be minimal such that  $|c^T_k|\ge |s|$.
Let $u,v$ extend $s^{\frown}0,s^{\frown}1$,  respectively,  leftmost in $T\re l^T_k$.
For each $(i,\delta)\in d\times\vec\delta_p$,
let $p'(i,\delta)$ be the leftmost extension
of $p(i,\delta)$
 in $T\re l^T_k$.
 By  Lemma \ref{lem.pnc},
there are  $q(d)\contains u$ and   $q(i,\delta)\contains p'(i,\delta)$, $(i,\delta)\in d\times \vec{\delta}_p$,
such that
\begin{enumerate}
\item
 $q(d)$ is a coding node;
\item
$q$  is valid over  $p$;
\item
 For each $j<2$,
$i\in I_j$ if and only if the immediate extension of $q(i,\delta)$ is $j$.
\end{enumerate}
Then $q\in \bP$ and $q\le p$.
Likewise  by  Lemma \ref{lem.pnc},
there is a condition $r\in\bP$ which extends
$\{p'(i,\delta): (i,\delta)\in d\times \vec{\delta}_p\}\cup\{v\}$ such that
$r\le p$.
Since the coding nodes $q(d)$ and $r(d)$ are incomparable, $q$ and $r$  are incompatible conditions in $\bP$.
\end{proof}

It follows from Claim \ref{claim.densehigh} that $\bP$ is atomless.
\end{proof}

From now on, whenever
 ambiguity will  not arise by doing so, for a condition $p\in\bP$,
we will use the terminology
{\em critical node} of $p$ to
 refer to $p(d)$, which is
a splitting node  in Case (a) and a coding node  in Case (b).
Let $\dot{b}_d$ be a $\bP$-name for the generic path through $T_d$;
that is, $\dot{b}_d=\{\lgl p(d),p\rgl:p\in\bP\}$.
Note that for each $p\in \bP$, $p$ forces that $\dot{b}_d\re l_p= p(d)$.
By Claim \ref{claim.densehigh}, it is dense to force a critical node in $\dot{b}_d$ above any given level in $T$,  so $\mathbf{1}_{\bP}$ forces that the set of levels of critical nodes in $\dot{b}_d$ is infinite.
Thus, given any generic  filter  $G$  for $\bP$, $\dot{b}_d^G=\{p(d):p\in G\}$ is a cofinal path of critical nodes in $T_d$.
Let $\dot{L}_d$ be a $\bP$-name for the set of lengths of critical  nodes in $\dot{b}_d$.
Note that $\mathbf{1}_{\bP}\forces \dot{L}_d\sse L$.
Let $\dot{\mathcal{U}}$ be a $\bP$-name for a non-principal ultrafilter on $\dot{L}_d$.
For  $i<d$ and $\al<\kappa$,  let $\dot{b}_{i,\al}$ be a $\bP$-name for the $\al$-th generic branch through $T_i$;
that is, $\dot{b}_{i,\al}=\{\lgl p(i,\al),p\rgl:p\in \bP$ and $\al\in\vec{\delta}_p\}$.
Then
for any $p\in \bP$ ,
\begin{equation}
 p\forces  (\forall i<d\  \forall \al\in \vec\delta_p\, (\dot{b}_{i,\al}\re l_p= p(i,\al)) )\wedge
( \dot{b}_d\re l_p=p(d)).
\end{equation}

For $j\in\bN$, we let $[\kappa]^j$ denote the collection of all $j$-element  subsets of $\kappa$.
We  shall write sets $\{\al_i:i< d\}$ in $[\kappa]^d$ as vectors $\vec{\al}=\lgl \al_0,\dots,\al_{d-1}\rgl$ in strictly increasing order.
For $\vec{\al}\in[\kappa]^d$,
we use the following abbreviation:
\begin{equation}
\dot{b}_{\vec{\al}}\mathrm{\ \ denotes \  \ }
\lgl \dot{b}_{0,\al_0},\dots, \dot{b}_{d-1,\al_{d-1}},\dot{b}_d\rgl.
\end{equation}
Since the branch $\dot{b}_d$ is  unique, this abbreviation introduces no ambiguity.
For any $l<\om$,
\begin{equation}
\mathrm{\ let\ \ }\dot{b}_{\vec\al}\re l
\mathrm{\ \ denote \  \ }
\lgl \dot{b}_{0,\al_0}\re l,\dots, \dot{b}_{d-1,\al_{d-1}}\re l,\dot{b}_d\re l\rgl.
\end{equation}
Using the abbreviations just defined,
$h$ is a coloring on sets of nodes of the form $\dot{b}_{\vec\al}\re l$
whenever this is
 forced to be a member of $\Ext_T(A,\tilde{X})$.
Given $\vec{\al}\in [\kappa]^d$ and a condition $p\in \bP$ with $\vec\al\sse\vec{\delta}_p$,
let
\begin{equation}
X(p,\vec{\al})=\{p(i,\al_i):i<d\}\cup\{p(d)\}.
\end{equation}

We now set up to prove
Lemma \ref{lem.compat}.
For each $\vec\al\in[\kappa]^d$,
choose a condition $p_{\vec{\al}}\in\bP$ such that
\begin{enumerate}
\item
 $\vec{\al}\sse\vec{\delta}_{p_{\vec\al}}$.

\item
$X(p_{\vec\al},\vec{\al})\in\Ext_T(A,\tilde{X})$.
\item
There is an $\varepsilon_{\vec{\al}}\in 2$
 such that
$p_{\vec{\al}}\forces$
``$h(\dot{b}_{\vec{\al}}\re l)=\varepsilon_{\vec{\al}}$
for $\dot{\mathcal{U}}$ many $l$ in $\dot{L}_d$''.
\item
$h(X(p_{\vec\al},\vec{\al}))=\varepsilon_{\vec{\al}}$.
\end{enumerate}

Properties (1) -  (4) can be guaranteed  as follows.
Recall that $\{s_i:i\le d\}$ enumerates $\tilde{X}$ and that $s_d$ is the critical node in $\tilde{X}$.
For each  $\vec{\al}\in[\kappa]^d$, define
$$
p^0_{\vec{\al}}=\{\lgl (i,\delta), t_i\rgl: i< d, \ \delta\in\vec{\al} \}\cup\{\lgl d,t_d\rgl\}.
$$
Then $p^0_{\vec{\al}}$ is a condition in $\bP$ with
$\ran(p^0_{\vec{\al}})=\tilde{X}$, and
$\vec\delta_{p_{\vec\al}^0}= \vec\al$ which implies (1) holds for any $p\le p^0_{\vec{\al}}$.
The following fact will be used many times.

\begin{claim}\label{claim.extensiongood}
Given $\vec{\al}\in[\kappa]^d$,
for any $p\le p_{\vec\al}^0$, the set of nodes
$X(p,\vec{\al})$
 is a member of
  $\Ext_T(A,\tilde{X})$.
\end{claim}

\begin{proof}
Suppose   $p\le p_{\vec\al}^0$.
Then
$p$ is valid over $ p_{\vec\al}^0$, so
$X(p,\vec{\al})$ has no new  pre-cliques over  $\tilde{X}$.
Since  $p$ is a condition of  $\bP$, $X(p,\vec{\al})$ is free in $T$ and
$A\cup X(p,\vec{\al})$ is strongly similar to $A\cup\tilde{X}$.
It follows  from Lemma \ref{lem.alternate}
that $X(p,\vec{\al})$ is in
$\Ext_T(A,\tilde{X})$.
\end{proof}

Thus, (2) holds for any $p\le p_{\vec\al}^0$.
Take  an extension $p^1_{\vec{\al}}\le p^0_{\vec{\al}}$ which
forces  $h(\dot{b}_{\vec{\al}}\re l)$ to be the same value for
$\dot{\mathcal{U}}$  many  $l\in \dot{L}_d$.
Since $\bP$ is a forcing notion, there is a $p^2_{\vec{\al}}\le p_{\vec{\al}}^1$ deciding a value $\varepsilon_{\vec{\al}}$ for which $p^2_{\vec{\al}}$ forces that $h(\dot{b}_{\vec{\al}}\re l)=\varepsilon_{\vec{\al}}$
for $\dot{\mathcal{U}}$ many $l$ in $\dot{L}_d$.
Then (3) holds for any $p\le p_{\vec\al}^2$.
If $ p_{\vec\al}^2$ satisfies (4), then let $p_{\vec\al}=p_{\vec\al}^2$.
Otherwise,
take  some  $p^3_{\vec\al}\le p^2_{\vec\al}$
which decides
some $l\in\dot{L}_d$
such that
$l_{p^2_{\vec\al}}< l_n^T< l\le l_{p^3_{\vec\al}}$,
 for some $n$,
and  $p^3_{\vec\al}$ forces
$h(\dot{b}_{\vec\al}\re l)=\varepsilon_{\vec\al}$.
Since $p^3_{\vec\al}$ forces ``$\dot{b}_{\vec\al}\re l=
\{p^3_{\vec\al}(i,\al_i)\re l:i<d\}   \cup\{p^3_{\vec\al}(d)\re l\}$'' and $h$ is defined in the ground model,
 this means that  $p^3_{\vec\al}(d)\re l$ is a splitting node in Case (a) and a coding node in Case (b), and
\begin{equation}\label{eq.hrest}
h(X(p^3_{\vec\al},\vec\al)\re l)
=\varepsilon_{\vec\al},
\end{equation}
where
$X(p^3_{\vec\al},\vec\al)\re l$ denotes
$\{p^3_{\vec\al}(i,\al_i)\re l:i<d\}   \cup\{p^3_{\vec\al}(d)\re l\}$.
If $l=l_{p^3_{\vec\al}}$, let $p_{\vec\al}=p_{\vec\al}^3$, and note that $p_{\vec\al}$ satisfies (1) - (4).

Otherwise,  $l<l_{p^3_{\vec\al}}$.
In Case (a),  let
$p_{\vec\al}$ be  defined as follows:
Let  $\vec\delta_{\vec\al}=\vec\delta_{p_{\vec\al}^2}$ and
 \begin{equation}
\forall (i,\delta)\in d\times\vec\delta_{\vec\al}, \mathrm{\ let\ }
p_{\vec\al}(i,\delta)=p^3_{\vec\al}(i,\delta)\re l\mathrm{\ \ and\  let\  }
p_{\vec\al}(d)=p^3_{\vec\al}(d)\re l.
\end{equation}
Since $p^3_{\vec\al}$ is a condition in $\bP$,
$\ran(p^3_{\vec\al})$ is free in $T$.
Furthermore, $p^3_{\vec\al}\le p^2_{\vec\al}$
implies that $\ran(p_{\vec\al}^3\re \vec\delta_{p^2_{\vec\al}})$ has no new pre-cliques over $\ran(p^2_{\vec\al})$.
Therefore, leftmost extensions of $\ran(p_{\vec\al})$ have no new pre-cliques, so $\ran(p_{\vec\al})$ is free in $T$.
Therefore,
$p_{\vec\al}$ is a condition in $\bP$ and  $p_{\vec\al}\le p_{\vec\al}^2$.
Thus, $p_{\vec\al}$  satisfies (1) - (3), and (4) holds  by equation (\ref{eq.hrest}).

In Case (b),
we  construct $p_{\vec\al}\le p^2_{\vec\al}$ as follows:
As in Case (a), let
 $\vec{\delta}_{\vec\al}=\vec\delta_{p^2_{\vec\al}}$.
For each $i<d$, define
$p_{\vec\al}(i,\al_i)=p^3_{\vec\al}(i,\al_i)\re l$,
and  let
$p_{\vec\al}(d)=p^3_{\vec\al}(d)\re l$.
Then $X(p_{\vec\al},\vec\al)=\{p^3_{\vec\al}(i,\al_i)\re l:i<d\}\cup\{p^3_{\vec\al}(d)\re l\}$,
 so   $h(X(p_{\vec\al},\vec\al))=\varepsilon_{\vec\al}$.
Let $U$ denote $X( p_{\vec\al}^2,\vec\al)$ and
let $U'=\ran (p_{\vec\al}^2)\setminus U$.
Let $X$ denote $X(p_{\vec\al},\vec\al)$ and note that $X$ end-extends $U$,  and $X$ is free in $T$ and has no new pre-cliques over $U$.
By Lemma
\ref {lem.HLCasebtruncate},
there is an $X'$ end-extending $U'$ to nodes in $T\re l$ so that  the following hold:
$X\cup X'$ is free in $T$ and has no new pre-cliques over $U\cup U'$;
furthermore, each node in $X'$ has the same passing  number at $l$ as it does at $l_{p_{\vec\al}^2}$.
Let $\ran(p_{\vec\al})$ be this set of nodes $X\cup X'$,
where
 for each $i<d$
and $(i,\delta)\in d\times\vec{\delta}_{p_{\vec\al}^3}$ with $\delta\ne\al_i$,
we
let $p_{\vec\al}(i,\delta)$ be the node in $Y'$ extending
 $p_{\vec\al}^3(i,\delta)$.
 This defines a condition $p_{\vec\al}\le p_{\vec\al}^2$
  satisfying  (1) - (4).

The rest of Part I follows by  arguments  in \cite{DobrinenJML20} for  the case $k=3$, with no modifications.
It is included here for the reader's convenience.
We are assuming $\kappa=\beth_{2d}$ so  that  $\kappa\ra(\aleph_1)^{2d}_{\aleph_0}$, by the  \Erdos-Rado Theorem (Theorem \ref{thm.ER}).
Given two sets of ordinals $J,K$ we shall write $J<K$ if every member of $J$ is less than every member of $K$.
Let $D_e=\{0,2,\dots,2d-2\}$ and  $D_o=\{1,3,\dots,2d-1\}$, the sets of  even and odd integers less than $2d$, respectively.
Let $\mathcal{I}$ denote the collection of all functions $\iota: 2d\ra 2d$ such that
$\iota\re D_e$
and $\iota\re D_o$ are strictly  increasing sequences
and $\{\iota(0),\iota(1)\}<\{\iota(2),\iota(3)\}<\dots<\{\iota(2d-2),\iota(2d-1)\}$.
Thus, each $\iota$ codes two strictly increasing sequences $\iota\re D_e$ and $\iota\re D_o$, each of length $d$.
For $\vec{\theta}\in[\kappa]^{2d}$,
$\iota(\vec{\theta}\,)$ determines the pair of sequences of ordinals $(\theta_{\iota(0)},\theta_{\iota(2)},\dots,\theta_{\iota(2d-2))}), (\theta_{\iota(1)},\theta_{\iota(3)},\dots,\theta_{\iota(2d-1)})$,
both of which are members of $[\kappa]^d$.
Denote these as $\iota_e(\vec\theta\,)$ and $\iota_o(\vec\theta\,)$, respectively.
To ease notation, let $\vec{\delta}_{\vec\al}$ denote $\vec\delta_{p_{\vec\al}}$,
 $k_{\vec{\al}}$ denote $|\vec{\delta}_{\vec\al}|$,
and let $l_{\vec{\al}}$ denote  $l_{p_{\vec\al}}$.
Let $\lgl \delta_{\vec{\al}}(j):j<k_{\vec{\al}}\rgl$
denote the enumeration of $\vec{\delta}_{\vec\al}$
in increasing order.

Define a coloring  $f$ on $[\kappa]^{2d}$ into countably many colors as follows:
Given  $\vec\theta\in[\kappa]^{2d}$ and
 $\iota\in\mathcal{I}$, to reduce the number of subscripts,  letting
$\vec\al$ denote $\iota_e(\vec\theta\,)$ and $\vec\beta$ denote $\iota_o(\vec\theta\,)$,
define
\begin{align}\label{eq.fiotatheta}
f(\iota,\vec\theta\,)= \,  &
\lgl \iota, \varepsilon_{\vec{\al}}, k_{\vec{\al}}, p_{\vec{\al}}(d),
\lgl \lgl p_{\vec{\al}}(i,\delta_{\vec{\al}}(j)):j<k_{\vec{\al}}\rgl:i< d\rgl,\cr
& \lgl  \lgl i,j \rgl: i< d,\ j<k_{\vec{\al}},\ \mathrm{and\ } \delta_{\vec{\al}}(j)=\al_i \rgl, \cr
& \lgl \lgl j,k\rgl:j<k_{\vec{\al}},\ k<k_{\vec{\beta}},\ \delta_{\vec{\al}}(j)=\delta_{\vec{\beta}}(k)\rgl\rgl.
\end{align}
Let $f(\vec{\theta}\,)$ be the sequence $\lgl f(\iota,\vec\theta\,):\iota\in\mathcal{I}\rgl$, where $\mathcal{I}$ is given some fixed ordering.
Since the range of $f$ is countable,
apply the \Erdos-Rado Theorem
to  obtain a subset $K\sse\kappa$ of cardinality $\aleph_1$
which is homogeneous for $f$.
Take $K'\sse K$ such that between each two members of $K'$ there is a member of $K$.
Take subsets $K_i\sse K'$ such that  $K_0<\dots<K_{d-1}$
and   each $|K_i|=\aleph_0$.

\begin{lem}\label{lem.onetypes}
There are $\varepsilon^*\in 2$, $k^*\in\om$, $t_d$,
and $ \lgl t_{i,j}: j<k^*\rgl$, $i< d$,
 such that
for all $\vec{\al}\in \prod_{i<d}K_i$ and  each $i< d$,
 $\varepsilon_{\vec{\al}}=\varepsilon^*$,
$k_{\vec\al}=k^*$,  $p_{\vec{\al}}(d)=t_d$, and
$\lgl p_{\vec\al}(i,\delta_{\vec\al}(j)):j<k_{\vec\al}\rgl
=
 \lgl t_{i,j}: j<k^*\rgl$.
\end{lem}

\begin{proof}
Let  $\iota$ be the member in $\mathcal{I}$
which is the identity function on $2d$.
For any pair $\vec{\al},\vec{\beta}\in \prod_{i<d}K_i$, there are $\vec\theta,\vec\theta'\in [K]^{2d}$
such that
$\vec\al=\iota_e(\vec\theta\,)$ and $\vec\beta=\iota_e(\vec\theta'\,)$.
Since $f(\iota,\vec\theta\,)=f(\iota,\vec\theta'\,)$,
it follows that $\varepsilon_{\vec\al}=\varepsilon_{\vec\beta}$, $k_{\vec{\al}}=k_{\vec{\beta}}$, $p_{\vec{\al}}(d)=p_{\vec{\beta}}(d)$,
and $\lgl \lgl p_{\vec{\al}}(i,\delta_{\vec{\al}}(j)):j<k_{\vec{\al}}\rgl:i< d\rgl
=
\lgl \lgl p_{\vec{\beta}}(i,\delta_{\vec{\beta}}(j)):j<k_{\vec{\beta}}\rgl:i< d\rgl$.
Thus, define  $\varepsilon^*$, $k^*$, $t_d$, $\lgl \lgl t_{i,j}:j<k^*\rgl:i<d\rgl$ to be
$\varepsilon_{\vec\al}$, $k_{\vec\al}$,
$p_{\vec\al}(d)$,
$\lgl \lgl p_{\vec{\al}}(i,\delta_{\vec{\al}}(j)):j<k_{\vec{\al}}\rgl:i< d\rgl$
 for any $\vec\al\in \prod_{i<d}K_i$.
\end{proof}

Let $l^*$ denote the length of $t_d$.
Then all the nodes  $t_{i,j}$,  $i< d$, $j<k^*$,   also  have  length $l^*$.

\begin{lem}\label{lem.j=j'}
Given any $\vec\al,\vec\beta\in \prod_{i<d}K_i$,
if $j,k<k^*$ and $\delta_{\vec\al}(j)=\delta_{\vec\beta}(k)$,
 then $j=k$.
\end{lem}

\begin{proof}
Let $\vec\al,\vec\beta$ be members of $\prod_{i<d}K_i$   and suppose that
 $\delta_{\vec\al}(j)=\delta_{\vec\beta}(k)$ for some $j,k<k^*$.
For each $i<d$, let  $\rho_i$ be the relation from among $\{<,=,>\}$ such that
 $\al_i\,\rho_i\,\beta_i$.
Let   $\iota$ be the member of  $\mathcal{I}$  such that for each $\vec\gamma\in[K]^{d}$ and each $i<d$,
$\theta_{\iota(2i)}\ \rho_i \ \theta_{\iota(2i+1)}$.
Then there is a
$\vec\theta\in[K']^{2d}$ such that
$\iota_e(\vec\theta)=\vec\al$ and $\iota_o(\vec\theta)= \vec\beta$.
Since between any two members of $K'$ there is a member of $K$, there is a
 $\vec\gamma\in[K]^{d}$ such that  for each $i< d$,
 $\al_i\,\rho_i\,\gamma_i$ and $\gamma_i\,\rho_i\, \beta_i$,
and furthermore, for each $i<d-1$,
$\{\al_i,\beta_i,\gamma_i\}<\{\al_{i+1},\beta_{i+1},\gamma_{i+1}\}$.
Given that  $\al_i\,\rho_i\,\gamma_i$ and $\gamma_i\,\rho_i\, \beta_i$ for each $i<d$,
there are  $\vec\mu,\vec\nu\in[K]^{2d}$ such that $\iota_e(\vec\mu)=\vec\al$,
$\iota_o(\vec\mu)=\vec\gamma$,
$\iota_e(\vec\nu)=\vec\gamma$, and $\iota_o(\vec\nu)=\vec\beta$.
Since $\delta_{\vec\al}(j)=\delta_{\vec\beta}(k)$,
the pair $\lgl j,k\rgl$ is in the last sequence in  $f(\iota,\vec\theta)$.
Since $f(\iota,\vec\mu)=f(\iota,\vec\nu)=f(\iota,\vec\theta)$,
also $\lgl j,k\rgl$ is in the last  sequence in  $f(\iota,\vec\mu)$ and $f(\iota,\vec\nu)$.
It follows that $\delta_{\vec\al}(j)=\delta_{\vec\gamma}(k)$ and $\delta_{\vec\gamma}(j)=\delta_{\vec\beta}(k)$.
Hence, $\delta_{\vec\gamma}(j)=\delta_{\vec\gamma}(k)$,
and therefore $j$ must equal $k$.
\end{proof}

For any $\vec\al\in \prod_{i<d}K_i$ and any $\iota\in\mathcal{I}$, there is a $\vec\theta\in[K]^{2d}$ such that $\vec\al=\iota_o(\vec\theta)$.
By homogeneity of $f$ and  by the first sequence in the second line of equation  (\ref{eq.fiotatheta}), there is a strictly increasing sequence
$\lgl j_i:i< d\rgl$  of members of $k^*$ such that for each $\vec\al\in \prod_{i<d}K_i$,
$\delta_{\vec\al}(j_i)=\al_i$.
For each $i< d$, let $t^*_i$ denote $t_{i,j_i}$.
Then  for each $i<d$ and each $\vec\al\in \prod_{i<d}K_i$,
\begin{equation}
p_{\vec\al}(i,\al_i)=p_{\vec{\al}}(i, \delta_{\vec\al}(j_i))=t_{i,j_i}=t^*_i.
\end{equation}
Let $t_d^*$ denote $t_d$.

\begin{lem}\label{lem.compat}
For any finite subset $\vec{J}\sse \prod_{i<d}K_i$,
the set of conditions $\{p_{\vec{\al}}:\vec{\al}\in \vec{J}\,\}$ is  compatible.
Moreover,
$p_{\vec{J}}:=\bigcup\{p_{\vec{\al}}:\vec{\al}\in \vec{J}\,\}$
is a member of $\bP$ which is below each
$p_{\vec{\al}}$, $\vec\al\in\vec{J}$.
\end{lem}

\begin{proof}
For any $\vec\al,\vec\beta\in \prod_{i<d}K_i$,
whenver
 $j,k<k^*$ and
 $\delta_{\vec\al}(j)=\delta_{\vec\beta}(k)$, then $j=k$, by Lemma  \ref{lem.j=j'}.
It then follows from Lemma \ref{lem.onetypes}
that for each $i<d$,
\begin{equation}
p_{\vec\al}(i,\delta_{\vec\al}(j))=t_{i,j}=p_{\vec\beta}(i,\delta_{\vec\beta}(j))
=p_{\vec\beta}(i,\delta_{\vec\beta}(k)).
\end{equation}
Thus, for each $\vec\al,\vec\beta\in\vec{J}$ and each
$\delta\in\vec{\delta}_{\vec\al}\cap
\vec{\delta}_{\vec\beta}$,
for all $i<d$,
\begin{equation}
p_{\vec\al}(i,\delta)=p_{\vec\beta}(i,\delta).
\end{equation}
Thus,
$p_{\vec{J}}:=
\bigcup \{p_{\vec{\al}}:\vec\al\in\vec{J}\}$
is a  function.
Let $\vec\delta_{\vec{J}}=
\bigcup\{
\vec{\delta}_{\vec\al}:
\vec\al\in\vec{J}\,\}$.
For each $\delta\in
\vec{\delta}_{\vec{J}}$ and   $i<d$,
$p_{\vec{J}}(i,\delta)$ is defined,
and it is exactly  $p_{\vec\al}(i,\delta)$, for any $\vec\al\in\vec{J}$ such that $\delta\in \vec\delta_{\vec\al}$.
Thus, $p_{\vec{J}}$ is a member of $\bP$, and $p_{\vec{J}}\le p_{\vec\al}$ for each $\vec\al\in\vec{J}$.
\end{proof}

The final  lemma of Part I  will be used in the next section.

\begin{lem}\label{lem.subclaimA}
If $\beta\in \bigcup_{i<d}K_i$,
$\vec{\al}\in\prod_{i<d}K_i$,
and $\beta\not\in\vec\al$,
 then
$\beta$ is not  a member of   $\vec{\delta}_{\vec{\al}}$.
\end{lem}

\begin{proof}
Suppose toward a contradiction that $\beta\in\vec{\delta}_{\vec{\al}}$.
Then there is a $j<k^*$ such that $\beta=\delta_{\vec{\al}}(j)$.
Let $i$ be such that $\beta\in K_i$.
Since $\beta\ne\al_i=\delta_{\vec{\al}}(j_i)$, it must be that $j\ne j_i$.
However,
letting $\vec\beta$ be any member of $\prod_{i<d}K_i$ with $\beta_i=\beta$,
then
$\beta=\delta_{\vec{\beta}}(j_i)=\delta_{\vec{\al}}(j)$, so Lemma  \ref{lem.j=j'}
implies that $j_i=j$, a contradiction.
\end{proof}
\vskip.1in

\noindent \underline{\bf{Part II.}}
In this last part of the proof,
we build a strong coding tree $S$ valid in $T$ on which the coloring $h$ is homogeneous.
Cases (a) and  (b) must be handled separately.
\vskip.1in

\noindent\underline{\bf{Part II Case (a).}}
Recall that $\{s_i:i\le d\}$ enumerates the members of $A_e$, which is a subset of $B^+$.
Let $m$ be the   integer such that
$B\in\mathcal{AT}_{m}^k$.
Let $M=\{ m_j:j\in\bN\}$ be the strictly increasing enumeration of those $m'> m$
 such that the splitting node in  $\max(r_{m'}(T))$ extends $s_d$.
By induction on $j$ we will  construct the following:
The base case splits into two subcases depending on whether $m_0>m+1$ or $m_0=m+1$.
In the first case, 
 we will find some $U_{m_{0}-1}\in r_{m_{0}-1}[B,T]$ which is valid in $T$. 
In the second, we will find some 
$U_{m_{1}-1}\in r_{m_{1}-1}[B,T]$ which is valid in $T$
 and such that 
$h$ takes color $\varepsilon^*$ on $\Ext_{U_{m_1-1}}(A,\tilde{X})$.
In general, 
given $U_{m_j-1}$, 
 we will use
the forcing 
to 
 find some $U_{m_{j}}\in r_{m_j}[U_{m_j-1},T]$ which is valid in $T$ such that  $h$ takes color $\varepsilon^*$ on $\Ext_{U_{m_j}}(A,\tilde{X})$.
Then
we will apply Theorem \ref{thm.GOODnonempty}  and  Lemma \ref{lem.HLconstruction}  to find an extension $U_{m_{j+1}-1}\in r_{m_{j+1}-1}[U_{m_{j}}, T]$ which is valid in $T$ and continue the induction.
Setting $S=\bigcup_{j\in\bN} U_{m_j}$ will yield $S$ to be a member of $[B,T]$ for which  $\Ext_S(A,\tilde{X})$ is homogeneous for $h$, with color $\varepsilon^*$.

First extend each node in $B^+$  to level $l^*$ as follows.
The set $\{t^*_i:i\le d\}$ end-extends
 $A_e$,
 has no new  pre-cliques
 over $A_e$, and is free in $T$.
For  each node $u$ in $B^+\setminus A_e$, let $u^*$ denote
 its  leftmost extension in $T\re l^*$.
Then the set
\begin{equation}
U^*=\{t^*_i:i\le d\}\cup\{u^*:u\in B^+\setminus A_e\}
\end{equation}
end-extends   $B^+$, is free in $T$, and has no new pre-cliques over $B$, by
Lemma \ref{lem.poc}.
Thus,
$U^*$ is free in $T$, and
 $B\cup U^*$
 satisfies the Witnessing Property
so   is valid in $T$.
If $m_0>m+1$,
apply Lemma \ref{lem.HLconstruction} and Theorem \ref{thm.GOODnonempty} to
extend above $U^*$  to construct
a member $U_{m_0-1}\in r_{m_0-1}[B,T]$ which is valid in $T$.
In this case, note that $\max(r_{m+1}(U_{m_0}))$ is not
$ U^*$, but rather  $\max(r_{m+1}(U_{m_0}))$
end-extends $U^*$.

If $m_0=m+1$,
then $B\cup U^*$ is a member of $r_{m_0}[B,T]$, by Lemma \ref{lem.HLconstruction}.
In this case,  we can simply let $U_{m_0}=B\cup U^*$.
Then $h$ takes color $\varepsilon^*$ on 
$\Ext_{U_{m_0}}(A,\tilde{X})$, since $U^*$ is the only member of that set.
Using Theorem \ref{thm.GOODnonempty},
 extend $U_{m_0}$ to a member $U_{m_1-1}\in r_{m_1-1}[U_{m_0},T]$ which is valid in $T$.

Assume  $j<\om$ and
 we have constructed $U_{m_j-1}$, valid in $T$,  so that every member of $\Ext_{U_{m_j-1}}(A,\tilde{X})$ is colored $\varepsilon^*$ by $h$.
Fix some  $C\in r_{m_j}[U_{m_j -1} ,T]$ with $C$ valid in $T$, and let $Z=\max(C)$.
The nodes in $Z$ will not be in the tree $S$ we are constructing;
rather,
we will extend the nodes in $Z$ to construct
$U_{m_j}\in r_{m_j}[U_{m_j-1},T]$.

We now start to construct a condition $q$ which will satisfy
Lemma \ref{lem.qbelowpal}, below.
Let $q(d)$ denote  the splitting node in $Z$ and let $l_q=|q(d)|$.
For each $i<d$,
let  $Z_i$ denote the set of those $z\in T_i\cap Z$
such that $z\in X$ for some $X\in\Ext_{Z}(A,\tilde{X})$.
For each $i<d$,
take  a set $J_i\sse K_i$ of cardinality $|Z_i|$
and label the members of $Z_i$ as
$\{z_{\al}:\al\in J_i\}$.
Notice that each member of $\Ext_T(A,\tilde{X})$ above $Z$ extends some set $\{z_{\al_i}:i<d\}\cup\{q(d)\}$, where each $\al_i\in J_i$.
Let $\vec{J}$ denote the set of those $\lgl \al_0,\dots,\al_{d-1}\rgl\in \prod_{i< d}J_i$ such that  the set $\{z_{\al_i}:i< d\}\cup\{q(d)\}$ is in $\Ext_T(A,\tilde{X})$.
Then for each $i<d$,
$J_i=\{\al_i:\vec\al\in\vec{J}\}$.
 It follows from  Lemma \ref{lem.compat} that
the set $\{p_{\vec\al}:\vec\al\in\vec{J}\}$ is compatible.
The fact that
$p_{\vec{J}}$ is a condition in $\bP$ will be used
to make  the construction of $q$ very precise.

Let
 $\vec{\delta}_q=\bigcup\{\vec{\delta}_{\vec\al}:\vec\al\in \vec{J}\}$.
For each $i<d$ and $\al\in J_i$,
define $q(i,\al)=z_{\al}$.
Notice that for each
$\vec\al\in \vec{J}$ and $i<d$,
\begin{equation}
q(i,\al_i)\contains t^*_i=p_{\vec\al}(i,\al_i)=p_{\vec{J}}(i,\al_i),
\end{equation}
and
\begin{equation}
q(d)\contains t^*_d=p_{\vec\al}(d)=p_{\vec{J}}(d).
\end{equation}
For each  $i<d$ and $\gamma\in\vec{\delta}_q\setminus
J_i$,
there is at least one $\vec{\al}\in\vec{J}$ and some $k<k^*$ such that $\delta_{\vec\al}(k)=\gamma$.
Let $q(i,\gamma)$ be the leftmost extension
 of $p_{\vec{J}}(i,\gamma)$ in $T$ of length $l_q$.
Define
\begin{equation}
q=\{q(d)\}\cup \{\lgl (i,\delta),q(i,\delta)\rgl: i<d,\  \delta\in \vec{\delta}_q\}.
\end{equation}
Since $C$ is valid in $T$ and $Z=\max(C)$, it follows that  $Z$ is free in $T$.
Since $\ran(q)$ consists of $Z$ along with leftmost extensions of nodes in $\ran(p_{\vec{J}}(i,\gamma))$,
all of which are free,
$\ran(q)$ is free.
Therefore, $q$ is a condition in $\bP$.

\begin{lem}\label{lem.qbelowpal}
For all $\vec\al\in\vec{J}$,
$q\le p_{\vec{\al}}$.
\end{lem}

\begin{proof}
Given  $\vec\al\in\vec{J}$,
it follows from the  definition of $q$ that
$\vec{\delta}_q\contains \vec{\delta}_{\vec{\al}}$,
$q(d)\contains p_{\vec{\al}}(d)$,
and
for each pair $(i,\gamma)\in d\times \vec{\delta}_{\vec\al}$,
$q(i,\gamma)\contains p_{\vec{\al}}(i,\gamma)$.
So it only remains to show that  $q$
is valid over $p_{\vec{\al}}$.
It follows from  Lemma \ref{lem.subclaimA}
that
 $\vec{\delta}_{\vec\al}\cap
\bigcup_{i<d}K_i=\vec\al$; so
for  each $i<d$ and $\gamma\in\vec{\delta}_{\vec\al}\setminus \{\al_i\}$,
 $q(i,\gamma)$ is the leftmost extension of $p_{\vec\al}(i,\gamma)$.
Since $\vec\al$ is in $\vec{J}$,
$ X(q,\vec\al)$ is in $\Ext_T(A,\tilde{X})$.
This  implies that $ X(q,\vec\al)$
has no new pre-cliques
 over $A$, and hence, none over $X(p_{\vec\al},\vec\al)$.
It follows that
$\ran(q\re \vec{\delta}_{\vec\al})$
is valid over $\ran(p_{\vec\al})$, by Lemma \ref{lem.poc}.
Therefore, $q\le p_{\vec\al}$.
 \end{proof}

\begin{rem}
Notice that
we did not prove that $q\le p_{\vec{J}}$; in fact that is generally false.
\end{rem}

To construct $U_{m_j}$,
take an $r\le q$ in  $\bP$ which  decides some $l_j$ in $\dot{L}_d$ for which   $h(\dot{b}_{\vec\al}\re l_j)=\varepsilon^*$, for all $\vec\al\in\vec{J}$.
This is possible since for all $\vec\al\in\vec{J}$,
$p_{\vec\al}$ forces $h(\dot{b}_{\vec\al}\re l)=\varepsilon^*$ for $\dot{\mathcal{U}}$ many $l\in \dot{L}_d$.
By the same argument as in creating the conditions $p_{\vec\al}\le p_{\vec\al}^2$ to satisfy (4)
 in Part I,
 we may assume that
 the nodes in the image of $r$ have length  $l_j$.
Since
$r$ forces $\dot{b}_{\vec{\al}}\re l_j=X(r,\vec\al)$
for each $\vec\al\in \vec{J}$,
and since the coloring $h$ is defined in the ground model,
it follows that
$h(X(r,\vec\al))=\varepsilon^*$ for each $\vec\al\in \vec{J}$.
Extend the splitting node $q(d)$ in $Z$
to $r(d)$.
For each $i<d$ and $\al_i\in J_i$,
extend $q(i,\al_i)$ to $r(i,\al_i)$.
Let
\begin{equation}
Z_0=\{q(i,\al_i):i<d,\ \al_i\in J_i\}\cup \{q(d)\}
\end{equation}
and let $Z_1= Z\setminus Z_0$.
Let
\begin{equation}
Y=\{r(i,\al_i):i<d,\ \al_i\in J_i\}\cup \{r(d)\}.
\end{equation}
Then $Y$ extends  $Z_0$ and has no new pre-cliques over $Z_0$, since $r\le q$. 
By Lemma \ref{lem.HLconstruction},
 there is a $U_{m_j}\in    r_{m_j}[U_{m_j-1},T]$ which is valid in $T$ 
such that 
$\max(U_{m_j})$ end-extends $Z$ and in particular, 
$Y\sse\max(U_{m_j})$.
Notice that every $X\in\Ext_{U_{m_j}}(A,\tilde{X})$
 with $X\sse\max(U_{m_j})$
 satisfies $h(X)=\varepsilon^*$.
 This holds since for each such $X$, 
the truncation  $ X \re l_q$
is a member of $\Ext_{Z}(A,\tilde{X})$.
So  there corresponds a sequence $\vec\al\in\vec{J}$ such that
$X\re l_q=  X(q,\vec\al)$.
Then
 $X= X(r,\vec\al)$,
 which  has $h$-color $\varepsilon^*$.

To finish the induction step, 
apply  Lemma \ref{lem.HLconstruction} and Theorem \ref{thm.GOODnonempty}
to extend $U_{m_{j}}$ to some $U_{m_{j+1}-1}\in r_{m_{j+1}-1}[U_{m_{j}},T]$ which is valid in $T$.
Note that every member of $\Ext_{U_{m_{j+1}-1}}(A,\tilde{X})$ is colored $\varepsilon^*$ by $h$, since 
$\Ext_{U_{m_{j+1}-1}}(A,\tilde{X})=\Ext_{U_{m_{j}}}(A,\tilde{X})$.

Let $S=\bigcup_{j\in\bN}U_{m_j}$.
Then for each $X\in\Ext_{S}(A,\tilde{X})$, there corresponds a $j$ such that $X\in\Ext_{U_{m_j}}(A,\tilde{X})$, and hence,
$h(X)=\varepsilon^*$.
Thus, $S\in [B,T]$ and satisfies the theorem.
This concludes the proof of the theorem for Case (a).
\vskip.1in

\noindent\underline{\bf{Part II Case (b).}}
Let  $m_0$ be the integer such that there is a $B'\in r_{m_0}[B,T]$ with
$\tilde{X}\sse \max(B')$.
Let $U_{m_0-1}$ denote $r_{m_0-1}(B')$.
Since $\tilde{X}\sse \max(B')$, it follows that
$l^*\ge l_{B'}$.
Let  $V=\{t^*_i:i\le d\}$, and recall that
 this set  has no new pre-cliques over $\tilde{X}$.
By Lemma \ref{lem.HLconstruction}
there is  a set of nodes $V'$ end-extending $\max(B')\setminus V$ such that
$U_{m_0-1}\cup V\cup V'$ is a member of $r_{m_0}[U_{m_0-1},T]$;
label  this $U_{m_0}$.
Since $\max(U_{m_0})$ is at the level of  the coding node $t^*_d$, $\max(U_{m_0})$  is free in $T$.
Moreover,  $U_{m_0}\in r_{m_0}[U_{m_0-1},T]$ implies that $U_{m_0}$
 satisfies the Strong Witnessing Property.
Therefore, $U_{m_0}$ is valid in $T$.
Notice that $\{t^*_i:i\le d\}$ is the only member of
$\Ext_{U_{m_0}}(A,\tilde{X})$,
and it has $h$-color $\varepsilon^*$.

Let $M=\{m_j:j\in\bN\}$ enumerate  the set of $m\ge m_0$
such that the coding node $c^T_{m}\contains c^T_{m_0}$.
Assume  that  $j\ge 1$ and
 we have constructed $U_{m_{j-1}}\in \mathcal{AT}^k_{m_{j-1}}$ valid in $T$ so that every member of $\Ext_{U_{m_{j-1}}}(A,\tilde{X})$ is colored $\varepsilon^*$ by $h$.
 By Theorem  \ref{thm.GOODnonempty},
we may  fix some
 $U_{m_j-1}\in r_{m_j-1}[U_{m_{j-1}},T]$ which is valid in $T$.
 Take some  $C\in r_{m_j}[U_{m_j-1} ,T]$, and
let $Z$ denote $\max(C)$.
The nodes in $Z$ will not be in the tree $S$ we are constructing;
rather,
we will construct
$U_{m_j}\in r_{m_j}[U_{m_j-1},T]$
so that
$\max(U_{m_j})$ extends
 $Z$.
Let $q(d)$ denote the coding   node in $Z$ and let $l_q=|q(d)|$.
Recall that for $e\in\{0,1\}$, $I_e$ denotes the set of
 $i<d$ for which $t^*_i$ has passing number $e$ at $t^*_d$.
For each pair $e\in\{0,1\}$ and
$i\in I_e$, let
$Z_i$ be the set
of nodes $z$ in $T_i\cap Z$ such that $z$ has passing number $e$ at  $q(d)$.

We now construct a condition $q$ similarly to, but not exactly as in,   Case (a).
For each $i<d$,
let $J_i$ be a subset of $K_i$ with the same size as $Z_i$.
For each $i< d$, label the nodes in $Z_i$ as
$\{z_{\al}:\al\in J_i\}$.
Let $\vec{J}$ denote the set of those $\lgl \al_0,\dots,\al_{d-1}\rgl\in \prod_{i< d}J_i$ such that  the set
$\{z_{\al_i}:i< d\}\cup\{q(d)\}$ is in $\Ext_T(A,\tilde{X})$.
Notice that for each $i<d$ and
$\vec\al\in \vec{J}$, $z_{\al_i}\contains t^*_i=p_{\vec{\al}}(i,\al_i)$, and $q(d)\contains t^*_d=p_{\vec{\al}}(d)$.
Furthermore, for each $i<d$ and   $\delta\in J_i$,
there is an $\vec\al\in\vec{J}$ such that $\al_i=\delta$.
Let
 $\vec{\delta}_q=\bigcup\{\vec{\delta}_{\vec\al}:\vec\al\in \vec{J}\,\}$.
For each pair $(i,\gamma)\in d\times\vec{\delta}_q$ with  $\gamma\in J_i$,
define $q(i,\gamma)=z_{\gamma}$.

Let $\mathcal{J}=\{(i,\gamma)\in
d\times\vec{\delta}_q : i<d$ and $
\gamma\in\vec{\delta}_q\setminus J_i\}$.
For  each pair $(i,\gamma)\in \mathcal{J}$,
there is at least one $\vec{\al}\in\vec{J}$ and some $k<k^*$ such that $\delta_{\vec\al}(k)=\gamma$.
By Lemma \ref{lem.compat},
$p_{\vec\beta}(i,\gamma)=p_{\vec{\al}}(i,\gamma)=t^*_{i,k}$,
for any  $\vec\beta\in\vec{J}$ for which $\gamma\in\vec{\delta}_{\vec\beta}$.
For each pair $(i,\gamma)\in\mathcal{J}$ with $i\in I_0$,
take $q(i,\gamma)$ to be the leftmost extension of $t^*_{i,k}$ in $T\re l_q$.
For  each pair $(i,\gamma)\in\mathcal{J}$ with $i\in I_1$,
let  $q(i,\gamma)$ be the node which extends $t^*_{i,k}$ leftmost until length of the longest coding node in $T$ strictly below $q(d)$, and then takes the rightmost path to length $l_q$.
Note that  $q(i,\gamma)$ has passing number $e$, where $e\in\{0,1\}$ is the number such that $i\in I_e$.
By similar arguments to those in 
Lemma \ref{lem.HLCasebtruncate},
the set $\{q(i,\gamma):(i,\gamma)\in\mathcal{J}\}$ has no new pre-cliques over
$\ran(p_{\vec{\al}})$ for $\vec\al\in\vec{J}$ (recall, these all have the same range);
moreover,
any new pre-cliques in the set
$\{q(i,\gamma):i< d,\ \gamma\in\vec{\delta}_q\}\cup\{q(d)\}$ over $\ran(p_{\vec{\al}})$ (for any $\vec\al\in\vec{J}$)
must occur among 
$\{q(i,\gamma): i<d,\ \gamma\in J_i\}\cup\{q(d)\}$.

Define
\begin{equation}
q=\{q(d)\}\cup \{\lgl (i,\delta),q(i,\delta)\rgl: i<d,\  \delta\in \vec{\delta}_q\}.
\end{equation}
By the construction, $q$ is a member of $\bP$.

\begin{claim}\label{claim.qbelowpal}
For each $\vec\al\in \vec{J}$,
$q\le p_{\vec\al}$.
\end{claim}

\begin{proof}
By construction, 
$q(i,\delta)\contains p_{\vec{\al}}(i,\delta)$ for all $(i,\delta)\in d\times \vec{\delta}_{\vec\al}$; so 
it suffices to show that for each $\vec\al\in\vec{J}$,
$\ran(q\re \vec\delta_{\vec\al})$ has no new pre-cliques  over $\ran(p_{\vec\al})$.

Let $\vec\al\in\vec{J}$ be given.
Then
\begin{equation}
\ran(q\re \vec\delta_{\vec\al})\sse
\{q(i,\gamma):(i,\gamma)\in\mathcal{J}\}\cup
X(q,\vec\al),
\end{equation}
recalling that $X(q,\vec\al)=\{q(i,\al_i):i<d\}\cup\{q(d)\}$.
By definition of $\vec{J}$,
$\vec{\al}\in\vec{J}$ implies that
$X(q,\vec\al)$  is a member of $\Ext_T(A,\tilde{X})$.
Thus,
$X(q,\vec\al)$
has no new pre-cliques over  $A\cup\tilde{X}$, by
Lemma \ref{lem.alternate}.
Since $\{t^*_i:i\le d\}$ end-extends $\tilde{X}$,
it follows that $X(q,\vec\al)$  has no new pre-cliques over  $\{t^*_i:i\le d\}$.
Since the set $\{q(i,\gamma):i\le d,\ \gamma\in\vec{\delta}_{\vec{\al}}, \ \gamma\ne \al_i\}$ has no new pre-cliques with $X(q,\vec\al)$ over  
$\ran(p_{\vec\al})$, 
it follows that 
$q\le p_{\vec\al}$.
\end{proof}

To construct $U_{m_j}$,
take an $r\le q$ in  $\bP$ which decides  $l_r\in \dot{L}_d$ such that
$h(\dot{b}_{\vec\al}\re l_r)=\varepsilon^*$ for all $\vec\al\in\vec{J}$,
using
the same ideas as in the construction of the $p_{\vec\al}$'s.
Let $Y=\bigcup\{ X(r,\vec\al):\vec\al\in\vec{J}\}$, and let $Z^*=\{r(d)\}\cup\bigcup_{i<d}Z_i$.
Since  $\ran(r\re \vec\delta_q)$ has no new pre-cliques over $\ran(q)$,
it follows that $Y$ has no new pre-cliques over $Z^*$.
Apply Lemma \ref{lem.HLCasebtruncate} to
extend the nodes in $Z\setminus Z^*$ to a set $Y'\sse T\re  l_r$ so that
each node in $Y'$ has the same passing number at $r(d)$ as it does at $q(d)$, and such that $Y\cup Y'$ has no new pre-cliques over $Z$.
Then  $U_{m_j-1}\cup Y\cup Y'$ is a member of
$r_{m_j}[U_{m_j-1},T]$ which is valid in $T$.


To finish the proof of the theorem for Case (b),
Define  $S=\bigcup_{j\in\bN}U_{m_j}$.
Then $S\in [B',T]$, and
for each $Z\in\Ext_{S}(A,\tilde{X})$, there is a $j\in\bN$ such that $Z\in\Ext_{U_{m_j}}(A,\tilde{X})$, so $h(Z)=
\varepsilon^*$.

This concludes the proof of the theorem.
\end{proof}


\section{Ramsey Theorem for  finite  trees with the Strict Witnessing Property}\label{sec.1SPOC}

The main theorem of this section is
Theorem \ref{thm.MillikenSWP}, which is
 an  analogue of Milliken's Theorem \ref{thm.Milliken}
for
 colorings of finite trees with  the following strong version of the Witnessing Property.

\begin{defn}[Strict Witnessing Property]\label{defn.SWP}
A  subtree $A$ of a strong coding tree satisfies the {\em Strict Witnessing Property (SWP)}
if $A$
 satisfies the Witnessing Property and
 the following hold:
 \begin{enumerate}
 \item
 For each interval  $(|d_m^A|,|d^A_{m+1}|]$,
 $A$  has  at most  one new pre-clique of size at least two, or a singleton with some new pre-cliques in $(|d_m^A|,|d^A_{m+1}|]$, but not both.
 \item
 If $X$ is a new pre-$a$-clique of size at least three 
 in $(|d_m^A|,|d^A_{m+1}|]$, 
 then every proper subset of $X$ has a new pre-$a$-clique in  an  interval $(|d_j^A|,|d^A_{j+1}|]$, for some  $j<m$. 
 \end{enumerate}
\end{defn}

\begin{lem}\label{lem.copy}
If $A\sse\bT_k$ has the Strict Witnessing Property and
$B\cong A$,
then $B$ also has the Strict Witnessing Property.
\end{lem}

\begin{proof}
If $B\cong A$ and $A$ has the WP,
then $B$ also has the WP
 by Lemma \ref{lem.concpresWP}.
Let $f:A\ra B$ be the strong isomorphism between them.
Since $A$ has the SWP,  each new pre-clique of size at least two in $A$ is the only new pre-clique occuring in  that interval of $A$, hence it is maximal in that interval.
By (2) of Definition \ref{defn.SWP}, 
each proper subset of a new pre-clique in a given interval of $A$ occurs as a maximal new pre-clique and is witnessed in some lower interval of $A$. 
Since $f$ preserves maximal new pre-cliques, 
each new pre-clique of size at least two in $B$ is a maximal  new pre-clique in $B$, and is the only new pre-clique of $B$ in the interval in which it  occurs. 
Thus, $B$ satisfies (2).
Furthermore, for any $t\in A$, $f(t)$ is a new singleton pre-$a$-clique in $B$ iff $t$ is a  is a new singleton pre-$a$-clique in $A$.
Therefore, $B$ has the SWP.
\end{proof}

Given a finite tree $A$ with  the SWP,
we say that $B$ is a {\em copy} of $A$ if $A\cong B$.
The main theorem of this section,
Theorem \ref{thm.MillikenSWP},
will  guarantee a Ramsey Theorem for  colorings of copies of a finite tree with the SWP inside a strong coding tree.

\begin{thm}\label{thm.MillikenSWP}
Let $T\in\mathcal{T}_k$
 be a strong coding tree
and
let $A$ be a finite subtree  of $T$ satisfying the Strict Witnessing Property.
Then for any coloring of the  copies of $A$ in $T$ into finitely many colors,
 there is a strong coding subtree $S\le T$ such that all  copies of $A$ in $S$ have the same color.
\end{thm}

 Theorem \ref{thm.MillikenSWP} will be proved via four  lemmas and an induction argument.
 The main difficulty is that
  Case (b) of Theorem \ref{thm.matrixHL}
 provides homogeneity for
 $\Ext_S(A,\tilde{X})$ for some strong coding tree $S$;  in particular, homogeneity  only holds for  level sets   $X$
 end-extending $\tilde{X}$.
The issue of new singleton pre-$a$-cliques will be handled similarly to how we handle  the case when $\tilde{X}$ has a coding node. 
We need a strong coding tree in which  {\em every} $X$  satisfying $A\cup X\cong A\cup\tilde{X}$ has the same color.
  This will be addressed by the following:
 Lemma \ref{lem.endhomog} will build a fusion sequence to obtain an $S\le T$ which is homogeneous
 on $\Ext_S(A,Y)$ for each
  minimal level set $Y$ extending $A_e$  such that $A\cup Y\cong A\cup\tilde{X}$.
Lemma \ref{lem.Case(c)} will use a new forcing and  arguments from the proof of Theorem \ref{thm.matrixHL} to
obtain a strong coding tree $S\in [B,T]$ in which every $X$
satisfying
 $A\cup X\cong A\cup\tilde{X}$  has the same color.
The last two lemmas involve fusion to construct a strong coding  subtree which is homogeneous for the induced color on
copies of  $A$.
The theorem then follows  by induction and an application of Ramsey's Theorem.

The following basic assumption, similar to  but stricter than Case (b) of Theorem \ref{thm.matrixHL},  will be  used in much of this section.

\begin{assumption}\label{assumption.6}
Let $A$  and $C$ be fixed non-empty finite valid subtrees
 of a strong coding tree $T\in\mathcal{T}_k$ such that
 \begin{enumerate}
 \item
  $A$ and $C$ both  satisfy  the Strict Witnessing Property; and
  \item
  $C\setminus A$ is a level set  containing both  a coding node and the sequence $0^{(l_C)}$.
  \end{enumerate}
Let $\tilde{X}$ denote $C\setminus A$, and
let  $A_e$ be the  subset of
$A^+$ which is extended to $\tilde{X}$.
Let $d+1$ be the number of nodes
 in $\tilde{X}$.
 List the nodes
 in $A_e$  as $\lgl s_i:i\le d\rgl$
and the nodes of $\tilde{X}$ as $\lgl t_i:i\le d\rgl$ so that each $t_i$ extends $s_i$ and  $t_d$ is the coding node in $\tilde{X}$.
For $j\in\{0,1\}$,
let
 $I_j$ denote the set of $i\le d$ such that  $t_i$
 has passing number $j$ at $t_d$.
If $\tilde{X}$ has a new pre-clique over $A$,
let $I_*$ denote the  set of $i\in I$ such that
$\{t_i:i\in I_*\}$ is the  new pre-clique in $\tilde{X}$ over $A$.
Note that $I_*\sse I_1$ and $t_d$ must be among the coding nodes in $C$ witnessing this new pre-cliqe.
\end{assumption}

For
any $X$ such that $A\cup X\cong C$,
let
 $\Ext_T(A,X)$
be defined as in equation (\ref{eq.ExtTAC}) of Section \ref{sec.5}.
Thus,
$\Ext_T(A,X)$ is the collection of  level sets $Y\sse T$ such that $Y$ end-extends $X$ and $A\cup Y\cong A\cup X$,
(equivalently, $A\cup Y\cong C$),
and $A\cup Y$ is valid in $T$.
Recall that, since $\tilde{X}$ contains a coding node,
 $A\cup X\cong A\cup \tilde{X}$ implicitly includes that
the strong isomorphism from $A\cup \tilde{X}$ to $A\cup X$ preserves  passing numbers between  $\tilde{X}^+$  and  $X^+$.
We hold to the convention that given  $Y$ such that $A\cup Y\cong C$,
the nodes in $Y$
 are labeled $y_i$, $i\le d$, where  each $y_i\contains s_i$.
In particular, $y_d$ is the coding node in $Y$.

In this section, we want to consider all copies of $C$ extending $A$.
To that end let
\begin{equation}
\Ext_T(A,C)=\bigcup\{\Ext_T(A,X): A\cup X\cong C\}.
\end{equation}
Now we define the notion of minimal pre-extension, which
 will be used in the next lemma.
For $x\in T$,
define  $\splitpred_T(x)$
to be $x\re l$ where $l< |x|$ is maximal such that $x\re l$ is a splitting node in $T$.

\begin{defn}[Minimal pre-extension of $A$ to a copy of $C$]\label{defn.mpe}
Given $A$, $\tilde{X}$, and $C$ as in Assumption  \ref{assumption.6},
for
 $X=\{x_i:i\le d\}$  a   level set extending $A_e$ such that $x_i\contains s_i$ for each $i\le d$
and such that
$l_X$
is the length of some coding node in $T$,
we say that  $X$ is  a
{\em minimal pre-extension in $T$ of $A$ to a copy of $C$} if the following hold:
\begin{enumerate}
\item[(i)]
$\{i\le d: $ the passing number of  $x_i$ at $x_d$ is $1\}=I_1$.
\item[(ii)]
$A\cup \SP_T(X)$
satisfies the Strict Witnessing Property,
 where
 \begin{equation}\label{eq.SP}
 \SP_T(X)=\{\splitpred_T(x_i):i\in I_1\}\cup\{x_i:i\in I_0\}.
 \end{equation}
 \item[(iii)]
 If $X$ has a new pre-clique over $A$,
 then $X$ has only one new maximal pre-clique over $A$  which is exactly  $\{x_i:i\in I_*\}\re l$, for some $l\in (l_A, l_X]$.
\end{enumerate}
\end{defn}

 Notice that 
 for (ii) to hold, $\SP_T(X)$  must have no new pre-cliques over $A$.
 Let $\MPE_T(A,C)$ denote the set of minimal pre-extenions in $T$ of $A$  to a copy of $C$.
  When $A$ and $C$ are clear,
 we call members of  $\MPE_T(A,C)$ simply
 {\em minimal pre-extensions}.
Minimal pre-extensions
are exactly the level sets  in $T$ which
 can be extended to a member of $\Ext_T(A,\tilde{X})$.

For $X\in\MPE_T(A,C)$,
define
\begin{equation}
\Ext_T(A,C;X)=\{Y\sse T: A\cup Y\cong C\mathrm{\  and\ }Y \mathrm{\ extends\ } X\}.
\end{equation}
Then
\begin{equation}
\Ext_T(A,C)=\bigcup\{\Ext_T(A,C;X): X\in \MPE_T(A,C)\},
\end{equation}

\begin{defn}\label{defn.endhomog}
A coloring on $\Ext_T(A,C)$ is {\em end-homogeneous} if  for each minimal pre-extension $X$,
 every member  of $\Ext_T(A,C;X)$ has the same color.
\end{defn}

The following lemma is a slightly modified version of Lemma 6.7 in \cite{DobrinenJML20}.

\begin{lem}[End-homogeneity]\label{lem.endhomog}
Assume \ref{assumption.6}, and
let $m$ be the integer  such that $\max(A)\sse  r_{m}(T)$.
Then for any coloring $h$
 of $\Ext_T(A,C)$ into  two  colors,
 there is a $T'\in[r_{m}(T),T]$ such that  $h$ is
 end-homogeneous  on $\Ext_{T'}(A,C)$.
\end{lem}

\begin{proof}
Let $(m_j)_{j\in\bN}$ enumerate those integers greater  than  $m$ such that there is a minimal pre-extension of $A$ to a copy of  $C$ from among  the maximal nodes in
$r_{m_j}(T)$.
Notice that for each $j\in\bN$,
 $\max(r_{m_j}(T))$ contains a coding node,
although there can  be  members of $\MPE_T(A,C)$
contained in $\max(r_{m_j}(T))$ not containing that coding node.

Let $T_{-1}$ denote $T$.
Suppose that $j\in\bN$ and $T_{j-1}$ is given  so that
the coloring $h$ is homogeneous on
$\Ext_{T_{j-1}}(A,C;X)$ for each minimal pre-extension $X$ in $r_{m_j-1}(T_{j-1})$.
Let  $U_{j-1}$ denote $r_{m_j-1}(T_{j-1})$.
Enumerate the   minimal pre-extensions contained in
 $\max(r_{m_j}(T_{j-1}))$ as $X_0,\dots, X_n$.
By induction on  $i\le n$, we will obtain
$T_j\in [U_{j-1},T_{j-1}]$ such that $\max(r_{m_j}(T_j))$ end-extends $\max(r_{m_j}(T_{j-1}))$
and   $\Ext_{T_j}(A,C;Z)$ is homogeneous
for each minimal pre-extension $Z$ in $\max(r_{m_j}(T_{j-1}))$.

Let $l$ denote the length of the  nodes in $\max(r_{m_j}(T_{j-1}))$, and 
let  $S_{-1}=T_{j-1}$.
Suppose $0\le i\le n$ and 
we have 
 strong coding trees 
$S_{-1},\dots, S_{i-1}$ such that 
for each $0\le  i'\le i-1$, 
 $S_{i'}\in [U_{j-1},S_{i'-1}]$
and
 $h$ is homogeneous on $\Ext_{S_{i'}}(A,C;X_{i'})$.
Note that
$X_i$ is contained in $r_{m_j}(S_{i-1})\re l$,
though $l$ does not have to  be  the length of any node in $S_{i-1}$.
The point is that the set of nodes $Y_i$ in $\max(r_{m_j}(S_{i-1}))$ end-extending $X_i$
is again a minimal pre-extension.
Extend the nodes in $Y_i$ to some $Z_i\in \Ext_{ S_{i-1}}(A,C;Y_i)$,
and let $l'$ denote the length of the nodes in $Z_i$.
Note that $Z_i$ has no new pre-cliques over $Y_i$.
Let $W_i$ consist of the nodes in $Z_i$ along with the leftmost extensions   of
the nodes in $\max(r_{m_j}(S_{i-1}))\setminus Y_i$
 to the length  $l'$ in $S_{i-1}$.

Let $S'_{i-1}$ be a strong coding tree in $[U_{j-1},S_{i-1}]$ such that $\max(r_{m_j}(S'_{i-1}))$ extends $W_i$.
Such an $S'_{i-1}$ exists by
Lemmas  \ref{lem.poc} and \ref{lem.pnc} and
Theorem \ref{thm.GOODnonempty}.
Apply
 Case (b)  of Theorem \ref{thm.matrixHL}
to obtain  a strong coding tree
 $S_i\in [U_{j-1},S'_{i-1}]$ such that the coloring on $\Ext_{S_i}(A,C;Z_i)$ is homogeneous.
At the end of this process, let $T_j=S_n$.
Note that for each minimal pre-extension $Z\sse\max(r_{m_j}(T_j))$,
there is a unique $i\le n$ such that
$Z$ extends $X_i$,
since each node in  $\max(r_{m_j}(T_j))$ is a unique  extension of one node in $\max(r_{m_j}(T_{j-1}))$,
and hence
$\Ext_{T_j}(A,C;Z)$ is homogeneous.

Having chosen each $T_j$ as above,
let $T'=\bigcup_{j\in\bN}r_{m_j}(T_j)$.
Then $T'$ is a strong coding tree which is a member of
 $[r_{m}(T),T]$,
and for each minimal pre-extension $Z$ in $T'$,
$\Ext_{T'}(A,C;Z)$ is homogeneous for $h$.
Therefore,  $h$ is end-homogeneous on $\Ext_{T'}(A,C)$.
\end{proof}

The next lemma provides a means for uniformizing the  end-homogeneity from the previous lemma
to obtain one color for all
members of  $\Ext_S(A,C)$.
The arguments are often similar to those of
Case (a) of
Theorem \ref{thm.matrixHL}, but sufficiently different to warrant a proof.

\begin{lem}\label{lem.Case(c)}
Assume \ref{assumption.6},  and suppose that $B$ is a finite strong coding tree valid in $T$ and  $A$ is a subtree of $B$
such that $\max(A)\sse\max(B)$.
Suppose that $h$ is end-homogeneous on $\Ext_{T}(A,C)$.
Then there is an $S\in[B,T]$ such that $h$ is homogeneous on
 $\Ext_S(A,C)$.
\end{lem}

\begin{proof}
Given any $U\in[B,T]$,
recall that  $\MPE_U(A,C)$ denotes  the set of all minimal pre-extensions of $A$ to a copy of $C$ in $U$.
We are under Assumption  \ref{assumption.6}.
Let $i_0\le d$ be such that $t_{i_0}= 0^{(l_C)}$, and note that $i_0$ is a member of  $I_0$.
Each member $Y$ of $\MPE_T(A,C)$
will be enumerated as $\{y_i:i\le d\}$ so that $y_i\contains s_i$ for each $i\le d$.
Recall notation (\ref{eq.SP})
of
$\SP_T(Y)$.

Since $C$ satisfies the SWP,
$\tilde{X}$ is in $\MPE_T(A,C)$.
Let  $P$ denote $\SP_T(\tilde{X})$.
Since  $\tilde{X}$ is contained in an interval of $T$ above the interval containing $\max(A)$,
each node of $P$ extends exactly one node of $A_e$.
For any $U\in[B,T]$, define
$\Ext_{U}(A,P)$ to consist of those $X$ such that 
\begin{equation}
X=\SP_U(Y) \mathrm{\ for\  some\ }
Y\in\MPE_U(A,C).
\end{equation}

By assumption,
the coloring $h$ on $\Ext_{T}(A,C)$  is end-homogeneous.
This
induces a coloring  $h$ on $\MPE_T(A,C)$
 by defining, for $Y\in\MPE_T(A,C)$,
 $h(Y)$ to be
 {\em the}  $h$-color that all members of
 $\Ext_T(A,C;Y)$ have.
This further
induces a coloring $h'$  on $\Ext_{T}(A,P)$
as follows:
For $Q\in \Ext_{T}(A,P)$, for the $Y\in \MPE_T(A,C)$ such that $\SP_T(Y)=Q$,
let
 $h'(Q)=h(Y)$.
Given $Q\in \Ext_{T}(A,P)$, the  extensions of the $q_i\in Q$ such that $i\in I_1$  to the level of next coding node in $T$, with passing number $1$ at  that coding node, recovers 
$Y$.
Thus, $h'$ is well-defined.

Let $L$ denote the collection of all  $l\in\bN$ such that there is a member of
 $\Ext_{T}(A,P)$ with maximal  nodes of length $l$.
For each
  $i\in (d+1)\setminus\{i_0\}$,   let  $T_i=\{t\in T:t\contains s_i\}$.
Let $T_{i_0}$ be  the collection of all leftmost nodes in $T$ extending $s_{i_0}$.
Let $\kappa=\beth_{2d+2}$.
The following forcing notion $\bQ$  will  add $\kappa$ many paths through each $T_i$, $i\in (d+1)\setminus\{i_0\}$ and
 one path through $T_{i_0}$, though with $\kappa$ many labels.
The present case is handled similarly to Case (a) of Theorem \ref{thm.matrixHL}.

Let
$\bQ$ be  the set of conditions $p$ such that
$p$ is a  function
of the form
$$
p:(d+1)\times\vec{\delta}_p\ra  T,
$$
where $\vec{\delta}_p$ is a finite subset of $\kappa$,
$l_p\in L$,
$\{p(i,\delta):\delta\in\vec\delta_p\}\sse T_i$ for each $i<d$,
and
\begin{enumerate}
\item[(i)]
There  is  
some coding node
$c^{T}_{n(p)}$ in $T$
 such that $l^{T}_{n(p)}=l_p$,
 and
 $l^{T}_{n(p)-1}< |p(i,\delta)|\le l_p$ for each
 $(i,\delta)\in (d+1)\times
\vec{\delta}_p$.
\item [(ii)]
\begin{enumerate}
\item[($\al$)]
If  $i\in I_1$,
 then
 $p(i,\delta)=\splitpred_{T}(y)$ for some $y\in T_i\re l_p$.
\item[$(\beta)$]
If $i\in I_0$,  then
$p(i,\delta)\in T_i\re l_p$ and has immediate extension $0$ in $T$.
\end{enumerate}
\end{enumerate}
It follows from the definition  that for $p\in \bQ$,
 $\ran(p):=\{p(i,\delta):(i,\delta)\in (d+1)\times\vec{\delta}_p\}$  is free in $T$: leftmost extensions add no new pre-cliques.
Furthermore, all nodes in $\ran(p)$ are contained in the $n(p)$-th interval of $T$.
We point out that  $\ran(p)$ may or may not contain a coding node.
If it does, then that coding node must appear as $p(i,\delta)$ for some $i\in I_0$; this $i$ may or may not equal $d$.

The partial ordering on $\bQ$ is defined as follows:
$q\le p$ if and only if
$l_q\ge l_p$, $\vec{\delta}_q\contains \vec{\delta}_p$,
\begin{enumerate}
\item[(i)]
$q(i,\delta)\contains p(i,\delta)$  for each $(i,\delta)\in (d+1)\times\vec{\delta}_p$; and

\item[(ii)]
$\ran(q\re \vec\delta_p):=\{q(i,\delta):(i,\delta)\in (d+1)\times\vec{\delta}_p\}$
has no new pre-cliques over
 $\ran(p)$.
\end{enumerate}

By arguments similar to those  in the proof of Theorem \ref{thm.matrixHL},
 $(\bQ,\le)$ is an atomless partial order, and any condition in $\bQ$ can be extended by two incompatible conditions of length greater than any given $l\in\bN$.

Let $\dot{\mathcal{U}}$ be a $\bQ$-name for a non-principal ultrafilter on $L$.
For each $i\le d$ and $\al<\kappa$,  let
$\dot{b}_{i,\al}=\{\lgl p(i,\al),p\rgl:p\in \bQ$ and $\al\in\vec{\delta}_p\}$,
 a $\bQ$-name for the $\al$-th generic branch through $T_i$.
For
any condition $p\in \bQ$, for
$(i,\al)\in
 I_0\times \vec\delta_p$,  $p$ forces that $\dot{b}_{i,\al}\re l_p= p(i,\al)$.
For $(i,\al)\in I_1\times\vec\delta_p$,  $p$ forces that $\splitpred_{T}(\dot{b}_{i,\al}\re l_p)= p(i,\al)$.
For $\vec{\al}=\lgl\al_0,\dots,\al_{d}\rgl\in[\kappa]^{d+1}$,
\begin{equation}
\mathrm{let\ \ }\dot{b}_{\vec{\al}}\mathrm{\  \  denote\ \ }
\lgl \dot{b}_{0,\al_0},\dots,\dot{b}_{d,\al_d}\rgl.
\end{equation}
For $l\in L$, we shall  use the abbreviation
\begin{equation}
\dot{b}_{\vec\al}\re l
\mathrm{\ \ to\ denote \ \ }
\SP_T(\dot{b}_{\vec\al}\re l),
\end{equation}
which is exactly
$\{\dot{b}_{i,\al_i}\re l :i\in I_0\}\cup \{\splitpred_T(\dot{b}_{i,\al_i}\re l):i\in I_1\}$.

Similarly to  the proof of Theorem \ref{thm.matrixHL},
we will find  infinite  pairwise disjoint sets $K_i\sse \kappa$, $i\le d$, such that $K_0<K_1<\dots <K_d$,
and conditions $p_{\vec\al}$, $\vec\al\in \prod_{i\le d}K_i$,
such that these conditions are pairwise compatible,
have the same images in $T$, and force the same color $\varepsilon^*$  for $h'(\dot{b}_{\vec\al}\re l)$  for  $\dot{\mathcal{U}}$ many levels  $l$ in $L$.
Moreover, the nodes $\{t^*_i:i\le d\}$ obtained from the application of the \Erdos-Rado Theorem for this setting
will  extend
 $\{s_i:i\le d\}$
and  form a member of $\Ext_{T}(A,P)$.
The arguments are  quite  similar to those in Theorem \ref{thm.matrixHL}, so we only fill in the details for arguments which are necessarily different.
\vskip.1in

\noindent{\bf \underline{Part I}.}
Given $p\in\bQ$ and $\vec\al\in   [\vec\delta_p]^{d+1}$,
let
\begin{equation}
P(p,\vec\al)=\{p_{\vec\al}(i,\al_i):i\le  d\}.
\end{equation}
For each $\vec\al\in[\kappa]^{d+1}$,
choose a condition $p_{\vec{\al}}\in\bQ$ such that
\begin{enumerate}
\item
 $\vec{\al}\sse\vec{\delta}_{p_{\vec\al}}$.

\item
$P(p,\vec\al) \in\Ext_T(A,P)$.
\item
There is an  $\varepsilon_{\vec{\al}}\in 2$ such that
$p_{\vec{\al}}\forces$ ``$h(\dot{b}_{\vec{\al}}\re l)=\varepsilon_{\vec{\al}}$
for $\dot{\mathcal{U}}$ many $l$ in $L$."

\item
$h'(P(p_{\vec\al},\vec\al))=\varepsilon_{\vec{\al}}$.
\end{enumerate}

Properties (1) -  (4) can be guaranteed  as follows.
For each $i\le d$, let $u_i$ denote the member of $P$ which extends $s_i$.
For each  $\vec{\al}\in[\kappa]^{d+1}$, let
$$
p^0_{\vec{\al}}=\{\lgl (i,\delta), u_i\rgl: i\le d, \ \delta\in\vec{\al} \}.
$$
Then $p^0_{\vec{\al}}$ is a condition in $\bP$ and
$\vec\delta_{p_{\vec\al}^0}= \vec\al$, so (1) holds for every $p\le p^0_{\vec{\al}}$.
Further,
$\ran(p^0_{\vec\al})$ is a member of $\Ext_T(A,P)$ since it  equals $P$.
For any $p\le p_{\vec\al}^0$,
(ii) of the definition of the partial ordering on $\bQ$ guarantees that
 $P(p,\vec\al)$
 has no new pre-cliques over $\ran(p)$, and hence
is also  a member of $\Ext_T(A,P)$.
Thus,  (2) holds for any $p\le p_{\vec\al}^0$.
Take  an extension $p^1_{\vec{\al}}\le p^0_{\vec{\al}}$ which
forces  $h'(\dot{b}_{\vec{\al}}\re l)$ to be the same value for
$\dot{\mathcal{U}}$  many  $l\in L$,
and which decides that value, denoted by  $\varepsilon_{\vec{\al}}$.
Then any $p\le p^1_{\vec{\al}}$ satisfies (3).

Take $p_{\vec\al}^3\le p^2_{\vec\al}$  which  decides  $h'(\dot{b}_{\vec\al}^3 \re l)=\varepsilon_{\vec\al}$, for some $l$ such that
$l_{p^2_{\vec\al}}<l \le l_{p^3_{\vec\al}}$.
If $l= l_{p^3_{\vec\al}}$, let $p_{\vec\al}=p_{\vec\al}^3$.
Otherwise,
let $\vec\delta_{\vec\al}=\vec\delta_{p^2_{\vec\al}}$ and define
 $p_{\vec\al}$ as follows:
For each $i\in I_0$,
for  $\delta\in \vec\delta_{\vec\al}$,
let
$p_{\vec\al}(i,\delta)=p_{\vec\al}^3(i,\delta)\re l$.
For each $i\in I_1$,
for $\delta\in \vec\delta_{\vec\al}$,
let
$p_{\vec\al}(i,\delta)=\splitpred_T (p_{\vec\al}^3(i,\delta)\re l)$.
Then $p_{\vec\al}$ is a condition in $\bQ$,
and $p_{\vec\al}\le p_{\vec\al}^2$, so it satisfies (1) - (3).
Furthermore,
$h'(P(p_{\vec\al},\vec\al))=\varepsilon_{\vec\al}$, so  $p_{\vec\al}$ satisfies (4).

We are assuming $\kappa=\beth_{2d+2}$.
Let $D_e=\{0,2,\dots,2d\}$ and  $D_o=\{1,3,\dots,2d+1\}$, the sets of  even and odd integers less than $2d+2$, respectively.
Let $\mathcal{I}$ denote the collection of all functions $\iota: (2d+2)\ra (2d+2)$ such that
$\iota\re D_e$
and $\iota\re D_o$ are strictly  increasing sequences
and $\{\iota(0),\iota(1)\}<\{\iota(2),\iota(3)\}<\dots<\{\iota(2d),\iota(2d+1)\}$.
For $\vec{\theta}\in[\kappa]^{2d+2}$,
$\iota(\vec{\theta}\,)$ determines the pair of sequences of ordinals $(\theta_{\iota(0)},\theta_{\iota(2)},\dots,\theta_{\iota(2d))}), (\theta_{\iota(1)},\theta_{\iota(3)},\dots,\theta_{\iota(2d+1)})$,
both of which are members of $[\kappa]^{d+1}$.
Denote these as $\iota_e(\vec\theta\,)$ and $\iota_o(\vec\theta\,)$, respectively.
Let $\vec{\delta}_{\vec\al}$ denote $\vec\delta_{p_{\vec\al}}$,
 $k_{\vec{\al}}$ denote $|\vec{\delta}_{\vec\al}|$,
and let $l_{\vec{\al}}$ denote  $l_{p_{\vec\al}}$.
Let $\lgl \delta_{\vec{\al}}(j):j<k_{\vec{\al}}\rgl$
denote the enumeration of $\vec{\delta}_{\vec\al}$
in increasing order.
Define a coloring  $f$ on $[\kappa]^{2d+2}$ into countably many colors as follows:
Given  $\vec\theta\in[\kappa]^{2d+2}$ and
 $\iota\in\mathcal{I}$, to reduce the number of subscripts,  letting
$\vec\al$ denote $\iota_e(\vec\theta\,)$ and $\vec\beta$ denote $\iota_o(\vec\theta\,)$,
define
\begin{align}\label{eq.fiotatheta(c)}
f(\iota,\vec\theta\,)= \,
&\lgl \iota, \varepsilon_{\vec{\al}}, k_{\vec{\al}},
\lgl \lgl p_{\vec{\al}}(i,\delta_{\vec{\al}}(j)):j<k_{\vec{\al}}\rgl:i\le d\rgl,\cr
& \lgl  \lgl i,j \rgl: i\le d,\ j<k_{\vec{\al}},\ \mathrm{and\ } \delta_{\vec{\al}}(j)=\al_i \rgl, \cr
&\lgl \lgl j,k\rgl:j<k_{\vec{\al}},\ k<k_{\vec{\beta}},\ \delta_{\vec{\al}}(j)=\delta_{\vec{\beta}}(k)\rgl\rgl.
\end{align}
Let $f(\vec{\theta}\,)$ be the sequence $\lgl f(\iota,\vec\theta\,):\iota\in\mathcal{I}\rgl$, where $\mathcal{I}$ is given some fixed ordering.
By the \Erdos-Rado Theorem,
there is a subset $K\sse\kappa$ of cardinality $\aleph_1$
which is homogeneous for $f$.

Take $K'\sse K$ such that between each two members of $K'$ there is a member of $K$.
Then take subsets $K_i\sse K'$ such that  $K_0<\dots<K_{d}$
and   each $|K_i|=\aleph_0$.
The following four lemmas are direct analogues of
Lemmas \ref{lem.onetypes}, \ref{lem.j=j'},  \ref{lem.compat}, and \ref{lem.subclaimA}.
Their proofs follow by simply making the correct notational substitutions, and so are omitted.

\begin{lem}\label{lem.onetypes(c)}
There are $\varepsilon^*\in 2$, $k^*\in\om$,
and $ \lgl t_{i,j}: j<k^*\rgl$, $i\le d$,
 such that
for all $\vec{\al}\in \prod_{i\le d}K_i$ and  each $i\le d$,
 $\varepsilon_{\vec{\al}}=\varepsilon^*$,
$k_{\vec\al}=k^*$,  and
$\lgl p_{\vec\al}(i,\delta_{\vec\al}(j)):j<k_{\vec\al}\rgl
=
 \lgl t_{i,j}: j<k^*\rgl$.
\end{lem}

Let $l^*=|t_{i_0}|$.
Then for each $i\in I_0$,
 the nodes  $t_{i,j}$,  $j<k^*$,    have  length $l^*$;
and for each $i\in I_1$,
the nodes $t_{i,j}$,  $j<k^*$,  have length in the interval $(l^T_{n-1},l^T_n)$,
where $n$ is the index of the coding node in $T$ of length $l^*$.

\begin{lem}\label{lem.j=j'(c)}
Given any $\vec\al,\vec\beta\in \prod_{i\le d}K_i$,
if $j,k<k^*$ and $\delta_{\vec\al}(j)=\delta_{\vec\beta}(k)$,
 then $j=k$.
\end{lem}

For any $\vec\al\in \prod_{i\le d}K_i$ and any $\iota\in\mathcal{I}$, there is a $\vec\theta\in[K]^{2d+2}$ such that $\vec\al=\iota_o(\vec\theta)$.
By homogeneity of $f$, there is a strictly increasing sequence
$\lgl j_i:i\le d\rgl$  of members of $k^*$ such that for each $\vec\al\in \prod_{i\le d}K_i$,
$\delta_{\vec\al}(j_i)=\al_i$.
For each $i\le d$, let $t^*_i$ denote $t_{i,j_i}$.
Then  for each $i\le d$ and each $\vec\al\in \prod_{i\le d}K_i$,
\begin{equation}
p_{\vec\al}(i,\al_i)=p_{\vec{\al}}(i, \delta_{\vec\al}(j_i))=t_{i,j_i}=t^*_i.
\end{equation}

\begin{lem}\label{lem.compat(c)}
For any finite subset $\vec{J}\sse \prod_{i\le d}K_i$,
the set of conditions $\{p_{\vec{\al}}:\vec{\al}\in \vec{J}\,\}$ is  compatible.
Moreover,
$p_{\vec{J}}:=\bigcup\{p_{\vec{\al}}:\vec{\al}\in \vec{J}\,\}$
is a member of $\bP$ which is below each
$p_{\vec{\al}}$, $\vec\al\in\vec{J}$.
\end{lem}

\begin{lem}\label{lem.subclaimA(c)}
If $\beta\in \bigcup_{i\le d}K_i$,
$\vec{\al}\in\prod_{i\le d}K_i$,
and $\beta\not\in\vec\al$,
 then
$\beta$ is not  a member of   $\vec{\delta}_{\vec{\al}}$.
\end{lem}

\noindent{\bf \underline{Part II}.}
Let $(n_j)_{j\in \bN}$ denote the set of  indices for which  there is an
$X\in \MPE_T(A,C)$
with $X=\max(V)$ for some
$V$ of $r_{n_j}[B,T]$.
 For  $i\in I_0$,
let $u^*_i=t^*_i$.
For $i\in I_1$,
let $u_i^*$ be the  leftmost extension of $t^*_i$ in $T\re l^*$.
Note that $\{u_i^*:i\le d\}$ has no new pre-cliques over $A_e$, since leftmost extensions of splitting nodes with no new pre-cliques add no new pre-cliques; this follows from the WP of  $T$.
Extend each node $u$ in $B^+\setminus A_e$   to its  leftmost extension in $T\re l^*$ and label that extension $u^*$.
Let
\begin{equation}
U^*=\{u^*_i:i\le d\}\cup\{u^*:u\in B^+ \setminus A_e\}.
\end{equation}
Then  $U^*$ extends $B^+$,
and $U^*$ has no new pre-cliques over $B$.
Let $n_{-1}$ be the integer such that
$B=r_{n_{-1}}(B)$.
Take
  $S_0\in r_{n_0}[B,T]$
such that
the nodes in $\max(r_{n_{-1}+1}(S_0))$ extend the nodes in $U^*$.
This is possible by Lemma \ref{lem.HLconstruction}.

Suppose that $j\in\bN$, and  for all $i<j$,  we have chosen  $S_i\in r_{n_i}[B,T]$ such that
 $i<i'<j$ implies
$S_i\sqsubset S_{i'}$, and
$h'$ is constant of value $\varepsilon^*$ on $\Ext_{S_{i}}(A,P)$.
Take $V_j\in r_{n_j}[S_{j-1},T]$, and
let $X$ denote $\max(V_j)$.
Notice that
each member of
$\Ext_X(A,P)$
extends the nodes in $U^*$.
By the definition of $n_j$, the set of nodes $X$ contains a coding node.
For each $i\in I_0$,
let $Y_i$ denote the set of all $t\in T_i\cap X$
which have immediate extension $0$ in $T$.
Let  $n$ be such that $l_X=l^T_n$.
For each $i\in I_1$,
let $Y_i$ denote the set of all
splitting  predecessors of nodes in
$T_i\cap X$ which split in the interval
$(l^T_{n-1},l^T_n]$ of $T$.
For each $i\le d$,
let $J_i$ be a subset of $K_i$ of size $|Y_i|$,
and enumerate the members of $Y_i$ as $q(i,\delta)$, $\delta\in J_i$.
Let $\vec{J}$ denote the set of $\vec\al\in\prod_{i\le d}J_i$ such that the set $\{q(i,\al_i):i\le d\}$
has no new pre-cliques over $A$.
Thus,
the collection of sets  $\{q(i,\al_i):i\le d\}$, $\vec\al\in \vec{J}$, is exactly the collection of sets of nodes in the interval  $(l^T_{n-1},l^T_n]$ of $T$
 which are members of
 $\Ext_{T}(A,P)$.
Moreover, for
$\vec\al\in \vec{J}$
and   $i\le d$,
\begin{equation}
q(i,\al_i)\contains t^*_i=p_{\vec{\al}}(i,\al_i).
\end{equation}

To complete the construction of the desired $q\in \bQ$ for which $q\le p_{\vec\al}$ for all $\vec\al\in \vec{J}$,
let  $\vec{\delta}_q=\bigcup\{\vec{\delta}_{\vec\al}:\vec\al\in \vec{J}\}$.
For each pair $(i,\gamma)$  with  $\gamma\in\vec{\delta}_q\setminus
J_i$,
there is at least one $\vec{\al}\in\vec{J}$ and some $m<k^*$ such that $\gamma=\delta_{\vec\al}(m)$.
As in Case (a) of Theorem \ref{thm.matrixHL},
for any other $\vec\beta\in\vec{J}$ for which $\gamma\in\vec{\delta}_{\vec\beta}$,
it follows
 that
$p_{\vec\beta}(i,\gamma)=p_{\vec{\al}}(i,\gamma)=t^*_{i,m}$ and  $\delta_{\vec\beta}(m)=\gamma$.
If $i\in I_0$,
let $q(i,\gamma)$ be the leftmost extension
 of $t_{i,m}^*$ in  $T\re l^{V_j}_{n_j}$.
If $i\in I_1$, let $q(i,\gamma)$ be the leftmost extension of $t^*_{i,m}$ to a splitting node in $T$
in the interval
 $(l^{V_j}_{n_j-1}, l^{V_j}_{n_j}]$.
Such a splitting node must exist, because
the coding node in $X$ must have  no pre-cliques
 with  $t^*_i$ (since $i\in I_1$).
 Thus, by Lemma \ref{lem.poc}
 the leftmost extension of $t^*_i$  in $T$ to length $l_X$ has no pre-cliques with the coding node in $X$, so it has  a splitting predecessor in the interval
  $(l^{V_j}_{n_j-1}, l^{V_j}_{n_j}]$.
Define
\begin{equation}
q=\bigcup_{i\le d}\{\lgl (i,\al),q(i,\al)\rgl: \al\in \vec{\delta}_q\}.
\end{equation}
By a  proof similar to that of
Claim \ref{claim.qbelowpal},
it follows that
$q\le p_{\vec\al}$,
for each $\vec\al\in \vec{J}$.

Take an $r\le q$ in  $\bP$ which  decides some $l_j$ in $L$,
and such that
for all $\vec\al\in\vec{J}$,  $h'(\dot{b}_{\vec\al}\re l_j)=\varepsilon^*$.
Without loss of generality, we may assume that the maximal nodes in $r$ have length $l_j$.
If $q(i',\al')$ is a coding node for some $i'\in I_0$ and $\al'\in J_{i'}$,
then let $c_r$ denote $r(i',\al')$;
otherwise, let $c_r$ denote the leftmost extension  in $T$  of the coding node in $X$ to  length $l_j$.
Let $c_X$ denote the coding node in $X$.

Let $Z_0$ denote those nodes in  $\SP_T(X)$
which have length equal to 
$l_X$
and
 are not in $\bigcup_{i\in I_0} Y_i$.
For each $z\in Z_0$,
let $s_z$ denote the leftmost extension of $z$ in $T$ to length $l_j$.
Let $Z_1$ denote the set of all 
nodes in $\splitpred_T(X)$ which are not in 
$Z_0\cup \bigcup_{i\in I} Y_i$.
For each $z\in Z_1$, let $s_z$ denote the  splitting predecessor  of the leftmost extension of $z$ in $T$ to length $l_j$.
This splitting node $s_z$  exists in $T$ for the following reason:
If $z$ is a  splitting predecessor of a node in $X$,
then $z$ has no pre-$k$-clique with  $c_X$,  so the leftmost extension of $z$ to any length has no pre-$k$-cliques  with any extension of $c_X$.
In particular, the set $\{s_z:z\in Z_0\cup Z_1\}$ has no new  pre-cliques  over $X$.

Let
\begin{equation}
Z^-=\{r(i,\al):i\le d,\ \al\in J_i\}\cup\{s_z:z\in Z_0\cup Z_1\}.
\end{equation}
Let $Z^*$ denote all extensions in $T$ of the members of $Z^-$ to length $l_j$.
Note that $Z^-=\splitpred_T(Z^*)$, which end-extends $\splitpred_T(X)$.
Let $m$ denote the index such that the maximal coding node in $V_j$ below $c_X$ is $c^{V_j}_{n_m}$.
Note that $Z^*$ has no new pre-cliques over $\splitpred_T(X)$;
furthermore,  the tree induced by
$r_{n_m}(V_j)\cup Z^*$ is strongly similar to $V_j$, except that  the coding node might possibly be in the wrong place.
Using Lemma \ref{lem.HLconstruction},
 there is an $S_j\in r_{n_j}[r_{n_m}(V_j),T]$  with  $\max(S_j)$ extending $Z^*$.
Then every member of $\Ext_{S_j}(A,P)$ has the same $h'$ color $\varepsilon^*$,
by the choice of $r$,
since each minimal pre-extension in $\MPE_{S_j}(A,C)$ extends some member of $\Ext_{S_j}(A,P)$
which extends members in $\ran(r)$ and so have $h'$-color $\varepsilon^*$.

Let $S=\bigcup_{j\in\bN} S_j$.
Then $S$ is a strong coding tree in $[B,T]$.
Given any $Y\in\Ext_S(A,C)$,
 there is some
 $X\in\MPE_S(A,C)$
such that $Y$ extends $X$.
Since
$\SP_S(X)$ is in $\Ext_{S_j}(A,P)$ for some $j\in\bN$,
$\SP_S(X)$ has
 $h'$ color $\varepsilon^*$.
Thus,
$Y$ has
 $h$-color $\varepsilon^*$.
\end{proof}


\begin{lem}\label{lem.Case(b)}
Assume \ref{assumption.6}.
Then there is a strong coding subtree $S\le T$ such that for each  copy $A'$ of $A$ in $S$,
$h$ is homogeneous on $\Ext_S(A',C)$.
\end{lem}

\begin{proof}
Let $(n_i)_{i\in\bN}$ be the sequence of integers such that $r_{n_i}(T)$ contains a  copy  of $A$
which is valid in
$r_{n_i}(T)$ and such that $\max(A)\sse\max(r_{n_i}(T))$.
Let $n_{-1}=0$, $T_{-1}=T$,
and
 $U_{-1}=r_{0}(T)$.

Suppose $i\in\bN$, and $U_{i-1}\cong r_{n_{i-1}}(T)$
 and $T_{i-1}$ are given satisfying that for each copy $A'$ of $A$
  valid in $U_{i-1}$ with $\max(A)\sse \max(U_{i-1})$,
$h$ is homogeneous on $\Ext_{U_{i-1}}(A',C)$.
Let $U_i$ be in $r_{n_i}[U_{i-1},T_{i-1}]$.
Enumerate all
copies
$A'$ of $A$ which are valid in $U_i$ and have $\max(A')\sse \max(U_i)$ as $\lgl A_0,\dots,A_m\rgl$.
Apply Lemma \ref{lem.endhomog} to obtain  $R_{0}\in [U_i,T_{i-1}]$ which is end-homogeneous for $\Ext_{R_0}(A_0,C)$.
Then
apply Lemma \ref{lem.Case(c)} to
obtain $R'_{0}\in [U_i,R_0]$
such that $\Ext_{R'_{0}}(A_0,C)$ is homogeneous for $h$.
Given $R'_{j}$  for $j<m$,
apply Lemma \ref{lem.endhomog} to obtain a $R_{j+1}\in [U_i,R'_{j}]$ which is end-homogeneous for $\Ext_{R_{j+1}}(A_{j+1},C)$.
Then
apply Lemma \ref{lem.Case(c)} to
obtain $R'_{j+1}\in [U_i,R_{j+1}]$
such that $\Ext_{R'_{j+1}}(A_{j+1},C)$ is homogeneous for $h$.
Let $T_i=R'_m$.

Let $U=\bigcup_{i\in\bN}U_i$.
Then $U\le T$ and $h$ has the same color on $\Ext_U(A',C)$
for   each copy
$A'$ of $A$ which is valid in $U$.
Finally, take $S\le U$ such that for each $m\in\bN$,
$r_m(S)$ is valid in $U$.
Then
 each copy  $A'$ of $A$  in $S$ is valid in $U$.
Hence,
$h$ is homogeneous on  $\Ext_S(A',C)$, for each copy $A'$ of $A$ in $S$.
\end{proof}

For the setting of Case (a) in Theorem \ref{thm.matrixHL}, a similar lemma holds.
The proof is omitted, as it is almost identical, making the obvious changes.

\begin{lem}\label{lem.fusionsplit}
Let $T$ be  a member of $\mathcal{T}_k$,  let $A,\tilde{X},h$ be as in Case (a) of Theorem \ref{thm.matrixHL}, and let $C=A\cup\tilde{X}$.
Then there is a strong coding tree $S\le T$ such that for each  $A'\sse S$ with $A'\cong A$,
$\Ext_S(A',C)$ is homogeneous for $h$.
\end{lem}

Finally, for the case of $k\ge 4$, a new phenomenon appears: new singleton pre-$a$-cliques for $a\ge 4$ must be dealt with. 
The steps are very similar to the case when the level set $\tilde{X}$ contains a coding node.
\vskip.1in

\noindent\bf Case (b${}'$). \rm
Assume $k\ge 4$. 
Let $A$ and $C$ be fixed non-empty finite valid subtrees of a strong coding tree $T\in\mathcal{T}_k$ such that
\begin{enumerate}
\item
$A$ and $C$ both satisfy the Strict Witnessing Property; and
\item
$C\setminus A$ is a level set containing exactly  one new singleton pre-$a$-clique, for some $a\in [4,k]$.
\end{enumerate}

Given a copy $A'$ of $A$ in  a strong $\mathcal{H}_k$-coding tree $T$, let $\Ext_T(A',C)$  denote the set of all $C'$ contained in $T$ which end-extend $A'$ and such that $C'\cong A'$.

\begin{lem}[Case (b${}'$)]\label{lem.Case(bb)}
Given $A$ and $C$ as above, there is a strong coding subtree $S\le T$ such that for each  copy $A'$ of $A$ in $S$,
$h$ is homogeneous on $\Ext_S(A',C)$.
\end{lem}

\begin{proof}
The proof follows from simple modifications of the proofs of 
Case (b) in Theorem 
\ref{thm.matrixHL} and 
Lemmas \ref{lem.endhomog}, \ref{lem.Case(c)}, and \ref{lem.Case(b)}.
Just replace the coding node  in Case (b) with  the new singleton
 pre-$a$-clique in  Case (b$'$).  
Splitting predecessors work as before.
\end{proof}

\noindent {\bf Proof of Theorem \ref{thm.MillikenSWP}}.
The proof is by induction on the number of critical nodes, where in this proof,  by {\em critical node} we mean a coding node, a splitting node, or 
a new singleton pre-$a$-clique for  some $a\in[4,k]$.
Suppose first that $A$ consists of a single node.
Then $A$ consists of a single  splitting node in $T$ on the leftmost branch of $T$,
so the  strongly isomorphic copies of
$A$ are exactly  the leftmost splitting nodes
in $T$.
Recall that $\Seq[0]$ denotes the set of all finite length sequences of $0$'s.
Thus, the copies of $A$ in $T$ are exactly those
 splitting nodes in $T$ which  are  members of $\Seq[0]$.
Let $h$ be any finite coloring on the splitting nodes in the leftmost branch of $T$.
By Ramsey's Theorem,
 infinitely many  splitting nodes in the leftmost branch of $T$ must have the same $h$ color.
 By the Extension Lemmas in Section \ref{sec.ExtLem},
 there is a  subtree $S\le T$
 in which all splitting nodes in the leftmost branch of $S$ have  the  same $h$ color.

Now assume that $n\ge 1$ and  the theorem holds
for each finite tree $B$  with  $n$ or less critical nodes
such that $B$ satisfies the SWP and  $\max(B)$ contains a node which is a sequence of all $0$'s.
Let $C$ be a finite tree with $n+1$ critical nodes
containing a maximal node in $\Seq[0]$, and suppose
 $h$  maps  the copies of $C$ in $T$ into finitely many colors.
Let  $d$ denote the maximal critical node in $C$ and let
  $B=\{t\in C: |t|<|d|\}$.
Apply 
Lemma \ref{lem.Case(b)}, \ref{lem.fusionsplit} or \ref{lem.Case(bb)}, as appropriate, 
to obtain $T'\le T$ so that for each copy $V$ of $B$ in $T'$,  the set $\Ext_{T'}(V,C)$ is homogeneous for $h$.
Define $g$ on the copies of $B$ in $T'$ by  letting $g(V)$ be the value of  $h$ on $V\cup X$ for any $X\in\Ext_{T'}(V,C)$.
By the induction hypothesis,
there is an $S\le T'$ such that $g$ is homogeneous on  all copies of $B$ in $S$.
It follows that $h$ is homogeneous on the copies of $C$ in $S$.

To finish, let $A$ be any finite tree satisfying the SWP.
If  $\max(A)$ does not contain a member of $\Seq[0]$,
let $l_A$ denote the longest length of nodes in $A$,
and let $\tilde{A}$ be the tree induced by $A\cup\{0^{(l_A)}\}$.
Otherwise, let $\tilde{A}=A$.
Let  $g$ be a  finite coloring of the copies of $A$ in  $T$.
To each copy $B$ of $\tilde{A}$ in $T$ there corresponds a unique copy of $A$ in $T$, denoted $\varphi(B)$:
If $\tilde{A}=A$, then $\varphi(B)=B$;
if $\tilde{A}\ne A$, then $\varphi(B)$ is $B$ with the leftmost node in $\max(B)$ removed.
For each copy $B$ of $\tilde{A}$, define
$h(B)=g(\varphi(B))$.
Take $S\le T$ homogeneous for $h$.
Then $S$ is homogeneous for $g$ on  the copies of $A$ in $S$.
\hfill $\square$


\section{Main Ramsey Theorem for strong $\mathcal{H}_k$-coding trees}\label{sec.MainThm}

The third phase of this article takes place in this and the next section.
Subsection \ref{sec.squiggle} develops the notion of incremental trees, which sets the stage for envelopes for incremental antichains.
These envelopes
transform finite antichains of coding nodes to finite trees with the Strict Witnessing Property,
enabling  applications of
Theorem \ref{thm.MillikenSWP} to deduce
 Theorem \ref{thm.mainRamsey}.
 This theorem takes a finite coloring of
 all antichains of coding nodes strictly similar to a given finite antichain of coding nodes
 and  finds a strong coding tree in which the coloring has one color.
After showing  in Lemma \ref{lem.bD} that any strong coding tree contains an antichain of coding nodes coding a Henson graph,
we will apply  Theorem \ref{thm.mainRamsey} to prove that each Henson graph has finite big Ramsey degrees, thus obtaining the main result of this paper in Theorem \ref{finalthm}.

\subsection{Incremental trees}\label{sec.squiggle}

The new notions of {\em incremental new pre-cliques} and {\em incrementally witnessed pre-cliques}, and {\em  incremental trees} are defined now.
The main lemma of this subsection, Lemma \ref{lem.squiggletree},  shows  that given a strong coding tree $T$,
there is an incremental strong coding subtree $S\le T$ and
   a set $W\sse T$ of coding nodes disjoint from $S$ such that
   all pre-cliques in $S$ are incrementally witnessed by coding nodes in $W$.
This sets the stage for the development of   {\em envelopes} with the Strict Witnessing Property  in the next subsection.

\begin{defn}[Incremental  Pre-Cliques]\label{defn.incrementalpo}
Let  $S$ be a subtree of $\bT_k$,
 and let $\lgl l_j:j<\tilde{j}\rgl$
list in increasing order the minimal lengths
of new  pre-cliques  in $S$, except for singleton new pre-$3$-cliques, where $\tilde{j}\in\bN$ or $\tilde{j}=\bN$.
We say that
$S$ has {\em incremental  new pre-cliques},
or simply $S$ is
{\em incremental},
 if
 letting
\begin{equation}
S_{l_j,1}:=\{t\re l_j:t\in S, \ |t|>l_j,\mathrm{\  and \ } t(l_j)=1\},
\end{equation}
 the following hold
 for each $j<\tilde{j}$:
\begin{enumerate}
\item
$S_{l_j,1}$
is a  new  pre-$a$-clique for some $a\in [3,k]$, and no proper subset of $S_{l_j,1}$ is a new pre-$b$-clique for any $b\in[3,k]$;
\item
If $a=3$ and
 $S_{l_j,1}$ has more than two  members, then for each
 proper subset $X\subsetneq S_{l_j,1}$ of size at least $2$,
 for some $i<j$,  $X\re l_i=S_{l_i,1}$ and is also a pre-$3$-clique;
\item
If $a>3$ and
 $S_{l_j,1}$ has  at least two  members, then for each
 proper subset $X\subsetneq S_{l_j,1}$,
   for some $i<j$,  $X\re l_i=S_{l_i,1}$ and is also a pre-$a$-clique;
\item
If $a>3$,
then there are   $l_{j-1}<l^3<\dots<l^a=l_j$ such that for each $3\le b\le a$,
$S_{l_j,1}\re l^b$ is a pre-$b$-clique.
Furthermore, for some $m$, $|d^S_m|<l^3<l^a=l_j<|d^S_{m+1}|$.
\end{enumerate}

A tree $S\in\mathcal{T}_k$ is called an {\em incremental strong coding tree}
if  $S$ is incremental and moreover,
the node
$d^S_{m+1}$  in  (4) is a coding node in $S$.
\end{defn}

Note that every subtree of an incremental  strong coding tree is  incremental, but a strong coding subtree of an incremental strong coding tree need not be an incremental strong coding tree.
Note also that in (4), the pre-$b$-cliques
for $3\le b<a$
 at the levels $l^3$ through $l^a$ are not new, but they build up to the new pre-$a$-clique in the  interval $(l_{j-1},l_j]$.
This  redundancy  will actually make the definition of the envelopes simpler.

\begin{defn}[Incrementally Witnessed  Pre-Cliques]\label{defn.incremental}
Let $S,T\in \mathcal{T}_k$ be such that $S$ is incremental and
$S\le T$.
We say that  the  pre-cliques   in $S$  are  {\em incrementally witnessed} by a set of witnessing coding nodes $W\sse T$
if the following hold.
Given that $\lgl l_j:j\in\bN\rgl$ is the increasing enumeration of the minimal lengths of new pre-cliques in $S$,
for each $j\in\bN$ the following hold:
\begin{enumerate}
\item
$|d^S_{m_n-1}|<l_j< l^S_n$ for some $n\in\bN$.
\item
If $S_{l_j,1}$ is a new pre-$a_j$-clique of size at least two, where
$a_j\in [3,k]$,
then there exist coding nodes
$w_j^3,\dots, w_j^{a_j}$  in $T$ such that,
letting $W$ denote
$\bigcup_{j\in\bN}\{ w_j^3,\dots w_j^{a_j}\}$,
the set of all these witnessing coding nodes,
\begin{enumerate}

\item[(a)]
The set of nodes
$\{ w_j^3,\dots ,w_j^{a_j} : j\in\bN\}$ forms a pre-$(a_j-2)$-clique which  witnesses the pre-$a_j$-clique in $S_{l_j,1}$.

\item[(b)]
The nodes in $\{ w_j^3,\dots ,w_j^{a_j}\}$  do not form  pre-cliques  with any nodes in
 $(W\setminus \{ w_j^3,\dots ,w_j^{a_j} \})\cup (S\re |w^{a_j}_j|\setminus S_{l_j,1}')$
 where $S_{l_j,1}'$ denotes the set  of nodes in  $S\re  |w^{a_j}_j|$ which  end-extend $S_{l_j,1}$.

\item[(c)]
If  $Z\sse \{ w_j^3,\dots ,w_j^{a_j}  \}\cup (S\re |w^{a_j}_j|)$ forms a pre-clique, then
$Z\re l_j\cap S$ must be contained in $S_{l_j,1}$.

\end{enumerate}
Recalling that
$w^{\wedge}$ denotes $w\re l$, where $l$ is least such that $w(l)\ne 0$, we have
\begin{enumerate}
\item[(d)]
$|d^S_{m_n-1}|<|(w_j^3)^{\wedge}|<\dots<|(w_j^{a_j})^{\wedge}|<
|w_j^3|<\dots<|w_j^{a_j}|$.
\item[(e)]
If $|d^S_{m_n-1}|<l_{j+1}< l^S_n$,
then
$\max(l_{j},|w^{a_j}_j|)<
|(w_{j+1}^3)^{\wedge}|$.
\end{enumerate}
\end{enumerate}
\end{defn}

For $k=3$, in the terminology of \cite{DobrinenJML20}, (c) says that the only nodes in $S$ with which
$\{w_j^3,\dots, w_j^{a_j} \}$ has parallel $1$'s  (pre-$3$-cliques) are  in $S_{l_j,1}$.

In what follows, we shall say that a strong coding tree $S$ such that $S\le T$ is {\em valid} in $T$ if for each $m\in\bN$, $r_m(S)$ is valid in $T$.
Since $S$ is a strong coding tree, this is equivalent to   $\max(r_m(S))$  being free in $T$ for each $m\in\bN$.

\begin{lem}\label{lem.squiggletree}
Let $T\in\mathcal{T}_k$ be a strong coding tree.
Then there is an incremental strong coding tree $S\le T$ and a set of coding nodes $W\sse T$ such that  each
 new pre-clique in $S$ is incrementally
 witnessed in $T$ by  coding nodes in $W$.
\end{lem}

\begin{proof}
Recall that for any tree $T\in\mathcal{T}_k$,
the sequence
$\lgl m_n:n\in\bN\rgl$ denotes the indices
such that $d^T_{m_n}=c^T_n$; that is,  the $m_n$-th critical node in $T$ is the $n$-th coding node in $T$.

If $k=3$, fix some  $S_0\in r_{m_0+1}[0,T]$ which is
 valid  in $T$.
Then  $S_0$ has exactly one coding node,   $c^{S_0}_0$, and it has ghost coding node $c^{S_0}_{-1}$, which is the shortest splitting node in $S_0$.
There are no pre-cliques in $S_0$.

For $k=3$ and $n\ge 1$, or for $k\ge 4$ and $n\ge 0$,  proceed as follows:
If $k\ge 4$, let $S_{-1}$ consist of  the  stem of  $T$, that is $d^T_0$.
Suppose  that   we have chosen $S_{n-1}\in r_{m_{n-1}+1}[0,T]$
 valid in $T$  and $W_{n-1}\sse T$ so that $S_{n-1}$ is incremental and
 each new pre-clique in $S_{n-1}$ is
incrementally witnessed by some coding nodes in $W_{n-1}$.
Take some $U_n\in r_{m_n+1}[S_{n-1},T]$  such that  $r_{m_n}(U_n)$ is valid in $T$.
Let   $V=\max(r_{m_n}(U_n))$.


Let $\lgl X_j:j<\tilde{j}\rgl$ enumerate  those subsets of $\max(U_n)$ which have new
pre-cliques over $r_{m_n}(U_n)$ so that
for each pair $j<j'<\tilde{j}$,
if $X_j$ is a new pre-$a$-clique
and $X_{j'}$ is a new pre-$a'$-clique, 
then
\begin{enumerate}
\item
 $a\le a'$; 
\item
If $a=a'$, then  $X_j\not\contains X_{j'}$.
\end{enumerate}
Note that (1) implies that, in the case that $k\ge 4$,  for each $a\in [3,k-1]$,  all new pre-$a$-cliques are enumerated before any new pre-$(a+1)$-clique is enumerated.
Furthermore,  every new pre-clique in $\max(U_n)$ over  $r_{k_n}(U_n)$ is enumerated in $\lgl X_j:j<\tilde{j}\rgl$  whether or not it is maximal.
By (2), all new pre-$a$-cliques  composed of two nodes  are listed before  any new pre-$a$-clique consisting of three nodes, etc.
For each $j<\tilde{j}$,
let $Y_j=X_j\re (l_V+1)$.

By properties (1) and (2), $X_0$ must be a pre-$3$-clique consisting of two nodes.
The construction process in this case is similar to  the construction above for  $S_0$ when  $k>3$.
By Lemma \ref{lem.perfect},
 there is a splitting node $s\in T$ such that $s$ is a sequence of $0$'s and $|s|> l_V+1$.
Extend all nodes in $V$ leftmost in $T$ to the length  $|s|$, and call this set of nodes $Z$.
Apply Lemma \ref{lem.pnc} to obtain $V_0$ end-extending $Z$ so that the following hold:
The
 node in $V_0$ extending $s^{\frown}1$ is a coding node, call it $w_{n,0}$;
 the two nodes in $V_0$ extending the nodes in $Y_0$ both have passing number $1$ at $w_{n,0}$;
all other nodes in $V_0$ are leftmost extensions of the nodes in $V^+\setminus Y_0$;
and the only new pre-clique in $V_0$ is the nodes in $V_0$ extending $Y_0$.
Let $W_{n,0}=\{w_{n,0}\}$.

 Given $j<\tilde{j}-1$ and $V_j$,
 let $Y'_{j+1}$ be the set of those nodes in $V_j$ which  extend the nodes in $Y_{j+1}$.
 Let $a\in[3,k]$ be such that $X_{j+1}$ is a new pre-$a$-clique.
Applying  Lemma \ref{lem.perfect} $a-2$ times,
obtain
splitting nodes $s_i$, $i<a-2$, in $T$ which are sequences of $0$ such that
 $l_{V_j}<|s_0|<\dots<|s_{a-3}|$.
 Extend all nodes ${s_i}^{\frown}1$, $i<a-2$,
 leftmost in $T$ to length $|s_{a-3}|+1$;
 and extend
  the nodes in $V_j$  leftmost in $T$ to length $|s_{a-3}|+1$ and denote this set of nodes as $Z$.
 By Lemma \ref{lem.poc}, this adds no new pre-cliques over $V_j$.
 Next apply  Lemma \ref{lem.pnc} $a-2$  times to obtain $V_{j+1}$ end-extending $Z$ and coding nodes $w_{n,j+1,i}\in T$,   $i<a-2$,
 such that
 letting  $Y''_{j+1}$ be those nodes in $V_{j+1}$ extending nodes in $Y'_{j+1}$,
 the following hold:
\begin{enumerate}
 \item
 $|w_{n,j+1,0}|<\dots<|w_{n,j+1,a-3}|$;
 \item
 The nodes in $V_{j+1}$ all have length
 $|w_{n,j+1,a-3}|$;
 \item
For each   $i<a-2$,
 all nodes in
 $\{w_{n,j+1,i'}:i<i'<a-2\}
 \cup Y''_{j+1}$ have passing number $1$ at
 $w_{n,j+1,i}$.
 \item
 All nodes in $V_{j+1}\setminus Y''_{j+1}$ are leftmost extensions of nodes in $V_j\setminus Y'_{j+1}$.
 \item
 The only new pre-clique
 in $V_{j+1}$ above $V^+$ is the set of nodes in
 $Y''_{j+1}$.
\end{enumerate}
Let $W_{n,j+1}=\{w_{n,j+1,i}:i<a-2\}$.

After $V_{\tilde{j}-1}$ has been constructed,
take some $S_n\in r_{m_n+1}[r_{m_n}(U_n),T]$ such that $\max(S_n)$ end-extends $V_{\tilde{j}-1}$, by Lemma  \ref{lem.HLconstruction}.
Let $W_n=\bigcup_{j<\tilde{j}}W_{n,j}$.

To finish, let $S=\bigcup_{n\in\bN} S_n$ and $W=\bigcup_{n\in\bN}W_n$.
Then $S\le T$, $S$ is incremental, and
the pre-cliques  in $S$ are incrementally witnessed  by coding nodes in $W$.
\end{proof}


\subsection{Ramsey theorem for  strict similarity types}\label{sec.1color}

The main Ramsey theorem   for strong coding trees
is Theorem \ref{thm.mainRamsey}:
Given a  finite coloring of all strictly similar copies (Definition \ref{defn.ssimtype})
 of a fixed  finite  antichain in an incremental strong coding tree,
 there is a subtree which is again a strong coding tree in which all strictly similar copies of the antichain have the same color.
Such antichains will have envelopes which  have the Strict Witnessing Property.
Moreover,  all  envelopes of a fixed incremental antichain of coding nodes will be  strongly isomorphic to each other.
This will allow for an application of
Theorem \ref{thm.MillikenSWP}  to obtain
 the same color for all copies of a given  envelope, in some subtree in $\mathcal{T}_k$.
From this, we will deduce Theorem \ref{thm.mainRamsey}.

Recall that a set of nodes $A$ is an {\em antichain} if no node in $A$ extends any other node in $A$.
In what follows, by   {\em antichain}, we mean  an antichain of coding nodes.
If $Z$ is an antichain,
then   the {\em tree induced by $Z$}
is the set of nodes
\begin{equation}
\{z\re |u|:z\in Z\mathrm{\ and\ } u\in Z^{\wedge}\}.
\end{equation}
We say that an antichain satisfies the  Witnessing Property (Strict Witnessing Property) if and only if the tree it induces satisfies the Witnessing Property (Strict Witnessing Property).

Fix, for the rest of this section,  an incremental strong coding tree
$T\in\mathcal{T}_k$, as in
Lemma \ref{lem.squiggletree}.
Notice that any strong coding subtree of $T$ will also be incremental.
Furthermore, any antichain in $T$ must be incremental.

\begin{defn}[Strict similarity type]\label{defn.ssimtype}
Suppose  $Z\sse T$ is a finite   antichain   of coding nodes.
Enumerate the nodes of $Z$ in  increasing order of length as $\lgl z_i:i<\tilde{i}\rgl$ (excluding, as usual, new singleton pre-$3$-cliques).
Enumerate all nodes in  $Z^{\wedge}$  as $\lgl u^Z_m:m< \tilde{m}\rgl$ in order of increasing length.
Thus, each $u^Z_m$ is either a splitting node in $Z^{\wedge}$ or else  a coding node  in $Z$.
 List the minimal levels of new pre-cliques  in $Z$
 in increasing order as $\lgl l_j: j<\tilde{j}\rgl$.
For each $j<\tilde{j}$, let $I^Z_{l_j}$ denote the set of those
 $i<\tilde{i}$ such that
$\{z_i\re l_j :i\in I^Z_{l_j} \}$ is the new pre-clique in $Z\re l_j$.
The sequence
\begin{equation}
\lgl  \lgl l_j:j<\tilde{j}\rgl,
\lgl I^Z_{l_j}:j< \tilde{j}\rgl,
\lgl |u^Z_m|:m<\tilde{m}\rgl\rgl
\end{equation}
is  {\em  the strict similarity  sequence of $Z$}.

Let $Y$ be another finite  antichain in  $T$, and
let
\begin{equation}
\lgl
\lgl p_j:j<\tilde{k}\rgl,
\lgl I^Y_{p_j}:j< \tilde{k}\rgl,
\lgl |u^Y_m|:m<\tilde{q}\rgl\rgl
\end{equation}
be its  strict similarity  sequence.
We say that $Y$ and $Z$ have the same {\em strict similarity type} or are {\em strictly similar},  and write $Y\sssim Z$, if
\begin{enumerate}
\item
The tree induced by
$Y$  is strongly isomorphic to the tree induced by $Z$, so in particular, $\tilde{m}=\tilde{q}$;
\item
$\tilde{j}=\tilde{k}$;
\item
For each $j<\tilde{j}$, $I^Y_{p_j}=I^Z_{l_j}$;
and
\item
The function $\varphi:\{p_{j}:j<\tilde{j}\}\cup\{|u^Y_m|:m<\tilde{m}\}\ra \{l_{j}:j<\tilde{j}\}\cup\{|u^Z_m|:m<\tilde{m}\}$,
defined by $\varphi(p_{j})=l_{j}$ and
$\varphi(u^Y_m)=u^Z_m$,
is an order preserving bijection between these two linearly ordered sets of natural numbers.
\end{enumerate}
Define
\begin{equation}
\Sim^{ss}_T(Z)=\{Y\sse T:Y\sssim Z\}.
\end{equation}
\end{defn}

Note that  if  $Y\sssim Z$, then
 the map $f:Y\ra Z$ defined  by $f(y_i)=z_i$, for each $i<\tilde{i}$, induces the  strong similarity map from the tree induced by $Y$ onto  the tree induced by $Z$.
Then  $f(u^Y_m)=u^Z_m$, for each $m<\tilde{m}$.
Further,
by (3) and  (4) of Definition \ref{defn.ssimtype},
 this map preserves the order in which  new pre-cliques appear, relative to all  other  new pre-cliques in $Y$ and $Z$  and the nodes in $Y^{\wedge}$ and $Z^{\wedge}$.



The following notion of envelope is defined in terms of structure without regard to an ambient strong coding tree.
Given a fixed  incremental strong coding tree $T$,
in any given strong coding subtree  $U\le T$, there will
 certainly be  finite subtrees  of  $U$ which have no envelope in $U$.
 The point of Lemma \ref{lem.squiggletree} is that
  there will be
a  strong coding subtree  $S\le U$ along with a set of witnessing coding nodes $W\sse U$ so that each finite antichain in  $S$ has an   envelope consisting of nodes from $W$.
Thus, envelopes of antichains in $S$ will exist in $U$.
Moreover, $S$ must be incremental, since $S\le U\le T$.

\begin{defn}[Envelopes]\label{defn.envelope}
Let $Z$ be a finite  incremental antichain of coding nodes.
An envelope of $Z$, denoted $E(Z)$, consists of $Z$ along with a set of coding nodes $W$ such that $Z\cup W$  satisfies Definition \ref{defn.incremental}.
\end{defn}

Thus, all  new pre-cliques in  an envelope
$E(Z)=Z\cup W$ are incrementally witnessed by coding nodes in $W$.
The set $W$ is called the set of {\em witnessing coding nodes} in the envelope.
The next fact follows immediately from the definitions.

\begin{fact}\label{fact.Claim1}
Let
 $Z$ be any antichain  in  an incremental strong coding tree.
Then any envelope of $Z$ has incrementally witnessed pre-cliques, which implies that $Z$ has
 the  Strict Witnessing Property.
\end{fact}


\begin{lem}\label{lem.envelopesbehave}
Let $Y$ and $Z$ be strictly similar  incremental antichains of coding nodes.
Then  any envelope of $Y$ is strongly isomorphic  to any envelope of $Z$, and both envelopes have the Strict Witnessing Property.
\end{lem}

\begin{proof}
Let $Y=\{y_i:i<\tilde{i}\}$ and $Z=\{z_i:i<\tilde{i}\}$ be the enumerations of $Y$ and $Z$ in order of increasing length,
and  let
\begin{equation}
\lgl \lgl l_j:j<\tilde{j}\rgl,\lgl I^Y_{l_j}:j<\tilde{j}\rgl, \lgl |u^Y_m|:m<\tilde{m}\rgl\rgl
\end{equation}
and
\begin{equation}
\lgl \lgl p_j:j<\tilde{j}\rgl,\lgl I^Z_{p_j}:j<\tilde{j}\rgl, \lgl |u^Z_m|:m<\tilde{m}\rgl\rgl
\end{equation}
be their strict similarity sequences, respectively.
Let $E=Y\cup V$ and $F=Z\cup W$ be any envelopes of $Y$ and $Z$, respectively.
For each $j<\tilde{j}$, let $a_j\ge 3$ be such that
$I_{l_j}^Y$ is a new pre-$a_j$-clique.
Then the members of  $V$ may be  labeled as
$\{v_j^3,\dots,v_j^{a_j}:j<\tilde{j}\}$ with the property that  for each $j<\tilde{j}$,
 given  the least $m<\tilde{m}$ such that
 $|v_j^{a_j}|<|u^Y_m|$,  we have
 $|u^Y_{m-1}|<|(v_j^3)^{\wedge}|$.
This  follows from Definition \ref{defn.incremental}.
Since $Y$ and $Z$ have the same strict similarity type,
it follows that  for each $j<\tilde{j}$, $I_{p_j}^Z$ is also a new pre-$a_j$-clique.
Furthermore,  $W=\{w_j^3,\dots,w_j^{a_j}:j<\tilde{j}\}$,
where for each $j<\tilde{j}$,
 given  the least $m<\tilde{m}$ such that
 $|w_j^{a_j}|<|u^Z_m|$,  we have that  $|u^Z_{m-1}|<|(w_j^3)^{\wedge}|$.
Thus, $V$ and $W$ both have the same size, label it $J$.

Let $\tilde{n}=\tilde{i}+J$,
and let $\{e_n:n<\tilde{n}\}$ and $\{f_n:n<\tilde{n}\}$ be the enumerations of $E$ and $F$ in order of increasing length, respectively.
For each $j<\tilde{j}$,
let $n_j$ be the index in $\tilde{n}$ such that
$e_{n_j}=v_j$ and $f_{n_j}=w_j$.
For  $n<\tilde{n}$,
let $E(n)$ denote
the tree induced by $E$ restricted to those nodes of length less than or equal to $|e_n|$; precisely,
$E(n)=\{e\re |t|: e\in E$, $t\in E^{\wedge}$, and $ |t|\le |e_n|\}$.
Define $F(n)$ similarly.

We prove  that $E\cong F$ by induction on $\tilde{j}$.
If $\tilde{j}=0$, then $E=Y$ and $F=Z$, so $E\cong F$ follows from $Y\sssim Z$.
Suppose now that $\tilde{j}\ge 1$ and that, letting $j=\tilde{j}-1$,
 the induction hypothesis gives that
$E(n)\cong F(n)$ for the maximal $n<\tilde{n}$ such that
$e_n\in Y^{\wedge}$ and $|e_n|<l_j$.
Let $m$ be the least integer below $\tilde{m}$ such that
$|u^Y_{m}|>l_j$.
Then $e_n=u^Y_{m-1}$ and
the only  nodes in $E^{\wedge}$
in the interval $(|u^Y_{m-1}|,|u^Y_{m}|)$
are
$(v^3_j)^{\wedge},\dots, (v^{a_j}_j)^{\wedge}, v^3_j,\dots, v^{a_j}_j$.
Likewise, the only nodes in $F^{\wedge}$ in the interval
$(|u^Z_{m-1}|,|u^Z_{m}|)$
are $(w^3_j)^{\wedge},\dots, (w^{a_j}_j)^{\wedge}, w^3_j,\dots, w^{a_j}_j$.

By the induction hypothesis, there is a strong isomorphism  $g:E(n)\ra F (n)$.
Extend it
to a strong isomorphism
$g^*:E(n')\ra F(n')$, where $n'=\tilde{n}-1$
as follows:
Define
$g^*=g$ on $E(n)$.
For each $i\in [3,a_j]$, let $g^*((v^i_j)^{\wedge})=
(w^i_j)^{\wedge}$ and $g^*(v^i_j)=
w^i_j$.
Recall that the nodes $\{v^3_j,\dots v^{a_j}_j\}$  form a pre-$(a_j-1)$-clique and only have mutual pre-cliques with nodes in $\{y_i:i\in I^Y_{l_j}\}$, witnessing this set, and no other members of $E$.
Likewise,  for
$\{w^3_j,\dots v^{a_j}_j\}$   and  $\{z_i:i\in I^Z_{lp_j}\}$.
Thus, $g^*$ from $E(n'')$ to $F(n'')$ is a strict similarity map, where $n''<\tilde{n}$ is the index such that $v^{a_j}_j=e_{n''}$.
If $n''<\tilde{n}-1$,
then  $\{e_q:n''<q<\tilde{n}\}\sse Y^{\wedge}$
and
$\{f_q:n''<q<\tilde{n}\}\sse Z^{\wedge}$.
Since these sets have no new pre-cliques and are strictly similar, the map $g^*(e_q)=f_q$, $n''<q<\tilde{n}$,
is a strong isomorphism.
Thus, we have constructed a strong isomorphism $g^*:E\ra F$.
It follows from the definitions that envelopes satisfy the Strict Witnessing Property.
\end{proof}

\begin{lem}\label{lem.lastpiece}
Suppose
 $Z$ is a finite antichain  of coding nodes and $E$ is   an envelope of $Z$ in $T$.
Enumerate the nodes
in $Z$ and $E$ in
order of increasing length as
$\lgl z_i:i<\tilde{i}\rgl$ and
 $\lgl e_k:k<\tilde{k}\rgl$,
respectively.
Given   any $F\sse T$ with   $F\cong E$,
let $F\re Z:=\{f_{k_i}:i<\tilde{i}\}$, where
$\lgl f_k:k<\tilde{k}\rgl$
enumerates the nodes in $F$ in order of increasing length
 and for each $i<\tilde{i}$,
$k_i$ is the index such that
 $e_{k_i}=z_i$.
Then
$F\re Z$
is strictly similar to $Z$.
\end{lem}

\begin{proof}
Recall that $E$ has incrementally witnessed new pre-cliques  and $F\cong E$  implies that   $F$ also has this  property, and hence has the SWP.
Let   $\iota_{Z,F}:Z\ra F$ be the injective map defined via $\iota_{Z,F}(z_i)=f_{k_i}$, $i<\tilde{i}$,
and let
 $F\re Z$ denote  $\{f_{k_i}:i<\tilde{i}\}$, the image of $\iota_{Z,F}$.
Then $F\re Z$ is a subset of $F$ which we claim is strictly similar to $Z$.

Since $F$ and $E$ each have incrementally witnessed new pre-cliques,
 the  strong similarity map $g:E\ra F$
satisfies that for each $j<\tilde{k}$, the  indices of the new pre-cliques  at level of the $j$-th coding node are the same:
\begin{equation}
\{k<\tilde{k}:e_k(|e_j|)=1\}=
\{k<\tilde{k}: g(e_k)(|g(e_j)|)=1\}
=\{k<\tilde{k}:f_k(|f_j|)=1\}.
\end{equation}
Since $\iota_{Z,F}$ is the restriction of $g$ to $Z$,
$\iota_{Z,F}$ also
takes each
  new pre-clique  in $Z$
 to the corresponding  new pre-clique  in  $F\re Z$, with the same set of indices.
Thus, $\iota_{Z,F}$ witnesses that $F\re Z$ is
strictly similar to  $Z$.
\end{proof}




\begin{thm}[Ramsey Theorem for Strict Similarity Types]\label{thm.mainRamsey}
Let $Z$ be a finite antichain of coding nodes in   an incremental  strong coding tree $T$,
and suppose  $h$  colors  all subsets of $T$ which are strictly similar to $Z$ into finitely many colors.
Then there is an incremental strong coding tree $S\le T$ such that
all subsets of $S$ strictly similar to $Z$ have the same $h$ color.
\end{thm}

\begin{proof}
First, note that there is an envelope $E$
of a copy of $Z$ in $T$:
By  Lemma \ref{lem.squiggletree},
 there is an incremental strong coding tree $U\le T$ and a set of coding nodes $V\sse T$
such that each  $Y\sse U$  which is strictly similar to $Z$
has an envelope  in $T$ by adding nodes from $V$.
Since $U$ is strongly isomorphic to $T$,
there is  subset $Y$ of $U$ which is  strictly similar  to  $Z$.
Let
$E$ be any envelope of $Y$ in $T$, using witnessing coding nodes from $V$.

By Lemma \ref{lem.envelopesbehave}, all envelopes of copies of $Z$ are strongly isomorphic and have the SWP.
For each $F\cong E$,
define
$h^*(F)=h(F\re Z)$,
where
 $F\re Z$ is  the  subset of $F$ provided by Lemma \ref{lem.lastpiece}.
The set $F\re Z$
is strictly similar to $Z$,
so  the coloring $h^*$ is well-defined.
By
Theorem \ref{thm.MillikenSWP}, there is
a strong coding tree $T'\le T$ such that
$h^*$ is monochromatic  on all  strongly isomorphic copies of $E$ in $T'$.
 Lemma \ref{lem.squiggletree} implies there is an incremental strong coding tree $S\le T'$ and a set of coding nodes $W\sse T'$
such that each  $Y\sse S$  which is strictly similar to $Z$
has an envelope $F$ in $T'$, so that
 $h(Y)=h^*(F)$.
Therefore, $h$ takes only one color on
all strictly similar copies of $Z$ in $S$.
\end{proof}


\section{The  Henson graphs  have  finite big Ramsey degrees}\label{sec.7}

From the results in previous sections, we now prove the
 main theorem of this paper, Theorem \ref{finalthm}.
This result follows from
Ramsey Theorem \ref{thm.mainRamsey} for strict similarity types
  along with Lemma \ref{lem.bD} below.

For a strong coding tree $T$,  let $(T,\sse)$
be the reduct of $(T,\mathbb{N};\sse,<,c)$.
Then  $(T,\sse)$
is simply the tree structure of $T$, disregarding the difference between coding nodes and non-coding nodes.
We say that two trees $(T,\sse)$ and $(S,\sse)$ are {\em strongly similar trees}
if they
 satisfy
Definition 3.1 in  \cite{Sauer06}.
This is the  same as modifying Definition \ref{def.3.1.likeSauer} by deleting
 (6) and changing (7) to apply to passing numbers of  {\em all} nodes in the trees.
By saying that two finite trees are strongly similar trees, we are implicitly  assuming that
their extensions  to their immediate successors of their
 maximal nodes  are still strongly similar.
Thus, strong similarity of finite trees implies passing numbers of  their immediate extensions are preserved.
Given an antichain $D$ of coding nodes from a strong coding tree, let $L_D$ denote the set of all
lengths of nodes $t\in D^{\wedge}$
such that
 $t$ is not the splitting predecessor of any coding node in $D$.
Define
\begin{equation}\label{eq.D^*}
D^*=\bigcup\{t  \re l:t\in D^{\wedge}\setminus D
\mathrm{\ and\ }l\in L_D\}.
\end{equation}
Then $(D^*,\sse)$ is a tree.

\begin{lem}\label{lem.bD}
Let $T\in\mathcal{T}_k$ be a strong coding tree.
Then there is an infinite  antichain of coding nodes  $D\sse T$  which code
 $\mathcal{H}_k$ in  the same way as $\bT_k$:
$c^{D}_n(l^{D}_i)=c^k_n(l^k_i)$,
for all $i<n$.
Moreover,
$(D^*,\sse)$  and $(\bT_k,\sse)$ are strongly similar as  trees.
\end{lem}

\begin{proof}
We will  construct a subtree
 $\bD\sse\bT_k$
such that $D$
the set of coding nodes in $\bD$  form
 an antichain
satisfying the lemma.
Then, since $T\in\mathcal{T}_k$  implies $T\cong \bT_k$,
letting $\varphi:\bT_k\ra T$ be the strong similarity map  between $\bT_k$ and $T$,
the image of $\varphi$ on the coding nodes of  $\bD$ will yield an antichain of coding nodes $D\sse T$ satisfying the lemma.

We will construct $\bD$ so that  for each $n$,
the node of length $l^{\bD}_{n}$ which  is going to be extended to the next  coding node $c^{\bD}_{n+1}$ will  extend to a splitting node in $\bD$
of length smaller than that of any other splitting node in the $(n+1)$-st interval of $\bD$.
Above that  splitting node, the splitting will be regular in the interval until  the next coding node.
Recall that  for each $i\in\mathbb{N}$,
 $\bT_k$
 has either a coding node or  else a splitting node of length $i$.
To avoid some superscripts, let
 $l_n=|c^k_n|$
and $p_n=|c^{\bD}_n|$.
Let
$j_n$  be the index such that $c^{\bD}_n=c^k_{j_n}$, so that
 $p_n$  equals $l_{j_n}$.
The set of  nodes in  $\bD\setminus\{c^{\bD}_n\}$ of length  $p_n$
 shall be indexed  as $\{d_t:t\in \bT_k\cap\Seq_{l_n}\}$.
Recall that $m_n$ is the index such that the $m_n$-th critical node $d^k_{m_n}$ of $\bT_k$ is the $n$-th coding node $c^k_n$ of $\bT_k$.

We define inductively  on $n\ge -1$ finite  trees with coding nodes,
 $r_{m_n+1}(\bD):=  \bD\cap\Seq_{ \le p_{n}}$,
 and
strong similarity maps  of the trees
$\varphi :
r_{m_n+1}(\bT_k)\ra r_{m_n+1}(\bD^*)$,
where $|\varphi(c^k_n)|=p_n$.
Recall that the node $\lgl\rgl$ is the ghost coding node $c^k_{-1}$ in $\bT_k$.
Define  $d_{\lgl\rgl}=\varphi(\lgl\rgl)=\lgl\rgl$.
The node $\lgl\rgl$ splits in $\bT_k$,
so the node $d_{\lgl\rgl}$ will split in $\bD$.
Suppose  that $n\in \mathbb{N}$ and we have constructed
$r_{m_{n-1}+1}(\bD)$
 satisfying the lemma.
By the induction hypothesis,
there is a strong similarity map  of the trees
$\varphi :
r_{m_{n-1}+1}(\bT_k)\ra r_{m_{n-1}+1}(\bD^*)$.
For  $t\in  \bT_k(m_{n-1})$, let
$d_t$ denote $\varphi(t)$.

Let $s$ denote the  node  in $\bT_k  (m_{n-1})$  which  extends to the coding node $c^k_n$.
Let $v_s$ be a splitting node in $\bT_k$ extending $d_s$.
Let $u_s={v_s}^{\frown}1$
and extend all nodes $d_t$,
$t\in \bT_k(m_{n-1})
\setminus\{s\}$,
 leftmost to length $|u_s|$ and label these $d_t'$.
Extend ${v_s}^{\frown}0$ leftmost to length $|u_s|$ and label it $d'_s$.
Let $X=\{d'_t:t\in \bT_k(m_{n-1})\}\setminus\{u_s\}$
 and let
 \begin{equation}
 \Spl(n)=\{t\in \bT_k(m_{n-1}): t\mathrm{\ extends\ to\ a\ splitting\ node\ in\ the\ } n\mathrm{-th\ interval\ of\ } \bT_k\}.
 \end{equation}

Apply Lemma \ref{lem.facts} to obtain a coding node $c^{\bD}_n$ extending $u_s$
and nodes $d_w$, $w\in \bT_k(m_n)$,
so that,
letting  $p_n=|c^{\bD}_n|$
and
\begin{equation}
\bD(m_n)=
\{d_t:t\in \bT_k(m_n)\}\cup\{c^{\bD}_n\},
\end{equation}
and
for $m\in (m_{n-1},m_n)$, defining $\bD(m)=\{ d_t\re |s_m|:t\in \bD(m_n)\}$,
where $s_m$ is the $m$-th splitting node in $\bD(m_n)^{\wedge}$,
 the following hold:
$r_{m_n+1}(\bD)$ satisfies the Witnessing Property
and
$r_{m_n+1}(\bD^*)$ is strongly similar  as  a tree
to
$r_{m_n+1}(\bT_k)$.
Thus, the coding nodes in $r_{m_n+1}(\bD)$
code exactly the same graph as the coding nodes in $r_{m_n+1}(\bT_k)$.

Let $\bD=\bigcup_{m\in\mathbb{N}}\bD(m)$.
Then the set of coding nodes in  $\bD$  forms an antichain of maximal nodes in $\bD$.
Further,
the tree generated by the
 the meet closure of the set $\{c^{\bD}_n:n<\om\}$
is exactly  $\bD$,
and
 $\bD^*$  and  $\bT_k$
are strongly similar as trees.
By the construction,
for each pair $i<n<\om$, $c^{\bD}_n(p_i)=c^k_n(l_i)$;
 hence they code $\mathcal{H}_k$ in the same order.

To finish,  let $f_T$ be the strong isomorphism  from $\bT_k$ to $T$.
Letting $D$ be the $f_T$-image of $\{c^{\bD}_n:n<\om\}$,
we  see that $D$ is an antichain  of coding nodes in $T$
such that $D^*$ and $\bD^*$ are strongly similar trees,
and hence $D^*$ is strongly similar as  a tree to $\bT_k$.
Thus, the antichain of coding nodes $D$
 codes $\mathcal{H}_k$ and satisfies  the lemma.
\end{proof}

Recall that the Henson graph $\mathcal{H}_k$ is, up to isomorphism, the homogeneous $k$-clique-free graph on countably many vertices which is universal for all $k$-clique-free graphs on countably many vertices.

\begin{mainthm}\label{finalthm}
For each $k\ge 3$,
the Henson graph $\mathcal{H}_k$  has finite big Ramsey degrees.
\end{mainthm}

\begin{proof}
Fix $k\ge 3$ and let $\G$ be a finite $K_k$-free graph.
Suppose $f$   colors of all the copies of $\G$ in
$\mathcal{H}_k$ into finitely many colors.
By Example
 \ref{ex.bTp},
there is a
 strong coding tree $\bT_k$  such that the coding nodes in $\bT_k$ code  a  $\mathcal{H}_k$.
Let $\mathcal{A}$ denote the set of all
antichains of coding nodes
 of $\bT_k$ which code a copy of $\G$.
For each $Y\in\mathcal{A}$,
let
$h(Y)=f(\G')$,
where
$\G'$ is the copy of $\G$
coded by the coding nodes in $Y$.
Then $h$  is a finite coloring on $\mathcal{A}$.

Let $n(\G)$ be the number of different strict similarity types
of incremental antichains of coding nodes in  of $\bT_k$ coding $\G$,
and let $\{Z_i:i<n(\G)\}$ be a set of one representative from  each of  these   strict similarity types.
Successively
apply Theorem \ref{thm.mainRamsey}
  to obtain incremental strong coding trees $\bT_k\ge T_0\ge\dots\ge T_{n(\G)-1}$ so that for each $i<n(\G)$,
$h$ is takes only one color on
 the set of
 incremental antichains of coding nodes $A\sse T_i$ such that $A$ is strictly similar to $Z_i$.
Let $S=T_{n(\G)-1}$.

By Lemma \ref{lem.bD}
there is an antichain of coding nodes $D\sse S$  which  codes $\mathcal{H}_k$ in the same way as $\bT_k$.
Every set of coding nodes in $D$ coding $\G$ is automatically  incremental,  since $S$ is incremental.
Therefore, every copy  of $\G$ in the copy of $\mathcal{H}_k$ coded by the coding nodes in $\D$
is coded by an incremental  antichain of coding nodes.
Thus, the number of
 strict similarity types of  incremental  antichains in $\bT_k$  coding $\G$
provides an upper bound for
the big Ramsey degree of $\G$ in $\mathcal{H}_k$.
\end{proof}

Thus, each Henson graph has finite big Ramsey degrees.
Moreover,  given a finite $k$-clique-free graph,
the big Ramsey degree  $T(\mathrm{G},\mathcal{H}_k)$
is bounded by the number of strict similarity types of incremental antichains coding copies of $\mathrm{G}$ in $\bT_k$.



\section{Future Directions}

This article developed a unified approach to proving upper bounds for big Ramsey degrees of all  Henson graphs.
The main phases of the proof were as follows:
I. Find the correct structures to code $\mathcal{H}_k$ and prove Extension Lemmas.
II. Prove an analogue of Milliken's Theorem for finite trees with certain structure.
In the case of the Henson graphs, this is the Strict Witnessing Property.
III.  Find a means for turning finite antichains into finite trees with the Strict Witnessing Property so as to deduce a Ramsey Theorem for finite antichains from the previous Milliken-style theorem.
This general approach should apply to a large class of ultrahomogeneous structures with forbidden configurations.
It will be interesting to see where the dividing line is between those structures for which this methodology works and those for which it does not.
The author conjectures that similar approaches will work for forbidden configurations which are irreducible in the sense of \cite{Nesetril/Rodl77} and \cite{Nesetril/Rodl83}.

Although we have not yet proved the lower bounds
to obtain the precise
 big Ramsey degrees  $T(\G,\mathcal{K}_k)$, we conjecture that they will be exactly the number of strict similarity
 types of
incremental
antichains  coding $\G$.
We further conjecture that once found, the lower bounds will satisfy the conditions   needed for  Zucker's work in \cite{Zucker19}.
If  so, then each Henson graph would admit a big Ramsey structure and  any big Ramsey flow will be a universal completion flow, and any two universal completion flows will be universal.
We mention some bounds  on big Ramsey degrees, found by computing the number of distinct strict similarity types of incremental antichains of coding nodes.
Let $\bar{K}_2$ denote  two vertices with no edge between them. 
$T(K_2,\mathcal{H}_3)=2$ was proved by Sauer in \cite{Sauer98}.  
This is in fact the number of strict similarity types of two coding nodes coding an edge in $\bT_3$.
The number of strict similarity types of incremental antichains coding a non-edge in $\bT_3$ is seven, so 
\begin{equation}
T(\bar{K}_2,\mathcal{H}_3) \le 7.
\end{equation}
The number of strict similarity types of incremental antichains coding an edge in $\bT_4$ is 44, so 
\begin{equation}
T(K_2,\mathcal{H}_4)\le 44.
\end{equation}
The number of strict similarity types of incremental antichains coding a non-edge in $\bT_4$ is quite large. 
These numbers  grow quickly as more pre-cliques are allowed.  
Since any copy of $\mathcal{H}_k$ can be enumerated and coded by a strong coding tree, which wlog can be assumed to be incremental, it seems that these strict similarity types should persist.

We point out that by a compactness argument,
one can obtain finite versions of the  two main Ramsey theorems  in this article.
In particular, the finite version of Theorem \ref{thm.mainRamsey} may well produce better  bounds for the sizes of finite $K_k$-free graphs instantiating that  the \Fraisse\  class  $\mathcal{G}_k^{<}$
has the Ramsey property.

Curiously, the methodology in this paper and \cite{DobrinenJML20}  is also having impact on \Fraisse\ structures without forbidden configurations.
In  \cite{DobrinenRado19}, the author recently  developed trees with coding nodes to code copies of the Rado graph
and used forcing  arguments similar to, but much simpler than, those  in Section \ref{sec.5}
 to answer
a question of \cite{Kechris/Pestov/Todorcevic05} regarding infinite dimensional Ramsey theory of copies of the Rado graph.
These methods work also for the rationals, and ongoing work is to discover all aspects of Ramsey and anti-Ramsey theorems for colorings of definable sets of spaces of such \Fraisse\ structures, the aim being to extend theorems of Galvin-Prikry in \cite{Galvin/Prikry73} and Ellentuck in \cite{Ellentuck74} to a wide collection of \Fraisse\ classes.
Lastly,
modifications and generalizations of this approach  seem likely to
produce   a  general theorem
for  big Ramsey degrees for a large collection of relational \Fraisse\ structures  without forbidden configurations.


\subsection{Updates: May, 2022}

Since  the posting of the  first version of this  paper on the arxiv in January 2019,  much progress has been made on big Ramsey degrees and related Ramsey theory of infinite structures. 
The author's paper \cite{Dobrinen_ICM}, which will appear in the Proceedings of the 2022 International Congress of Mathematicians, 
provides an overview of current knowledge  of big Ramsey degrees of countable  homogeneous structures.
Here, we briefly mention   recent work which has been directly influenced by this paper and its predecessor \cite{DobrinenJML20}, which found upper bounds for big Ramsey degress of the triangle-free Henson graph, $\mathcal{H}_3$.

Via a small tweak of the  strong coding trees in \cite{DobrinenJML20},
exact big Ramsey degrees for $\mathcal{H}_3$ were characterized  by the author in \cite{DobrinenH_3ExactDegrees20}.
(The small tweak is an instance of the more general notion  of ``controlled coding levels"--Definition 5.6.5
of  \cite{Balko7}.)
An equivalent characterization
(based on upper bounds in \cite{Hubicka_CS20} and \cite{Zucker20}) was found independently around the same time  by 
Balko,  Chodounsk\'{y}, 
 Hubi\v{c}ka,   Kone\v{c}n\'{y}, Vena, and Zucker in an unposted preprint.
Big Ramsey degrees in $\mathcal{H}_3$ are exactly characterized by diagonal antichains  of coding nodes in the strong coding trees defined in \cite{DobrinenH_3ExactDegrees20}
and the order in which pairs of new
 pre-$3$-cliques  appear.
This is the  part  of the author's notion of strict similarity type  essential to  unavoidable colorings.
In particular, 
the big Ramsey degree $T(\bar{K}_2,\mathcal{H}_3)$ is now known to be exactly 5.

Building on methods in this paper, Zucker  in \cite{Zucker20} 
developed an approach using fully branching  (similar to our Definition 4.1)  ``left-leaning" coding trees, forcing, and a new style of envelope
to  handle
 all free amalgamation classes with finitely many  relations of arity at most two and finitely many forbidden irreducible structures;
 he proved that their 
\Fraisse\ limits have  finite big Ramsey degrees. 
Zucker  crystalized the essential properties  of new pre-cliques  in this paper to  so-called `age-changes' which made possible his broad approach.
The use of fully branching coding trees 
 led  to a shorter proof of upper bounds, though at the cost of  looser  bounds
for Henson graphs  than those in this paper.
By 
combining Zucker's approach in  \cite{Zucker20} with 
 properties inherent in  the strong coding trees and strict similarity types in this paper  as well as   new ideas, 
exact big Ramsey degrees  were characterized  for 
all free amalgamation classes with finitely many  relations of arity at most two and finitely many forbidden irreducible structures in the recent paper \cite{Balko7} by 
Balko,  Chodounsk\'{y}, Dobrinen, 
 Hubi\v{c}ka,   Kone\v{c}n\'{y}, Vena, and Zucker.
In particular, we now know that 
the big Ramsey degree $T(K_2,\mathcal{H}_4)$ is  exactly 17.

Building on ideas from this paper in another direction is the work of Coulson, Dobrinen, and Patel in \cite{CDP21}.
That work distills an amalgamation property true for the rationals and the Rado graph, and false for Henson graphs,  called SDAP,  which  guarantees that the forcing partial order of the sort used in Theorem \ref{thm.matrixHL} here 
 can be simply end-extension.
By working with diagonal coding trees for \Fraisse\ structures satisfying a strengthened version of SDAP, called SDAP$^+$, 
the following results were obtained:
(1) Indivisibility for \Fraisse\ limits  satisfying SDAP$^+$ with finitely many relations of finite arity;
(2) Exact big Ramsey degrees for  \Fraisse\ limits  satisfying SDAP$^+$ with finitely many relations of arity at most two, moreover characterized just by diagonal antichains and passing types.
Some interesting byproducts of the approach in  \cite{CDP21} are that  exact big Ramsey degrees are found directly, with no need for envelopes, and  previous results 
\cite{Laflamme/NVT/Sauer10},  \cite{Laflamme/Sauer/Vuksanovic06},   and  \cite{Sauer06} are recovered by the one method.

We point out that the exact big Ramsey degrees in \cite{Balko7} and in \cite{CDP21} satisfy the conditions of Zucker in  \cite{Zucker19} to admit big Ramsey structures.

A third line of results using coding trees  involve
infinite dimensional Ramsey theory for \Fraisse\ structures. 
The goal is to extend the Galvin-Prikry  \cite{Galvin/Prikry73} or Ellentuck \cite{Ellentuck74} theorem to subspaces of the Baire space in which all points (infinite subsets of $\om$) represent a subcopy of a given \Fraisse\ structure; coding trees help ensure that one is actually taking a subcopy of the given \Fraisse\ structure, not just any infinite  substructure.
An investigation of this area was
 asked for in \cite{Kechris/Pestov/Todorcevic05} and begun in the author's paper \cite{DobrinenRado19}, which obtained a Galvin-Prikry analogue for the Rado graph. 
The author has extended this to  \Fraisse\ structures satisfying SDAP$^+$ with finitely many relations of arity at most two, building on work in \cite{CDP21}; the Galvin-Prikry analogue there recovers exact big Ramsey degrees.
Work in this direction continues.

In 2020, Hubi\v{c}ka found  a new  way 
 to obtain upper bounds for big Ramsey degrees of the generic partial order as well as the triangle-free Henson graph in \cite{Hubicka_CS20}.
His method used 
parameter spaces of words and a coding of pre-$3$-cliques in that  setting,
producing
the first forcing-free proof of upper bounds for $\mathcal{H}_3$. 
This formed the basis for the characterization of exact big Ramsey degrees of the generic partial order in 
\cite{BalkoPO},
by 
Balko,  Chodounsk{\'{y}}, Dobrinen, 
 Hubi{\v{c}}ka,   Kone{\v{c}}n{\'{y}}, Vena,  and Zucker.
Hubi\v{c}ka's  method for upper bounds has been extended to 
upper bounds for 
relational structures in  a finite relational language with relations of arity at most 2 described by forbidden induced cycles
 in \cite{BalkoForbCycles}.
Big Ramsey degrees for structures with relations of  higher arities was initiated in \cite{Hubicka_et4_19} and \cite{Hubicka_et4_19withproofs} for the generic $3$-regular hypergraph, using the product Milliken theorem.
Category-theoretic methods are also currently being employed to discover exact big Ramsey degrees as well as transport existence of upper bounds in the work of Dasilva Barbosa \cite{Barbosa20} and Ma\v{s}ulovi\'{c} in \cite{Masulovic18}.
Work on further types of structures with  an expanding collection  of methods is continuing to be fruitful.

\bibliographystyle{amsplain}
\bibliography{references}

\providecommand{\bysame}{\leavevmode\hbox to3em{\hrulefill}\thinspace}
\providecommand{\MR}{\relax\ifhmode\unskip\space\fi MR }
\providecommand{\MRhref}[2]{%
  \href{http://www.ams.org/mathscinet-getitem?mr=#1}{#2}
}
\providecommand{\href}[2]{#2}
\begin{thebibliography}{10}

\bibitem{Abramson/Harringon78}
F.G. Abramson and L.A. Harrington, \emph{Models without indiscernibles},
  Journal of Symbolic Logic \textbf{43} (1978), no.~3, 572--600.

\bibitem{BalkoPO}
M.~Balko, D.~Chodounsk{\'{y}}, N.~Dobrinen, J.~Hubi{\v{c}}ka,
  M.~Kone{\v{c}}n{\'{y}}, L.~Vena, and A.~Zucker, \emph{Big {R}amsey degrees of
  the generic partial order}, Extended Abstracts EuroComb 2021, Springer, 2021,
  pp.~637--643.

\bibitem{Balko7}
\bysame, \emph{Exact big {R}amsey degrees via coding trees},  (2021), 97 pp,
  Submitted. arXiv:2110.08409.

\bibitem{BalkoForbCycles}
M.~Balko, D.~Chodounsk{\'{y}}, J.~Hubi{\v{c}}ka, M.~Kone{\v{c}}n{\'{y}},
  J.~Ne{\v{s}}et{\v{r}}il, and L.~Vena, \emph{Big {R}amsey degrees and
  forbidden cycles}, Extended Abstracts EuroComb 2021, Springer, 2021,
  pp.~436--441.

\bibitem{Hubicka_et4_19}
M.~Balko, D.~Chodounsk{\'{y}}, J.~Hubi{\v{c}}ka, M.~Kone{\v{c}}n{\'{y}}, and
  L.~Vena, \emph{Big {R}amsey degrees of 3-uniform hypergraphs}, Acta
  Mathematica Universitatis Comenianae \textbf{LXXXVIII} (2019), no.~3,
  415--422.

\bibitem{Hubicka_et4_19withproofs}
\bysame, \emph{Big {R}amsey degrees of 3-uniform hypergraphs are finite},
  Combinatorica (2022), 14 pp, https://doi.org/10.1007/s00493-021-4664-9.

\bibitem{Barbosa20}
Keegan~Dasilva Barbosa, \emph{A categorical notion of precompact expansions},
  Archive for Mathematical Logic (2020), 29 pp, Submitted. arXiv:2002.11751.

\bibitem{CDP21}
Rebecca Coulson, Natasha Dobrinen, and Rehana Patel, \emph{Fra{\"{i}}ss{\'{e}}
  classes with simply characterized big {R}amsey structures},  (2021), 69 pp,
  Submitted. arXiv:2010.02034.

\bibitem{DevlinThesis}
Dennis Devlin, \emph{Some partition theorems for ultrafilters on $\omega$},
  Ph.D. thesis, Dartmouth College, 1980.

\bibitem{DobrinenRado19}
Natasha Dobrinen, \emph{Borel sets of {R}ado graphs and {R}amsey's theorem},
  European Journal of Combinatorics, Proceedings of the 2016 Prague DocCourse
  on Ramsey Theory, 29 pp, To appear. arXiv:1904.00266v1.

\bibitem{Dobrinen_ICM}
\bysame, \emph{Ramsey theory of homogeneous structures: current trends and open
  problems}, 22 pp, arXiv:2110.00655.

\bibitem{DobrinenH_3ExactDegrees20}
\bysame, \emph{Ramsey theory of the universal homogeneous triangle-free graph,
  {P}art {II}: {E}xact big {R}amsey degrees}, 22pp, arXiv:2009.01985.

\bibitem{DobrinenRIMS17}
\bysame, \emph{Forcing in {R}amsey theory}, Proceedings of 2016 RIMS Symposium
  on Infinite Combinatorics and Forcing Theory (2017), 17--33.

\bibitem{DobrinenJML20}
\bysame, \emph{The {R}amsey theory of the universal homogeneous triangle-free
  graph}, Journal of Mathematical Logic \textbf{20} (2020), no.~2, 2050012, 75
  pp.

\bibitem{Dobrinen/Hathaway16}
Natasha Dobrinen and Daniel Hathaway, \emph{The {H}alpern-{L}{\"{a}}uchli
  {T}heorem at a measurable cardinal}, Journal of Symbolic Logic \textbf{82}
  (2017), no.~4, 1560--1575.

\bibitem{Dobrinen/Hathaway18}
\bysame, \emph{Forcing and the {H}alpern-{L}{\"{a}}uchli {T}heorem}, Journal of
  Symbolic Logic \textbf{85} (2020), no.~1, 87--102.

\bibitem{DodosKBK}
Dodos and Kanellopoulos, \emph{Ramsey theory for product spaces}, American
  Mathematical Society, 2016.

\bibitem{Dzamonja/Larson/MitchellQ09}
M.~D{\v{z}}amonja, J.~Larson, and W.~J. Mitchell, \emph{A partition theorem for
  a large dense linear order}, Israel Journal of Mathematics \textbf{171}
  (2009), 237--284.

\bibitem{Dzamonja/Larson/MitchellRado09}
\bysame, \emph{Partitions of large {R}ado graphs}, Archive for Mathematical
  Logic \textbf{48} (2009), no.~6, 579--606.

\bibitem{El-Zahar/Sauer89}
Mohamed El-Zahar and Norbert Sauer, \emph{The indivisibility of the homogeneous
  ${K}_n$-free graphs}, Journal of Combinatorial Theory, Series B \textbf{47}
  (1989), no.~2, 162--170.

\bibitem{Ellentuck74}
Erik Ellentuck, \emph{A new proof that analytic sets are {R}amsey}, Journal of
  Symbolic Logic \textbf{39} (1974), no.~1, 163--165.

\bibitem{Erdos/Hajnal/Posa75}
Paul Erd{\H{o}}s, Andr{\'{a}}s Hajnal, and Lajos P{\'{o}}sa, \emph{Strong
  embeddings of graphs into coloured graphs}, Colloquia Mathematica Societatis
  J{\'{a}}nos Bolyai, 10, vol. I, Infinite and finite sets (A.~Hajanal,
  R.~Rado, and V.~S{\'{o}}s, eds.), vol.~10, North-Holland, 1973, pp.~585--595.

\bibitem{Fouche98}
W.L. Fouch{\'{e}}, \emph{Symmetries in {R}amsey theory}, East-West Journal of
  Mathematics \textbf{1} (1998), 43--60.

\bibitem{Fraisse54}
Roland Fra{\"{i}}ss{\'{e}}, \emph{Sur l'extension aux relations de quelques
  propri{\'{e}}t{\'{e}}s des ordres}, Annales Scientifiques de l'{\'{E}}cole
  Normale Sup{\'{e}}rieure \textbf{71} (1954).

\bibitem{Galvin/Prikry73}
Fred Galvin and Karel Prikry, \emph{Borel sets and {R}amsey's {T}heorem},
  Journal of Symbolic Logic \textbf{38} (1973), no.~2, 193--198.

\bibitem{Glasner/Weiss02}
Eli Glasner and Benjamin Weiss, \emph{Minimal actions of the group
  $\mathbb{S}(\mathbb{Z})$ of permutations of the integers}, Geometric and
  Functional Analysis \textbf{12} (2002), no.~5, 964--988.

\bibitem{Graham/Leeb/Rothschild72}
R.~L. Graham, K.~Leeb, and B.~L. Rothschild, \emph{Ramsey's theorem for a class
  of categories}, Advances in Mathematics \textbf{8} (1972), 417--433.

\bibitem{Graham/Leeb/Rothschild73}
\bysame, \emph{Errata: ``{R}amsey's theorem for a class of categories},
  Advances in Mathematics \textbf{10} (1973), 326--327.

\bibitem{Graham/Rothschild71}
R.~L. Graham and B.~L. Rothschild, \emph{Ramsey's theorem for $n$-parameter
  sets}, Transactions of the American Mathematical Society \textbf{159} (1971),
  257--292.

\bibitem{Halpern/Lauchli66}
J.~D. Halpern and H.~L{\"{a}}uchli, \emph{A partition theorem}, Transactions of
  the American Mathematical Society \textbf{124} (1966), 360--367.

\bibitem{Halpern/Levy71}
J.~D. Halpern and A.~L{\'{e}}vy, \emph{The {B}oolean prime ideal theorem does
  not imply the axiom of choice}, Axiomatic Set Theory, Proc. Sympos. Pure
  Math., Vol. XIII, Part I, Univ. California, Los Angeles, Calif., 1967,
  American Mathematical Society, 1971, pp.~83--134.

\bibitem{Henson71}
C.~Ward Henson, \emph{A family of countable homogeneous graphs}, Pacific
  Journal of Mathematics \textbf{38} (1971), no.~1, 69--83.

\bibitem{Hubicka_CS20}
Jan Hubi{\v{c}}ka, \emph{Big {R}amsey degrees using parameter spaces},  (2020),
  19 pp, Preprint. arXiv:2009.00967.

\bibitem{Kechris/Pestov/Todorcevic05}
Alexander Kechris, Vladimir Pestov, and Stevo Todorcevic,
  \emph{Fra{\"{i}}ss{\'{e}} limits, {R}amsey theory, and topological dynamics
  of automorphism groups}, Geometric and Functional Analysis \textbf{15}
  (2005), no.~1, 106--189.

\bibitem{Komjath/Rodl86}
P{\'{e}}ter Komj{\'{a}}th and Vojt{\v{e}}ch R{\"{o}}dl, \emph{Coloring of
  universal graphs}, Graphs and Combinatorics \textbf{2} (1986), no.~1, 55--60.

\bibitem{Laflamme/NVT/Sauer10}
Claude Laflamme, Lionel Nguyen Van~Th{\'{e}}, and Norbert Sauer,
  \emph{Partition properties of the dense local order and a colored version of
  {M}illiken's theorem}, Combinatorica \textbf{30} (2010), no.~1, 83--104.

\bibitem{Laflamme/Sauer/Vuksanovic06}
Claude Laflamme, Norbert Sauer, and Vojkan Vuksanovic, \emph{Canonical
  partitions of universal structures}, Combinatorica \textbf{26} (2006), no.~2,
  183--205.

\bibitem{Larson08}
Jean Larson, \emph{Counting canonical partitions in the random graph},
  Combinatorica \textbf{28} (2008), no.~6, 659--678.

\bibitem{Laver84}
Richard Laver, \emph{Products of infinitely many perfect trees}, Journal of the
  London Mathematical Society (2) \textbf{29} (1984), no.~3, 385--396.

\bibitem{Masulovic18}
Dragan Ma{\v{s}}ulovi{\'{c}}, \emph{Finite big {R}amsey degrees in universal
  structures}, Journal of Combinatorial Theory, Series A \textbf{170} (2020),
  30 pp.

\bibitem{Milliken79}
Keith~R. Milliken, \emph{A {R}amsey theorem for trees}, Journal of
  Combinatorial Theory, Series A \textbf{26} (1979), 215--237.

\bibitem{Milliken81}
\bysame, \emph{A partition theorem for the infinite subtrees of a tree},
  Transactions of the American Mathematical Society \textbf{263} (1981), no.~1,
  137--148.

\bibitem{Nesetril/Rodl77}
Jaroslav Ne{\v{s}}et{\v{r}}il and Vojt{\v{e}}ch R{\"{o}}dl, \emph{Partitions of
  finite relational and set systems}, Journal of Combinatorial Theory Series A
  \textbf{22} (1977), no.~3, 289--312.

\bibitem{Nesetril/Rodl83}
\bysame, \emph{Ramsey classes of set systems}, Journal of Combinatorial Theory
  Series A \textbf{34} (1983), no.~2, 183--201.

\bibitem{NVT08}
Lionel Nguyen Van~Th{\'{e}}, \emph{Big {R}amsey degrees and divisibility in
  classes of ultrametric spaces}, Canadian Mathematical Bulletin \textbf{51}
  (2008), no.~3, 413--423.

\bibitem{NVT13}
\bysame, \emph{More on the {K}echris-{P}estov-{T}odorcevic correspondence},
  Fundamenta Mathematicae \textbf{222} (2013), no.~1, 19--47.

\bibitem{NVTHabil}
\bysame, \emph{Structural {R}amsey theory with the
  {K}echris-{P}estov-{T}odorcevic correspondence in mind}, Habilitation thesis,
  Universit{\'{e}} d'Aix-Marseille, 2013, p.~48 pp.

\bibitem{Pestov98_2}
Vladimir Pestov, \emph{On free actions, minimal flows, and a problem by ellis},
  Transactions of the American Mathematical Society \textbf{350} (1998),
  no.~10, 4149--4165.

\bibitem{Pouzet/Sauer96}
Maurice Pouzet and Norbert Sauer, \emph{Edge partitions of the {R}ado graph},
  Combinatorica \textbf{16} (1996), no.~4, 505--520.

\bibitem{Ramsey30}
Frank~P. Ramsey, \emph{On a problem of formal logic}, Proceedings of the London
  Mathematical Society \textbf{30} (1929), 264--296.

\bibitem{Sauer98}
Norbert Sauer, \emph{Edge partitions of the countable triangle free homogenous
  graph}, Discrete Mathematics \textbf{185} (1998), no.~1--3, 137--181.

\bibitem{Sauer06}
\bysame, \emph{Coloring subgraphs of the {R}ado graph}, Combinatorica
  \textbf{26} (2006), no.~2, 231--253.

\bibitem{Shelah91}
Saharon Shelah, \emph{Strong partition relations below the power set:
  consistency -- was {S}ierpinski right? {II}}, Sets, Graphs and Numbers
  (Budapest, 1991), vol.~60, Colloq. Math. Soc. J{\'{a}}nos Bolyai,
  North-Holland, 1991, pp.~637--688.

\bibitem{TodorcevicBK10}
Stevo Todorcevic, \emph{Introduction to {R}amsey {S}paces}, Princeton
  University Press, 2010.

\bibitem{Farah/TodorcevicBK}
Stevo Todorcevic and Ilijas Farah, \emph{Some {A}pplications of the {M}ethod of
  {F}orcing}, Yenisei Series in Pure and Applied Mathematics, 1995.

\bibitem{Zhang17}
Jing Zhang, \emph{A tail cone version of the {H}alpern-{L}{\"{a}}uchli theorem
  at a large cardinal}, Journal of Symbolic Logic \textbf{84} (2019), no.~2,
  473--496.

\bibitem{Zucker19}
Andy Zucker, \emph{Big {R}amsey degrees and topological dynamics}, Groups,
  Geometry and Dynamics (2018), 235--276.

\bibitem{Zucker20}
\bysame, \emph{A {N}ote on {B}ig {R}amsey degrees},  (2020), 21pp, Submitted.
  arXiv:2004.13162.

\end{thebibliography}

\end{document}